\documentclass[20pt]{book}
\usepackage{amsmath,amssymb,amsthm,graphicx,pstricks,xypic,comment,fancyhdr}
\usepackage[all]{xy}
\usepackage[english]{babel}
\graphicspath{{figthese/}}
\title{\textbf{Th{\'e}orie homotopique des DG-cat{\'e}gories}} 
\author{Gon{\c c}alo Tabuada}
\date{}
\begin{document}

\titlepage{
\begin{center}\Large{\textbf{THESE DE L'UNIVERSIT{\'E} PARIS DIDEROT - PARIS 7}}\end{center}
\begin{center}\textbf{MATH{\'E}MATIQUES}\end{center}
\begin{center}\textbf{E.D. SCIENCES MATH{\'E}MATIQUES DE PARIS - CENTRE}\end{center}
\vspace{2cm}
\begin{center}{\small{\textbf{Gon{\c c}alo Nery Tabuada}}}\end{center}
\vspace{3cm}
\begin{center}\huge{\textbf{Th{\'e}orie homotopique des DG-cat{\'e}gories}}\end{center}
\vspace{4cm}
\begin{center}
\texttt{Soutenue le 20 septembre 2007, devant le jury compos{\'e} de~:}
\end{center}
$$
\begin{array}{lcr}
\mbox{\textbf{M. Pierre Cartier (I.H.{\'E}.S.)}} & & \mbox{\texttt{Pr{\'e}sident}}\\
\mbox{\textbf{M. Denis-Charles Cisinski (Universit{\'e} Paris 13)}} & & \\
\mbox{\textbf{M. Bernhard Keller (Universit{\'e} Paris 7)}} &&  \texttt{Directeur}\\
\mbox{\textbf{M. Maxim Kontsevich (I.H.{\'E}.S.)}} && \\
\mbox{\textbf{M. Rapha{\"e}l Rouquier (University of Oxford)}} && \\
\mbox{\textbf{M. Bertrand To{\"e}n (C.N.R.S.-Laboratoire Emile Picard)}}
&& \texttt{Rapporteur} \\
\end{array}
$$

\vspace{0.4cm}
\begin{center}
\texttt{Rapporteur externe (absent {\`a} la soutenance)}
\end{center}
\begin{center}
\textbf{M. Mikhail Kapranov (Yale University)}
\end{center}
}

\def\og{\leavevmode\raise.3ex\hbox{$\scriptscriptstyle\langle\!\langle$~}}
\def\fg{\leavevmode\raise.3ex\hbox{~$\!\scriptscriptstyle\,\rangle\!\rangle$}}
\newcommand{\ul}{ \rule[0.2mm]{0.12mm}{2.2mm} \rule[2.3mm]{2.2mm}{0.12mm}}
\newcommand{\dgcat}{\mathsf{dgcat}}
\newcommand{\kar}{\rule[-0.1mm]{0.12mm}{1.5mm}\hspace{-0.36mm} \ni}
\newcommand{\ko}{\: , \;}
\newcommand{\la}{\langle}
\newcommand{\ra}{\rangle}
\newcommand{\HdC}{\bold H^{2}_{\mathbb C}}
\newcommand{\HdR}{\bold H^{2}_{\mathbb R}}
\newcommand{\HuC}{\bold H^{1}_{\mathbb C}}
\newcommand{\HnC}{\bold H^{n}_{\mathbb C}}
\newcommand{\Cdu}{\mathbb C^{2,1}}
\newcommand{\Ct}{\mathbb C^{3}}
\newcommand{\C}{\mathbb C}
\newcommand{\Mn}{\mbox{M}_n\left(\mathbb C\right)}
\newcommand{\R}{\mathbb R}
\newcommand{\A}{\mathbb A}
\newcommand{\cC}{\mathfrak C}
\newcommand{\cCp}{\mathfrak C_{\mbox{p}}}
\newcommand{\cT}{\mathfrak T}
\newcommand{\rR}{\mathfrak R}
\newcommand{\hH}{\mathfrak H}
\newcommand{\Td}{T^{2}_{\left(1,1\right)}}
\newcommand{\T}{T_{\left(1,1\right)}}
\newcommand{\PSL}{\mbox{\rm PSL(2,$\R$)}}
\newcommand{\Pu}{\mbox{\rm PU(2,1)}}
\newcommand{\X}{\mbox{\textbf{X}}}
\newcommand{\tr}{\mbox{\rm tr}}
\newcommand{\M}{{\rm M}}
\newcommand{\N}{{\rm N}}
\newcommand{\B}{\mathcal{B}}
\newcommand{\n}{\noindent}
\newcommand{\p}{{\bf p}}
\newcommand{\q}{{\bf q}}
\renewcommand{\Re}{\mbox{\rm Re}\,}
\renewcommand{\Im}{\mbox{\rm Im}\,}
\renewcommand{\leq}{\leqslant}
\renewcommand{\geq}{\geqslant}

\newtheorem{theorem}{Theorem}[chapter]
\newtheorem{theoreme}{Th{\'e}or{\`e}me}[chapter]
\newtheorem{theo}{Th{\'e}or{\`e}me}[chapter]
\newtheorem{corollaire}[theoreme]{Corollaire}
\newtheorem{lemma}[theorem]{Lemma}
\newtheorem{lemme}[theoreme]{Lemme}
\newtheorem{corollary}[theorem]{Corollary}
\newtheorem{proposition}[theorem]{Proposition}
\newtheorem{propositionf}[theoreme]{Proposition}
\newtheorem{prop}[theoreme]{Proposition}
\theoremstyle{remark}
\newtheorem{terminology}[theorem]{Terminology}
\newtheorem{notation}[theoreme]{Notation}
\newtheorem{Notation}[theorem]{Notation}
\newtheorem{remark}[theorem]{Remark}
\newtheorem{remarque}[theoreme]{Remarque}
\newtheorem{example}[theorem]{Example}
\newtheorem{exemple}[theoreme]{Exemple}
\newtheorem{nonexample}[theorem]{Nonexample}
\newtheorem{definitionf}[theoreme]{D{\'e}finition}
\newtheorem*{theo*}{\bf Theorem}
\newenvironment{modif}[1]{\bf (MODIF #1):}{(end MODIF)\rm}

\theoremstyle{definition}
\newtheorem{definition}[theorem]{Definition}
\newtheorem{defi}[theoreme]{Definition}
\newenvironment{proof1}{\noindent{\sc D{\'e}monstration.}}{$\,\Box$\smallskip}

\newcommand{\Ho}{\mathsf{Ho}}
\newcommand{\Map}{\mathsf{Map}}
\newcommand{\rep}{\mathsf{rep}}
\newcommand{\idem}{\mathrm{idem}}
\newcommand{\idemh}{\mathrm{idemh}}
\newcommand{\fact}{\mathrm{fact}}
\newcommand{\facth}{\mathrm{facth}}
\newcommand{\ca}{{\mathcal A}}
\newcommand{\cl}{{\mathcal L}}
\newcommand{\cm}{{\mathcal M}}
\newcommand{\cb}{{\mathcal B}}
\newcommand{\cc}{{\mathcal C}}
\newcommand{\cd}{{\mathcal D}}
\newcommand{\cp}{{\mathcal P}}
\newcommand{\cu}{{\mathcal U}}
\newcommand{\cv}{{\mathcal V}}
\newcommand{\cw}{{\mathcal W}}
\newcommand{\ten}{\otimes}
\newcommand{\iso}{\stackrel{_\sim}{\rightarrow}}
\newcommand{\id}{\mathbf{1}}

\chapter*{Remerciements\markboth{Remerciements}{}}

Je suis tr{\`e}s heureux de pouvoir exprimer ici toute ma gratitude envers Bernhard Keller qui a dirig{\'e} cette th{\`e}se et guid{\'e} mes premiers pas dans la
recherche. Il est impossible de r{\'e}sumer en quelques
phrases tout ce que je lui dois. Il m'a soutenu d{\`e}s le premier
jour et m'a enseign{\'e} le m{\'e}tier avec beaucoup de d{\'e}vouement.

\smallskip
{\em`Merci'}
\smallskip

Mes remerciements les plus chaleureux vont aussi {\`a} Denis-Charles
Cisinski et Bertrand To{\"e}n. Ils m'ont fait profiter de leur immense
richesse math{\'e}matique et cela m'a ouvert de nouveaux horizons de
recherche. Je suis tr{\`e}s heureux qu'ils fassent partie du jury.

\smallskip

Je tiens ensuite {\`a} exprimer ma gratitude {\`a} Mikhail Kapranov et
Bertrand To{\"e}n, qui m'ont fait
l'honneur de rapporter cette th{\`e}se. Je remercie Mikhail Kapranov pour
sa lecture et pour avoir fait le lien entre cette th{\`e}se et son
travail fondateur avec Alexei Bondal. Je remercie profond{\'e}ment
Bertrand To{\"e}n pour toutes
ses corrections et suggestions qui ont permis d'am{\'e}liorer {\'e}norm{\'e}ment la
qualit{\'e} scientifique de ce texte.

\smallskip
L'un des objets math{\'e}matiques au centre de cette th{\`e}se est une dg-cat{\'e}gorie
$\mathcal{K}$ due {\`a} Maxim~\mbox{Kontsevich}. Je lui suis
tr{\`e}s reconnaissant d'avoir accept{\'e}, si promptement, de si{\'e}ger dans le
jury.

\smallskip

Messieurs Pierre Cartier et Rapha{\"e}l Rouquier me font un tr{\`e}s grand
honneur de si{\'e}ger dans le jury. Leurs expos{\'e}s au cours de ces
derni{\`e}res ann{\'e}es m'ont apport{\'e} beaucoup.

\smallskip

La source principale d'inspiration pour cette th{\`e}se a {\'e}t{\'e} l'article
{\em DG quotients of DG categories} par Vladimir Drinfeld. J'esp{\`e}re que cette th{\`e}se lui fera plaisir.

\smallskip

Pendant la pr{\'e}paration de la th{\`e}se j'ai profit{\'e} de discussions
{\'e}lectroniques aussi bien que de vive voix avec plusiers math{\'e}maticiens~: Joseph Ayoub, Michael
Batanin, Clemens Berger, Rick Jardine, Francesco Lemma, Jean-Louis
Loday, Orlando Neto, Stefan Schwede et Mark Weber. Je les remercie
tous, ainsi que ceux que j'ai oubli{\'e}s de citer.

\smallskip

Je voudrait dire {\em`Obrigado'} {\`a} Gustavo Granja, qui depuis Lisbonne,
m'a encourag{\'e} et a r{\'e}pondu {\`a} mes questions na{\"\i}ves sur la th{\'e}orie d'homotopie
(stable).

\smallskip

D{\`e}s mon arriv{\'e}e {\`a} Paris, j'ai eu la chance de profiter des cours et
groupes de travail organis{\'e}s par Jean Barge et Fabien Morel. Je les
remercie tous les deux pour l'influence qu'ils ont eue dans mes choix
math{\'e}matiques.

\smallskip

L'excellente ambiance r{\'e}gnant sur le plateau des doctorants a {\'e}t{\'e}
capitale pendant ces ann{\'e}es. Je tiens donc {\`a} exprimer toute mon amiti{\'e}
{\`a} tous ceux que j'ai pu rencontrer.

\smallskip

Je voudrait remercier Monique Douchez et Mich{\`e}le Wasse pour toute leur
aide et efficacit{\'e} dans la partie administrative ainsi que la {\em Funda{\c c}{\~a}o
para a Ci{\^e}ncia e Tecnologia - Portugal} pour m'avoir
soutenu financi{\`e}rement avec la bourse SFRH/BD/144035/2003.

\smallskip

Pour finir, je remercie profond{\'e}ment mes parents, mes oncles et Liliana
pour leur amour, ainsi que mon fr{\`e}re Paulo Tabuada pour m'avoir
enseign{\'e} {\`a} ne pas avoir peur {\em`des math{\'e}matiques, du futur, de~la~vie.'}

\tableofcontents
\addcontentsline{toc}{chapter}{Introduction}

\chapter*{Introduction\markboth{Introduction}{}}

Nous renvoyons {\`a} l'expos{\'e} de Keller~\cite{ICM} pour une pr{\'e}sentation et motivation de la
plupart des concepts math{\'e}matiques qui interviennent dans cette
th{\`e}se. Nous proposons plusieurs contributions originales {\`a}
l'{\'e}tude~:
\begin{itemize}
\item[-] des dg-cat{\'e}gories et de leurs invariants,
\item[-] des cat{\'e}gories triangul{\'e}es (alg{\'e}briques) bien engendr{\'e}es au sens
  de Neeman et
\item[-] de l'approche des alg{\`e}bres `cluster' au
  sens de Fomin-Zelevinsky par la th{\'e}orie des repr{\'e}sentations. 
\end{itemize}
Cette th{\`e}se correspond aux articles \cite{cras}
\cite{addendum} \cite{IMRN} \cite{Documenta} \cite{dgquot}
\cite{Triangcat} \cite{univ} et {\`a} un appendice, o{\`u} une preuve simple et purement
homotopique d'un th{\'e}or{\`e}me
d{\^u} {\`a} Drinfeld est present{\'e}e.
Dans le r{\'e}sum{\'e} suivant, nous pr{\'e}sentons les r{\'e}sultats principaux de
cette th{\`e}se d'une fa{\c c}on plus d{\'e}taill{\'e}e.

\subsection*{DG-cat{\'e}gories}

Les cat{\'e}gories diff{\'e}rentielles gradu{\'e}es (=dg-cat{\'e}gories)
`enrichissent' notre comprehension des cat{\'e}gories triangul{\'e}es qui
apparaissent naturellement en alg{\`e}bre et g{\'e}om{\'e}trie.

L'id{\'e}e d'utiliser les dg-cat{\'e}gories pour `enrichir' les cat{\'e}gories
triangul{\'e}es remonte aux travaux de Bondal-Kapranov~\cite{Bon-Kap}. Leur motivation
principale {\'e}tait l'{\'e}tude des collections exceptionnelles de faisceaux
coh{\'e}rents sur les vari{\'e}t{\'e}s projectives. Elles ont {\'e}t{\'e} utilis{\'e}es aussi par
Keller~\cite{DerivingDG} dans l'{\'e}tude de la th{\'e}orie de Morita d{\'e}riv{\'e}e
et de la  dualit{\'e} de Koszul.

Actuellement, les dg-cat{\'e}gories sont consid{\'e}r{\'e}es comme des sch{\'e}mas
non-commutatifs par Drinfeld~\cite{Chicagotalk}~\cite{Drinfeld} et
Kontsevich~\cite{ENS}~\cite{IHP} dans leur programme de g{\'e}ometrie alg{\'e}brique non-commutative.

L'une des op{\'e}rations importantes qu'on peut r{\'e}aliser dans les cat{\'e}gories
triangul{\'e}es est le passage au quotient. Cette op{\'e}ration de quotient {\`a}
{\'e}t{\'e} relev{\'e}e au monde de dg-cat{\'e}gories par Keller dans
\cite{DerivingDG} et r{\'e}cemment par Drinfeld dans \cite{Drinfeld}.

\subsection*{Chapitre 1}

Dans le but de r{\'e}interpreter la construction du
dg-quotient de Drinfeld d'un point de vue purement homotopique afin
de mieux comprendre sa propri{\'e}t{\'e} universelle, on construit une structure de cat{\'e}gorie de mod{\`e}les de Quillen sur la
cat{\'e}gorie des petites cat{\'e}gories diff{\'e}rentielles gradu{\'e}es $\dgcat$. 
Rappelons qu'un dg-foncteur $F:\mathcal{C} \rightarrow \mathcal{D}$
est une {\it quasi-equivalence} si :
\begin{itemize}
\item[-] pour tous objets $c_1$ et $c_2$ dans
$\mathcal{C}$, le morphisme de complexes de $\mathrm{Hom}_{\mathcal{C}}(c_1,
c_2)$ vers $\mathrm{Hom}_{\mathcal{D}}(F(c_1), F(c_2))$ est un
quasi-isomorphisme et
\item[-] le foncteur $\mathrm{H}^0(F)$ de
$\mathrm{H}^0(\mathcal{C})$ vers $\mathrm{H}^0(\mathcal{D})$ est
essentiellement surjectif. 
\end{itemize} 

\begin{theoreme}[\ref{mal}]\label{thm1}
La cat{\'e}gorie $\dgcat$ admet une structure de cat{\'e}gorie de mod{\`e}les
de Quillen {\`a} engendrement cofibrant dont les {\'e}quivalences faibles sont
les quasi-equivalences. Les fibrations sont les dg-foncteurs $F:
\mathcal{A} \rightarrow \mathcal{B}$ qui induisent des surjections de
complexes 
$\mathsf{Hom}_{\mathcal{A}}(X,Y)\rightarrow \mathsf{Hom}_{\mathcal{B}}(F(X),F(Y))$
pour tous $X$, $Y$ dans $\mathcal{A}$ et tels que pour tout objet $X
\in \mathcal{A}$ et tout morphisme $v \in
\mathsf{Hom}_{\mathcal{B}}(F(X),Z)$ qui devient un isomorphisme dans
$\mathsf{H}^0(\mathcal{B})$, il existe un morphisme $u \in
\mathsf{Hom}_{\mathcal{A}}(X,Y)$ tel que $F(u)=v$ et qui devient un isomorphisme dans
$\mathsf{H}^0(\mathcal{A})$.
\end{theoreme} 

\begin{remarque}
Notre construction est inspir{\'e}e par des arguments dans \cite{Rezk} et par la construction
du dg-quotient donn{\'e}e par Drinfeld dans \cite{Drinfeld}.
La clef pour cette construction est une certaine dg-cat{\'e}gorie $\mathcal{K}$
d{\'e}finie par Drinfeld dans \cite[3.7.1]{Drinfeld} et qui est due {\`a} Kontsevich.
Conceptuellement, elle joue le m{\^e}me role dans $\dgcat$ que
l'intervalle dans la cat{\'e}gorie des espaces topologiques. 
\end{remarque}

\begin{remarque}
Dans l'appendice, on donne une preuve simple et purement
homotopique de la propri{\'e}te universelle du dg-quotient de Drinfeld \cite{Drinfeld}, qui est bas{\'e}e seulement sur le th{\'e}or{\`e}me~\ref{thm1}. 
\end{remarque}

Rappelons maintenant quelques r{\'e}sultats du travail fondamental de
To{\"e}n~\cite{Toen}, rendus possibles par le th{\'e}or{\`e}me~\ref{thm1}.

On note $\mathsf{Heq}$ la localisation de $\dgcat$ par rapport {\`a}
la classe des quasi-{\'e}quivalences.
Les morphismes dans la localisation (de Dwyer-Kan \cite{Dwyer}) sont d{\'e}crits de la
fa{\c c}on suivante: soient $\mathcal{A}$ et $\mathcal{B}$ deux dg-cat{\'e}gories. Si necessaire,
on peut remplacer $\mathcal{A}$ par une dg-cat{\'e}gorie quasi-{\'e}quivalente de
fa{\c c}on que $\mathcal{A}$ soit $k$-plate, c'est-{\`a}-dire le foncteur
$\mathcal{A}(X,Y)\otimes ?$ pr{\'e}serve les quasi-isomorphismes pour tous
$X$,$Y$ dans $\mathcal{A}$ (par exemple, on peut prendre une
r{\'e}solution cofibrante de $\mathcal{A}$).
Soit $\mathsf{rep}(\mathcal{A},\mathcal{B})$ la sous cat{\'e}gorie pleine
de la cat{\'e}gorie d{\'e}riv{\'e}e $\mathcal{D}(\mathcal{A}^{op} \otimes
\mathcal{B})$ des $\mathcal{A}$-$\mathcal{B}$-bimodules form{\'e}e des
bimodules $X$ tels que le foncteur produit tensoriel d{\'e}riv{\'e}
$$ ?\overset{\mathbb{L}}{\otimes_{\mathcal{A}}}X :
\mathcal{D}(\mathcal{A}) \rightarrow \mathcal{D}(\mathcal{B}) $$
envoie les $\mathcal{A}$-modules repr{\'e}sentables vers des objets qui
sont isomorphes dans $\mathcal{D}(\mathcal{B})$ {\`a} des $\mathcal{B}$-modules repr{\'e}sentables. D'une
fa{\c c}on {\'e}quivalente, on demande que $X(?,A)$ soit isomorphe dans
$\mathcal{D}(\mathcal{B})$ {\`a} un
$\mathcal{B}$-module repr{\'e}sentable pour tout objet $A$ dans
$\mathcal{A}$.

\begin{theoreme}[\cite{Toen}]
Les morphismes de $\mathcal{A}$ vers $\mathcal{B}$ dans
$\mathsf{Heq}$ sont en bijection naturelle avec les classes
d'isomorphisme de la cat{\'e}gorie $\mathsf{rep}(\mathcal{A},\mathcal{B})$.
\end{theoreme}

Maintenant, soit $\mathcal{R}(\mathcal{A},\mathcal{B})$ la cat{\'e}gorie
qui a les m{\^e}mes objets que $\mathsf{rep}(\mathcal{A},\mathcal{B})$ et
dont les morphismes sont les quasi-isomorphismes de dg-bimodules. La
  cat{\'e}gorie $\mathcal{R}(\mathcal{A},\mathcal{B})$ est donc une sous
  categorie non pleine de la cat{\'e}gorie des dg-bimodules
  $\mathcal{C}(\mathcal{A}^{op}\otimes\mathcal{B})$.

\begin{theoreme}[\cite{Toen}]
Il existe une {\'e}quivalence faible canonique d'ensembles simpliciaux
entre \newline
$\mathsf{Map}(\mathcal{A},\mathcal{B})$ et le nerf de la
cat{\'e}gorie $\mathcal{R}(\mathcal{A},\mathcal{B})$.
\end{theoreme}

\begin{theoreme}[\cite{Toen}]\label{thmToen}
La cat{\'e}gorie mono{\"\i}dale
$(\mathsf{Heq},- \overset{\mathbb{L}}{\otimes} -)$ admet un foncteur
$\mbox{Hom}$-interne \newline 
$\mathcal{R}\mathsf{Hom}(-,-)$. Pour deux
dg-cat{\'e}gories $\mathcal{A}$ et $\mathcal{B}$, telles que $\mathcal{A}$
est $k$-plate, la dg-cat{\'e}gorie
$\mathcal{R}\mathsf{Hom}(\mathcal{A},\mathcal{B})$ est isomorphe dans
$\mathsf{Heq}$ {\`a} la dg-cat{\'e}gorie
$\mathsf{rep}_{dg}(\mathcal{A},\mathcal{B})$, c'est-{\`a}-dire la sous
cat{\'e}gorie pleine de la dg-cat{\'e}gorie des
$\mathcal{A}$-$\mathcal{B}$-bimodules, dont les objets sont ceux de
$\mathsf{rep}(\mathcal{A},\mathcal{B})$ qui sont en plus cofibrants
comme bimodules. 
\end{theoreme}

\subsection*{Chapitre 2}

On remarque que tous les invariants fonctoriels classiques comme la $K$-th{\'e}orie alg{\'e}brique, l'homologie
de Hochschild, l'homologie cyclique, $\ldots$ se prolongent
naturellement des $k$-alg{\`e}bres vers les dg-cat{\'e}gories. D'une fa{\c c}on
analogue au cas des $k$-alg{\`e}bres ordinaires, ces invariants sont
pr{\'e}serv{\'e}s par les {\'e}quivalences de Morita d{\'e}riv{\'e}es.
Cela conduit au probl{\`e}me de donner une description explicite de la
`cat{\'e}gorie d'homotopie de Morita', c'est-{\`a}-dire la localisation de 
$\dgcat$ par rapport aux {\'e}quivalences de Morita d{\'e}riv{\'e}es, puisque tous ces
invariants descendent {\`a} cette cat{\'e}gorie.

On r{\'e}sout ce probl{\`e}me en utilisant le formalisme d'alg{\`e}bre
homotopique de Quillen. En effet, on
construit une structure de cat{\'e}gorie de mod{\`e}les de Quillen sur
$\dgcat$, dont les {\'e}quivalences faibles sont les dg-foncteurs {\it
  de Morita}, c'est-{\`a}-dire les dg-foncteurs $F:\mathcal{A} \rightarrow
\mathcal{B}$ qui induisent une {\'e}quivalence $\mathcal{D}(\mathcal{B})
\stackrel{\sim}{\rightarrow} \mathcal{D}(\mathcal{A})$ entre
cat{\'e}gories d{\'e}riv{\'e}es.

\begin{theoreme}[\ref{theorem2}]\label{thm2}
La cat{\'e}gorie $\dgcat$ admet une structure de cat{\'e}gorie de mod{\`e}les de
Quillen {\`a} engendrement cofibrant, dont les equivalences faibles sont
les dg-foncteurs de Morita et dont les cofibrations sont les m{\^e}mes que
celles du th{\'e}or{\`e}me~\ref{thm1}.
\end{theoreme}

\begin{remarque}
Notre structure a {\'e}t{\'e} construite {\`a} partir de celle du
th{\'e}or{\`e}me~\ref{thm1} en deux {\'e}tapes. (On remarque qu'une quasi-equivalence
est un dg-functor de Morita).

Premi{\`e}rement, on a construit une structure de cat{\'e}gorie de mod{\`e}les de Quillen
intermediaire, dont les {\'e}quivalences faibles sont les dg-foncteurs
quasi-{\'e}quiconiques, voir section~\ref{secquasicon}. Dans cette
{\'e}tape, on s'est inspir{\'e} de la construction de Bondal-Kapranov de l'envelope
pre-triangul{\'e}e d'une dg-cat{\'e}gorie, voir \cite{Bon-Kap}.

Dans une deuxi{\`e}me {\'e}tape, on a relev{\'e} la construction de la
Karoubianisation d'une cat{\'e}gorie triangul{\'e}e, voir~\cite{Balmer}, au
monde des dg-cat{\'e}gories en utilisant des facteurs directs {\`a} homotopie pr{\`e}s. Cela nous a permis de contr{\^o}ler les objets compacts de la cat{\'e}gorie d{\'e}riv{\'e}e d'une dg-cat{\'e}gorie. 
\end{remarque}

On d{\'e}signe par $\mathsf{Hmo}$ la localisation de $\dgcat$ par
rapport aux dg-foncteurs de Morita. Dans $\mathsf{Hmo}$, les {\'e}quivalences
d{\'e}riv{\'e}es au sens de Rickard-Keller \cite{Rickard}
\cite{Rickard1} \cite{DerivingDG} correspondent {\`a} des isomorphismes et
le `groupe de Picard d{\'e}riv{\'e}' \cite{Rou-Zim} au sens de Rouquier-Zimmermann y appara{\^\i}t comme un groupe
d'automorphismes.

\begin{prop}[\ref{nova4}]
Une dg-cat{\'e}gorie $\mathcal{A}$ est Morita fibrante si est non vide et l'image essentielle du plongement
$\mathsf{H}^0(\mathcal{A}) \hookrightarrow \mathcal{D}(\mathcal{A})$
est stable par suspensions, c{\^o}nes et facteurs directs.
\end{prop}

\begin{prop}[\ref{Bousf}, \ref{Bousfil}]
La structure de cat{\'e}gorie de mod{\`e}les de Quillen du th{\'e}or{\`e}me~\ref{thm2}
est une localisation de Bousfield {\`a} gauche de la structure de
cat{\'e}gorie de mod{\`e}les de Quillen du th{\'e}or{\`e}me~\ref{thm1}.
\end{prop}

On note $\mathcal{A} \mapsto \mathcal{A}_{fib}$ une resolution Morita fibrante
fonctorielle de $\mathcal{A}$. 

\begin{corollaire}[\ref{adjonct2}, \ref{adjonct5}]\label{prop1}
Le foncteur $\mathcal{A} \mapsto \mathcal{A}_{fib}$ nous
fournit un adjoint {\`a} droite du foncteur quotient $\mathsf{Heq}
\rightarrow \mathsf{Hmo}$ et induit une equivalence entre
$\mathsf{Hmo}$ et la sous cat{\'e}gorie pleine de $\mathsf{Heq}$ form{\'e}e des
dg-cat{\'e}gories Morita fibrantes.
\end{corollaire} 

On d{\'e}signe par $\mathsf{rep}_{mor}(\mathcal{A},\mathcal{B})$,
resp. $\mathcal{R}_{mor}(\mathcal{A},\mathcal{B})$,
resp. $\mathcal{R}\mathsf{Hom}_{mor}(\mathcal{A},\mathcal{B})$ les
cat{\'e}gories  $\mathsf{rep}(\mathcal{A},\mathcal{B}_{fib})$,
resp. $\mathcal{R}(\mathcal{A},\mathcal{B}_{fib})$,
resp.
$\mathcal{R}\mathsf{Hom}(\mathcal{A},\mathcal{B}_{fib})$.

La corollaire~\ref{prop1} nous permet d'{\'e}tendre les r{\'e}sultats de
To{\"e}n au cadre des dg-foncteurs de Morita. Soient $\mathcal{A}$ et
$\mathcal{B}$ deux dg-cat{\'e}gories.

\begin{corollaire}[\ref{adjonct6}]\label{lemmeGon}

\begin{itemize}
\item[-] Les morphismes de $\mathcal{A}$ vers $\mathcal{B}$ dans
  $\mathsf{Hmo}$ sont en bijection naturelle avec les classes
  d'isomorphime de la cat{\'e}gorie $\mathsf{rep}_{mor}(\mathcal{A},\mathcal{B})$.
\item[-] On dispose d'une {\'e}quivalence faible canonique d'ensembles simpliciaux entre $\mathsf{Map}(\mathcal{A},\mathcal{B})$ et le nerf
  de la cat{\'e}gorie $\mathcal{R}_{mor}(\mathcal{A},\mathcal{B})$.
\item[-] La cat{\'e}gorie mono{\"\i}dale sym{\'e}trique $(\mathsf{Hmo},
  - \overset{\mathbb{L}}{\otimes} -)$ admet un foncteur
    $\mathsf{Hom}$-interne $\mathcal{R}\mathsf{Hom}_{mor}(-,-)$.
\end{itemize}
\end{corollaire}

Rappelons que tous les invariants habituels comme la $K$-th{\'e}orie alg{\'e}brique,
l'homologie de Hochschild, l'homologie cyclique, $\ldots$ descendent {\`a} la cat{\'e}gorie
$\mathsf{Hmo}$. Ces invariants d{\'e}pendent en fait de moins
de structure, voir l'exemple~\ref{exemple1}. Cela motive la
construction suivante.
\par
Soit $\mathsf{Hmo}_0$ la cat{\'e}gorie qui a pour objets les petites
dg-cat{\'e}gories et telle que
$\mathsf{Hom}_{\mathsf{Hmo}_0}(\mathcal{A},\mathcal{B})$ est le groupe
de Grothendieck de la cat{\'e}gorie triangul{\'e}e
$\mathsf{rep}_{mor}(\mathcal{A},\mathcal{B})$. La composition est
induite par celle de $\mathsf{Hmo}$.
On dispose d'un foncteur canonique $\mathsf{Hmo} \rightarrow
\mathsf{Hmo}_0$. 

\begin{lemme}[\ref{ad}]
La cat{\'e}gorie $\mathsf{Hmo}_0$ est additive et le produit tensoriel
$-\otimes^{\mathbb{L}}-$ de $\mathsf{Hmo}$ induit une structure mono{\"\i}dale sym{\'e}trique sur $\mathsf{Hmo}_0$.
\end{lemme} 

Soit maintenant $F:\mathsf{Hmo}\rightarrow \mathsf{C}$ un foncteur {\`a} valeurs
dans une cat{\'e}gorie additive $\mathsf{C}$.

\begin{theoreme}[\ref{Univ}]\label{universal}
Les conditions suivantes sont {\'e}quivalentes~:
\begin{itemize}
\item[1)] Le foncteur $F$ est compos{\'e} d'un foncteur additif
  $\mathsf{Hmo}_0 \rightarrow \mathsf{C}$ et du foncteur canonique
  $\mathsf{Hmo}\rightarrow \mathsf{Hmo}_0$.
\item[2)] Pour toutes dg-cat{\'e}gories $\mathcal{A}$, $\mathcal{B}$, l'identit{\'e}
  $F(\left[X\right])+F(\left[Z\right])=F(\left[Y\right])$, est v{\'e}rifi{\'e}e
  dans $\mathrm{Hom}_{\mathsf{C}}(F(\mathcal{A}), F(\mathcal{B}))$ pour
  tout triangle $X \rightarrow Y \rightarrow Z \rightarrow X\left[1\right]$
  de $\mathsf{rep}_{mor}(\mathcal{A}, \mathcal{B})$.
\item[3)] Pour toute dg-cat{\'e}gorie $\mathcal{A}$, le morphisme 
$$
\xymatrix{
F(\mathcal{A})\oplus F(\mathcal{A})
\ar[rr]^-{\left[F(i_1)
        \, , \,F(i_2)\right]} && F(\mbox{T}(\mathcal{A}))
}
$$
est un isomorphisme dans $\mathsf{C}$.

\item[4)] Pout toute dg-cat{\'e}gorie pr{\'e}triangul{\'e}e $\mathcal{A}$ munie de sous-dg-cat{\'e}gories pleines pr{\'e}triangul{\'e}es
  $\mathcal{B}$ et $\mathcal{C}$ qui donnent lieu {\`a} une d{\'e}composition
  semi-orthogonale
  $\mathrm{H}^0(\mathcal{A})=(\mathrm{H}^0(\mathcal{B}),
  \mathrm{H}^0(\mathcal{C}))$, voir
  \cite{Bondal}, le morphisme
$$
F(\mathcal{B})\oplus F(\mathcal{C}) \rightarrow F(\mathcal{A})
$$
induit par les inclusions est un isomorphisme dans $\mathsf{C}$.
\end{itemize}
\end{theoreme}

\begin{remarque}
Toute {\'e}quivalence d{\'e}riv{\'e}e, voir \cite{Rickard}, donne un isomorphisme dans $\mathsf{Hmo}$ et donc dans
$\mathsf{Hmo}_0$. Cependant, il existe d'autres isomorphismes dans
$\mathsf{Hmo}_0$~: si $k$ est un corps alg{\'e}briquement clos et $A$ une $k$-alg{\`e}bre (ordinaire) de dimension finie sur $k$ et de dimension
globale finie, alors $A$ est isomorphe {\`a} son plus grand quotient
semisimple $A/\mathrm{rad}(A)$ dans $\mathsf{Hmo}_0$ (voir
\cite[2.5]{cyclicDG}) mais dans $\mathsf{Hmo}$, $A$ est isomorphe {\`a}
$A/\mathrm{rad}(A)$ seulement si $A$ est semisimple.
\end{remarque}

On appelle un foncteur $F:\mathsf{Hmo} \rightarrow \mathsf{C}$ qui
satisfait aux conditions du th{\'e}or{\`e}me~\ref{universal} un {\em invariant additif}.

\begin{exemple}\label{exemple1}
\begin{itemize}
\item[-] {\bf $K$-th{\'e}orie} $\,$ Soit $\mathcal{A}$ une petite dg-cat{\'e}gorie. On
  definit la $K$-th{\'e}orie de $\mathcal{A}$ comme la $K$-th{\'e}orie de
  Waldhausen associ{\'e} {\`a} la cat{\'e}gorie des dg-modules cofibrants et
  parfaits dont les cofibrations sont les monomorphismes (scind{\'e}es
  dans la cat{\'e}gorie gradu{\'e}e) et les
  {\'e}quivalences faibles les quasi-isomorphismes. Cette
  construction nous fournit un foncteur $K$ de $\dgcat$ vers la
  cat{\'e}gorie homotopique des spectres (sym{\'e}triques). Par des r{\'e}sultats
  dus {\`a} Waldhausen \cite{Waldhausen} le foncteur $K$ est un invariant
  additif.

\item[-] {\bf Homologie de Hochschild, cyclique, $\ldots$} Soit $\mathcal{A}$ une
  petite dg-cat{\'e}gorie (qu'on peut supposer cofibrante). On associe {\`a}
  $\mathcal{A}$ son complexe mixte $C(\mathcal{A})$ au sens de
  Kassel \cite{Kassel}, voir aussi \cite{cyclichomology}. Cela nous fournit un foncteur $C$ de
  $\dgcat$ vers $\mathcal{D}(\Lambda)$, o{\`u} $\Lambda$ est la dg-alg{\`e}bre
  $k[B]/(B^2)$, avec $B$ de degr{\'e} $-1$ et $\partial B=0$. Toutes les
  variantes de l'homologie cyclique d{\'e}pendent seulement de
  $C(\mathcal{A})$ consid{\'e}r{\'e} comme un object dans
  $\mathcal{D}(\Lambda)$. Par des r{\'e}sultats dus {\`a} Keller
  \cite{cyclicDG} \cite{cyclichomology}, le foncteur $C$ est un invariant additif. 
\end{itemize}
\end{exemple}

Maintenant on remarque que les th{\'e}or{\`e}mes~\ref{thm2} et \ref{universal} nous
fournissent une construction explicite de `l'invariant additif
universel', c'est-{\`a}-dire un foncteur
$$
\mathcal{U}_a:\dgcat \rightarrow \mathsf{Hmo}_0
$$
{\`a} valeurs dans une cat{\'e}gorie additive qui rend inversibles les
dg-foncteurs de Morita, transforme les d{\'e}compositions
semi-orthogonales au sens de Bondal-Orlov \cite{Bon-Orl} en sommes directes et qui est
universel pour ces propri{\'e}t{\'e}s.

Cette construction a {\'e}t{\'e} reli{\'e}e r{\'e}cemment {\`a} la th{\'e}orie des
motifs par Kontsevich dans \cite[4]{Konnew}~:
\\

{\it Triangulated categories of the type $D(X)$ where $X$ is a smooth
  projective variety over a field {\bf k} belong to a larger class of smooth
  proper triangulated {\bf k}-linear dg-categories (another name is
  `saturated categories'), see \cite{KonY}, \cite{Toen-Vaq}. We see
  that the above quotient category of pure motives is a full
  subcategory of $\mbox{K}_0$-decategorification (with $\mathbb{Q}$
  coefficients) of the $2$-category of smooth proper {\bf k}-linear
  dg-categories. This construction was descibed recently (without
  mentioning the relation to motives) in \cite{IMRN}.

Analogously, if one takes the quotient of the Voevodsky triangulated
category of mixed motives by the endofunctor
$\mathbb{Q}(1)[2]\otimes.$, the resulting triangulated category seems
to be similar to a full subcategory of the full decategorification of
the $2$-category of smooth proper {\bf k}-linear dg-categories.
}
\\

Un exemple d'un invariant additif est le groupe de Grothendieck
$K_0$. Nous observons que dans $\mathsf{Hmo}_0$, le foncteur $K_0$
devient corepr{\'e}sentable. Soit $\underline{k}$ la dg-cat{\'e}gorie avec un
seul object et dont la dg-alg{\`e}bre d'endomorphismes est $k$.

\begin{lemme}[\ref{jolie}]\label{etant}
On a un isomorphisme naturel de groupes ab{\'e}liens
$$\mathrm{Hom}_{\mathsf{Hmo}_0}(\underline{k}, \mathcal{C})
\stackrel{\sim}{\longrightarrow} \mbox{K}_0(\mathcal{C})\, .$$
\end{lemme}

Ce lemme nous fournit imm{\'e}diatement les caract{\`e}res de Chern. En effet,
soit $n \in \mathbb{N}$ et $\mathrm{HC}_n$ le $n$-i{\`e}me groupe
d'homologie cyclique. Comme $\mathrm{HC}_n$ est un foncteur
additif, on dispose d'un foncteur 
$$\mathrm{HC}_n: \mathsf{Hmo}_0 \rightarrow
\mbox{Mod}\,\mathbb{Z}\,.$$
Maintenant, {\`a} partir de l'isomorphisme
$$ \mathrm{HC}_*(k)\simeq k \left[ u \right], \, \, \left| u \right|
=2\, , $$
et du lemma de Yoneda, on obtient les caract{\`e}res de Chern
$$ch_{2n}: \mbox{K}_0 \rightarrow \mathrm{HC}_{2n} \, . $$

Dans une tentative de construction d'une {\it mesure motivique}, Bondal-Larsen-Lunts ont introduit dans \cite{Bondal} un anneau de
Grothendieck $\mathcal{PT}$. On remarque que l'anneau $\mathcal{PT}$ s'annule si on n'impose pas des
conditions de finitude. On le modifie en d{\'e}finissant
$\mathcal{PT}^{cl}_{kar}$ comme le groupe
ab{\'e}lien engendr{\'e} par les classes de quasi-{\'e}quivalences $[\mathcal{A}]$
de petites dg-cat{\'e}gories $\mathcal{A}$ Morita-fibrantes qui sont en
outre compactes et lisses au sens de Kontsevich, voir \cite{IHP}, soumis
aux relations qui proviennent des d{\'e}compositions semi-orthogonales au
sens de Bondal-Orlov \cite{Bon-Orl}.
On relie $\mathcal{PT}^{cl}_{kar}$ {\`a} notre construction de
la cat{\'e}gorie
$\mathsf{Hmo}_0$.

Soit $\mbox{K}_0 (\mathsf{Hmo}_0^{cl})$
l'anneau de Grothendieck de la cat{\'e}gorie additive
$(\mathsf{Hmo}_0^{cl} , \oplus)$ o{\`u} l'on se restreint aux dg-cat{\'e}gories
compactes et lisses, voir \cite{IHP}
\begin{prop}[\ref{surj}]
On a un morphisme surjectif d'anneaux commutatifs
$$
\mathcal{PT}^{cl}_{kar} \rightarrow \mbox{K}_0(\mathsf{Hmo}_0^{cl})\,.
$$
\end{prop}

\subsection*{Chapitre 3}

Dans ce chapitre en utilisant le formalisme des d{\'e}rivateurs de
Grothendieck, on construit `l'invariant localisant universel des
dg-cat{\'e}gories'. On entend par cela un morphisme $\mathcal{U}_l$ du
d{\'e}rivateur point{\'e} $\mathsf{HO}(\mathsf{dgcat})$ associ{\'e} {\`a} la th{\'e}orie
homotopique de Morita des dg-cat{\'e}gories vers un derivateur triangul{\'e}
fort $\mathcal{M}_{dg}^{loc}$, tel que $\mathcal{U}_l$ commute aux
colimites homotopiques filtrantes, pr{\'e}serve le point, envoie chaque
suite exacte courte de dg-cat{\'e}gories vers un triangle et est
universel pour ces propriet{\'e}s.

{\`A} cause de cette propri{\'e}t{\'e} universelle {\em motivique}, \emph{cf.} section
$4.1$ de \cite{Konnew}, on appelle $\mathcal{M}_{dg}^{loc}$ {\em le
  motivateur localisant} (stable) des dg-cat{\'e}gories. Par exemple, la construction du complexe
mixte \cite{cyclichomology} et la $K$-th{\'e}orie non-connective
\cite{Marco} sont des invariants localisants et se factorisent donc par
$\mathcal{U}_l$.

On construit aussi `l'invariant additif universel des
dg-cat{\'e}gories', c'est-{\`a}-dire le morphisme universel de d{\'e}rivateurs
$\mathcal{U}_a$ du derivateur $\mathsf{HO}(\mathsf{dgcat})$ vers un
d{\'e}rivateur triangul{\'e}e fort $\mathcal{M}^{add}_{dg}$ qui satisfait les
deux premi{\`e}res propri{\'e}t{\'e}s et la troisi{\`e}me seulement pour les suites
exactes scind{\'e}es. On appelle $\mathcal{M}^{add}_{dg}$ le {\em
  motivateur additif} (stable) des dg-cat{\'e}gories. 

On montre que le spectre de $K$-th{\'e}orie de Waldhausen appara{\^\i}t comme
un spectre de morphismes dans la cat{\'e}gorie de base $\mathcal{M}_{dg}^{add}(e)$ du
motivateur additif. Cela est la premi{\`e}re caract{\'e}risation conceptuelle
de la $K$-th{\'e}orie de Quillen-Waldhausen
\cite{Quillen}~\cite{Waldhausen} et nous donne une toute nouvele fa{\c c}on
de penser la $K$-th{\'e}orie.

La corepr{\'e}sentation du spectre de $K$-th{\'e}orie de Waldhausen comme un
spectre de morphismes {\'e}tend
le lemme~\ref{etant}. La diff{\'e}rence essentielle est qu'on consid{\`e}re des morphismes de
d{\'e}rivateurs {\`a} valeurs dans des d{\'e}rivateurs triangul{\'e}es forts au lieu
des foncteurs {\`a} valeurs dans des cat{\'e}gories additives.

Afin de construire les motivateurs $\mathcal{M}_{dg}^{loc}$ et $\mathcal{M}_{dg}^{add}$, on a {\'e}t{\'e} amen{\'e} {\`a}
d{\'e}velopper plusieurs outils originaux et {\`a} {\'e}tablir de nouveaux liens entre la th{\'e}orie des d{\'e}rivateurs de Grothendieck
\cite{Grothendieck} et la th{\'e}orie des cat{\'e}gories de mod{\`e}les de Quillen
\cite{Quillen}. Le formalisme des d{\'e}rivateurs de Grothendieck nous
permet d'{\'e}noncer et de d{\'e}montrer des propriet{\'e}s universelles rigoureuses
et de nous d{\'e}barasser de plusieurs probl{\`e}mes techniques pr{\'e}sents dans
l'approche de Quillen.

Notre construction se d{\'e}compose en plusieurs parties. On commence
par developper les id{\'e}es de Cisinski sur la localisation des d{\'e}rivateurs.
Soit $\mathbb{D}$ un d{\'e}rivateur et $S$ une classe de morphismes de
$\mathbb{D}(e)$, o{\`u} $e$ designe la cat{\'e}gorie terminale.

\begin{defi}[\ref{defCis} Cisinski]
Le derivateur $\mathbb{D}$ admet une {\em localisation de Bousfield {\`a}
gauche} par rapport {\`a} $S$ s'il existe un morphisme de derivateurs
$$\gamma : \mathbb{D} \rightarrow \mathsf{L}_{S}\mathbb{D}\,,$$
qui commute aux colimites homotopiques, envoie les elements de $S$
vers des isomorphismes dans $\mathsf{L}_S\mathbb{D}(e)$ et satisfait {\`a}
la propri{\'e}t{\'e} universelle suivante~: pour chaque d{\'e}rivateur
$\mathbb{D}'$, le morphisme $\gamma$ induit une {\'e}quivalence de cat{\'e}gories
$$\underline{\mathsf{Hom}}_!(\mathsf{L}_S\mathbb{D},\mathbb{D}')
  \stackrel{\gamma^{\ast}}{\longrightarrow}
  \underline{\mathsf{Hom}}_{!,S}(\mathbb{D},\mathbb{D}')\,,$$
o{\`u} $\underline{\mathsf{Hom}}_{!,S}(\mathbb{D},\mathbb{D}')$ d{\'e}signe la
cat{\'e}gorie des morphismes de d{\'e}rivateurs qui commutent aux colimites
homotopiques et envoient les {\'e}l{\'e}ments de $S$ vers des isomorphismes dans
$\mathbb{D}(e)$.
\end{defi}
Soient maintenant $\mathcal{M}$ une cat{\'e}gorie de mod{\`e}les de Quillen cellulaire et
propre {\`a} gauche, $S$ un ensemble de morphismes de $\mathcal{M}$ et
$\mathsf{L}_S\mathcal{M}$ la localisation de Bousfield {\`a} gauche,
\emph{cf.} la section $3$ de \cite{Hirschhorn}. On demontre le th{\'e}or{\`e}me clef suivant.

\begin{theoreme}[\ref{Cisinsk} Cisinski]\label{intcin}
Le morphisme de d{\'e}rivateurs
$$\gamma: \mathsf{HO}(\mathcal{M})
\stackrel{\mathbb{L}Id}{\longrightarrow} \mathsf{HO}(\mathsf{L}_S
\mathcal{M})$$
est la localisation de Bousfield {\`a} gauche de
$\mathsf{HO}(\mathcal{M})$ par l'image de l'ensemble $S$ dans $\mathsf{Ho}(\mathcal{M})$.
\end{theoreme}

\begin{remarque}
Ce th{\'e}or{\`e}me nous permet de caract{\'e}riser le d{\'e}rivateur associ{\'e} {\`a} la
localisation de Bousfield {\`a} gauche $\mathsf{L}_S\mathcal{M}$ de
$\mathcal{M}$ par une propri{\'e}t{\'e} universelle.

La preuve de ce th{\'e}or{\`e}me est bas{\'e}e sur une description constructive
des $S$-{\'e}quivalences locales, voir proposition~\ref{Cisin}.
\end{remarque}

On d{\'e}veloppe ensuite une th{\'e}orie des `ind-objets'. Supposons que $\mathcal{M}$ est une cat{\'e}gorie de
mod{\`e}les de Quillen cellulaire qui satisfait certaines conditions de
finitude, voir section~\ref{homotopy}. A partir de $\mathcal{M}$, on construit
une petite sous-cat{\'e}gorie $\mathcal{M}_f$ de $\mathcal{M}$ telle que le
d{\'e}rivateur $\mathsf{HO}(\mathcal{M})$ correspond aux `ind-objets' du
pr{\'e}d{\'e}rivateur associ{\'e} {\`a} la localisation de $\mathcal{M}_f$ par rapport
aux {\'e}quivalences faibles.

\begin{theoreme}[\ref{flt}]
Le morphisme de d{\'e}rivateurs
$$ \mathsf{HO}(\mathcal{M})
\stackrel{\mathbb{R}\underline{h}}{\longrightarrow}
\mathsf{L}_{\Sigma}\mathsf{Hot}_{\mathcal{M}_f}\,,$$
induit une {\'e}quivalence de cat{\'e}gories
$$ \underline{\mathsf{Hom}}_!(\mathsf{L}_{\Sigma}\mathsf{Hot}_{\mathcal{M}_f},
\mathbb{D}) \stackrel{\mathbb{R}\underline{h}^{\ast}}{\longrightarrow}
\underline{\mathsf{Hom}}_{flt}(\mathsf{HO}(\mathcal{M}),
\mathbb{D})\,,$$
o{\`u}
$\underline{\mathsf{Hom}}_{flt}(\mathsf{HO}(\mathcal{M}),\mathbb{D})$
d{\'e}signe la cat{\'e}gorie des morphismes de d{\'e}rivateurs qui commutent aux
colimites homotopiques filtrantes.
\end{theoreme}

\begin{remarque}
Les conditions de finitude sur $\mathcal{M}$ nous permettent de construire
une petite cat{\'e}gorie g{\'e}n{\'e}ratrice $\mathcal{M}_f$ {\`a} l'interieur de
$\mathcal{M}$. On obtient alors
$\mathsf{L}_{\Sigma}\mathsf{Hot}_{\mathcal{M}_f}$ en consid{\'e}rant une
localisation de Bousfield {\`a} gauche de la cat{\'e}gorie des pr{\'e}-faisceaux simpliciaux sur
$\mathcal{M}_f$.

La preuve de ce th{\'e}or{\`e}me fait intervenir la th{\'e}orie des extensions
d{\'e}riv{\'e}es de Kan, au sens de Cisinski~\cite{Cisinski}, th{\'e}or{\`e}me~\ref{intcin} et des r{\'e}sultats dus {\`a} To{\"e}n-Vezzosi~\cite{HAG}.
\end{remarque}

Afin de mieux comprendre les contextes stables on r{\'e}interpr{\`e}te les
diff{\'e}rentes constructions de stabilisation existant dans la
litt{\'e}rature. On {\'e}tablit le lien entre la stabilisation de Heller et la
stabilisation de Hovey/Schwede en d{\'e}montrant que si on commence avec
une cat{\'e}gorie de mod{\`e}les de Quillen point{\'e}e $\mathcal{M}$, qui
satisfait quelques hypoth{\`e}ses de g{\'e}n{\'e}ration, alors les deux
constructions nous fournissent des r{\'e}sultats {\'e}quivalents.
On d{\'e}montre le th{\'e}or{\`e}me de comparaison suivant.

\begin{theoreme}[\ref{repre}]
Le morphisme induit entre les d{\'e}rivateurs triangul{\'e}s 
$$ \varphi: \mathsf{St}(\mathsf{HO}(\mathcal{M})) \longrightarrow
\mathsf{HO}(\mathsf{Sp}^{\mathbb{N}}(\mathcal{M}))$$
est une {\'e}quivalence, o{\`u} $\mathsf{St}(\mathsf{HO}(\mathcal{M}))$
d{\'e}signe la stabilisation au sens de Heller.
\end{theoreme}

\begin{remarque}
Les conditions de g{\'e}n{\'e}ration sur $\mathcal{M}$ nous permettent d'avoir une description
explicite de l'espace des morphismes de
$\mathsf{Sp}^{\mathbb{N}}(\mathcal{M})$ et de r{\'e}duire l'analyse du
morphisme $\varphi$ {\`a} la cat{\'e}gorie terminale $e$. Une analyse de
l'espace des morphismes entre les g{\'e}n{\'e}rateurs dans
$\mathsf{St}(\mathsf{HO}(\mathcal{M}))$ et un argument sur les
cat{\'e}gories triangul{\'e}es {\`a} engendrement compact nous permettent alors de d{\'e}montrer
le th{\'e}or{\`e}me.

Ce th{\'e}or{\`e}me nous permet de caract{\'e}riser la construction de
Hovey/Schwede par une propri{\'e}t{\'e} universelle
\end{remarque}

En appliquant tous les outils d{\'e}j{\`a} developp{\'e}s {\`a} la th{\'e}orie
d'homotopie de Morita des dg-cat{\'e}gories, on construit un morphisme de
d{\'e}rivateurs
$$ \mathcal{U}_t: \mathsf{HO}(\mathsf{dgcat})
\rightarrow
\mathsf{St}(\mathsf{L}_{\Sigma,P}\mathsf{Hot}_{\mathsf{dgcat}_f})$$
{\`a} valeurs dans un d{\'e}rivateur triangul{\'e} fort. Le
morphisme $\mathcal{U}_t$ est caract{\'e}ris{\'e} par la propri{\'e}t{\'e} universelle suivante.
Soit $\mathbb{D}$ un derivateur triangul{\'e} fort.

\begin{prop}[\ref{Trin}]
On dispose d'une {\'e}quivalence de cat{\'e}gories
$$\underline{\mathsf{Hom}}_!(\mathsf{St}(\mathsf{L}_{\Sigma,P} \mathsf{Hot}_{\mathsf{dgcat}_f}),\mathbb{D})
\stackrel{\mathcal{U}_t^{\ast}}{\longrightarrow} \underline{\mathsf{Hom}}_{flt,p}(\mathsf{HO}(\mathsf{dgcat}),\mathbb{D})\,,$$
o{\`u}
$\underline{\mathsf{Hom}}_{flt,p}(\mathsf{HO}(\mathsf{dgcat}),\mathbb{D})$
d{\'e}signe la cat{\'e}gorie des morphismes de d{\'e}rivateurs qui commutent aux
colimites homotopiques filtrantes et qui pr{\'e}servent le point.
\end{prop}

Pour chaque inclusion $\mathcal{A} \stackrel{K}{\hookrightarrow}
\mathcal{B}$ d'une sous dg-cat{\'e}gorie pleine, on dispose d'un morphisme
induit
$$ S_K: \mathsf{cone}(\mathcal{U}_t(\mathcal{A}
\stackrel{K}{\hookrightarrow} \mathcal{B})) \rightarrow
\mathcal{U}_t(\mathcal{B}/\mathcal{A})$$
dans $\mathsf{L}_{\Sigma,P} \mathsf{Hot}_{\mathsf{dgcat}_f}(e)$, o{\`u}
$\mathcal{B}/\mathcal{A}$ d{\'e}signe le dg-quotient de Drinfeld. On
d{\'e}montre que tous les morphismes $S_K$ peuvent {\^e}tre invers{\'e}s, ce qui nous
permet d'obtenir finalement {\em l'invariant localisant universel des dg-cat{\'e}gories}
$$ \mathcal{U}_l : \mathsf{HO}(\mathsf{dgcat}) \rightarrow \mathcal{M}_{dg}^{loc}\,.$$

\begin{defi}[\ref{defloc}]
Un morphisme de d{\'e}rivateurs $F$ de $\mathsf{HO}(\mathsf{dgcat})$ vers
$\mathbb{D}$ satisfait la {\em condition Dr)} si:
\begin{quote}
 - pour chaque inclusion  $\mathcal{A}
  \stackrel{K}{\hookrightarrow} \mathcal{B}$ d'une sous dg-cat{\'e}gorie
  pleine, le morphisme induit
$$ S_K: \mathsf{cone}(F(\mathcal{A}
\stackrel{K}{\hookrightarrow} \mathcal{B})) \rightarrow
F(\mathcal{B}/\mathcal{A})$$
est un isomorphisme dans $\mathbb{D}(e)$.
\end{quote}
\end{defi}

\begin{theoreme}[\ref{principal}]
Le morphisme $\mathcal{U}_l$ induit une {\'e}quivalence de cat{\'e}gories
$$
\underline{\mathsf{Hom}}_!(\mathcal{M}_{dg}^{loc},
\mathbb{D}) \stackrel{\mathcal{U}_l^{\ast}}{\longrightarrow}
\underline{\mathsf{Hom}}_{flt,\,Dr,\,p}(\mathsf{HO}(\mathsf{dgcat}),\mathbb{D})\,,$$
o{\`u} $\underline{\mathsf{Hom}}_{flt,\,Dr\,,
  p}(\mathsf{HO}(\mathsf{dgcat}),\mathbb{D})$ d{\'e}signe la cat{\'e}gorie des
morphismes de d{\'e}rivateurs qui commutent aux colimites homotopiques
filtrantes, satisfont la condition Dr) et pr{\'e}servent le point.
\end{theoreme}

\begin{remarque}
La preuve de ce th{\'e}or{\`e}me se d{\'e}compose en plusieurs propositions et
th{\'e}or{\`e}mes. On commence par construire un ensemble $\mathcal{E}$
d'inclusions $\mathcal{G} \hookrightarrow \mathcal{H}$, avec
$\mathcal{H}$ une $I$-cellule strictement finie. Cela nous permet d'exprimer chaque inclusion $\mathcal{A} \stackrel{K}{\hookrightarrow}
\mathcal{B}$ d'une sous dg-cat{\'e}gorie pleine comme la colimite
homotopique d'un diagramme filtrant o{\`u} chaque terme appartient {\`a}
$\mathcal{E}$.
On en d{\'e}duit le th{\'e}or{\`e}me suivant~:
\end{remarque}
\begin{theoreme}[\ref{invert}]
Si 
$$G:
\mathsf{St}(\mathsf{L}_{\Sigma,P}\mathsf{Hot}_{\mathsf{dgcat}_f})
\rightarrow \mathbb{D}$$
est un morphisme de d{\'e}rivateurs triangul{\'e}s qui commute aux colimites
homotopiques et tel que $G(e)(S_L)$ est inversible pour chaque $L$
dans $\mathcal{E}$, alors $G(e)(S_K)$ est inversible pour chaque
inclusion $\mathcal{A} \stackrel{K}{\hookrightarrow} \mathcal{B}$
d'une sous dg-cat{\'e}gorie pleine. 
\end{theoreme}
Finalement, en utilisant les t{\'e}chniques de localisation des d{\'e}rivateurs
d{\'e}velopp{\'e}es jusqu'ici et le fait que le d{\'e}rivateur
$\mathsf{St}(\mathsf{L}_{\Sigma,P}\mathsf{Hot}_{\mathsf{dgcat}_f})$
admet un mod{\`e}le de Quillen stable, on construit le motivateur
localisant des dg-cat{\'e}gories $\mathcal{M}_{dg}^{loc}$.

On {\'e}tablit aussi un lien entre la construction $S_{\bullet}$ de
Waldhausen et le foncteur de suspension dans la cat{\'e}gorie triangul{\'e}e
$\mathcal{M}_{dg}^{loc}(e)$. Soit $\Delta$ la cat{\'e}gorie simpliciale et $p:
\Delta \rightarrow e$ le foncteur de projection.

\begin{prop}[\ref{real}]
On dispose d'un isomorphisme canonique dans $\mathcal{M}_{dg}^{loc}(e)$
$$ p_!\mathcal{U}_l(S_{\bullet}\mathcal{A})
\stackrel{\sim}{\rightarrow} \mathcal{U}_l(\mathcal{A})[1]\,.$$
\end{prop}

On montre aussi que le derivateur $\mathcal{M}_{dg}^{loc}$ admet un mod{\`e}le de
Quillen stable donn{\'e} par une localisation de Bousfield {\`a} gauche d'une
certaine cat{\'e}gorie de pr{\'e}-faisceaux en spectres.

\begin{prop}[\ref{prespectres}]
On dispose d'une {\'e}quivalence de d{\'e}rivateurs
$$ \mathsf{HO}(\mathsf{L}_{\widetilde{\mathcal{E}_{st}},G}\mathsf{Fun}(\mathsf{dgcat}_f^o,\mathsf{Sp}^{\mathbb{N}}(Sset_{\bullet}))) \stackrel{\sim}{\longrightarrow} \mathcal{M}_{dg}^{loc} \,.$$
\end{prop} 

Ensuite, afin d'{\'e}tudier les suites exactes courtes scind{\'e}es on introduit la
notion de {\em dg-cat{\'e}gorie triangulaire sup{\'e}rieure}. Une dg-cat{\'e}gorie triangulaire
sup{\'e}rieure $\underline{\mathcal{B}}$ est donn{\'e}e par une matrice
$$
\begin{array}{rcl}
\underline{\mathcal{B}} &  := & \begin{pmatrix} \mathcal{A} &
  \mathcal{B} \\ 0 & \mathcal{C}   \end{pmatrix} 
\end{array}
$$
o{\`u} $\mathcal{A}$ et $\mathcal{C}$ sont des petites dg-cat{\'e}gories et
$\mathcal{B}$ est un $\mathcal{A}$-$\mathcal{C}$-bimodule. Un
morphisme de dg-cat{\'e}gories triangulaires sup{\'e}rieures $\underline{F}:
\underline{\mathcal{B}} \rightarrow \underline{\mathcal{B}'}$ est
donn{\'e} par $\underline{F}:=(F_{\mathcal{A}},F_{\mathcal{C}},F_X)$, o{\`u} $F_{\mathcal{A}}$, resp. $F_{\mathcal{C}}$, est un dg-foncteur de
$\mathcal{A}$ vers $\mathcal{A}'$, resp. de $\mathcal{C}$ vers
$\mathcal{C}'$, et $F_X$ est un morphisme de
$\mathcal{A}$-$\mathcal{C}$-bimodules de $X$ vers $X'$. On d{\'e}signe par
$\mathsf{dgcat}^{tr}$ la cat{\'e}gorie des petites dg-cat{\'e}gories
triangulaires sup{\'e}rieures. On dispose d'une adjonction
$$
\xymatrix{
\mathsf{dgcat}^{tr} \ar@<-1ex>[d]_{|-|} \\
\mathsf{dgcat} \ar@<-1ex>[u]_I \,,
}
$$
o{\`u} 
$$
\begin{array}{rcl}
I(\mathcal{B}') & := & \begin{pmatrix} \mathcal{B}' & \mathsf{Hom}_{\mathcal{B}'}(-,-) \\
0 & \mathcal{B}' \end{pmatrix}\,.
\end{array}
$$

\begin{theoreme}[\ref{theo1}]\label{triang}
La cat{\'e}gorie $\mathsf{dgcat}^{tr}$ admet une structure de cat{\'e}gorie de
mod{\`e}les de Quillen {\`a} engendrement cofibrant, dont les {\'e}quivalences
faibles sont les morphismes $\underline{F}$ tels que $F_{\mathcal{A}}$
et $F_{\mathcal{B}}$ sont des dg-foncteurs de Morita et $F_X$ est un
quasi-isomorphisme de $\mathcal{A}$-$\mathcal{C}$-bimodules.
\end{theoreme}

\begin{prop}[\ref{comspl}]\label{realization}
Si $\underline{\mathcal{B}}$ est une $I$-cellule strictement finie
dans $\mathsf{dgcat}^{tr}$, alors $\mathcal{A}$, $\mathcal{C}$ et
$|\underline{\mathcal{B}}|$ sont des $I$-cellules strictement finies
dans $\mathsf{dgcat}$.
\end{prop}

On d{\'e}montre le r{\'e}sultat `d'approximation' suivant, qui est un
engr{\'e}dient important dans la preuve du th{\'e}or{\`e}me de corepresentabilit{\'e}e de la $K$-th{\'e}orie de Waldhausen, voir th{\'e}or{\`e}me~\ref{representabilidade}.

\begin{prop}[\ref{aproxsplit}]\label{apr}
Chaque suite exacte courte scind{\'e}e de dg-cat{\'e}gories est faiblement
{\'e}quivalente {\`a} une colimite homotopique filtrante de suites exactes
courtes scind{\'e}es dont les composantes sont des $I$-cellules
strictement finies dans $\mathsf{dgcat}$.
\end{prop}

\begin{remarque}
Notre preuve est bas{\'e}e sur le lien entre dg-cat{\'e}gories
triangulaires sup{\'e}rieures et suites exactes courtes scind{\'e}es. On
utilise les propri{\'e}t{\'e}es de finitude de la cat{\'e}gorie de mod{\`e}les de
Quillen du th{\'e}or{\`e}me~\ref{triang} et la proposition~\ref{realization}.
\end{remarque}

En utilisant les techniques de localisation du th{\'e}or{\`e}me~\ref{intcin},
on construit le morphisme universel de d{\'e}rivateurs
$$ \mathcal{U}_u :\mathsf{HO}(\mathsf{dgcat}) \rightarrow
\mathcal{M}_{dg}^{unst}$$
qui commute aux colimites homotopiques filtrantes, preserve le point et
envoie chaque suite exacte scind{\'e}e vers une cofibration
homotopique. Le d{\'e}rivateur $\mathcal{M}_{dg}^{unst}$ est le {\em
  motivateur unstable des dg-cat{\'e}gories}. Il admet un mod{\`e}le de
Quillen et l'espace de $K$-th{\'e}orie de Waldhausen y appara{\^\i}t comme un
espace de morphismes.

\begin{prop}[\ref{kth}]
On dispose d'une {\'e}quivalence faible d'ensembles simpliciaux
$$ \mathsf{Map}(\mathcal{U}_u(k),S^1 \wedge \mathcal{U}_u(\mathcal{A}))
\stackrel{\sim}{\rightarrow} |N.wS.\mathcal{A}_f|\,.$$
En particulier, on dispose des isomorphismes
$$\pi_{i+1}
\mathsf{Map}(\mathcal{U}_a(k),S^1 \wedge \mathcal{U}_u(\mathcal{A}))
\stackrel{\sim}{\rightarrow} K_i(\mathcal{A}), \, \forall i \geq
0\,.$$
\end{prop} 

Finalement on stabilise le d{\'e}rivateur $\mathcal{M}_{dg}^{unst}$ et on
obtient l'invariant additif universel des dg-cat{\'e}gories
$$ \mathcal{U}_a : \mathsf{HO}(\mathsf{dgcat}) \rightarrow \mathcal{M}^{add}_{dg}\,.$$

On appelle $\mathcal{M}_{dg}^{add}$ le {\em motivateur additif des dg-cat{\'e}gories}.

Soit $\mathbb{D}$ un d{\'e}rivateur triangul{\'e} fort.

\begin{defi}[\ref{condadit}]
Un morphisme de d{\'e}rivateurs $F$ de $\mathsf{HO}(\mathsf{dgcat})$ vers
$\mathbb{D}$ satisfait la {\em condition A)} si:
\begin{quote}
 - pour chaque suite exacte courte scind{\'e}e (voir definition~\ref{adjsplt})
$$ 
\xymatrix{
0 \ar[r] & \mathcal{A} \ar[r]_{i_{\mathcal{A}}} & \mathcal{B} \ar@<-1ex>[l]_R
\ar[r]_P & \mathcal{C} \ar@<-1ex>[l]_{i_{\mathcal{C}}} \ar[r] & 0 \,,
}
$$
les dg-foncteurs
$i_{\mathcal{A}}$ et $i_{\mathcal{C}}$ induisent un isomorphisme
$$ \mathcal{U}_a(\mathcal{A}) \oplus \mathcal{U}_a(\mathcal{C})
\stackrel{\sim}{\rightarrow} \mathcal{U}_a(\mathcal{B})$$
dans $\mathbb{D}(e)$.
\end{quote}
\end{defi}

\begin{theoreme}[\ref{principal1}]
Le morphisme $\mathcal{U}_a$ induit une {\'e}quivalence de cat{\'e}gories
$$
\underline{\mathsf{Hom}}_!(\mathcal{M}_{dg}^{add},
\mathbb{D}) \stackrel{\mathcal{U}_a^{\ast}}{\longrightarrow}
\underline{\mathsf{Hom}}_{flt,\,A),\,p}(\mathsf{HO}(\mathsf{dgcat}),\mathbb{D})\,,$$
o{\`u} $\underline{\mathsf{Hom}}_{flt,\,A)\,,
  p}(\mathsf{HO}(\mathsf{dgcat}),\mathbb{D})$ d{\'e}signe la cat{\'e}gorie des
morphismes de d{\'e}rivateurs qui commutent aux colimites homotopiques
filtrantes, satisfont la condition $A)$ et pr{\'e}servent le point.
\end{theoreme}

La $K$-th{\'e}orie connective de Waldhausen est un invariant additif et
descend donc {\`a} $\mathcal{M}_{dg}^{add}$. En utilisant le fait que le
d{\'e}rivateur $\mathcal{M}_{dg}^{add}$ admet un mod{\`e}le de Quillen enrichi
sur les spectres, on
d{\'e}montre le th{\'e}or{\`e}me de cor{\'e}pr{\'e}sentabilit{\'e}e suivant.

\begin{theoreme}[\ref{corep}]\label{representabilidade}
On dispose d'une {\'e}quivalence faible de spectres
$$ \mathsf{Hom}^{\mathsf{Sp}^{\mathbb{N}}}(\mathcal{U}_a(\mathcal{A}),\mathcal{U}_a(\mathcal{B})[1])
\stackrel{\sim}{\rightarrow}
K^c(\mathsf{rep}_{mor}(\mathcal{A},\mathcal{B}))\,,$$
o{\`u} $K^c(\mathsf{rep}_{mor}(\mathcal{A},\mathcal{B}))$ d{\'e}signe le
spectre de $K$-th{\'e}orie connective de Waldhausen de
$\mathsf{rep}_{mor}(\mathcal{A},\mathcal{B})$.

En particulier, on dipose d'une {\'e}quivalence faible d'ensembles simpliciaux
$$\mathsf{Map}(\mathcal{U}_a(\mathcal{A}),\mathcal{U}_a(\mathcal{B})[1])
\stackrel{\sim}{\rightarrow}
|N.wS_{\bullet}\mathsf{rep}_{mor}(\mathcal{A},\mathcal{B})|$$
et des isomorphismes
$$\pi_{i+1}
\mathsf{Map}(\mathcal{U}_a(A),\mathcal{U}_a(\mathcal{B})[1])
\stackrel{\sim}{\rightarrow} K_i(\mathsf{rep}_{mor}(\mathcal{A},\mathcal{B})), \,\,\, \forall i \geq 0\,.$$
\end{theoreme}

\begin{remarque}
Notre preuve est constitu{\'e}e de deux parties. Premi{\`e}rement, on montre
que le spectre de $K$-th{\'e}orie connective de Waldhausen est un objet fibrant par rapport {\`a}
notre mod{\`e}le de Quillen de $\mathcal{M}_{dg}^{add}$. Pour cela, on
utilise proposition~\ref{apr} et le th{\'e}or{\`e}me de fibration de
Waldhausen, voir \cite{Waldhausen}.

Dans une deuxi{\`e}me partie, on construit un morphisme naturel de
$\mathcal{U}_a(\mathcal{A})[1]$ vers la $K$-th{\'e}orie connective et on
montre qu'il est une {\'e}quivalence faible en utilisant inductivement
certaines suites exactes courtes scind{\'e}es.
\end{remarque}

\begin{remarque}
Si dans le th{\'e}or{\`e}me pr{\'e}c{\'e}dent, on prend $\mathcal{A}=k$, on obtient
$$ \mathsf{Hom}^{\mathsf{Sp}^{\mathbb{N}}}(\mathcal{U}_a(k),\mathcal{U}_a(\mathcal{B})[1])
\stackrel{\sim}{\rightarrow}
K^c(\mathcal{B})\,.$$
Le spectre de $K$-th{\'e}orie connective de Waldhausen devient donc
corepresentable dans $\mathcal{M}^{add}_{dg}$.
\end{remarque}
 
\subsection*{Chapitre 4}

Dans ce chapitre, $k$ d{\'e}signe un corps.

Le probl{\`e}me suivant a {\'e}t{\'e} formul{\'e} par To{\"e}n dans \cite{Toen}~:
\\

{\it The model category $\dgcat$ together with the symmmetric monoidal
  structure $-\otimes-$ is not a symmetric monoidal model category, as
  the tensor product of two cofibrant objects in $\dgcat$ is not
  cofibrant in general. A direct consequence of this fact is that the
  internal Hom object between cofibrant-fibrant objects in $\dgcat$
  can not be invariant by quasi-equivalences, and thus does not
  provide internal Hom's for the homotopy categories
  $\mathsf{Heq}$. This fact is the main difficulty in computing the
  mapping spaces in $\dgcat$, as the naive approach simply does not
  work.}
\\

Clairement ce probl{\`e}me est analogue pour la cat{\'e}gorie homotopique
$\mathsf{Hmo}$. Rappelons que To{\"e}n dans \cite{Toen}, a construit le
foncteur $\mathsf{Hom}$-interne $\mathsf{rep}_{dg}(-,-)$ pour
$\mathsf{Heq}$ en utilisant une certaine dg-cat{\'e}gorie de bimodules
quasi-repr{\'e}sentables {\`a} droite, voir le th{\'e}or{\`e}me~\ref{thmToen}.
On a {\'e}tendu cet object $\mathsf{Hom}$-interne {\`a} $\mathsf{Hmo}$, voir le lemme~\ref{lemmeGon}.

Pour r{\'e}soudre le probl{\`e}me pos{\'e} par To{\"e}n, on construit une
nouvelle cat{\'e}gorie de mod{\`e}les de Quillen $\mathsf{Lp}$ en utilisant
les concepts de {\it paire de localisation} et la construction
explicite du dg-quotient donn{\'e}e par Drinfeld dans \cite{Drinfeld}.

Rappelons qu'une paire de localisation $\mathcal{A}$
est donn{\'e}e par une petite dg-cat{\'e}gorie $\mathcal{A}_1$ et une sous
dg-cat{\'e}gorie pleine $\mathcal{A}_0 \subset \mathcal{A}_1$. On d{\'e}signe par $\mathsf{Lp}$ la cat{\'e}gorie des paires de localisation. Un
morphisme de paires de localisation $F:\mathcal{A} \rightarrow
\mathcal{B}$ est une $Q$-{\it {\'e}quivalence faible} si le dg-foncteur induit
$$ \mathcal{A}_1 / \mathcal{A}_0 \rightarrow \mathcal{B}_1 /
\mathcal{B}_0 $$
entre les dg-quotients de Drinfeld est un dg-foncteur de Morita.

\begin{theoreme}[\ref{main}]\label{thm4}
La cat{\'e}gorie $\mathsf{Lp}$ admet une structure de cat{\'e}gorie de mod{\`e}les
de Quillen dont les {\'e}quivalences faibles sont les $Q$-{\'e}quivalences faibles.
\end{theoreme}

\begin{remarque}
Notre construction est constitu{\'e}e de plusieurs parties et est inspir{\'e}e
par plusieurs arguments presents dans la construction de la
cat{\'e}gorie homotopique stable des spectres au sens de
Bousfield-Friedlander \cite{Bos-Fri}. Rappelons que dans \cite{Drinfeld}, Drinfeld a donn{\'e} une construction
explicite du dg-quotient d'une dg-cat{\'e}gorie $\mathcal{A}$ modulo une
sous dg-cat{\'e}gorie pleine $\mathcal{B}$ en imposant certaines
conditions de platitude, qui sont automatiquement verifi{\'e}es si on
travaille sur un corps $k$.
Notre preuve repose fortement sur cette construction explicite.
\end{remarque}

Dans section~\ref{secmon}, on construit un produit tensoriel $-\otimes-$
sur $\mathsf{Lp}$ et un foncteur $\mathsf{Hom}(-,-)$.

\begin{prop}[\ref{monferme}]
La cat{\'e}gorie $\mathsf{Lp}$ munie des foncteurs $-\otimes-$ et
$\mathsf{Hom}(-,-)$ est une cat{\'e}gorie mono{\"\i}dale sym{\'e}trique ferm{\'e}e.
\end{prop}

\begin{prop}[\ref{tensorprod}]
Le produit tensoriel $-\otimes-$ admet un foncteur d{\'e}riv{\'e} total {\`a}
gauche
$$ - \overset{\mathbb{L}}{\otimes}- : \mathsf{Ho}(\mathsf{Lp}) \times
\mathsf{Ho}(\mathsf{Lp}) \longrightarrow
\mathsf{Ho}(\mathsf{Lp})\,.$$
\end{prop}

\begin{theoreme}[\ref{rep}]\label{thm5}
Le foncteur $\mathsf{Hom}(-,-)$ admet un foncteur d{\'e}riv{\'e} total {\`a}
droite 
$$ \mathcal{R}\mathsf{Hom}(-,-):  \mathsf{Ho}(\mathsf{Lp}^{op} \times \mathsf{Lp}) \longrightarrow
\mathsf{Ho}(\mathsf{Lp})\,.$$
\end{theoreme}

\begin{remarque}
Notre preuve est bas{\'e}e sur une analyse profonde de la notion
d'homotopie entre dg-foncteurs. Pour cela, on r{\'e}-interpr{\`e}te une
construction d'une certaine dg-cat{\'e}gorie de morphismes donn{\'e}e par
Drinfeld dans \cite[2.9]{Drinfeld}, comme un objet de chemins fonctoriel
pour la structure de mod{\`e}les de Quillen du th{\'e}or{\`e}me~\ref{thm1}. Cela
nous permet d'obtenir une description plus simple de la notion
d'homotopie entre dg-foncteurs et de d{\'e}montrer le th{\'e}or{\`e}me. 
\end{remarque}

Rappelons qu'on dispose d'une adjonction
$$
\xymatrix{
\mathsf{Lp} \ar@<1ex>[d]^{Ev_1}\\
\mathsf{dgcat} \ar@<1ex>[u]^F \,,
}
$$
o{\`u} $Ev_1$ est le foncteur d'{\'e}valuation dans la premi{\`e}re composante et
le foncteur $F$ associe {\`a} une dg-cat{\'e}gorie $\mathcal{A}$ la paire de
localisation $(\emptyset \subset \mathcal{A})$.

\begin{prop}[\ref{relat}]
Si on consid{\`e}re dans $\dgcat$ la structure de mod{\`e}les de Quillen du
th{\'e}or{\`e}me~\ref{thm2} et dans $\mathsf{Lp}$ la structure de mod{\`e}les de
Quillen du th{\'e}or{\`e}me~\ref{thm4}, l'adjonction pr{\'e}c{\'e}dente est une {\'e}quivalence de Quillen.
\end{prop}

\begin{prop}[\ref{Quilleneq}]
Les foncteurs d{\'e}riv{\'e}s totaux $-\overset{\mathbb{L}}{\otimes}-$ et
$\mathcal{R}\mathsf{Hom}(-,-)$ dans $\mathsf{Ho}(\mathsf{Lp})$
correspondent par l'{\'e}quivalence
$$
\xymatrix{
\mathsf{Ho}(\mathsf{Lp}) \ar@<1ex>[d]^{\mathcal{R}Ev_1} \\
\mathsf{Ho}(\mathsf{dgcat}) \ar@<1ex>[u]^F
}
$$  
aux foncteurs $-\overset{\mathbb{L}}{\otimes}-$ et
$\mathsf{rep}_{dg}^{mor}(-,-)$ du corollaire~\ref{lemmeGon}
\end{prop}

Ces r{\'e}sultats nous permettent de r{\'e}-interpreter le foncteur
$\mathsf{rep}_{dg}^{mor}(-,-)$ comme un foncteur d{\'e}riv{\'e} total {\`a}
droite $\mathcal{R}\mathsf{Hom}(-,-)$ dans
$\mathsf{Ho}(\mathsf{Lp})$. Cela donne une solution au probl{\`e}me formul{\'e} par
To{\"e}n au d{\'e}but de ce chapitre.

\subsection*{Chapitre 5}
Dans son livre \cite{Neeman}, Neeman introduit une classe importante
de cat{\'e}gories triangul{\'e}es appel{\'e}es {\it bien engendr{\'e}es}. Rappelons
qu'une cat{\'e}gorie triangul{\'e}e $\mathcal{T}$ est bien engendr{\'e}e si elle
est {\`a} engendrement $\alpha$-compact, pour un cardinal r{\'e}gulier
$\alpha$, voir \cite{Krause} \cite{Neeman}. Neeman montre le
th{\'e}or{\`e}me de repr{\'e}sentabilit{\'e} de Brown pour les cat{\'e}gories triangul{\'e}es
bien engendr{\'e}es et aussi que cette classe de cat{\'e}gories triangul{\'e}es
est stable par localisations et passage {\`a} des sous-cat{\'e}gories localisantes
engendr{\'e}es par un ensemble d'objets.

\begin{exemple}
Soit $\mathcal{B}$ une cat{\'e}gorie ab{\'e}lienne de Grothendieck, par
exemple la cat{\'e}gorie des modules sur un espace annel{\'e}. Par le th{\'e}or{\`e}me
de Popescu-Gabriel~\cite{Pop-Gab}, $\mathcal{B}$ est une localisation
de la cat{\'e}gorie $\mbox{Mod}\,A$ des $A$-modules sur un anneau $A$. On
d{\'e}duit de cela~\cite{DocNeeman} que la cat{\'e}gorie deriv{\'e}e non born{\'e}e de
la cat{\'e}gorie ab{\'e}lienne $\mathcal{B}$ est une localisation de
$\mathcal{D}(A)$ et est donc bien engendr{\'e}e.
\end{exemple}
Keller~\cite{ICM} introduit la notion de cat{\'e}gorie
triangul{\'e}e alg{\'e}brique afin de r{\'e}sumer les propri{\'e}t{\'e}s de
toutes les cat{\'e}gories triangul{\'e}es qui apparaissent naturellement en
alg{\`e}bre et g{\'e}om{\'e}trie alg{\'e}brique. Rappelons qu'une cat{\'e}gorie triangul{\'e}e
$\mathcal{T}$ est {\it alg{\'e}brique} si elle est {\'e}quivalente au sens
triangul{\'e} {\`a} la cat{\'e}gorie stable $\underline{\mathcal{E}}$ d'une
cat{\'e}gorie de Frobenius $\mathcal{E}$. Cela revient {\`a} demander
que $\mathcal{T}$ soit {\'e}quivalente {\`a} une sous-cat{\'e}gorie pleine de la
cat{\'e}gorie des complexes {\`a} homotopie pr{\`e}s sur une cat{\'e}gorie additive.

Afin de mieux comprendre cette classe importante de
cat{\'e}gories triangul{\'e}es d'un point de vue homotopique, on construit,
pour chaque cardinal r{\'e}gulier $\alpha$, une cat{\'e}gorie
$\dgcat_{ex,\alpha}$ dont les objets sont essentiellement les
dg-cat{\'e}gories qui sont stables par suspensions, cosuspensions, c{\^o}nes
et sommes $\alpha$-petites.
On proc{\`e}de en deux {\'e}tapes.

\begin{prop}[\ref{monade}, \ref{sums}]
On dispose d'une monade $\mathsf{T}_{\alpha}$ sur la cat{\'e}gorie
$\dgcat$ dont les alg{\`e}bres sont les dg-cat{\'e}gories qui admettent tous
les coproduits $\alpha$-petits.
\end{prop}
\begin{remarque}
En utilisant la th{\'e}orie des ordinaux et son arithm{\'e}tique, on construit
la monade $\mathsf{T}_{\alpha}$ en adaptant au monde des dg-cat{\'e}gories
la monade $\mathsf{Fam}$, voir \cite{Fam} et \cite{Kelly}.

On montre ensuite que les propri{\'e}t{\'e}es d'unit{\'e} et d'associativit{\'e} d'une
$\mathsf{T}_{\alpha}$-alg{\`e}bre correspondent {\`a} l'existence de
coproduits $\alpha$-petits.

\end{remarque}
On d{\'e}signe par $\mathsf{T}_{\alpha}$-$\mathsf{alg}$
la cat{\'e}gorie des $\mathsf{T}_{\alpha}$-alg{\`e}bres. On dispose d'une adjonction naturelle
$$
\xymatrix{
\mathsf{T}_{\alpha}\mbox{-}\mathsf{alg} \ar@<1ex>[d]^U \\
\mathsf{dgcat} \ar@<1ex>[u]^F\,.
}
$$

\begin{theoreme}[\ref{main2}]\label{thm6}
La cat{\'e}gorie $\mathsf{T}_{\alpha}$-$\mathsf{alg}$ est munie d'une
structure de cat{\'e}gorie de mod{\`e}les de Quillen {\`a} engendrement cofibrant,
telle qu'un morphisme $F : A \rightarrow B$ de
$\mathsf{T}_{\alpha}$-alg{\`e}bres est une {\'e}quivalence faible,
resp. fibration, si et seulement si $U(F)$ est une {\'e}quivalence faible,
resp. fibration, dans la structure de mod{\`e}les de Quillen du th{\'e}or{\`e}me~\ref{thm1}.
\end{theoreme}

\begin{remarque}
Notre preuve est inspir{\'e}e d'un argument de rel{\`e}vement d{\^u} {\`a} Quillen~\cite{Quillen}. Cet argument nous permet d'{\'e}viter une analyse des
sommes amalgam{\'e}es dans la cat{\'e}gorie
$\mathsf{T}_{\alpha}$-$\mathsf{alg}$. On montre, en effet, qu'un objet de chemins fonctoriel dont on dispose dans $\dgcat$ (voir la
preuve du th{\'e}or{\`e}me~\ref{thm5}) admet une structure d'alg{\`e}bre sur la monade
$\mathsf{T}_{\alpha}$. 
\end{remarque}

Dans une deuxi{\`e}me {\'e}tape, on construit une cat{\'e}gorie
$\dgcat_{ex, \alpha}$ en consid{\'e}rant certains diagrammes dans la
cat{\'e}gorie $\mathsf{T}_{\alpha}$-$\mathsf{alg}$. On dispose d'un
foncteur d'oubli
$$ U_1 : \dgcat_{ex,\alpha} \longrightarrow
\mathsf{T}_{\alpha}\mbox{-}\mathsf{alg} \,. $$

\begin{prop}[\ref{gauche}]
Le foncteur $U_1$ admet un adjoint {\`a} gauche.
\end{prop}

\begin{prop}[\ref{monadic}]
Le foncteur $U_1$ est monadique.
\end{prop}

\begin{theoreme}[\ref{final}]\label{thm7}
La cat{\'e}gorie $\dgcat_{ex,\alpha}$ est munie d'une structure de
cat{\'e}gorie de mod{\`e}les de Quillen {\`a} engendrement cofibrant, dont un
morphisme $F:\underline{A} \rightarrow \underline{B}$ dans
$\dgcat_{ex,\alpha}$ est une {\'e}quivalence faible, resp. fibration, si et
seulement si $U_1(F)$ est une {\'e}quivalence faible, resp. fibration, dans
la structure de mod{\`e}les de Quillen du th{\'e}or{\`e}me~\ref{thm6}.
\end{theoreme}
Finalement, on definit un foncteur $\mathcal{D}_{\alpha}$ de
$\dgcat_{ex,\alpha}$ vers la cat{\'e}gorie $\mathsf{Tri}_{\alpha}$ des
cat{\'e}gories triangul{\'e}es alg{\'e}briques {\`a} engendrement $\alpha$-compact, qui par
\cite{Porta} \cite{Ober} v{\'e}rifie les conditions suivantes~:
\begin{itemize}
\item[-] toute cat{\'e}gorie dans $\mathsf{Tri}_{\alpha}$ est {\'e}quivalente {\`a}
  $\mathcal{D}_{\alpha}(\underline{A})$, pour $\underline{A}$ dans
  $\dgcat_{ex,\alpha}$ et 
\item[-] un morphisme $F$ dans $\dgcat_{ex,\alpha}$ est une
  equivalence faible si et seulement si $\mathcal{D}_{\alpha}(F)$ est
  une {\'e}quivalence de cat{\'e}gories triangul{\'e}es. 
\end{itemize}

Cela nous montre que les cat{\'e}gories triangul{\'e}es alg{\'e}briques
bien engendr{\'e}es admettent un enrichissement homotopique donn{\'e} par notre
structure de mod{\`e}les de Quillen du th{\'e}or{\`e}me~\ref{thm7}.

\subsection*{Chapitre 6}

Dans ce chapitre, on propose une description d'une certaine classe de
cat{\'e}gories de Calabi-Yau en utilisant le formalisme des dg-cat{\'e}gories
et la notion de `stabilisation' au sens de la description
de la cat{\'e}gorie d'orbites triangul{\'e}e suivant Keller, voir
\cite{orbit}. On fixe un corps $k$.

Soit $d \geq 2$ et $\mathcal{C}$ une cat{\'e}gorie triangul{\'e}e alg{\'e}brique
$d$-Calabi-Yau munie d'une sous-cat{\'e}gorie $\mathcal{T}$ qui est
$d$-amas-basculante ($d$-cluster-tilting), voir section~\ref{preli}.

Ces cat{\'e}gories apparaissent (pour $d=2$) naturellement dans
\begin{itemize}
\item[-] l'approche conceptuelle des alg{\`e}bres amass{\'e}es (cluster algebras) au sens de
  Fomin-Zelevinsky,
\item[-] dans l'{\'e}tude des modules de Cohen-Macaulay sur certaines
  singularit{\'e}es isol{\'e}es et dans l'{\'e}tude des r{\'e}solutions cr{\'e}pantes non-commutatives.
\end{itemize}

{\`A} partir de $\mathcal{C}$ et $\mathcal{T}$, on construit le carr{\'e}
commutatif cart{\'e}sien

{\small
$$
\xymatrix{
*+<1pc>{\mathcal{M}} \ar@{^{(}->}[r] \ar@{->>}[d] \ar@{}[dr]|{\ulcorner} & \mathcal{E} \ar@{->>}[d] \\
*+<1pc>{\mathcal{T}}  \ar@{^{(}->}[r]  & \underline{\mathcal{E}} =\mathcal{C} \,,
}
$$}
o{\`u} $\mathcal{E}$ est un `mod{\`e}le alg{\'e}brique' de la cat{\'e}gorie
$\mathcal{C}$, c'est-{\`a}-dire une cat{\'e}gorie de Frobenius $k$-lin{\'e}aire o{\`u}
les idempotents se scindent et telle que la cat{\'e}gorie stable
$\underline{\mathcal{E}}$ est {\'e}quivalente {\`a} $\mathcal{C}$ en tant que
cat{\'e}gorie triangul{\'e}e.

D'apr{\`e}s \cite{Palu}, on dispose d'une suite exacte courte de cat{\'e}gories triangul{\'e}es
$$ 0 \longrightarrow \mathcal{H}^b_{\mathcal{E}\mbox{-}ac}(\mathcal{M}) \longrightarrow
\mathcal{H}^b(\mathcal{M})/ \mathcal{H}^b\mathcal(\mathcal{P})
\longrightarrow \mathcal{C} \longrightarrow 0 \,,$$ 
o{\`u} $\mathcal{H}^b_{\mathcal{E}\mbox{-}ac}(\mathcal{M})$ d{\'e}signe la
cat{\'e}gorie homotopique des complexes born{\'e}s sur $\mathcal{M}$ et
$\mathcal{E}$-acycliques. Maintenant, soit $\mathcal{B}$ la sous
dg-cat{\'e}gorie pleine de $\mathcal{C}^b(\mathcal{M})_{dg}$ form{\'e}e des
complexes $\mathcal{E}$-acycliques.

On definit un foncteur $G: \mathcal{H}^{-}(\mathcal{M}) \rightarrow
\mathcal{D}(\mathcal{B}^{op})^{op}$, qui envoie un complexe born{\'e} {\`a}
droite $X$ sur le dg-module
$$ B \mapsto \mathsf{Hom}^{\bullet}_{\mathcal{M}}(X,B)\,,$$
o{\`u} $B$ appartient {\`a} $\mathcal{B}$.

Cette construction induit un foncteur
$$ G:\mathcal{H}^b(\mathcal{M})/\mathcal{H}^b(\mathcal{P}) \rightarrow
\mathcal{D}(\mathcal{B}^{op})^{op}\,.$$

\begin{prop}[\ref{pleinfidele}]
Le foncteur $G$ est pleinement fid{\`e}le.
\end{prop}

\begin{remarque}
La d{\'e}monstration consiste {\`a} factoriser le foncteur $G$ comme le foncteur
compos{\'e} suivant
$$ \mathcal{H}^b(\mathcal{M})/\mathcal{H}^b(\mathcal{P})
\stackrel{\Upsilon}{\hookrightarrow}
\mathcal{H}^{-}_{\mathcal{E}-ac}(\mathcal{M})
\stackrel{\Psi}{\rightarrow}
\mathcal{D}^{-}_{\underline{\mathcal{M}}}(\mathcal{M})
  \stackrel{\Phi}{\rightarrow} \mathcal{D}(\mathcal{B}^{op})^{op}\,.$$
On montre ensuite que les foncteurs $\Upsilon$ et $\Psi$ sont pleinement fid{\`e}les
et que $\Phi$ l'est aussi lorsqu'on le restreint {\`a} l'image
essentielle de $\Psi$.
\end{remarque}
Afin de caracteriser l'image essentielle du foncteur $G$, on introduit de
nouvelles cat{\'e}gories et foncteurs r{\'e}sum{\'e}s dans le diagramme suivant~:
$$
\xymatrix{
& & & \mathcal{D}(\mathcal{B}^{op})^{op} & \mathcal{B} \ar[d] \\
\mathcal{H}^b(\mathcal{M})/\mathcal{H}^b(\mathcal{P})
\ar@{^{(}->}[r]^-{\Upsilon} & \mathcal{H}^-_{\mathcal{E}\mbox{-}ac}(\mathcal{M})
\ar[r]^{\Psi} & \mathcal{D}^-_{\underline{\mathcal{M}}}(\mathcal{M})
\ar[d]_L \ar[r]^{{\Phi}'} \ar[ur]^{\Phi}  & \mathcal{D}(\mathcal{B}'^{op})^{op}
\ar[d]^{R'} \ar[u]_{\sim} & \mathcal{B}' \\
 &  & \mathcal{D}^-(\underline{\mathcal{M}}) \ar[r]_{\Gamma} &
 \mathcal{D}(\underline{\mathcal{M}}^{op})^{op} & \underline{\mathcal{M}} \ar[u] \\
}
$$

\begin{defi}[\ref{stable}]
Soit $\mathcal{D}(\mathcal{B}^{op})^{op}_f$ la sous-cat{\'e}gorie pleine
de $\mathcal{D}(\mathcal{B}^{op})^{op}$ dont les objets sont les $Y$
tels que $\tau_{\geq -n}Y$ appartient {\`a}
$\mathrm{per}(\mathcal{B}^{op})^{op}$, pour tout $n \in \mathbb{Z}$,
et $R'(Y)$ appartient {\`a} $\mathrm{per}(\underline{\mathcal{M}}^{op})^{op}$.
\end{defi}

\begin{prop}[\ref{caracterisation2}]
Un objet $Y$ dans $\mathcal{D}(\mathcal{B}^{op})^{op}$ appartient {\`a}
l'image essentielle du foncteur
$G:\mathcal{H}^b(\mathcal{M})/\mathcal{H}^b(\mathcal{P}) \rightarrow
\mathcal{D}(\mathcal{B}^{op})^{op}$ si et seulement s'il appartient {\`a} $\mathcal{D}(\mathcal{B}^{op})^{op}_f$.
\end{prop}

\begin{remarque}
Pour la preuve de cette proposition, on caract{\'e}rise d'abord l'image
essentielle du foncteur $\Psi \circ \Upsilon$. En effet, un objet $Y$
dans $\mathcal{D}^{-}_{\underline{\mathcal{M}}}(\mathcal{M})$ appartient
  {\`a} l'image essentielle de $\Psi \circ \Upsilon$ si et seulement si
  $\tau_{\geq -n}Y$ appartient {\`a}
  $\mathrm{per}_{\underline{\mathcal{M}}}(\mathcal{M})$, pour tout $n
    \in \mathbb{Z}$, et $L(Y)$ appartient {\`a}
    $\mathrm{per}(\underline{\mathcal{M}})$. En utilisant le fait que la
    $t$-structure canonique sur $\mathcal{D}(\mathcal{M})$ se
    restreint {\`a} $\mathrm{per}_{\underline{\mathcal{M}}}(\mathcal{M})$
      et l'{\'e}quivalence
      $\mathrm{per}_{\underline{\mathcal{M}}}(\mathcal{M})^{op}
        \stackrel{\sim}{\rightarrow} \mathrm{per}(\mathcal{B}^{op})$,
        on {\'e}tend cette $t$-structure {\`a}
        $\mathcal{D}(\mathcal{B}^{op})^{op}$.

Finalement, on d{\'e}montre que le foncteur $\Phi :
\mathcal{D}_{\underline{\mathcal{M}}}(\mathcal{M}) \rightarrow
  \mathcal{D}(\mathcal{B}^{op})^{op}$ est exact par rapport {\`a} ces $t$-structures, lorsqu'on se
  restreint {\`a} l'image essentielle de $\Psi$. 
\end{remarque}

On d{\'e}signe par $M$ un objet de $\mathcal{M}$ ainsi que le complexe
associ{\'e} dans $\mathcal{H}^b(\mathcal{M})$. Puisque la cat{\'e}gorie
$\mathcal{H}^b(\mathcal{M})/\mathcal{H}^b(\mathcal{P})$ est engendr{\'e}e
par les objets $M \in \mathcal{M}$ et que le foncteur $G$ est pleinement
fid{\`e}le, on remarque que la cat{\'e}gorie
$\mathcal{D}(\mathcal{B}^{op})^{op}_f$ s'identifie avec la sous
cat{\'e}gorie triangul{\'e}e de $\mathcal{D}(\mathcal{B}^{op})^{op}$ engendr{\'e}e
par les objets de $G(M)$.

En utilisant un argument d{\^u} {\`a} M.~Van den Bergh, on caract{\'e}rise les
objets $G(M)$, $M \in \mathcal{M}$ de la fa{\c c}on suivante: soit $P_M$ le $\underline{\mathcal{M}}$-module
$\underline{\mathcal{M}}(?,M)$ projectif associ{\'e} {\`a} $M \in
\mathcal{M}$ et $X_M$ l'image de $M$ par le foncteur $\Psi \circ \Upsilon$.

\begin{lemme}[\ref{repre2}]
On dispose d'un isomorphisme
$$
\mathrm{Hom}_{\mathcal{D}_{\underline{\mathcal{M}}}^-(\mathcal{M})}(X_M,Y)
\stackrel{\sim}{\longleftarrow}
\mathrm{Hom}_{\mathrm{mod}\,\underline{\mathcal{M}}}(P_M,
\mathrm{H}^0(Y))\,,$$
pour tous $Y \in \mathcal{D}_{\underline{\mathcal{M}}}^-(\mathcal{M})$.
\end{lemme}

On commence par caract{\'e}riser les objets $G(M)=\Phi(X_M)$, $M \in \mathcal{M}$,
dans la cat{\'e}gorie triangul{\'e}e $\mathcal{D}(\mathcal{B}^{op})$,
c'est-{\`a}-dire, on caract{\'e}rise le foncteur~:
$$ R_M:= \mathrm{Hom}_{\mathcal{D}(\mathcal{B}^{op})}(?,
\Phi(X_M)):\mathcal{D}(\mathcal{B}^{op})^{op} \rightarrow
\mathrm{Mod}\,k\,.$$
Pour cela on consid{\'e}re le foncteur
$$ F_M:=
\mathrm{Hom}_{\mathrm{per}(\mathcal{B}^{op})}(\mathrm{H}^0(?),
\Phi(P_M)): \mathrm{per}(\mathcal{B}^{op})^{op} \rightarrow
\mathrm{mod}\,k$$
et le foncteur $DF_M$ obtenue en composent $F_M$ avec le foncteur de
dualit{\'e} $D=\mathrm{Hom}(?,k)$.

\begin{lemme}[\ref{DF}]
On dispose d'un isomorphisme de foncteurs
$$
DF_M \stackrel{\sim}{\longrightarrow}
\mathrm{Hom}_{\mathcal{D}(\mathcal{B}^{op})}(\Phi(X_M), ?[d+1])\,.$$
\end{lemme}
On d{\'e}signe par $E_M$ l'extension de Kan {\`a} gauche de $DF_M$ par
l'inclusion $\mathrm{per}(\mathcal{B}^{op}) \hookrightarrow
  \mathcal{D}(\mathcal{B}^{op})$. On remarque que le foncteur $E_M$
  est homologique et pr{\'e}serve les coproduits et donc $DE_M$ est
  cohomologique et transforme coproduits en produits. Puisque $\mathcal{D}(\mathcal{B}^{op})$ est une cat{\'e}gorie triangul{\'e}e {\`a}
  engendrement compact, le th{\'e}or{\`e}me de repr{\'e}sentabilit{\'e} de Brown,
  \emph{cf.}~\cite{Neeman}, nous fournit un objet $Z_M \in
  \mathcal{D}(\mathcal{B}^{op})$ tel que
$$ DE_M \stackrel{\sim}{\longrightarrow}
\mathrm{Hom}_{\mathcal{D}(\mathcal{B}^{op})}(?, Z_M)\,.$$
On dispose de la caract{\'e}risation suivante des objets
$G(M),\, M\in \mathcal{M}$.

\begin{theoreme}[\ref{G(M)}]
On dispose d'un isomorphisme
$$ G(M) \stackrel{\sim}{\rightarrow} Z_M\,.$$
\end{theoreme} 

Jusqu'ici on a construit {\`a} partir de $\mathcal{C}$ et $\mathcal{T}$ une
dg-cat{\'e}gorie $\mathcal{B}$ et une $t$-structure $\mathcal{U}$ sur
$\mathsf{H}^0(\mathcal{B})$ qui satisfont certaines conditions de
finitude. On retrouve alors $\mathcal{C}$ et $\mathcal{T}$ en
utilisant le processus de `stabilisation' suivant, voir section~\ref{mainsec}~:
\begin{itemize}

\item[-] Soit $\mathcal{Q}$ la cat{\'e}gorie des projectifs du coeur
  $\mathcal{H}$ de la $t$-structure sur
  $\mathsf{H}^0(\mathcal{B})$. On construit un morphisme $\mathcal{Q}
  \stackrel{j}{\rightarrow} \mathcal{B}$ dans la cat{\'e}gorie homotopique
  des dg-cat{\'e}gories qui induit un foncteur de restriction
  $j^{\ast}:\mathcal{D}(\mathcal{B}) \rightarrow
  \mathcal{D}(\mathcal{Q})$.

\item[-] Soit $\mathcal{D}(\mathcal{A}^{op})^{op}_f$ la sous-cat{\'e}gorie
triangul{\'e}e pleine de $\mathcal{D}(\mathcal{A}^{op})^{op}$ form{\'e}e des
objets $Y$ tels que $\tau_{\geq -n}$ appartient {\`a}
$\mathrm{per}(\mathcal{A}^{op})^{op}$, pour tout $n \in \mathbb{Z}$,
et $j^{\ast}(Y)$ appartient {\`a} $\mathrm{per}(\mathcal{Q}^{op})^{op}$.

\item[-] La cat{\'e}gorie stable est le quotient triangul{\'e}
$$ \mbox{stab}(\mathcal{B},\mathcal{U}) := \mathcal{D}(\mathcal{B}^{op})^{op}_f / \mathrm{per}(\mathcal{B}^{op})^{op}\,.$$
\end{itemize}

\begin{theoreme}[\ref{main3}]
Le foncteur $G$ induit une {\'e}quivalence de cat{\'e}gories
$$ \widetilde{G}:\mathcal{C} \rightarrow \mathrm{stab}(\mathcal{B},\mathcal{U})\,.$$
\end{theoreme}

Soient $\mathcal{T}$ une cat{\'e}gorie triangul{\'e}e {\`a} engendrement compact
et $\mathcal{T}_c$ la sous-cat{\'e}gorie de ses objets compacts.
Dans l'appendice on montre comment {\'e}tendre une
$t$-structure $\mathcal{U}$ sur $\mathcal{T}_c$ en une $t$-structure
sur $\mathcal{T}$.

\begin{prop}[\ref{extension2}]
\begin{itemize}
\item[a)]
L'aile gauche $\mathcal{U}$ admet une extension minimale en une aile
gauche $\mathcal{T}_{\leq 0}$ dans $\mathcal{T}$.
\item[b)] Si $\mathcal{U}\subseteq \mathcal{T}_c$ est non degener{\'e}
  (c'est-{\`a}-dire que $f:X \rightarrow Y$ est inversible si et seulement
  si $\mathsf{H}^p(f)$ est inversible pour tout $p \in \mathbb{Z}$)
  et pour chaque $X$ dans $\mathcal{T}_c$, il existe un entier $N$ tel
  que $\mathsf{Hom}(X,S^N\mathcal{U})=0$ pour tout $U$ dans
  $\mathcal{U}$, alors $\mathcal{T}_{\leq0}$ est encore non deg{\'e}ner{\'e}.
\end{itemize}
\end{prop}
\part{{\`A} la poursuite du Motivateur des DG-cat{\'e}gories}
\chapter{Une structure de cat{\'e}gorie de mod{\`e}les de Quillen sur la cat{\'e}gorie des DG-cat{\'e}gories}

\textit{\small{Ce chapitre correspond {\`a} l'article \cite{cras}.}}

\section{Introduction}
Nous construisons une structure de cat{\'e}gorie de mod{\`e}les de Quillen {\`a}
engendrement cofibrant sur la cat{\'e}gorie des petites cat{\'e}gories
diff{\'e}rentielles gradu{\'e}es.

\section{Pr{\'e}liminaires}
Dans toute la suite, $k$ d{\'e}signe un anneau commutatif avec $\mathbf{1}$.
Le produit tensoriel $\otimes$ d{\'e}signe toujours le produit tensoriel
sur $k$. Nous commen{\c c}ons par introduire les d{\'e}finitions de
base de la th{\'e}orie des dg-cat{\'e}gories suivant \cite{ICM}.
\begin{definitionf}\label{dgcategorie}
\begin{itemize}
\item[-] Une {\em dg-cat{\'e}gorie} $\mathcal{A}$ est la donn{\'e}e d'une
  classe d'objets $\mbox{obj}(\mathcal{A})$, d'un dg $k$-module
  $\mathsf{Hom}_{\mathcal{A}}(X,Y)$ pour tous $X,Y \in
  \mbox{obj}(\mathcal{A})$, et d'applications de composition associatives
$$ \mathsf{Hom}_{\mathcal{A}}(Y,Z) \otimes
\mathsf{Hom}_{\mathcal{A}}(X,Y) \rightarrow
\mathsf{Hom}_{\mathcal{A}}(X,Z), \,\, (f,g) \mapsto fg$$
qui admettent des unit{\'e}s $\mathbf{1}_X \in
\mathsf{Hom}_{\mathcal{A}}(X,X)$.
\item[-] Un {\em dg-foncteur} $F:\mathcal{A} \rightarrow \mathcal{A}'$
  est donn{\'e} par une application $F:\mbox{obj}(\mathcal{A}) \rightarrow
  \mbox{obj}(\mathcal{A}')$ et par des morphismes de dg $k$-modules
$$ F(X,Y): \mathsf{Hom}_{\mathcal{A}}(X,Y) \rightarrow
\mathsf{Hom}_{\mathcal{A}'}(FX,FY), \,\, X,Y \in \mbox{obj}(\mathcal{A}),$$
compatibles avec la composition et les unit{\'e}s.
\end{itemize}
\end{definitionf}
Une dg-cat{\'e}gorie $\mathcal{A}$ est {\em petite} si $\mbox{obj}(\mathcal{A})$
est un ensemble. On d{\'e}signe par $\dgcat$ la cat{\'e}gorie des
petites dg-cat{\'e}gories.

\begin{exemple}
Soit $B$ une $k$-alg{\`e}bre et $\mathcal{C}(B)$ la cat{\'e}gorie
des complexes de $B$-modules {\`a} droite
$$ \cdots \rightarrow M^p \stackrel{d_M}{\rightarrow} M^{p+1}
\rightarrow \cdots , \, \, p \in \mathbb{Z}\,.$$
Pour deux complexes $L,M$ et un entier $n \in \mathbb{Z}$, on
d{\'e}finit $\mathsf{Hom}^n(L,M)$ comme le $k$-module des morphismes
$f:L\rightarrow M$ d'objets gradu{\'e}s de degr{\'e} $n$, c'est-{\`a}-dire des
familles $f=(f^p)$ de morphismes $f^p:L^p \rightarrow M^{p+n}, \, p
\in \mathbb{Z}$, de $B$-modules. On d{\'e}finit $\mathsf{Hom}(L,M)$ comme
le $k$-module gradu{\'e} {\`a} composantes $\mathsf{Hom}^n(L,M)$ et dont la
diff{\'e}rentielle est le commutateur
$$ d(f)=d_M\circ f - (-1)^{n} f \circ d_L\,.$$
La {\em dg-cat{\'e}gorie} $\mathcal{C}_{dg}(B)$ a pour objets tous les
complexes de $B$-modules et les morphismes sont d{\'e}finis par
$$\mathcal{C}_{dg}(B)(L,M) = \mathsf{Hom}(L,M)\,.$$
La composition est la composition des applications gradu{\'e}es.
\end{exemple}
Soit $\mathcal{A}$ une dg-cat{\'e}gorie. La {\em dg-cat{\'e}gorie oppos{\'e}e}
$\mathcal{A}^{op}$ a les m{\^e}mes objets que $\mathcal{A}$ et ses
espaces de morphismes sont d{\'e}finis par
$$ \mathcal{A}^{op}(X,Y) = \mathcal{A}(Y,X)\,;$$
la composition de $f \in \mathcal{A}^{op}(Y,X)^p$ avec $g \in
\mathcal{A}^{op}(Z,Y)^q$ est donn{\'e}e par $(-1)^{pq}gf$.
\begin{definitionf}
Un {\em dg $\mathcal{A}$-module {\`a} droite} est un dg-foncteur
$$ M:\mathcal{A}^{op} \rightarrow \mathcal{C}_{dg}(k)\,.$$
\end{definitionf}
\noindent
Si $L$ et $M$ sont des dg $\mathcal{A}$-modules et $n$ est un
entier, un {\em morphisme gradu{\'e} $f$ de degr{\'e} $n$} de $L$ dans $M$
est la donn{\'e}e de morphismes $fA$, $A\in\mathcal{A}$, appartenant {\`a}
\[
\mathsf{Hom}(LA,MA)^n
\]
tels que $(fA)(La)=(Ma)(fB)$ pour tous morphismes $a:A\to B$ de $\ca$.
La {\em diff{\'e}rentielle $d(f)$} est alors donn{\'e}e par les morphismes
$d(fA)$, $A\in \ca$. On note $\mathsf{Hom}(L,M)$
le complexe ainsi obtenu. La {\em composition} de morphismes gradu{\'e}s
est d{\'e}finie de la fa{\c c}on naturelle.
\begin{definitionf} La {\em dg-cat{\'e}gorie des $\mathcal{A}$-modules {\`a} droite},
not{\'e}e $\mathcal{C}_{dg}(\mathcal{A})$  ou
$\mathrm{Mod}\,\ca$ est la dg-cat{\'e}gorie dont les objets sont les dg $\ca$-modules
et les complexes de morphismes les complexes $\mathsf{Hom}(L,M)$.
Le {\em dg-foncteur de Yoneda} est d{\'e}fini par
$$
\widehat{  }\, : \mathcal{A}  \longrightarrow \mathcal{C}_{dg}(\mathcal{A}) \ko
X  \mapsto  \widehat{X}:= \mathsf{Hom}_{\mathcal{A}}(?,X).
$$
\end{definitionf}
\begin{definitionf}
\begin{itemize}
\item[-] La {\em cat{\'e}gorie $\mathsf{Z}^0(\mathcal{A})$} a les m{\^e}mes
  objets que $\mathcal{A}$ et ses morphismes sont d{\'e}finis par
$$ \mathsf{Hom}_{\mathsf{Z}^0(\mathcal{A})}(X,Y) = \mathsf{Z}^0
\mathsf{Hom}_{\mathcal{A}}(X,Y)\,,$$
o{\`u} $\mathsf{Z}^0$ est le noyau de $d:\mathsf{Hom}^0_{\mathcal{A}}(X,Y)
\rightarrow \mathsf{Hom}^1_{\mathcal{A}}(X,Y)$.
\item[-] La {\em cat{\'e}gorie $\mathsf{H}^0(\mathcal{A})$} a les m{\^e}mes
  objets que $\mathcal{A}$ et ses espaces de morphismes sont d{\'e}finis
  par
$$ \mathsf{Hom}_{\mathsf{H}^0(\mathcal{A})}(X,Y) =
\mathsf{H}^0(\mathsf{Hom}_{\mathcal{A}}(X,Y))\,,$$
o{\`u} $\mathsf{H}^0$ d{\'e}signe la cohomologie en degr{\'e} z{\'e}ro
du complexe $\mathsf{Hom}_{\mathcal{A}}(X,Y)$.
\end{itemize}
\end{definitionf}

\begin{definitionf}\label{quasieq}
Un dg-foncteur $F$ de $\mathcal{C}$ vers $\mathcal{D}$ est une {\em
  quasi-{\'e}quivalence} si:
\begin{itemize}
\item[-] pour tous objects $c_1$ et $c_2$ de
$\mathcal{C}$, le morphisme de complexes de $\mathrm{Hom}_{\mathcal{C}}(c_1,
c_2)$ vers $\mathrm{Hom}_{\mathcal{D}}(F(c_1), F(c_2))$ est un
quasi-isomorphisme et
\item[-] le foncteur $\mathrm{H}^0(F)$ de
$\mathrm{H}^0(\mathcal{C})$ vers $\mathrm{H}^0(\mathcal{D})$ est
essentiellement surjectif.
\end{itemize}
\end{definitionf}

Pour les cat{\'e}gories de mod{\`e}les de Quillen, nous renvoyons {\`a}
\cite{Hovey}. On introduira une structure de cat{\'e}gorie de mod{\`e}les de
Quillen {\`a} engendrement cofibrant dans $\dgcat$ dont les {\'e}quivalences faibles sont les
quasi-{\'e}quivalences.

\section{Th{\'e}or{\`e}me principal}\label{secdef}

Suivant \cite[3.7.1]{Drinfeld}, nous d{\'e}finissons $\mathcal{K}$ comme
la dg-cat{\'e}gorie avec deux objets $1$, $2$ et dont
les morphismes sont engendr{\'e}s par $f \in \mathrm{Hom}_{\mathcal{K}}^0 (1,2)$,
$g \in \mathrm{Hom}_{\mathcal{K}}^0 (2,1)$,
$r_1 \in\mathrm{Hom}_{\mathcal{K}}^{-1} (1,1)$,
$r_2 \in \mathrm{Hom}_{\mathcal{K}}^{-1} (2,2)$ et $r_{12}
\in \mathrm{Hom}_{\mathcal{K}}^{-2} (1,2)$ soumis aux relations $df=dg=0$,
$dr_1=gf-\mathbf{1}_1$, $dr_2 =fg-\mathbf{1}_2$ et $dr_{12}=fr_1 - r_2f$.
$$\xymatrix{
    1 \ar@(ul,dl)[]_{r_1} \ar@/^/[r]^f \ar@/^0.8cm/[r]^{r_{12}} &
    2 \ar@(ur,dr)[]^{r_2} \ar@/^/[l]^g }
$$
Soit $\mathcal{B}$ une petite dg-cat{\'e}gorie.
\begin{propositionf}\label{nova}
On a une bijection naturelle entre
$\mathsf{Hom}_{\dgcat}(\mathcal{K},\mathcal{B})$ et l'ensemble des
couples $(s,h)$, o{\`u} $s$ est un morphisme de degr{\'e} $0$
dans $\mathcal{B}$ tel que $d(s)=0$ et $h$ est une contraction du c{\^o}ne
de $\widehat{s}$ dans $\mathcal{C}_{dg}(\mathcal{B})$.
\end{propositionf}

\begin{proof}[D{\'e}monstration]
Soit $F$ un dg-foncteur de $\mathcal{K}$ vers $\mathcal{B}$. Notons
que l'image de $f$ par $F$ est un morphisme $s:X \rightarrow
Y$ de degr{\'e} $0$ dans $\mathcal{B}$ tel que $d(s)=0$. On remarque que
le c{\^o}ne de $\widehat{s}$ dans $\mathcal{C}_{dg}(\mathcal{B})$, qu'on
note $\mbox{cone}(s)$, est le
$\mathcal{B}$-module gradu{\'e}
$$ \widehat{Y} \oplus \widehat{X}[1]\,,$$
muni de la diff{\'e}rentielle
$$
\begin{bmatrix}
d_Y & \widehat{s} \\
0 & -d_X \\
\end{bmatrix}\,.
$$
Un morphisme $h$ dans
$\mathsf{Hom}^{-1}_{\mathcal{C}_{dg}(\mathcal{B})}(\mbox{cone}(\widehat{s}),
  \mbox{cone}(\widehat{s}))$ correspond donc {\`a} une matrice
$$
\begin{bmatrix}
r_2 & r_{12} \\
g & r_1 \\
\end{bmatrix}\,,
$$
o{\`u} $g \in \mathsf{Hom}^0_{\mathcal{B}}(Y,X)$, $r_1 \in
\mathsf{Hom}_{\mathcal{B}}^{-1}(X,X)$, $r_2 \in
\mathsf{Hom}_{\mathcal{B}}^{-1}(Y,Y)$ et $r_{12} \in
\mathsf{Hom}_{\mathcal{B}}^{-2}(X,Y)$. Finalement on observe que $h$ est
une contraction de $\mbox{cone}(\widehat{s})$ si et seulement si l'{\'e}quation
suivante est satisfaite~:
$$
\begin{bmatrix} r_2 & r_{12} \\
g & r_1\\ \end{bmatrix} \begin{bmatrix} d_Y & s \\
0 & -d_X\\ \end{bmatrix} - (-1)^{-1} \begin{bmatrix} d_Y & s \\
0 & -d_X\\ \end{bmatrix} \begin{bmatrix} r_2 & r_{12} \\
g & r_1\\ \end{bmatrix} = \begin{bmatrix} \mathbf{1}_2 & 0 \\
0 & \mathbf{1}_1\\ \end{bmatrix}\,,
$$
o{\`u}, pour am{\'e}liorer la lisibilit{\'e}, nous avons omis tous les `chapeaux'.
On r{\'e}cup{\`e}re donc les relations impos{\'e}es dans la d{\'e}finition de la
dg-cat{\'e}gorie $\mathcal{K}$.

Cela montre la proposition.
\end{proof}

Soit $\mathcal{A}$ la dg-cat{\'e}gorie avec un seul object
$3$ et telle que $\mathrm{Hom}_{\mathcal{A}}(3,3)=k$. Soit $F$ le
dg-foncteur de $\mathcal{A}$ vers $\mathcal{K}$ qui envoie $3$
sur $1$.
Soit $\mathcal{B}$ la dg-cat{\'e}gorie avec deux
objects $4$ et $5$ telle que
$
\mathrm{Hom}_{\mathcal{B}}(4,4)=k \ko
\mathrm{Hom}_{\mathcal{B}}(5,5)=k \ko
\mathrm{Hom}_{\mathcal{B}}(4,5)=0  \ko
\mathrm{Hom}_{\mathcal{B}}(5,4)=0
$.
Soit $n \in \mathbb{Z}$. On note $S^{n-1}$ le complexe $k[n-1]$
et $D^n$ le c{\^o}ne sur le morphisme identique de $S^{n-1}$.
On note $\mathcal{P}(n)$ la dg-cat{\'e}gorie avec deux
objets $6$ et $7$ et telle que
$
\mathrm{Hom}_{\mathcal{P}(n)}(6,6)=k \ko
\mathrm{Hom}_{\mathcal{P}(n)}(7,7)=k \ko
\mathrm{Hom}_{\mathcal{P}(n)}(7,6)=0 \ko
\mathrm{Hom}_{\mathcal{P}(n)}(6,7)=D^n
$.
Soit $R(n)$ le dg-foncteur de $\mathcal{B}$ vers
$\mathcal{P}(n)$ qui envoie $4$ sur $6$ et $5$ sur $7$.
On consid{\`e}re la dg-cat{\'e}gorie $\mathcal{C}(n)$ avec deux objects
$8$ et $9$ telle que
$
\mathrm{Hom}_{\mathcal{C}(n)}(8,8)=k \ko
\mathrm{Hom}_{\mathcal{C}(n)}(9,9)=k \ko
\mathrm{Hom}_{\mathcal{C}(n)}(9,8)=0  \ko
\mathrm{Hom}_{\mathcal{C}(n)}(8,9)=S^{n-1}
$.
Soit $S(n)$ le dg-foncteur de $\mathcal{C}(n)$ vers
$\mathcal{P}(n)$ qui envoie $8$ sur $6$, $9$ sur $7$ et $S^{n-1}$ dans $D^n$
par l'identit{\'e} sur $k$ en degr{\'e} $n-1$.
Soit finalement $Q$ le dg-foncteur de la dg-cat{\'e}gorie vide $\mathcal{O}$, qui est l'objet initial dans $\dgcat$,
vers $\mathcal{A}$.

\begin{theoreme}\label{mal}
Si on consid{\`e}re pour cat{\'e}gorie $\mathcal{C}$ la cat{\'e}gorie $\dgcat$,
pour classe $W$ la sous-cat{\'e}gorie de $\dgcat$ des quasi-{\'e}quivalences,
pour classe $J$ les dg-foncteurs $F$ et $R(n)$,
$n\in \mathbb{Z}$, et pour classe $I$ les dg-foncteurs $Q$ et $S(n)$,
$n\in \mathbb{Z}$, alors les conditions du th{\'e}or{\`e}me \cite[2.1.19]{Hovey} sont satisfaites.
\end{theoreme}
{\`A} la proposition~\ref{nova3} apr{\`e}s la d{\'e}monstration du th{\'e}or{\`e}me,
nous allons donner une description explicite des fibrations de la
structure de cat{\'e}gorie de mod{\`e}les obtenue ainsi.
\begin{remarque} Un r{\'e}sultat analogue pour les cat{\'e}gories simpliciales a {\'e}t{\'e}
obtenu dans \cite{Bergner}. Notre construction est inspir{\'e}e par le
r{\'e}sultat principal de \cite{Rezk} montr{\'e} d'abord dans \cite{Joyal} et par la construction des dg-quotients dans \cite{Drinfeld}.

En s'appuyant sur le th{\'e}or{\`e}me~\ref{mal} B.To{\"e}n a d{\'e}crit la
localization de Dwyer-Kan \cite{Dwyer} de $\dgcat$ dans \cite{Toen}.
\end{remarque}

\subsection{Preuves}
Rappelons d'abord les conditions du th{\'e}or{\`e}me 2.1.19 de \cite{Hovey}~:
\bigskip
\begin{itemize}
\item[(i)] La classe $W$ poss{\`e}de la propri{\'e}t{\'e} $2$ sur $3$ et est stable par r{\'e}tracts.
\item[(ii)] Les domaines des morphismes de $I$ sont petits par rapport {\`a}
$I-\mbox{cell}$.
\item[(iii)] Les domaines des morphismes de $J$ sont petits par rapport {\`a}
$J-\mbox{cell}$.
\item[(iv)] On a $J -\mbox{cell} \subseteq W\cap I-\mbox{cof}$.
\item[(v)] On a $I-\mbox{inj}\subseteq W \cap J-\mbox{inj}$.
\item[(vi)] On a $W\cap I-\mbox{cof}\subseteq J-\mbox{cof}$ ou
$W\cap J-\mbox{inj} \subseteq I-\mbox{inj}$.
\end{itemize}
\bigskip
On observe facilement que la classe $W$ poss{\`e}de la
  propri{\'e}t{\'e} $2$ sur $3$ et est stable par r{\'e}tracts. On observe
  aussi que les domaines des morphismes de $I$ et de  $J$ sont petits dans
  la cat{\'e}gorie $\dgcat$. Ils
  le sont donc en particulier relativement aux
  classes $I-\mbox{cell}$ et $J-\mbox{cell}$. Ainsi, les
  trois premi{\`e}res conditions sont v{\'e}rifi{\'e}es.

\begin{lemme}\label{J-cell-W} On a $J-\mbox{cell}\subseteq W$.
\end{lemme}

\begin{proof}
On montre d'abord que la classe $W$ est stable par
  compositions transfinies (voir \cite{Hirschhorn} pour la d{\'e}finition de
  cette construction).
  En effet, le
  foncteur $\mathsf{H}^0(-)$ commute aux compositions transfinies et
  la classe des foncteurs essentiellement surjectifs est stable par
  compositions transfinies. Puisque la classe des
  quasi-isomorphismes de complexes de $k$-modules est aussi stable par compositions
  transfinies, la classe $W$ satisfait la condition.

Soient maintenant $n\in \mathbb{Z}$ et $T : \mathcal{B} \rightarrow \mathcal{J}$
un dg-foncteur quelconque dans $\dgcat$. On consid{\`e}re la somme
amalgam{\'e}e suivante
$$
\xymatrix{
\mathcal{B} \ar[r]^{T} \ar[d]_{R(n)} \ar@{}[dr]|{\lrcorner} & \mathcal{J}
\ar[d]^{\mbox{inc}} \\
\mathcal{P}(n) \ar[r] & \mathcal{U}
}
$$
dans $\dgcat$. Il s'agit de v{\'e}rifier que $\mbox{inc}$
est une quasi-{\'e}quivalence.
La dg-cat{\'e}gorie $\mathcal{U}$ s'obtient {\`a} partir de la dg-cat{\'e}gorie
$\mathcal{J}$ en rajoutant un nouveau morphisme
$l$ de $T(4)$ vers $T(5)$ de degr{\'e} $n-1$ et un
nouveau morphisme $j$ de $T(4)$ vers $T(5)$ de
degr{\'e} $n$ tels que $dl=j$. Pour des objets $X$ et $Y$ de $\mathcal{J}$,
on a donc une d{\'e}composition de
$\mathrm{Hom}_{\mathcal{U}}(X,Y)$ en  somme directe de complexes
$$
\mathrm{Hom}_{\mathcal{U}}(X,Y)=\bigoplus_{m \geq 0} \mathrm{Hom}_{\mathcal{U}}^{(m)}(X,Y)
$$
avec
$$
\mathrm{Hom}_{\mathcal{U}}^{(m)}(X,Y)= \underbrace{(T(5),Y) \otimes D^n \otimes
(T(5),T(4)) \otimes D^n \otimes \cdots \otimes D^n
\otimes (X,T(4))}_{m\textrm{ {\scriptsize facteurs} } D^n} \ko
$$
o{\`u} l'on {\'e}crit $(,)$ pour $\mathrm{Hom}_{\mathcal{J}}(,)$.
Puisque le complexe $D^n$ est contractile, l'inclusion
$$
\mathrm{Hom}_{\mathcal{J}}(X,Y) \hookrightarrow \mathrm{Hom}_{\mathcal{U}}(X,Y)
$$
est un quasi-isomorphisme.
Comme le dg-foncteur d'inclusion est l'identit{\'e} au niveau des objets,
c'est une quasi-{\'e}quivalence.
Soit maintenant $N : \mathcal{A} \rightarrow \mathcal{L}$ un dg-foncteur
quelconque dans $\dgcat$.
On consid{\`e}re la somme amalgam{\'e}e suivante
$$
\xymatrix{
\mathcal{A} \ar[r]^{N} \ar[d]_{F} \ar@{}[dr]|{\lrcorner} & \mathcal{L}
\ar[d]^{\mbox{inc}} \\
\mathcal{K} \ar[r]_{N'} & \mathcal{M}
}
$$
dans $\dgcat$. Il s'agit de montrer que $\mbox{inc}$ est
une quasi-{\'e}quivalence.
La dg-cat{\'e}gorie  $\mathcal{M}$ s'obtient {\`a} partir de la dg-cat{\'e}gorie
$\mathcal{L}$ en rajoutant la dg-cat{\'e}gorie $\mathcal{K}$ {\`a}
$\mathcal{L}$ en identifiant les objets $N(3)$ et $F(3)$. Nous allons
maintenant donner une autre description de $\cm$:
soit $\mathcal{L}_{0}$ la dg-cat{\'e}gorie $\mathcal{L}$ {\`a} laquelle on rajoute
un morphisme $s$ avec $ds=0$ de $N(3)$ vers un nouvel objet $H$.
Notons $\mathrm{Mod} \, \mathcal{L}_0=\mathcal{C}_{dg}(\mathcal{L}_0)$
la dg-cat{\'e}gorie des dg-modules ({\`a} droite) sur
$\mathcal{L}_0$. On consid{\`e}re le plongement de Yoneda
$$
\mathcal{L}_{0} {\hookrightarrow} \mathrm{Mod}\, \mathcal{L}_0 \ko X \mapsto \widehat{X}.
$$
Soit $\mathcal{L}_{1}$ la sous-dg-cat{\'e}gorie pleine de
$\mathrm{Mod}\, \mathcal{L}_{0}$ dont les objets sont le c{\^o}ne $C$ sur $\widehat{s}$
et les dg-foncteurs repr{\'e}sentables. Soit
$\mathcal{L}_{2}$ la dg-cat{\'e}gorie obtenue en rajoutant dans $\mathcal{L}_{1}$ un
morphisme $h$ de degr{\'e} $-1$ {\`a} l'anneau d'endomorphismes de
$C$  tel que $dh$ est {\'e}gal {\`a} l'identit{\'e} de $C$.
R{\'e}sumons les notations
pour les objets dans le diagramme suivant (o{\`u} nous avons omis les `chapeaux')~:
\[
\xymatrix{N'(2) \ar@/^/[r]^{N'(g)} & N'(1) \ar@/^/[l]^{N'(f)} \ar[r]_{s} & H \ar[r] & C \ar@(dr,ur)[]_h}
\quad .
\]
Soit $\cl_3$ la sous-dg-cat{\'e}gorie pleine de $\cl_2$ dont les objets
sont les images dans $\cl_2$ des objets de $\cl_0$.
Nous allons montrer
que notre dg-cat{\'e}gorie $\mathcal{M}$ s'identifie naturellement {\`a} $\cl_3$.
En effet, la d{\'e}monstration de la proposition~\ref{nova} montre que le dg-foncteur naturel
$\cl\to \cl_3$ admet un unique prolongement en un dg-foncteur $\cm\to\cl_3$
qui envoie $N'(2)$ sur $H$ (plus pr{\'e}cis{\'e}ment~: l'image de $H$ dans $\cl_3$ \ldots),
$N'(f)$ sur $s$ et $N'(g)$, $N'(r_1)$, $N'(r_2)$ et $N'(r_{12})$ sur les quatre composantes de
la contraction $h$. Nous affirmons que ce dg-foncteur est un isomorphisme.
Pour le montrer, construisons un inverse~: tout d'abord, la d{\'e}finition
de $\cl_0$ montre qu'il existe un dg-foncteur $\cl_0\to \cm$ qui prolonge
l'inclusion de $\cl$ et qui envoie $H$ sur $N'(2)$ et $s$ sur $N'(f)$. Ce
dg-foncteur admet un prolongement $\cl_1 \to \mathrm{Mod}\,\cm$ qui envoie le
c{\^o}ne $C$ sur $\widehat{s}$ sur le c{\^o}ne $C_\cm$ sur l'image de $N'(f)$ par le foncteur
de Yoneda. Par la proposition~\ref{nova}, le c{\^o}ne $C_\cm$ est muni
d'une contraction naturelle $h_\cm$. Par la d{\'e}finition de $\cl_2$, nous
obtenons un unique dg-foncteur $\cl_2 \to \mathrm{Mod}\,\cm$ qui prolonge
$\cl_1 \to \mathrm{Mod}\cm$ et qui envoie $h$ sur $h_\cm$. On v{\'e}rifie
facilement que la restriction de $\cl_2 \to \mathrm{Mod}\,\cm$ {\`a} $\cl_3$
prend ses valeurs dans $\cm$ et est inverse du dg-foncteur naturel
$\cm\to\cl_3$.

Soient $X$ et $Y$ des objets de $\mathcal{L}$.
On a alors une d{\'e}composition de $k$-modules gradu{\'e}s comme dans \cite[3.1]{Drinfeld}
$$
\mathrm{Hom}_{\mathcal{M}}(X,Y)\simeq
\mathrm{Hom}_{\mathcal{L}_2}(\widehat{X},\widehat{Y}) = \bigoplus_{n\geq 0}
\mathrm{Hom}_{\mathcal{L}_2}^{(n)}(\widehat{X},\widehat{Y}) \ko
$$
o{\`u}
$$
\mathrm{Hom}_{\mathcal{L}_2}^{(n)}(\widehat{X},\widehat{Y})=
\underbrace{\mathrm{Hom}_{\mathcal{L}_1}(C,
\widehat{Y})\otimes S^{1}\otimes \mathrm{Hom}_{\mathcal{L}_1}(C,C)
\otimes S^{1} \otimes \cdots \otimes S^{1}
\otimes \mathrm{Hom}_{\mathcal{L}_1}(\widehat{X},C)}_{n \textrm{ {\scriptsize facteurs} } S^{1}}.
$$
Mais dans cette situation, on n'a pas une somme directe de complexes. Soit
$g_{n+1}\cdotp h \cdotp g_n \cdot h \cdots h \cdotp g_1 \in \mathrm{Hom}_{\mathcal{L}_2}^{(n)}(\widehat{X},\widehat{Y})$.
Comme on a $dh=\mathbf{1}$, l'image par $d$ de cet {\'e}l{\'e}ment est {\'e}gale {\`a}
$$
d(g_{n+1})\cdotp h \cdotp g_n \cdotp h  \cdots
 h \cdotp g_1 + \underbrace{(-1)^{\mid g_{n+1} \mid}\cdotp
   g_{n+1}\cdotp \mathbf{1} \cdotp  g_n \cdotp h  \cdots  h \cdotp g_1}_{(n-1) \textrm{ {\scriptsize facteurs} } h} + \cdots \;\; .
$$
On remarque que, pour tout $m\geq 0$, la somme
$$
\bigoplus_{n \geq 0}^m \mathrm{Hom}_{\mathcal{L}_2}^{(n)}(\widehat{X},\widehat{Y})
$$
est un sous-complexe de $\mathrm{Hom}_{\mathcal{L}_2}(\widehat{X},\widehat{Y})$ et on dispose donc d'une
filtration exhaustive de $\mathrm{Hom}_{\mathcal{L}_2}(\widehat{X},\widehat{Y})$. Le $n$-i{\`e}me
sous-quotient s'identifie {\`a}
$\mathrm{Hom}_{\mathcal{L}_2}^{(n)}(\widehat{X},\widehat{Y})$.
Le complexe $\mathrm{Hom}_{\mathcal{L}_1}(\widehat{X},C)$ s'identifie
au c{\^o}ne sur l'isomorphisme
$$
s_{\ast} : \mathrm{Hom}_{\mathcal{L}_0}(X,N(3))
\stackrel{\sim}{\rightarrow}
\mathrm{Hom}_{\mathcal{L}_0}(X,H).
$$
Il est donc contractile et l'inclusion
$$
\mathrm{Hom}_{\mathcal{L}}(X,Y) {\hookrightarrow}
\mathrm{Hom}_{\mathcal{M}}(X,Y) \simeq \mathrm{Hom}_{\mathcal{L}_2}(\widehat{X},\widehat{Y})
$$
est un quasi-isomorphisme. Comme $s$ devient un
isomorphisme dans $\mathrm{H}^0(\mathcal{M})$ et que le dg-foncteur d'inclusion est
l'identit{\'e} au niveau des objets, il est bien une quasi-{\'e}quivalence.
\end{proof}

D{\'e}montrons maintenant que   $J-\mbox{inj}\cap W = I-\mbox{inj}$.
Pour cela, on consid{\`e}re la classe $\mbox{{\bf Surj}}$ form{\'e}e des foncteurs
$G : \mathcal{H} \rightarrow \mathcal{I}$ dans $\dgcat$ qui v{\'e}rifient:
\begin{itemize}
\item[-] $G$ induit une surjection de l'ensemble des objets de $\mathcal{H}$ sur l'ensemble des
objets de $\mathcal{I}$ et
\item[-] $G$ induit des quasi-isomorphismes surjectifs dans les complexes de morphismes.
\end{itemize}

\begin{lemme}\label{I-inj--Surj}
On a $I-\mbox{inj} = \mbox{{\bf Surj}}$.
\end{lemme}
\begin{proof}[D{\'e}monstration]
Soit $\mathcal{C}$ une cat{\'e}gorie quelconque et $\mathcal{V}$ une classe quelconque
de morphismes dans $\mathcal{C}$. On note $\mathcal{V}-\mbox{drt}$ la classe de morphismes qui
ont la propri{\'e}t{\`e} de rel{\`e}vement {\`a} droite par rapport {\`a} $\mathcal{V}$.
 La classe $Q-\mbox{drt}$ est form{\'e}e des foncteurs qui
  sont surjectifs au niveau des objets. La classe
    $S(n)-\mbox{drt}, \,n \in \mathbb{Z}$ est form{\'e}e des foncteurs qui sont des
  quasi-isomorphismes surjectifs au niveau des complexes de
  morphismes. En effet, un carr{\'e} commutatif dans $\dgcat$
$$
\xymatrix{
\mathcal{C}(n) \ar[r]^{D} \ar[d]_{S(n)} & \mathcal{H} \ar[d]^{G} \\
\mathcal{P}(n) \ar[r]^{E} & \mathcal{I}
}
$$
correspond {\`a} la donn{\'e}e d'un carr{\'e} commutatif dans la cat{\'e}gorie des complexes
$$
\xymatrix{
S^{n-1} \ar@{^{(}->}[d]_{i_n} \ar[r]^-{D} &
  \mathrm{Hom}_{\mathcal{H}}(D(8),D(9)) \ar@{|->}[d]^{G}\\
D^{n} \ar[r]^-{E} & \mathrm{Hom}_{\mathcal{I}}(E(6),E(7))
}
$$
o{\`u} $D(8)$ et  $D(9)$ sont des objets
quelconques dans $\mathcal{H}$. La propri{\'e}t{\'e} r{\'e}sulte de la
caract{\'e}risation des quasi-isomorphismes surjectifs dans la cat{\'e}gorie
des complexes sur $k$. Voir \protect{\cite[2.3.5]{Hovey}}.
\end{proof}

\begin{lemme}\label{J-inj-Surj}
On a $J-\mbox{inj} \cap W = \mbox{{\bf Surj}}$.
\end{lemme}
\begin{proof}[D{\'e}monstration]
Montrons l'inclusion $\supseteq$.
Soit $H$ un foncteur de $\mathcal{N}$ vers
$\mathcal{E}$ dans la classe $\mbox{{\bf Surj}}$. Comme $H$
est surjectif au niveau des objets et un quasi-isomorphisme au niveau
des complexes de morphismes, on a $H\in W$.

La classe $R(n)-\mbox{drt}$ est form{\'e}e des foncteurs surjectifs aux
niveau des complexes de morphismes. Il suffit donc de montrer que
$H \in F-\mbox{drt}$. La donn{\'e}e d'un carr{\'e} commutatif
$$
\xymatrix{
\mathcal{A} \ar[r]^{P} \ar[d]_{F} &
  \mathcal{N} \ar[d]^{H} \\
\mathcal{K} \ar[r]^{U} & \mathcal{E}
}
$$
correspond {\`a} la donn{\'e}e de la partie inf{\'e}rieure gauche du diagramme
$$
\xymatrix{
P(3) \ar@{|->}[d]_-{H} \ar@{-->}[r]^-{\overline{U(f)}} & D
\ar@{|->}[d]^-{H} \\
U(1) \ar[r]^-{U(f)} & U(2)
}
$$
dans les dg-cat{\'e}gories $\mathcal{N}$ et $\mathcal{E}$ et {\`a} la donn{\'e} d'une contraction $h$ du c{\^o}ne $C_1$ de
$\widehat{U(f)}$ dans $\mathcal{C}_{dg}(\mathcal{E})$.
Comme $H$ est surjectif au niveau des objets, il
existe $D \in \mathcal{N}$ telle que $H(D)=U(2)$. Le
foncteur $H$ est un quasi-isomorphisme surjectif au niveau
des complexes de morphismes. Donc on peut relever $U(f)$ en
$\overline{U(f)}$.
Dans les cat{\'e}gories des dg-modules respectives, on obtient le
diagramme suivant
$$
\xymatrix{
\widehat{P(3)}\ar@{|->}[d]_-{\widehat{H}} \ar@{-->}[r]^-{\widehat{\overline{U(f)}}} &
\widehat{D} \ar@{|->}[d]^-{\widehat{H}}  \ar[r] &
C_2 \ar@(dr,ur)[]_{h^{\star}} \ar@{|->}[d]^-{\widehat{H}} \\
\widehat{U(1)} \ar[r]^-{\widehat{U(f)}} &
\widehat{U(2)} \ar[r] &  C_1 \ar@(dr,ur)[]_h
}
$$
o{\`u} $C_1$ et $C_2$ d{\'e}signent les c{\^o}nes sur les morphismes
respectifs et  $h$ est la contraction de $C_1$. Puisque
  $H$ est un quasi-isomorphisme surjectif au niveau des
  complexes de morphismes et $C_1$ et $C_2$ sont des c{\^o}nes sur des
  morphismes entre r{\'e}presentables, le dg-foncteur $\widehat{H}$ induit
  aussi un quasi-isomorphisme surjectif de l'alg{\`e}bre d'endomorphismes
  de $C_2$ sur celle de $C_1$.
  On peut relever $h$ en une contraction $h^*$ de $C_2$ par application du lemme
\cite[2.3.5]{Hovey} au couple $(h,1)$.

Montrons maintenant l'inclusion $\subseteq$.
Soit $L$ un foncteur de $\mathcal{D}$ vers
$\mathcal{S}$ qui appartient {\`a} $J-\mbox{inj} \cap W$. La classe
$R(n)-\mbox{drt}$ est form{\'e}e des foncteurs surjectifs au niveau
des complexes de morphismes. Comme $L \in W$, il
suffit de montrer que $L$ est surjectif au niveau des
objets.
Soit $E \in \mathcal{S}$ un objet quelconque. Comme $L \in W$,
il existe $C \in \mathcal{D}$ et un morphisme $q \in
\mathrm{Hom}_{\mathcal{S}}(L(C),E)$ qui devient un isomorphisme dans
$\mathrm{H}^{0}(\mathcal{S})$
$$
\xymatrix{
C \ar@{|->}[d]_{L} & \\
L(C) \ar[r]^{q} & E .
}
$$
Ainsi, $q$ est l'image de $f$ par un dg-foncteur de
$\mathcal{K}$ vers $\mathcal{S}$.
Comme on a $L \in J-\mbox{inj}$, on peut relever le
morphisme  $q$ et par cons{\'e}quence l'objet $E$. Le foncteur $L$ est donc bien
surjectif au niveau des  objets.
\end{proof}

Nous avons v{\'e}rifi{\'e} que $J-\mbox{cell}\subseteq W$ (lemme~\ref{J-cell-W}) et
que $I-\mbox{inj}$ est {\'e}gal {\`a} $J-\mbox{inj}\cap W$
(lemmes~\ref{I-inj--Surj} et \ref{J-inj-Surj}). Ainsi les trois
derni{\`e}res conditions du th{\'e}or{\`e}me $2.1.19$ de \cite{Hovey} sont aussi verifi{\'e}es. Cela demontre le th{\'e}or{\`e}me~\ref{mal}.

\begin{propositionf}\label{nova3}
Un dg-foncteur $G$ de $\mathcal{C}$ vers $\mathcal{D}$
est une fibration, pour la structure de cat{\'e}gorie de mod{\`e}les de
Quillen du th{\'e}or{\`e}me~\ref{mal}, si et seulement si:
\begin{itemize}
\item[1)] pour tous objets $c_1$ et $c_2$ dans $\mathcal{C}$, le morphisme
  de complexes de $\mathrm{Hom}_{\mathcal{C}}(c_1, c_2)$ vers
  $\mathrm{Hom}_{\mathcal{D}}(G(c_1),G(c_2))$ est surjectif en chaque
  composante et
\item[2)] pour tout objet $c_1 \in \mathcal{C}$ et tout morphisme $v
  \in \mathsf{Hom}_{\mathcal{D}}(G(c_1),d)$ qui devient un isomorphisme
    dans $\mathsf{H}^0(\mathcal{D})$, il existe un morphisme $u \in
    \mathsf{Hom}_{\mathcal{C}}(c_1,c_2)$ tel que $G(u)=v$ et qui
      devient un isomorphisme dans $\mathsf{H}^0(\mathcal{C})$.
\end{itemize}
\end{propositionf}

\begin{remarque}\label{toutfibrant}
Puisque l'objet final dans $\dgcat$ est la dg-cat{\'e}gorie
nulle, cette proposition implique que tout objet est fibrant.
\end{remarque}

\begin{proof}[D{\'e}monstration]
Le dg-foncteur $G$ est une fibration si et seulement s'il a la
propri{\'e}t{\'e} de r{\'e}l{\`e}vement {\`a} droite par rapport {\`a} l'ensemble des
cofibrations g{\'e}n{\'e}ratrices $J=\{ F, R(n), \, n \in
\mathbb{N}\}$. Clairement un dg-foncteur satisfait la condition $1)$ si
et seulement s'il a la propri{\'e}t{\'e} de r{\'e}l{\`e}vement {\`a} droite par rapport
{\`a} l'ensemble $\{R(n), \, n \in \mathbb{Z}\}$.

Notons que pour tout isomorphisme $[v]$ dans
$\mathsf{H}^0(\mathcal{D})$ on dispose d'une contraction du c{\^o}ne de
$\widehat{v}$ dans $\mathcal{C}_{dg}(\mathcal{D})$. Par la
proposition~\ref{nova}, cela nous fournit un dg-foncteur $T$ de
$\mathcal{K}$ vers $\mathcal{D}$ tel que $T(f)=v$. Il est donc clair que si un
dg-foncteur a la propri{\'e}t{\'e} de rel{\`e}vement {\`a} droite par rapport {\`a} $F$
alors il satisfait la condition $2)$. Montrons maintenant qu'un
dg-foncteur qui satisfait conditions $1)$ et $2)$ a la propri{\'e}t{\'e} de
rel{\`e}vement {\`a} droite par rapport {\`a} $F$.

Pour cela, on utilise un argument analogue {\`a} celui de la d{\'e}montration
de l'inclusion $\supseteq$ du lemme~\ref{J-inj-Surj} ci-dessus. On
remarque que la condition $2)$ nous fournit le rel{\`e}vement
$\overline{U(f)}$ et implique que la dg-alg{\`e}bre d'endomorphismes de
$C_2$ est acyclique. Par la condition $1)$ le dg-foncteur $\widehat{H}$ induit donc un
quasi-isomorphisme surjectif de l'alg{\`e}bre d'endomorphismes de $C_2$
sur celle de $C_1$ et l'affirmation suit.

Cela demontre la proposition.
\end{proof}
\chapter{Invariants additifs de DG-cat{\'e}gories}

\textit{\small{Ce chapitre correspond aux articles \cite{addendum} et \cite{IMRN}.}}

\section{Introduction}
Dans cet article, nous poursuivons l'{\'e}tude \cite{cras} \cite{Toen} de
la cat{\'e}gorie des petites cat{\'e}gories diff{\'e}rentielles gradu{\'e}es
($=$ dg-cat{\'e}gories) du point de vue des cat{\'e}gories de mod{\`e}les de Quillen
\cite{Quillen}. Plus pr{\'e}cis{\'e}ment, nous munissons la cat{\'e}gorie des
petites dg-cat{\'e}gories d'une structure de cat{\'e}gorie de mod{\`e}les {\`a} engendrement
cofibrant dont les {\'e}quivalences
faibles sont exactement les dg-foncteurs {\em de Morita}, c'est-{\`a}-dire
les dg-foncteurs $F:\mathcal{A} \rightarrow \mathcal{B}$ qui
induisent une {\'e}quivalence $\mathcal{D}(\mathcal{B})
\stackrel{\sim}{\rightarrow} \mathcal{D}(\mathcal{A})$ entre
cat{\'e}gories d{\'e}riv{\'e}es. Dans la cat{\'e}gorie homotopique $\mathsf{Hmo}$
obtenue ainsi, les {\'e}quivalences d{\'e}riv{\'e}es au sens de \cite{Rickard}
\cite{Rickard1} \cite{DerivingDG} correspondent {\`a} des isomorphismes et
le `groupe de Picard d{\'e}riv{\'e}' de \cite{Rou-Zim} y appara{\^\i}t comme un groupe
d'automorphismes.
\par
La cat{\'e}gorie $\mathsf{Hmo}$ est fortement reli{\'e}e {\`a} la cat{\'e}gorie des
petites cat{\'e}gories triangul{\'e}es mais s'en distingue par sa structure
plus riche. Elle donne un cadre commode pour formuler des propri{\'e}t{\'e}s
universelles comme celles du dg-quotient de \cite{Drinfeld}, de
l'envelope pr{\'e}triangul{\'e}e de \cite{Bon-Kap} ou des cat{\'e}gories d'orbites
de \cite{orbit}.
\par
Une autre motivation pour son {\'e}tude provient de la
g{\'e}om{\'e}trie alg{\'e}brique non commutative au sens de Drinfeld
\cite{Chicagotalk} et Kontsevich \cite{IHP} \cite{ENS},
c'est-{\`a}-dire l'{\'e}tude des dg-cat{\'e}gories et de leurs invariants
homologiques. Dans cette veine, nous construisons `l'invariant additif
universel', c'est-{\`a}-dire un foncteur
$$
\mathcal{U}_a:\dgcat \rightarrow \mathsf{Hmo}_0
$$
{\`a} valeurs dans une cat{\'e}gorie additive qui rend inversibles les
dg-foncteurs de Morita, transforme les d{\'e}compositions
semi-orthogonales \cite{Bon-Orl} en sommes directes et qui est universel pour ces propri{\'e}t{\'e}s. Par exemple, la $K$-th{\'e}orie et
l'homologie cyclique sont des invariants additifs et se factorisent
donc par $\mathcal{U}_a$. Nous observons que dans $\mathsf{Hmo}_0$, le
foncteur $K_0$ devient corepr{\'e}sentable, ce qui donne imm{\'e}diatement les
caract{\`e}res de Chern.
\par
 La cat{\'e}gorie $\mathsf{Hmo}_0$ est {\'e}troitement
reli{\'e}e au anneau de Grothendieck ${\mathcal{P}\mathcal{T}}_{kar}$ des
  dg-cat{\'e}gories pr{\'e}triangul{\'e}es Karobiennes de \cite{Bondal}. Nous pr{\'e}cisons
  ce lien en exhibant une surjection
$$
    {\mathcal{P}\mathcal{T}}^{cl}_{kar} \rightarrow  \mbox{K}_0(\mathsf{Hmo}_0^{cl})
$$
apr{\`e}s avoir impos{\'e} des conditions de finitude convenables.

\section{Conventions}

Dans toute la suite, $k$ d{\'e}signe un anneau commutatif avec $\mathbf{1}$.
Le produit tensoriel $\otimes$ d{\'e}signe toujours le produit tensoriel
sur $k$. Soit $n \in \mathbb{Z}$. On note $S^{n-1}$ le complexe
$k[n-1]$ et $D^n$ le cone sur le morphisme identique de $S^{n-1}$.
Par une {\em dg-cat{\'e}gorie}, nous entendons une $k$-cat{\'e}gorie
diff{\'e}rentielle gradu{\'e}e, voir definition~\ref{dgcategorie}.
Soit $\dgcat$ la cat{\'e}gorie des petites dg-cat{\'e}gories.
Pour la construction des foncteurs $\mathbb{Z}$ et $\mbox{pre-tr}$,
voir \cite{cyclichomology} \cite{Bon-Kap}.
Pour une dg-cat{\'e}gorie $\mathcal{A}$, on note $\widehat{ } :\mathcal{A}
\rightarrow \mathrm{Mod}\,\mathcal{A}$ le dg-foncteur de Yoneda et
$\mathcal{D}(\mathcal{A})$ la cat{\'e}gorie d{\'e}riv{\'e}e, voir \cite{Drinfeld} \cite{DerivingDG}.
Soit $\mathcal{C}$ une cat{\'e}gorie quelconque et $\mathcal{V}$ une
classe quelconque de morphismes dans $\mathcal{C}$. On note
$\mathcal{V}-\mbox{drt}$, $\mbox{gch}-\mathcal{V}$, la classe de
morphismes qui ont la propri{\'e}t{\'e} de rel{\`e}vement {\`a} droite, repectivement
{\`a} gauche, par rapport {\`a} $\mathcal{V}$. On note
$\mbox{cof}(\mathcal{V})$ la classe $\mbox{gch}-(\mathcal{V}-\mbox{drt})$.
Pour les cat{\'e}gories de mod{\`e}les de Quillen, nous renvoyons {\`a}
\cite{Hovey} \cite{Quillen}. Rappelons seulement th{\'e}or{\`e}me $2.1.19$ dans
\cite{Hovey}.
\begin{theoreme}\label{thm}
Soit $\mathcal{C}$ une cat{\'e}gorie qui admet toutes les petites limites
inductives et projectives. Soit $\mathcal{W}$ une classe de morphismes
de $\mathcal{C}$ et $I$ et $J$ des ensembles de morphismes de
$\mathcal{C}$. Alors, il existe une structure de cat{\'e}gorie de mod{\`e}les
ferm{\'e}e {\`a} engendrement cofibrant dans $\mathcal{C}$, avec $I$ les
cofibrations g{\'e}n{\'e}ratrices, $J$ les cofibrations acycliques
g{\'e}n{\'e}ratrices et $\mathcal{W}$ les {\'e}quivalences faibles si et seulement
si les conditions suivantes sont satisfaites~:
\begin{itemize}
\item[(i)] La classe $\mathcal{W}$ est stable par retracts et dans
  tout triangle commutatif de $\mathcal{C}$, si deux fl{\`e}ches sont des
  {\'e}quivalences faibles, alors la troisi{\`e}me en est une.
\item[(ii)] Les sources de l'ensemble $I$ sont petites par rapport {\`a}
  la classe $I-\mbox{cell}$.
\item[(ii)] Les sources de l'ensemble $J$ sont petites par rapport {\`a}
  la classe $J-\mbox{cell}$.
\item[(iv)] $J-\mbox{cell} \subset \mathcal{W} \cap \mbox{cof}(I)$.
\item[(v)] $I-\mbox{drt} \subset \mathcal{W} \cap J-\mbox{drt}$.
\item[(vi)] $\mathcal{W}\cap \mbox{cof}(I) \subset \mbox{cof}(J)$ ou
  $\mathcal{W}\cap J-\mbox{drt} \subset I-\mbox{drt}$.
\end{itemize}

\end{theoreme}

\section{DG-foncteurs quasi-{\'e}quiconiques}\label{secquasicon}

Un dg-foncteur $F$ de $\mathcal{C}$ vers $\mathcal{D}$ est
{\em quasi-{\'e}quiconique} si:
\begin{itemize}
\item[(i)] pour tous objets $c_1$ et $c_2$ dans $\mathcal{C}$, le morphisme
  de complexes de $\mathrm{Hom}_{\mathcal{C}}(c_1, c_2)$ vers
  $\mathrm{Hom}_{\mathcal{D}}(F(c_1),F(c_2))$ est un
  quasi-isomorphisme et
\item[(ii)] le foncteur $\mathrm{H}^0(\mbox{pre-tr}(F))$ de
  $\mathrm{H}^0(\mbox{pre-tr}(\mathcal{C}))$ vers
  $\mathrm{H}^0(\mbox{pre-tr}(\mathcal{D}))$ est essentiellement surjectif.
\end{itemize}
On introduira une structure de cat{\'e}gorie de mod{\`e}les de
Quillen {\`a} engendrement cofibrant dans $\dgcat$ dont les
{\'e}quivalences faibles sont les dg-foncteurs
quasi-{\'e}quiconiques. Pour cela, on se servira du th{\'e}or{\`e}me~\ref{thm}.
Nous d{\'e}finissons $\mathcal{K}(n)$, $ n \in \mathbb{Z}$, comme
la dg-cat{\'e}gorie avec deux objets $1$, $2$ et dont
les morphismes sont engendr{\'e}s par $f \in \mathrm{Hom}_{\mathcal{K}(n)}^n (1,2)$,
$g \in \mathrm{Hom}_{\mathcal{K}(n)}^{-n} (2,1)$,
$r_1 \in\mathrm{Hom}_{\mathcal{K}(n)}^{-1} (1,1)$,
$r_2 \in \mathrm{Hom}_{\mathcal{K}(n)}^{-1} (2,2)$ et $r_{12}
\in \mathrm{Hom}_{\mathcal{K}(n)}^{n-1} (1,2)$ soumis aux relations $df=dg=0$,
$dr_1=gf-\mathbf{1}_1$, $dr_2 =fg-\mathbf{1}_2$ et $dr_{12}=fr_1 - r_2f$.
$$\xymatrix{
    1 \ar@(ul,dl)[]_{r_1} \ar@/^/[r]^f \ar@/^0.8cm/[r]^{r_{12}} &
    2 \ar@(ur,dr)[]^{r_2} \ar@/^/[l]^g }
$$
On pose $\mathcal{K}=\mathcal{K}(0)$. Soit $\mathcal{A}$ la dg-cat{\'e}gorie avec un seul object
$3$ et telle que $\mathrm{Hom}_{\mathcal{A}}(3,3)=k$. Soit $F(n)$, $n
\in \mathbb{Z}$, le dg-foncteur de $\mathcal{A}$ vers $\mathcal{K}(n)$ qui envoie $3$
sur $1$.
Pour un entier $n >0$ et des entiers $k_0, \ldots , k_n$, soit
$\mathcal{M}_n(k_0, \ldots , k_n)$ la dg-cat{\'e}gorie avec $n+1$ objets
$0, \ldots , n$ et dont les morphismes sont engendr{\'e}s par des
morphismes $q_{i,j}$ d{\'e}finis pour $0 \leq j < i \leq n$, de source $j$ et de
but $i$, de degr{\'e} $k_i-k_j +1$ et soumis aux relations
$$d(Q)+Q^2=0\,,$$
o{\`u} $Q$ est la matrice strictement triangulaire inf{\'e}rieure de
coefficients $q_{i,j}$ pour $0 \leq j < i \leq n$.

Soit $\mathrm{cone}_n(k_0, \ldots , k_n)$ la sous-dg-cat{\'e}gorie pleine de la
dg-cat{\'e}gorie des dg-modules ({\`a} droite) sur $\mathcal{M}_n(k_0, \ldots , k_n)$
dont les objets sont les $\hat{l}$, $0 \leq l \leq n$, et le {\em c{\^o}ne
it{\'e}r{\'e}} $X_n$, qui a le m{\^e}me module gradu{\'e} sous-jacent que le dg-module
$$ X=\bigoplus^n_{l=0} \hat{l}\left[ k_l \right]$$
et dont la diff{\'e}rentielle est $d_X+\widehat{Q}$.

Soit $L:\mathcal{A} \rightarrow \mathrm{cone}_n(k_0, \ldots , k_n)$ le dg
foncteur qui envoie $3$ sur $X_n$. On consid{\`e}re la somme amalgam{\'e}e
$$\xymatrix{
\mathcal{A} \ar[r]^-{L} \ar[d]_F \ar@{}[dr]|{\lrcorner} & \mathrm{cone}_n(k_0, \ldots , k_n) \ar[d] \\
\mathcal{K} \ar[r] & \mathrm{cone}_n(k_0, \ldots ,
k_n) \amalg_{\mathcal{A}} \mathcal{K}
}$$
et on d{\'e}finit $\mathrm{coneh}_n(k_0, \ldots , k_n)$ comme la sous-dg-cat{\'e}gorie
pleine de la somme amalgam{\'e}e dont les objets sont les images des
objets $\hat{l}$, $0 \leq l \leq n$, de $\mathrm{cone}_n(k_0, \ldots , k_n)$
et l'image $X\!h\,_n$ de l'object $2$ de $\mathcal{K}$. On note $I_n(k_0, \ldots , k_n)$ le dg foncteur fid{\`e}le, mais non plein, de
$\mathcal{M}_n(k_0, \ldots , k_n)$ dans $\mathrm{coneh}_n(k_0, \ldots , k_n)$.

\begin{theoreme}\label{theorem}
 Si on consid{\`e}re pour cat{\'e}gorie $\mathcal{C}$ la
  cat{\'e}gorie $\dgcat$, pour classe $W$ la sous-cat{\'e}gorie de $\dgcat$ des
dg-foncteurs quasi-{\'e}quiconiques, pour classe $J$ les dg-foncteurs $F$ et $R(n)$,
$n\in \mathbb{Z}$, (voir section~\ref{secdef}), $F(n)$, $n \in \mathbb{Z}$, et $I_n(k_0,
 \ldots , k_n)$ et pour classe $I$ les dg-foncteurs $Q$ et $S(n)$,
$n\in \mathbb{Z}$, (voir section~\ref{secdef}) alors les conditions du th{\'e}or{\`e}me~\ref{thm} sont satisfaites.
\end{theoreme}

{\`A} la proposition~\ref{fibrants} apr{\`e}s la d{\'e}monstration du th{\'e}or{\`e}me,
nous allons donner une description explicite des objets fibrants de la
structure de cat{\'e}gorie de mod{\`e}les obtenue ainsi.

On observe facilement que les conditions $\mbox{{\it (i)}}$,
$\mbox{{\it (ii)}}$ et $\mbox{{\it (iii)}}$ du th{\'e}or{\`e}me~\ref{thm} sont verifi{\'e}es.

D{\'e}montrons maintenant que   $J-\mbox{inj}\cap W = I-\mbox{inj}$.
Pour cela, on consid{\`e}re la classe $\mbox{{\bf Surj}}$ form{\'e}e des foncteurs
$G : \mathcal{H} \rightarrow \mathcal{I}$ dans $\dgcat$ qui
v{\'e}rifient~:
\begin{itemize}
\item[-] $G$ induit une surjection de l'ensemble des objets de $\mathcal{H}$ sur l'ensemble des
objets de $\mathcal{I}$ et
\item[-] $G$ induit des quasi-isomorphismes surjectifs
dans les complexes de morphismes.
\end{itemize}

\begin{lemme}\label{I-inj-Surj}
On a $\mbox{{\bf Surj}} = I-\mbox{inj}$.
\end{lemme}
\begin{proof}[D{\'e}monstration]
 La classe $Q-\mbox{drt}$ est form{\'e}e des foncteurs qui
  sont surjectifs au niveau des objets. La classe
    $S(n)-\mbox{drt}, \,n \in \mathbb{Z}$ est form{\'e}e des foncteurs qui sont des
  quasi-isomorphismes surjectifs au niveau des complexes de
  morphismes. En effet, un carr{\'e} commutatif dans $\dgcat$
$$
\xymatrix{
\mathcal{C}(n) \ar[r]^{D} \ar[d]_{S(n)} & \mathcal{H} \ar[d]^{G} \\
\mathcal{P}(n) \ar[r]^{E} & \mathcal{I}
}
$$
correspond {\`a} la donn{\'e}e d'un carr{\'e} commutatif dans la cat{\'e}gorie des complexes
$$
\xymatrix{
S^{n-1} \ar@{^{(}->}[d]_{i_n} \ar[r]^-{D} &
  \mathrm{Hom}_{\mathcal{H}}(D(8),D(9)) \ar@{|->}[d]^{G}\\
D^{n} \ar[r]^-{E} & \mathrm{Hom}_{\mathcal{I}}(E(6),E(7))
}
$$
o{\`u} $D(8)$ et  $D(9)$ sont des objets
quelconques dans $\mathcal{H}$. La propri{\'e}t{\'e} r{\'e}sulte de la
caract{\'e}risation des quasi-isomorphismes surjectifs dans la cat{\'e}gorie
des complexes sur $k$. Voir \protect{\cite[2.3.5]{Hovey}}.
\end{proof}

\begin{lemme}\label{J-inj-Surj 1}
On a $\mbox{{\bf Surj}} \supseteq J-\mbox{inj} \cap W$.
\end{lemme}
\begin{proof}[D{\'e}monstration]
Soit $L$ un dg-foncteur de $\mathcal{D}$ vers
$\mathcal{S}$ qui appartient {\`a} $J-\mbox{inj}\cap W$. La classe
$R(n)-\mbox{drt}$ est form{\'e}e des dg-foncteurs surjectifs au niveau des
complexes de morphismes. Comme $L \in W$, il suffit de montrer que $L$
est surjectif au niveau des objets. Soit $E \in \mathcal{S}$ un
object quelconque. Comme $L \in W$, il existe $C \in
\mbox{pre-tr}(\mathcal{D})$ et un morphisme ferm{\'e} $q$ de
$\mbox{pre-tr}(\mathcal{S})$ qui fournit un isomorphisme entre l'image
par $\mbox{pre-tr}(L)$ de $C$ et $E$ dans
$\mathrm{H}^0(\mbox{pre-tr}(\mathcal{S}))$.  On identifie les
dg-cat{\'e}gories $\mathcal{D}$ et $\mathbb{Z}\,\mathcal{D}$ avec leurs
images dans $\mbox{pre-tr}(\mathcal{D})$. Il existe alors trois possibilit{\'e}s~:
\begin{enumerate}
\item l'objet $C$ appartient {\`a} la dg-cat{\'e}gorie $\mathcal{D}$. Alors on
  est dans les conditions de l'inclusion $\subseteq$ du lemme \ref{J-inj-Surj}.
\item l'objet $C$ appartient {\`a} la dg-cat{\'e}gorie $\mathbb{Z}\,\mathcal{D}$. Alors $C$ est de la forme
  $C=(\overline{C}, n)$, o{\`u} $ \overline{C} \in \mathcal{D}$ et $n \in \mathbb{Z}$.
On a la situation suivante
$$\xymatrix{
C=((\overline{C}), n) \ar@{|->}[d]^{\mathbb{Z}(L)} & \\
\mathbb{Z}(L)(C)=(L(\overline{C}), n) \ar[r]^-{q} & (E,0).
}$$
Ainsi, $q$ est l'image de $f$ par un dg-foncteur de $\mathcal{K}(-n)$
vers $\mathcal{S}$. Comme on a $L \in J-\mbox{inj}$, on peut relever
le morphisme $q$ et par cons{\'e}quence l'objet $E$.
\item l'objet $C$ appartient {\`a} la dg-cat{\'e}gorie
  $\mbox{pre-tr}(\mathcal{D})$ mais non pas {\`a} la dg-cat{\'e}gorie
  $\mathbb{Z}\,\mathcal{D}$. On sait d'apr{\`e}s \cite{Bon-Kap} que, dans
  $\mbox{pre-tr}(\mathcal{D})$, l'objet $C$ s'{\'e}crit d'une fa{\c c}on
  canonique comme un c{\^o}ne it{\'e}r{\'e} sur des morphismes de
  $\mathcal{D}$. Comme le dg-foncteur $\mbox{pre-tr}(L)$ pr{\'e}serve les c{\^o}nes, l'object
  $\mbox{pre-tr}(L)(C)$ s'identifie au c{\^o}ne it{\'e}r{\'e} sur leurs images par
  $L$. On peut donc construire le carr{\'e} commutatif suivant
$$\xymatrix @R=1pc{
*+<1pc>{\mathcal{M}_n(k_0, \ldots , k_n)} \ar[r] \ar@{^{(}->}[dd]_{can}
& \mathcal{D} \ar[d]^L\\
 & \mathcal{S} \ar[d] \\
\mathrm{cone}_n(k_0, \ldots , k_n) \ar[r]^-H & \mbox{pre-tr}(\mathcal{S})\,.
}$$
Ainsi, $q$ est l'image de $f$ par un dg-foncteur de $\mathcal{K}$ vers
$\mbox{pre-tr}(\mathcal{S})$ qui envoie l'object $1$ sur $H(X_n)$ et l'object $2$
sur l'object $E$. Le dg-foncteur $H$ s'{\'e}tend donc en un dg-foncteur
$\overline{H}$ de $\mathrm{cone}_n(k_0, \ldots , k_n) \amalg_{\mathcal{A}}
\mathcal{K}$ vers $\mbox{pre-tr}(\mathcal{S})$. On observe que la
restriction de $\overline{H}$, qu'on note $\widetilde{H}$, {\`a} la dg-cat{\'e}gorie
pleine $\mathrm{coneh}_n(k_0, \ldots , k_n)$, a son image dans la dg-cat{\'e}gorie
$\mathcal{S}$. Cela permet de construire le diagramme commutatif suivant
$$\xymatrix@!0 @R=3pc @C=4pc{
 & & \mathcal{M}_n \ar[rr] \ar[dr]_-{I_n} \ar[dl]^-{can} &
 & \mathcal{D} \ar[d]^L \\
\mathcal{A} \ar@{}[drr]|{\lrcorner} \ar[r] \ar[dr]_-F & \mathrm{cone}_n \ar[dr] \ar[drrr]^(0.3){H} & &
 \mathrm{coneh}_n \ar[r]^-{\widetilde{H}}
 \ar@{^{(}->}[dl]|\hole \ar@{.>}[ur]  & \mathcal{S} \ar[d] \\
 & \mathcal{K} \ar[r] & \mathrm{cone}_n \amalg_{\mathcal{A}} \mathcal{K} \ar[rr]_-{\overline{H}} &  & \mbox{pre-tr}(\mathcal{S})\,.
}$$
Comme $L$ appartient {\`a} $J-\mbox{inj}$, l'objet $E$ est bien l'image d'un objet de $\mathcal{D}$.
\end{enumerate}
\end{proof}
Soit $\mathcal{B}$ une dg-cat{\'e}gorie et $C$ un dg-foncteur de $\mathcal{M}_n(k_0, \ldots , k_n)$ vers $\mathcal{B}$.

\begin{remarque}
Le dg-foncteur $C$ se prolonge en un dg-foncteur
$$ C_!: \mathrm{Mod}\,\mathcal{M}_n(k_0, \ldots , k_n)
\rightarrow \mathrm{Mod}\,\mathcal{B}$$
adjoint {\`a} gauche de la restriction le long de $C$. Le dg-foncteur $C$
{\em admet un c{\^o}ne} $Y$ dans $\mathcal{B}$ si $C_!(X_n)$  est
repr{\'e}sentable par un objet $Y$ de $\mathcal{B}$. C'est le cas ssi le
foncteur $C$ se prolonge en un dg-foncteur
$$\widetilde{C}: \mathrm{cone}_n(k_0, \ldots , k_n) \rightarrow \mathcal{B}\,.$$
Si $G: \mathcal{B} \rightarrow \mathcal{B}'$ est un dg-foncteur et
$C$ admet le c{\^o}ne $Y$ dans $\mathcal{B}$, alors $G\circ C$ admet le
c{\^o}ne $G(Y)$ dans $\mathcal{B}'$.
Si $C$ se factorise par la sous-dg-cat{\'e}gorie pleine
$\mathcal{B}\,\backslash \{Y \}$ et admet le c{\^o}ne $Y$ dans
$\mathcal{B}$, alors le carr{\'e}
$$\xymatrix{
*+<1pc>{\mathcal{M}_n(k_0, \ldots , k_n)} \ar[r] \ar@{^{(}->}[d] \ar@{}[dr]|{\lrcorner} & *+<1pc>{\mathcal{B}\,\backslash
\{Y \}} \ar@{^{(}->}[d]\\
\mathrm{cone}_n(k_0, \ldots , k_n) \ar[r] & \mathcal{B}
}$$
est cocart{\'e}sien.
\end{remarque}
Ces remarques impliquent le lemme suivant.
\begin{lemme} \label{lemme-clef}
Le carr{\'e} commutatif
$$\xymatrix{
*+<1pc>{\mathcal{M}_n(k_0, \ldots , k_n)} \ar[rr]^-{I_n(k_0, \ldots , k_n)}
\ar@{^{(}->}[d]_-{can}  \ar@{}[drr]|{\lrcorner} & &  *+<1pc>{\mathrm{coneh}_n(k_0, \ldots , k_n)}
\ar@{^{(}->}[d] \\
\mathrm{cone}_n(k_0, \ldots , k_n) \ar[rr]  & &  \mathrm{cone}_n(k_0, \ldots , k_n) \amalg_{\mathcal{A}}\mathcal{K}
}$$
est cocart{\'e}sien.
\end{lemme}

Soit $F: \mathcal{A}\rightarrow \mathcal{B}$ un dg-foncteur.
\begin{remarque}\label{rem-pre}
Rappelons de \cite{Bon-Kap} \cite[2.4]{Drinfeld} que les objets de
$\mbox{pre-tr}(\mathcal{A})$ sont des expressions formelles
$C=(\bigoplus^n_{i=0} C_i[r_i],q)$, o{\`u} $C_i\in \mathcal{A}$, $r_i \in
\mathbb{Z}$, $n\geq 0$, $q=(q_{ij})$, $q_{ij} \in
\mathsf{Hom}^1_{\mathcal{A}}(C_j,C_i)[r_i-r_j]$, $q_{ij}=0$ pour $i\geq
  j$, $dq +q^2 =0$ et le complexe de morphismes
  $\mathsf{Hom}_{\mbox{pre-tr}(\mathcal{A})}(C,C')$ est un espace de
  matrices formelles $f=(f_{ij})$, $f_{ij} \in
  \mathsf{Hom}(C_j,C_i')[r_i'-r_j]$, voir \cite{Bon-Kap}. Le
  dg-foncteur $\mbox{pre-tr}(F)$ envoie donc les expressions formelles
en expressions formelles. Cela implique que la classe $\mbox{{\bf
    Surj}}$ est stable par application du foncteur $\mbox{pre-tr}(-)$ et
qu'un objet de $\mbox{pre-tr}(\mathcal{A})$ dont l'image par
$\mbox{pre-tr}(F)$ est dans $\mathcal{B}$ est en fait dans
$\mathcal{A}$.
\end{remarque}

\begin{lemme}\label{J-inj-Surj 2}
On a $\mbox{{\bf Surj}} \subseteq J-\mbox{inj} \cap W$.
\end{lemme}
\begin{proof}[D{\'e}monstration]
Soit $H$ un dg-foncteur
de $\mathcal{N}$ vers $\mathcal{E}$ dans la classe $\mbox{{\bf Surj}}$. Comme $H$ est surjectif au niveau des objets et un
quasi-isomorphisme au niveau des complexes de morphismes, on a $H \in
W$. La classe $R(n)-\mbox{drt}$ est form{\'e}e des dg-foncteurs surjectifs
au niveau des complexes de morphismes. Il suffit de montrer que $H$ a
la propri{\'e}t{\'e} de rel{\`e}vement {\`a} droite par rapport {\`a} $F$, $F(n)$, $n \in
\mathbb{Z}$, et $I_n(k_0, \ldots , k_n)$. On consid{\`e}re ces trois cas~:
\begin{enumerate}
\item on sait par le lemme~\ref{J-inj-Surj} que $H \in F-\mbox{drt}$;
\item en effet, on a $H \in F(n)-\mbox{drt}$, $n \in \mathbb{Z}$, par un argument compl{\`e}tement
  analogue au cas pr{\'e}c{\'e}dent;
\item on consid{\`e}re le diagramme suivant~:
$$\xymatrix{
\mathcal{M}_n(k_0,\ldots, k_n) \ar[r] \ar[d]_{I_n(k_0,
  \ldots , k_n)} & \mathcal{N} \ar[d]^H\\
\mathrm{coneh}_n(k_0, \ldots , k_n) \ar[r]_-{L} & \mathcal{E}
}$$
Le lemme~\ref{lemme-clef} permet d'{\'e}tendre le dg-foncteur $L$ en un dg-foncteur
$\overline{L}$ de $\mathrm{cone}_n(k_0, \ldots , k_n)\amalg_{\mathcal{A}}
 \mathcal{K}$ vers $\mbox{pre-tr}(\mathcal{E})$.
$$\xymatrix@!0 @R=2pc @C=4pc{
&  & \mathcal{M}_n \ar[rr] \ar[dl] \ar[dd]|\hole & &  *+<1pc>{\mathcal{N}}
\ar[dd]^-H \ar@{^{(}->}[dl] \\
\mathcal{A} \ar@{}[ddr]|{\lrcorner} \ar[dd]_F \ar[r] & \mathrm{cone}_n \ar[dd] \ar[rr]  & &
\mbox{pre-tr}(\mathcal{N}) \ar[dd] & \\
& & \mathrm{coneh}_n \ar[rr]|\hole^(0.3){L} \ar[dl] &    & *+<1pc>{\mathcal{E}}
\ar@{^{(}->}[dl] \\
\mathcal{K} \ar[r] & \mathrm{cone}_n \amalg_{\mathcal{A}} \mathcal{K}
\ar[rr]^-{\overline{L}} &   & \mbox{pre-tr}(\mathcal{E}) &
}$$
Puisque $H$ appartient {\`a} la classe $\mbox{{\bf Surj}}$ et cette
classe est stable par application du foncteur
$\mbox{pre-tr}(-)$, voir remarque~\ref{rem-pre}, on peut
relever le dg-foncteur $\overline{L}$ vers
$\mbox{pre-tr}(\mathcal{N})$.
Finalement, par la remarque~\ref{rem-pre}, la restriction du rel{\`e}vement du dg-foncteur $\overline{L}$ {\`a} la sous-dg-cat{\'e}gorie pleine $\mathrm{coneh}_n(k_0, \ldots
,k_n)$ nous fournit un rel{\`e}vement du dg-foncteur $L$ vers la
dg-cat{\'e}gorie $\mathcal{N}$.
\end{enumerate}
\end{proof}

\begin{lemme} \label{J-cell-dans-W} On a $J-\mbox{cell}\subset W$.
\end{lemme}
\begin{proof}[D{\'e}monstration]
On sait d{\'e}j{\`a} par le lemme~\ref{J-cell-W} que les classes $F-\mbox{cell}$ et
$R(n)-\mbox{cell}$ sont form{\'e}es des quasi-{\'e}quivalences et donc
contenues dans la classe $W$. Soit $n \in \mathbb{Z}$ et $T:\mathcal{A}
\rightarrow \mathcal{J}$ un dg-foncteur quelconque. On consid{\`e}re la
somme amalgam{\'e}e suivante
$$\xymatrix{
\mathcal{A} \ar[d]_{F(n)} \ar[r]^T \ar@{}[dr]|{\lrcorner} & \mathcal{J} \ar[d]^{inc}\\
\mathcal{K}(n) \ar[r] & \mathcal{U}
}$$
dans $\dgcat$. Il s'agit de v{\'e}rifier que $inc \in W$. Il
faut verifier deux conditions:
\begin{enumerate}
\item L'inclusion $\mathrm{Hom}_{\mathcal{J}}(X,Y) \rightarrow \mathrm{Hom}_{\mathcal{U}} (inc(X), inc(Y))$ est un quasi-isomorphisme
  pour tous $X,Y \in \mathcal{J}$. Pour cela, on raisonne comme dans
  le cas o{\`u} on a le dg-foncteur $F$ au lieu de
  $F(n)$, voir lemme~\ref{J-cell-W}.
\item Le foncteur
  $\mathrm{H}^0(\mathbb{Z}(inc)):\mathrm{H}^0(\mathbb{Z}\,\mathcal{J})
  \rightarrow \mathrm{H}^0(\mathbb{Z}\,\mathcal{U})$ est essentiellement
  surjectif. En effet, la dg-cat{\'e}gorie $\mathcal{U}$ s'obtient {\`a} partir de
  $\mathcal{J}$ en rajoutant un nouvel objet $S$ homotopiquement {\'e}quivalent {\`a} un d{\'e}cal{\'e} de l'objet $T(3)$. Plus
  pr{\'e}cis{\'e}ment, l'object $(S,0)$
  devient isomorphe {\`a} $(T(3),-n)$ dans $\mathrm{H}^0(\mathbb{Z}\,\mathcal{U})$.
\end{enumerate}

Soit maintenant $T:\mathcal{M}_n(k_0,\ldots, k_n) \rightarrow \mathcal{J}$ un
dg-foncteur quelconque. On consid{\`e}re la somme amalgam{\'e}e suivante
$$\xymatrix{
\mathcal{M}_n(k_0,\ldots, k_n) \ar[d]_{I_n(k_0, \ldots ,
  k_n)} \ar[r]^-{T} \ar@{}[dr]|{\lrcorner}   & \mathcal{J} \ar[d]^{inc} \\
\mathrm{coneh}_n(k_0, \ldots , k_n) \ar[r] & \mathcal{U}
}$$
dans $\dgcat$. Il s'agit de v{\'e}rifier que $inc \in W$. Il
faut verifier deux conditions:
\begin{enumerate}
\item L'inclusion $\mathrm{Hom}_{\mathcal{J}}(X,Y) \rightarrow \mathrm{Hom}_{\mathcal{U}} (inc(X), inc(Y))$ est un quasi-isomor\-phisme
  pour tous $X,Y \in \mathcal{J}$. Dans le diagramme commutatif
  suivant, la dg-cat{\'e}gorie $\mathcal{N}$ est la somme amalgam{\'e}e de
  $\mathcal{J}$ et $\mathrm{cone}_n$, et $\mathcal{E}$ est la somme
  amalgam{\'e}e de $\mathcal{N}$ et $\mathcal{K}$.
$$\xymatrix@!0 @R=2pc @C=4pc{
&  & \mathcal{M}_n \ar[rr]^-T \ar[dl] \ar[dd]|\hole & &  \mathcal{J}
\ar[dd]^-{inc} \ar[dl]^L \\
\mathcal{A} \ar@{}[ddr]|{\lrcorner} \ar[dd]_F \ar[r] & \mathrm{cone}_n \ar[dd] \ar[rr] & &
\mathcal{N} \ar[dd]^(0.3){H}  & \\
& & \mathrm{coneh}_n \ar[rr]|\hole \ar@{^{(}->}[dl] &    & *+<1pc>{\mathcal{U}}
\ar@{^{(}->}[dl] \\
\mathcal{K} \ar[r] & \mathrm{cone}_n \amalg_{\mathcal{A}} \mathcal{K}
\ar[rr] &   & \mathcal{E} &
}$$
On peut identifier
 la dg-cat{\'e}gorie $\mathcal{N}$ {\`a} la sous-dg-cat{\'e}gorie pleine des dg-modules ({\`a}
 droite) sur $\mathcal{J}$ dont les objets sont les dg-modules
 represent{\'e}es et le $n$-i{\`e}me c{\^o}ne it{\'e}r{\'e} sur l'image de la famille
 $T(q_{i,j})$ par le foncteur de Yoneda. Le dg-foncteur $L$ s'identifie alors au plongement
 de Yoneda et il est donc pleinement fid{\`e}le. Le dg-foncteur $H$ s'identifie au dg-foncteur de $\mathcal{N}$ vers
 la somme amalgam{\'e}e avec $\mathcal{K}$. On sait d'apr{\`e}s le lemme~\ref{J-cell-W}
 que $H$ est donc une quasi-{\'e}quivalence. Notons maintenant que la
 dg-cat{\'e}gorie $\mathcal{U}$ s'identifie {\`a} la sous-dg-cat{\'e}gorie pleine
 de $\mathcal{E}$ dont les objets sont ceux dans l'image du
 dg-foncteur
$$ \mathrm{coneh}_n \hookrightarrow \mathrm{cone}_n
\amalg_{\mathcal{A}} \mathcal{K} \rightarrow \mathcal{E}\,.$$
Cela implique que $inc$ satisfait la condition.

\item  Le foncteur
  $\mathrm{H}^0(\mbox{pre-tr}(inc)):\mathrm{H}^0(\mbox{pre-tr}(\mathcal{J}))
  \rightarrow \mathrm{H}^0(\mbox{pre-tr}(\mathcal{U}))$ est essentiellement
  surjectif. En effet, la dg-cat{\'e}gorie $\mathcal{U}$ poss{\`e}de un object
  de plus que la dg-cat{\'e}gorie $\mathcal{J}$, {\`a} savoir l'image de $X\!h\,_n$. Soit $T(X_n)$,
  dans $\mbox{pre-tr}(\mathcal{J})$, le $n$-i{\`e}me c{\^o}ne it{\'e}r{\'e} associ{\'e} au
  dg-foncteur $T$. Alors le lemme~\ref{lemme-clef} nous montre que l'image de $T(X_n)$ par le dg-foncteur
  $\mbox{pre-tr}(inc)$ et $X\!h\,_n$ sont homotopiquement {\'e}quivalents
  dans $\mbox{pre-tr}(\mathcal{U})$.
\end{enumerate}
\end{proof}

Nous avons v{\'e}rifi{\'e} que $J-\mbox{cell}\subseteq W$ (lemme~\ref{J-cell-dans-W}) et
que $I-\mbox{inj}$ est {\'e}gal {\`a} $J-\mbox{inj}\cap W$
(lemmes~\ref{I-inj-Surj}, \ref{J-inj-Surj 1} et \ref{J-cell-dans-W}).
Ces conditions impliquent celles du th{\'e}or{\`e}me~\ref{thm}.

\begin{prop}\label{fibrants}
Les objets fibrants, pour la structure de cat{\'e}gorie de mod{\`e}les de
Quillen du th{\'e}or{\`e}me~\ref{theorem}, sont les dg-cat{\'e}gories
$\mathcal{B}$ telles que l'image essentielle du plongement $\mathrm{H}^0(\mathcal{B}) \hookrightarrow
\mathcal{D}(\mathcal{B})$ est stable par suspensions et c{\^o}nes.
\end{prop}

\begin{proof}[D{\'e}monstration]
Soit $n \in \mathbb{Z}$ et $\mathcal{B}$ une dg-cat{\'e}gorie
fibrante. Soit $X\in \mathsf{H}^0(\mathcal{B})$ et $X[n] \in
\mathcal{D}(\mathcal{B})$ la $n$-i{\`e}me suspension de $X$ dans
$\mathcal{D}(\mathcal{B})$. L'objet $X$ correspond {\`a} un dg-foncteur
$L:\mathcal{A} \rightarrow \mathcal{B}$ et puisque $\mathcal{B}$
est fibrante on dispose du diagram commutatif suivant
$$
\xymatrix{
\mathcal{A} \ar[r]^L \ar[d]_{F(n)} & \mathcal{B} \\
\mathcal{K}(n) \ar@{-->}[ur]_{\overline{L}} & .
}
$$
Notons que l'objet $\overline{L}(2)$ correspond {\`a} la $n$-i{\`e}me
suspension de $X$ dans la cat{\'e}gorie
$\mathsf{H}^0(\mbox{Mod}(\mathcal{B}))$ et donc dans
$\mathcal{D}(\mathcal{B})$. Les objets $\overline{L}(2)$ et $X[n]$
sont donc isomorphes dans $\mathcal{D}(\mathcal{B})$.

Soit maintenant $[f]$ un morphisme dans $\mathsf{H}^0(\mathcal{B})$ et
$\mbox{cone}([f])$ le c{\^o}ne de $[f]$ dans
$\mathcal{D}(\mathcal{B})$. Le morphisme $f \in
\mathsf{Z}^0(\mathcal{B})$ correspond {\`a} un dg-foncteur
$L:\mbox{cone}_0(0) \rightarrow \mathcal{B}$ et puisque
$\mathcal{B}$ est fibrante on dispose du diagram commutatif suivant
$$
\xymatrix{
\mbox{cone}_0(0) \ar[r]^L \ar[d]_{I_0(0)} & \mathcal{B} \\
\mbox{coneh}_0(0) \ar@{-->}[ur]_{\overline{L}} & .
}
$$
Notons maintenant que l'objet $\overline{L}(X\!h\,_0)$ correspond au
c{\^o}ne de $[f]$ dans $\mathsf{H}^0(\mbox{Mod}(\mathcal{B}))$ et
donc dans $\mathcal{D}(\mathcal{B})$. Les objets
$\overline{L}(X\!h\,_0)$ et $\mbox{cone}([f])$ sont donc isomorphes
dans $\mathcal{D}(\mathcal{B})$. Cela montre que l'image essentielle
du plongement $\mathrm{H}^0(\mathcal{B}) \hookrightarrow
\mathcal{D}(\mathcal{B})$ est stable par suspensions et c{\^o}nes.

Soit maintenant $\mathcal{B}$ une dg-cat{\'e}gorie telle que l'image essentielle
du plongement $\mathrm{H}^0(\mathcal{B}) \hookrightarrow
\mathcal{D}(\mathcal{B})$ est stable par suspension et c{\^o}nes.
On consid{\'e}re les diagrammes suivants~:
$$
\xymatrix{
\mathcal{A} \ar[d]_{F(n)} \ar[r]^L & \mathcal{B} & &
\mathcal{M}_n(k_0, \ldots, k_n) \ar[r]^-S \ar[d]_{I_n(k_0,\ldots,k_n)}
& \mathcal{B} \\
\mathcal{K}(n) & & & \mbox{cone}h_n(k_0,\ldots,k_n) & .
}
$$
Puisque la suspension dans $\mathcal{D}(\mathcal{B})$ des
$\mathcal{B}$-modules representables s'identifie {\`a} la suspension dans
$\mathsf{H}^0(\mbox{Mod}(\mathcal{B}))$, il existe un dg-foncteur
$\overline{L}: \mathcal{K}(n) \rightarrow \mathcal{B}$ tel que $\overline{L}\circ F(n) = L$. D'une fa{\c c}on
analogue les c{\^o}nes it{\'e}r{\'e}s dans $\mathcal{D}(\mathcal{B})$ entre
$\mathcal{B}$-modules representables s'identifient {\`a} les c{\^o}nes it{\'e}r{\'e}s
dans $\mathsf{H}^0(\mbox{Mod}(\mathcal{B}))$. Il existe donc un
dg-foncteur $\overline{S}:\mbox{cone}h_n(k_0,\ldots, k_n) \rightarrow
\mathcal{B}$ tel que $\overline{S} \circ I_n(k_0,\ldots,k_n)=S$. Cela
montre que $\mathcal{B}$ est une dg-cat{\'e}gorie fibrante et la
proposition est d{\'e}montr{\'e}e.
\end{proof}
\begin{remarque}
On remarque que les conditions de la proposition~\ref{fibrants} sont {\'e}quivalents {\`a} ce que
le dg-foncteur $\mathcal{B} \hookrightarrow
\mbox{pre}$-$\mbox{tr}(\mathcal{B})$, voir \cite{Bon-Kap}, soit une quasi-{\'e}quivalence.
\end{remarque}

\begin{notation}
On note $\mathcal{Q}eq$ la classe des quasi-{\'e}quivalences, voir
definition~\ref{quasieq}, et $\mathcal{Q}ec$ la classe plus large des
dg-foncteurs quasi-{\'e}quiconiques. On note par $\mathsf{Heq}$ et
$\mathsf{Hec}$ les cat{\'e}gories homotopiques de $\dgcat$ par rapport {\`a}
ces classes et par $\mathcal{B}_{fib}$ un remplacement fibrant
de la dg-cat{\'e}gorie $\mathcal{B}$ par rapport {\`a} la structure de
cat{\'e}gorie de mod{\`e}les de Quillen quasi-{\'e}quiconique.
\end{notation}

\begin{lemme}\label{homotopes}
Les dg-cat{\'e}gories $\mathcal{B}_{fib}$ et
$\mbox{pre-tr}(\mathcal{B})$ sont isomorphes dans $\mathsf{Heq}$.
\end{lemme}
\begin{proof}[D{\'e}monstration]
On consid{\`e}re le carr{\'e} commutatif
$$
\xymatrix{
*+<2pc>{\mathcal{B}} \ar[r] \ar@{>->}[d]_-F & \mbox{pre-tr}(\mathcal{B})
\ar[d]^-{pre-tr(F)} \\
*+<1pc>{\mathcal{B}_{fib}} \ar@{^{(}->}[r] &
\mbox{pre-tr}(\mathcal{B}_{fib})\,.
}
$$
Le dg-foncteur $F$ appartient {\`a} la classe $\mathcal{Q}ec$ et donc par
definition $\mbox{pre-tr}(F)$ appartient {\`a} la classe
$\mathcal{Q}eq$. Puisque la dg-cat{\'e}gorie $\mathcal{B}_{fib}$ est
fibrante, la proposition~\ref{fibrants} implique que le dg-foncteur $\mathcal{B}_{fib}
\hookrightarrow \mbox{pre-tr}(\mathcal{B}_{fib})$ appartient {\`a} la
classe $\mathcal{Q}eq$. Cela implique le lemme.
\end{proof}

\begin{prop}\label{Bousf}
La structure de cat{\'e}gorie de mod{\`e}les de Quillen quasi-{\'e}quiconique est
la localisation de Bousfield {\`a} gauche, voir \cite[3.3.1]{Hirschhorn},
de celle du th{\'e}or{\`e}me~\ref{mal}, par rapport {\`a} l'ensemble
$$ S: = \{ F(1), F(-1), I_0(0) \}\,.$$
\end{prop}

\begin{proof}[D{\'e}monstration]
 On commence par montrer qu'une dg-cat{\'e}gorie $\mathcal{B}$ est
 $S$-locale, voir \cite[3.4]{Hirschhorn}, si et seulement si elle est
 fibrante pour la structure quasi-{\'e}quiconique. Notons que puisque
 les deux structures de mod{\`e}les de Quillen sur $\dgcat$ ont les m{\^e}mes
 cofibrations, et donc les m{\^e}mes fibrations triviales, le foncteur de
 remplacement cofibrant simplicial $\Gamma^{\ast}$, voir
 \cite{Hirschhorn}, est le m{\^e}me dans les deux situations. Puisque
 $S$ est form{\'e} de foncteurs quasi-{\'e}quiconiques, cela
 implique que si $\mathcal{B}$ est fibrante pour la structure
 quasi-{\'e}quiconique, alors pour tout dg-foncteur
 $L:\mathcal{A} \rightarrow \mathcal{C}$ dans $S$, le morphisme
 induit
$$ \mathsf{Map}(\mathcal{C},\mathcal{B})
\stackrel{\sim}{\longrightarrow}
\mathsf{Map}(\mathcal{A},\mathcal{B})$$
est une {\'e}quivalence faible d'ensembles simpliciaux.

Soit maintenant $\mathcal{B}$ un objet $S$-local dans $\dgcat$.
Montrons qu'il est fibrant. Par la proposition~\ref{fibrants},
il s'agit de montrer que l'image essentielle du plongement
$\mathsf{H}^0(\mathcal{B}) \hookrightarrow \mathcal{D}(\mathcal{B})$
est stable par suspensions et c{\^o}nes. On consid{\'e}re maintenant le
cas des c{\^o}nes (l'argument pour le cas des suspensions est
analogue). La donn{\'e}e d'un morphisme ferm{\'e} $f$ de degr{\'e} $0$ de $\cb$
fournit un morphisme $\mathcal{M}_0(0) \to \cb$ et donc un objet
de $\rep(\mathcal{M}_0(0),\cb)$.
Par la description de $\mathsf{Map}(-,-)$ dans \cite{Toen} et l'hypoth{\`e}se
sur $\cb$, le foncteur
\[
\rep(\mbox{coneh}_0(0), \mathcal{B}) \to \rep(\mathcal{M}_0(0),\mathcal{B})
\]
est une {\'e}quivalence. Cela implique qu'il existe bien dans $\cb$ un objet
$C'(f)$ qui, dans $\cd\cb$, devient isomorphe au c{\^o}ne sur l'image
de $f$ par le foncteur de Yoneda.

On montre maintenant qu'un dg-foncteur est quasi-{\'e}quiconique si et
seulement s'il est une $S$-{\'e}quivalence locale, voir
\cite[3.1.4]{Hirschhorn}. Un dg-foncteur $F$ est une
$S$-{\'e}quivalence locale ssi $\Map(F,\cc)$ est une {\'e}quivalence faible pour tout
$\cc$ $S$-locale, ssi $\Map(F,\cc)$ est une {\'e}quivalence faible pour tout $\cc$ fibrant pour
la structure quasi-{\'e}quiconique. Une fois que les foncteurs
$\mathsf{Map}(-,\mathcal{C})$ sont {\'e}quivalents pour les deux structures, avec
$\mathcal{C}$ fibrante (pour la structure quasi-{\'e}quiconique), $F$ est
quasi-{\'e}quiconique ssi est une $S$-{\'e}quivalence locale.
\end{proof}

Soient $\mathcal{A}$ et $\mathcal{B}$ des dg-cat{\'e}gories et $\mbox{can}_1:\mathsf{Heq} \rightarrow \mathsf{Hec}$ le
foncteur canonique.
\begin{corollaire}\label{adjonct1}
On a une adjonction
$$
\mathrm{Hom}_{\mathsf{Hec}}(\mbox{can}_1(\mathcal{A}),
\mathcal{B}) \stackrel{\sim}{\longrightarrow} \mathrm{Hom}_{\mathsf{Heq}}(\mathcal{A},
\mathcal{B}_{fib})\,.
$$
\end{corollaire}

\begin{notation}\label{notexact}
On note $\mathsf{Heq}_{ex}$ la sous-cat{\'e}gorie pleine de
$\mathsf{Heq}$ dont les objets sont les dg-cat{\'e}gories exactes,
voir \cite{cyclichomology}. On note $\mbox{inc}: \mathsf{Heq}_{ex} \hookrightarrow
\mathsf{Heq}$ l'inclusion. 
\end{notation}

On remarque que le corollaire~\ref{adjonct1} et le lemme~\ref{homotopes} nous
fournissent une {\'e}quivalence de cat{\'e}gories entre $\mathsf{Hec}$ et
$\mathsf{Heq}_{ex}$. On a donc le diagramme suivant
$$
\xymatrix{
\mathsf{Heq} \ar[d]_-{can_1} &
\mathsf{Heq}_{ex} \ar@{^{(}->}[l]_-{inc}\\
\mathsf{Hec}  \ar[ur]^-{\sim}_{\mbox{pre-tr}(-)} & \,.
}
$$
On note $\mathsf{rep}_{ec}(\mathcal{A}, \mathcal{B})$ la sous-cat{\'e}gorie pleine
de la cat{\'e}gorie d{\'e}riv{\'e}e $\mathcal{D}(\mathcal{A}^{op} \otimes^{\mathbb{L}} \mathcal{B})$,
  voir \cite{Drinfeld} \cite{DerivingDG}, dont les objets sont les dg
  $\mathcal{A}\mbox{-}\mathcal{B}$-bimodules $X$ tels que $X(?, A)$ est
  isomorphe dans $\mathcal{D}(\mathcal{B})$ {\`a} un objet de l'image du
  dg-foncteur canonique $\mbox{pre-tr}(\mathcal{B})\rightarrow
  \mathrm{Mod}\, \mathcal{B}$, pour tout $A \in \mathcal{A}$. On note
  $\left[X \right]$ la classe d'isomorphisme de $X$ dans
  $\mathsf{rep}_{ec}(\mathcal{A},\mathcal{B})$. Pour une
  cat{\'e}gorie essentiellement petite $\mathcal{C}$, on note
  $\mathrm{Iso}(\mathcal{C})$ l'ensemble de ses classes d'isomorphisme.

\begin{corollaire}\label{adjonct4}
On a une bijection
$$\mathrm{Hom}_{\mathsf{Hec}}(\mathcal{A},
\mathcal{B}) \stackrel{\sim}{\longrightarrow} \mathrm{Iso}(\mathsf{rep}_{ec}(\mathcal{A}, \mathcal{B}))\,.
$$
\end{corollaire}
\begin{proof}[D{\'e}monstration]
D'apr{\`e}s le corollaire~\ref{adjonct1} et le lemme~\ref{homotopes}, $\mathrm{Hom}_{\mathsf{Hec}}(\mathcal{A},
\mathcal{B})$ s'identifie {\`a}  $\mathrm{Hom}_{\mathsf{Heq}}(\mathcal{A},
\mbox{pre-tr}(\mathcal{B}))$. Par \cite{Toen} on sait que
cet ensemble s'identifie {\`a} $\mbox{Iso}(\mathsf{rep}(\mathcal{A},\mbox{pre-tr}(\mathcal{B}))$.
On montre maintenat que la cat{\'e}gorie
$\mathsf{rep}_{ec}(\mathcal{A},\mathcal{B})$ s'identifie {\`a}
$\mathsf{rep}(\mathcal{A},\mbox{pre-tr}(\mathcal{B}))$. Le dg-foncteur
canonique $\mathcal{A}^{op}\otimes \mathcal{B} \rightarrow
\mathcal{A}^{op}\otimes \mbox{pre-tr}(\mathcal{B})$ est pleinement
fid{\`e}le et le foncteur induit
$$ \mathcal{D}(\mathcal{A}^{op} \otimes \mathcal{B}) \rightarrow
\mathcal{D}(\mathcal{A}^{op}\otimes \mbox{pre-tr}(\mathcal{B}))$$
envoie l'ensemble des dg-modules strictement representables dans
$\mathcal{D}(\mathcal{A}^{op}\otimes \mathcal{B})$ vers un ensemble de
petits g{\'e}n{\'e}rateurs. Ce foncteur est donc une {\'e}quivalence
triangul{\'e}e. Via cette {\'e}quivalence la sous-cat{\'e}gorie
$\mathsf{rep}_{ec}(\mathcal{A},\mathcal{B})$ s'identifie {\`a}
$\mathsf{rep}(\mathcal{A}, \mbox{pre-tr}(\mathcal{B}))$.

Cela montre le corollaire.
\end{proof}
\begin{remarque}\label{monoidale}
On sait que les structures de cat{\'e}gorie de mod{\`e}les de Quillen dans
$\dgcat$ du th{\'e}or{\`e}me~\ref{mal} et du th{\'e}or{\`e}me~\ref{theorem} ont les m{\^e}mes
cofibrations et donc les m{\^e}mes fibrations acycliques. Cela a pour
consequence que
\begin{itemize}
\item[-] Le foncteur de remplacement cofibrant simplicial $\Gamma^{\ast}$, voir
\cite{Hovey}, dans $\dgcat$ est le m{\^e}me dans les deux situations. Cela
implique que les corollaires~\ref{adjonct1}, et~\ref{adjonct4} sont encore vrais si on remplace $\mathrm{Hom}$ par
l'espace de morphismes, $\underline{\mathrm{Map}}$, voir \cite{Hovey}, et
$\mbox{Iso}(\mathsf{rep}_{ec}(\mathcal{A},\mathcal{B}))$,
dans le corollaire~\ref{adjonct4}, par le nerf de la sous-cat{\'e}gorie de
$\mathrm{Mod}(\mathcal{A}^{op} \otimes^{\mathbb{L}}\mathcal{B})$ dont les objets
sont les m{\^e}mes que ceux de $\mathsf{rep}_{ec}(\mathcal{A},
\mathcal{B})$ et dont les morphismes sont les quasi-isomorphismes.
\item[-] Le produit tensoriel $-\otimes^{\mathbb{L}}-$ de $\mathsf{Heq}$
  descend {\`a} la localis{\'e}e $\mathsf{Hec}$. La cat{\'e}gorie mono{\"\i}dale sym{\'e}trique $(\mathsf{Hec},\,-\otimes^{\mathbb{L}}-)$ est ferm{\'e}e, voir \cite{Hovey}, et l'espace de
  morphismes interne de $\mathsf{Hec}$, $\mathbb{R}\underline{\mathrm{Hom}}_{\mathsf{Hec}}(\mathcal{A},
    \mathcal{B})$ s'identifie {\`a} $\mathbb{R}\underline{\mathrm{Hom}}_{\mathsf{Heq}}(\mathcal{A},
      \mbox{pre-tr}(\mathcal{B}))$, voir \cite{Toen}.
\end{itemize}
\end{remarque}

\section{Additivisation}
Soient $\mathcal{A}$, $\mathcal{B}$ et $\mathcal{C}$ des
dg-cat{\'e}gories. La composition dans $\mathsf{Hec}$ est induite par le
bifoncteur
$$
\xymatrix@!0 @R=1,5pc @C=12pc{
-\otimes^{\mathbb{L}}_{\mathcal{B}}-: \mathsf{rep}_{ec}(\mathcal{A},
 \mathcal{B}) \times \mathsf{rep}_{ec}(\mathcal{B}, \mathcal{C})
\ar[r]   & \mathsf{rep}_{ec}(\mathcal{A}, \mathcal{C})\\
*+<2pc>{(X , Y)} \ar@{|->}[r]  & X_{cof}\otimes_{\mathcal{B}}Y\, ,
}
$$
o{\`u} $X_{cof}$ est un remplacement cofibrant de $X$ dans la cat{\'e}gorie
$\mathrm{Mod}(\mathcal{A}^{op}\otimes^{\mathbb{L}}\mathcal{B})$,
par rapport {\`a} sa structure de cat{\'e}gorie de mod{\`e}les de Quillen, voir \cite{Toen}.

\begin{remarque}\label{bitri}
Le bifoncteur $-\otimes^{\mathbb{L}}-$ est bi-triangul{\'e}, puisqu'il est induit par le dg-bifoncteur produit tensoriel de bimodules.
\end{remarque}

Soit $\mathsf{Hec}_0$ la cat{\'e}gorie qui a pour objets les petites
dg-cat{\'e}gories et telle que $\mathrm{Hom}_{\mathsf{Hec}_0}(\mathcal{A},
\mathcal{B})$ est le groupe de Grothendieck de la cat{\'e}gorie triangul{\'e}e
$\mathsf{rep}_{ec}(\mathcal{A}, \mathcal{B})$. L'operation de
composition est induite par le bifoncteur $-\otimes^{\mathbb{L}}-$
d'apr{\`e}s la remarque~\ref{bitri}. On dispose d'un foncteur canonique
$\mbox{add}_1:\mathsf{Hec}\rightarrow \mathsf{Hec}_0$.

\begin{lemme}\label{addit}
La cat{\'e}gorie $\mathsf{Hec}_0$ est additive et le foncteur canonique
$\mathsf{Hec}\rightarrow \mathsf{Hec}_0$ transforme les produits finis
et les coproduits finis en sommes directes.
\end{lemme}
\begin{proof}[D{\'e}monstration]
Par construction, les ensembles de morphismes de $\mathsf{Hec}_0$ sont
des groupes ab{\'e}liens et la composition est $\mathbb{Z}$-bilin{\'e}aire. Il
reste {\`a} verifier que $\mathsf{Hec}_0$ poss{\`e}de des sommes directes. En
effet, montrons que la somme directe et le produit direct dans $\mathsf{Hec}_0$
sont les m{\^e}mes que dans $\dgcat$. Puisque l'on a
$$ \mathsf{rep}_{ec}(\mathcal{A}\amalg \mathcal{B}, \mathcal{C}) \backsimeq
\mathsf{rep}_{ec}(\mathcal{A}, \mathcal{C})\times  \mathsf{rep}_{ec}(\mathcal{B}, \mathcal{C})$$
et que le groupe de Grothendieck pr{\'e}serve les produits, on a
$$\mathrm{Hom}_{\mathsf{Hec}_0}(\mathcal{A}\amalg\mathcal{B},
  \mathcal{C}) \backsimeq \mathrm{Hom}_{\mathsf{Hec}_0}(\mathcal{A},
  \mathcal{C})\oplus \mathrm{Hom}_{\mathsf{Hec}_0}(\mathcal{B},
  \mathcal{C})\,.$$
Un argument analogue montre aussi que $\mathcal{A}\times\mathcal{B}$
  est bien le produit dans $\mathsf{Hec}_0$.
\end{proof}

\begin{lemme}\label{tens}
Le produit tensoriel $-\otimes^{\mathbb{L}}-$ de $\mathsf{Hec}$, voir remarque~\ref{monoidale}, induit
une structure mono{\"\i}dale sym{\'e}trique sur $\mathsf{Hec}_0$.
\end{lemme}
\begin{proof}[D{\'e}monstration]
Soient $\mathcal{A}$, $\mathcal{B}$, $\mathcal{C}$ et $\mathcal{D}$
des dg-cat{\'e}gories. Comme dans la remarque~\ref{bitri}, le bifoncteur
$$
\xymatrix@!0 @R=1,5pc @C=12pc{
-\otimes^{\mathbb{L}}_k-: \mathsf{rep}_{ec}(\mathcal{A},
 \mathcal{B}) \times \mathsf{rep}_{ec}(\mathcal{C}, \mathcal{D})
\ar[r] & \mathsf{rep}_{ec}(\mathcal{A}\otimes\mathcal{C},
 \mathcal{B}\otimes \mathcal{D})\\
*+<2pc>{(X , Y)} \ar@{|->}[r]  & X_{cof}\otimes_k Y\, ,
}
$$
o{\`u} $X_{cof}$ est un remplacement cofibrant de $X$ dans la cat{\'e}gorie
$\mathrm{Mod}(\mathcal{A}^{op}\otimes^{\mathbb{L}}\mathcal{B})$,
par rapport {\`a} sa structure de cat{\'e}gorie de mod{\`e}les de Quillen, est
bi-triangul{\'e} puisqu'il provient aussi d'un dg-bifoncteur.
\end{proof}

\begin{remarque}
Soit $\mathsf{Heq}_{{ex}_0}$ la sous-cat{\'e}gorie pleine de
$\mathsf{Hec}_0$ dont les objets sont les dg-cat{\'e}gories
exactes. L'{\'e}quivalence entre $\mathsf{Hec}$ et $\mathsf{Heq}_{ex}$ permet d'{\'e}tablir une {\'e}quivalence entre $\mathsf{Hec}_0$ et $\mathsf{Heq}_{{ex}_0}$.
\end{remarque}

Soit $\mathcal{I}$ la dg-cat{\'e}gorie avec objets $1$, $2$ et dont les morphismes sont engendr{\'e}s par $m_{12}\in \mathrm{Hom}_{\mathcal{I}}^0(1, 2)$ soumis {\`a}
la relation $dm_{12}=0$. Pour une dg-cat{\'e}gorie
    $\mathcal{A}$, on note $\mbox{T}(\mathcal{A})$ la dg-cat{\'e}gorie
    $\mathcal{I}\otimes\mathcal{A}$. On dispose de deux inclusions
    canoniques $\mathcal{A} \stackrel{i_1}{\hookrightarrow} \mbox{T}(\mathcal{A})$ et $\mathcal{A} \stackrel{i_2}{\hookrightarrow} \mbox{T}(\mathcal{A})$ et d'une projection canonique
    $\mbox{T}(\mathcal{A})\rightarrow \mathcal{A}$.
\begin{remarque}\label{bimod}
La donn{\'e}e d'un $\mathcal{A}$-$\mathcal{B}$-bimodule $X$ est {\'e}quivalente
{\`a} la donn{\'e}e d'un dg-foncteur $X:\mathcal{A}\rightarrow
\mathrm{Mod}$-$\mathcal{B}$. Puisque la cat{\'e}gorie mono{\"\i}dale sym{\'e}trique
$(\dgcat,\otimes)$ est ferm{\'e}e, un
$\mathcal{A}$-$\mbox{T}(\mathcal{B})$-bimodule $X$ s'identifie {\`a} la
donn{\'e}e d'un morphisme de $\mathcal{A}$-$\mathcal{B}$-bimodules.
\end{remarque}
On note {\'e}galement $i_1$ et $i_2$ les morphismes de $\mathsf{Hec}$ associ{\'e}s
respectivement aux inclusions $i_1$ et $i_2$. Soit
$F:\mathsf{Hec}\rightarrow \mathsf{C}$ un foncteur {\`a} valeurs dans une
cat{\'e}gorie additive $\mathsf{C}$.

\begin{theoreme}\label{semiloc}
Les conditions suivantes sont {\'e}quivalentes~:
\begin{itemize}
\item[1)] Le foncteur $F$ est compos{\'e} d'un foncteur additif $\mathsf{Hec}_0\rightarrow \mathsf{C}$ et du foncteur canonique
  $\mathsf{Hec}\rightarrow \mathsf{Hec}_0$.
\item[2)] Pour toutes dg-cat{\'e}gories $\mathcal{A}$, $\mathcal{B}$, l'identit{\'e}
  $F(\left[X\right])+F(\left[Z\right])=F(\left[Y\right])$, est v{\'e}rifi{\'e}e
  dans $\mathrm{Hom}_{\mathsf{C}}(F(\mathcal{A}), F(\mathcal{B}))$ pour
  tout triangle $X \rightarrow Y \rightarrow Z \rightarrow X\left[1\right]$
  de $\mathsf{rep}_{ec}(\mathcal{A}, \mathcal{B})$.
\item[3)] Pour toute dg-cat{\'e}gorie $\mathcal{A}$, le morphisme
$$
\xymatrix{
F(\mathcal{A})\oplus F(\mathcal{A})
\ar[rr]^-{\left[F(i_1)
        \, , \,F(i_2)\right]} && F(\mbox{T}(\mathcal{A}))
}
$$
est un isomorphisme dans $\mathsf{C}$.

\item[4)] Pout toute dg-cat{\'e}gorie pr{\'e}triangul{\'e}e $\mathcal{A}$ munie de sous-dg-cat{\'e}gories pleines pr{\'e}triangul{\'e}es
  $\mathcal{B}$ et $\mathcal{C}$ qui donnent lieu {\`a} une d{\'e}composition
  semi-orthogonale
  $\mathrm{H}^0(\mathcal{A})=(\mathrm{H}^0(\mathcal{B}),
  \mathrm{H}^0(\mathcal{C}))$, voir \cite{Bondal}, le morphisme
$$
F(\mathcal{B})\oplus F(\mathcal{C}) \rightarrow F(\mathcal{A})
$$
induit par les inclusions est un isomorphisme dans $\mathsf{C}$.

\end{itemize}
\end{theoreme}

\begin{proof}[D{\'e}monstration]
Les conditions $1)$ et $2)$ sont {\'e}quivalentes par construction du groupe
de Grothendieck d'une cat{\'e}gorie triangul{\'e}e. Clairement la condition
$4)$ implique la condition $3)$. Montrons maintenant que la
condition $3)$ entra{\^\i}ne la condition $2)$. Pour une dg-cat{\'e}gorie
$\mathcal{E}$, notons $\mathrm{Mod}_{cf}\, \mathcal{E}$ la
sous-cat{\'e}gorie des objets cofibrants de
$\mathrm{Mod}\,\mathcal{E}$. Dans la suite, nous supposons que
$\mathcal{B}$ est cofibrant en tant que dg-cat{\'e}gorie. La cat{\'e}gorie
$\mathrm{Mod}\, \mbox{T}(\mathcal{B})$ s'identifie {\`a} la cat{\'e}gorie des
morphismes (ferm{\'e}s de degr{\'e} $0$)
$$ M_2 \stackrel{f}{\rightarrow} M_1$$
de $\mathrm{Mod}\, \mathcal{B}$. Les objets cofibrants correspondent aux cofibrations entre objets cofibrants de $\mathrm{Mod}\,
\mathcal{B}$. Pour un objet cofibrant $M$, on dispose d'une suite
exacte
$$ 0 \rightarrow M_2 \stackrel{f}{\rightarrow} M_1 \rightarrow
\mbox{Cok}f\rightarrow 0$$
de $\mathrm{Mod}\, \mathcal{B}$ fonctorielle en $M$. Notons
$$ 0 \rightarrow P_2 \rightarrow P_1 \rightarrow P_1/P_2 \rightarrow
0$$
la suite de dg-foncteurs de $\mathrm{Mod}_{cf}\,\mbox{T}(\mathcal{B})$ dans
$\mathrm{Mod}_{cf}\, \mathcal{B}$ obtenue ainsi. Clairement, les
foncteurs $P_2$, $P_1$, $P_1/P_2$ envoient les dg-modules
repr{\'e}sentables sur des dg-modules repr{\'e}sentables et donnent donc lieu
{\`a} des objets dans $\mathsf{rep}(\mbox{T}(\mathcal{B}), \mathcal{B})$. Soit
$$ X \rightarrow Y \rightarrow Z \rightarrow X\left[1\right]$$
un triangle de $\mathsf{rep}_{ec}(\mathcal{A}, \mathcal{B})$. Il est
isomorphe dans
$\mathcal{D}(\mathcal{A}^{op}\otimes^{\mathbb{L}}\mathcal{B})$ au
triangle associ{\'e} {\`a} une suite exacte courte
$$ 0 \rightarrow X' \stackrel{i}{\rightarrow} Y' \rightarrow Z'
\rightarrow 0$$
de bimodules cofibrants. Nous pouvons consid{\'e}rer le morphisme $M=(X'
\stackrel{i}{\rightarrow}Y')$ comme un dg-foncteur de
$\mathrm{Mod}_{cf}(\mathcal{A})$ dans $\mathrm{Mod}_{cf}\,
\mbox{T}(\mathcal{B})$. Clairement, il donne lieu {\`a} un objet de
$\mathsf{rep}(\mathcal{A}, \mbox{T}(\mathcal{B}))$. Nous avons
$$ P_2(M)=X', \, \, \, \, P_1(M)=Y', \, \,\, \, (P_1/P_2)(M)=Z'$$
dans $\mathsf{Hec}$. Pour montrer que l'on a $F(X')+F(Z')=F(Y')$, il
suffit donc de v{\'e}rifier que
$$ F(P_2)+F(P_1/P_2)=F(P_1)\,.$$
Notons
$$
\xymatrix@!0 @R=1,5pc @C=5pc{
I_1: L \ar@{|->}[r] & (0\rightarrow L)\\
I_2:L \ar@{|->}[r] & ( L \stackrel{\mathbf{1}}{\rightarrow}L)
}
$$
 les dg-foncteurs de $\mathrm{Mod}_{cf}\, \mathcal{A}$ dans
 $\mathrm{Mod}_{cf}\, \mbox{T}(\mathcal{A})$ induits par $i_1$ et
 $i_2$. Clairement, nous avons
$$
(P_1/P_2) \circ I_1=\mathbf{1}, \, \, \, (P_1/P_2) \circ I_2=0, \, \, \,
P_2 \circ I_1=0, \, \, \, P_2 \circ I_2=\mathbf{1}$$
dans $\mathsf{Hec}$.
Il s'ensuit que dans la cat{\'e}gorie additive $\mathsf{C}$, nous avons
$$
\begin{bmatrix}F(P_1/P_2)\\
F(P_2)
\end{bmatrix}
\circ
\begin{bmatrix}
F(I_1) & F(I_2)
\end{bmatrix}
=
\begin{bmatrix}
\mathbf{1} & 0 \\
0 & \mathbf{1}
\end{bmatrix}
$$
et donc
$$
(F(P_1/P_2)+F(P_2))\circ \begin{bmatrix}F(I_1) & F(I_2) \end{bmatrix}
= \begin{bmatrix}\mathbf{1} & \mathbf{1}\end{bmatrix}\,.$$
De l'autre c{\^o}t{\'e}, nous avons
$$P_1 \circ I_1=\mathbf{1}, \, \, \, P_1 \circ I_2=\mathbf{1}$$
dans $\mathsf{Hec}$ et donc
$$
F(P_1)\circ \begin{bmatrix}F(I_1) & F(I_2)
\end{bmatrix}=\begin{bmatrix}\mathbf{1} &  \mathbf{1}\end{bmatrix}$$
dans $\mathsf{C}$. Comme $ \begin{bmatrix} F(I_1) & F(I_2)\end{bmatrix}$ est
inversible, il s'ensuit que l'on a bien
$$
F(P_1/P_2)+F(P_2)=F(P_1)\,.$$

Montrons maintenant que la condition $1)$ implique la condition
$3)$. Ecrivons $F$ comme compos{\'e} d'un foncteur additif
$\overline{F}:\mathsf{Hec}_0 \rightarrow \mathsf{C}$ et du foncteur
canonique $\mathsf{Hec}\rightarrow \mathsf{Hec}_0$. Nous avons
$$
\overline{F}\,\mathcal{A}\oplus \overline{F}\,\mathcal{A}\simeq
\overline{F}(\mathcal{A}\oplus \mathcal{A})\, .
$$
Il suffit donc de montrer que dans $\mathsf{Hec}_0$, le morphisme
canonique
$$
\mathcal{A}\oplus\mathcal{A} \longrightarrow \mbox{T}(\mathcal{A})
$$
devient inversible. Par le lemme de Yoneda, il suffit de montrer que
pour toute dg-cat{\'e}gorie $\mathcal{U}$, l'application
$$
\mbox{K}_0(\mathsf{rep}_{ec}(\mathcal{U},\mathcal{A}))\oplus
\mbox{K}_0(\mathsf{rep}_{ec}(\mathcal{U},\mathcal{A})) \longrightarrow
\mbox{K}_0(\mathsf{rep}_{ec}(\mathcal{U}, \mbox{T}(\mathcal{A})))
$$
est bijective. Ceci r{\'e}sulte du fait que la cat{\'e}gorie triangul{\'e}e
$\mathsf{rep}_{ec}(\mathcal{U}, \mbox{T}(\mathcal{A}))$ admet une
d{\'e}composition semi-orthogonale en deux sous-cat{\'e}gories {\'e}quivalentes {\`a}
$\mathsf{rep}_{ec}(\mathcal{U}, \mathcal{A})$. En effet, si
$X$ est un objet de $\mathsf{rep}_{ec}(\mathcal{U},
\mbox{T}(\mathcal{A}))$ qui est cofibrant dans
$$
\mathrm{Mod}(\mathcal{U}^{op}\otimes
\mbox{T}(\mathcal{A}))\stackrel{\sim}{\longrightarrow}
\mathrm{Mod}(\mbox{T}(\mathcal{U}^{op}\otimes \mathcal{A}))\, ,$$
nous pouvons l'identifier avec une cofibration $X_2
\stackrel{i}{\rightarrow} X_1$ entre objets cofibrants appartenant {\`a}
$\mathsf{rep}_{ec}(\mathcal{U},\mathcal{A})$. Alors le triangle
associ{\'e} {\`a} $X$ par la d{\'e}composition semi-orthogonale est induit par la
suite exacte de morphismes
$$
\xymatrix{
0 \ar[r] \ar[d] & X_2 \ar@{=}[r] \ar@{=}[d] & X_2 \ar[d]^-{i} \ar[r] & 0
\ar[d] \ar[r] & 0 \ar[d]\\
0 \ar[r] & X_2 \ar[r]_-{i} & X_1 \ar[r] & X_1/X_2 \ar[r] & 0\,.
}
$$
Un argument analogue au pr{\'e}c{\'e}dent montre que la condition $1)$
implique la condition $4)$.
\end{proof}

\section{DG-foncteurs de Morita}\label{secmor}

Soit $\mathcal{T}$ une cat{\'e}gorie triangul{\'e}e. On note
$\mathcal{T}^{\,\kar}$ sa compl{\'e}tion idempotente, voir \cite{Balmer}.
Ici, l'exposant repr{\'e}sente la moiti{\'e} du symbole $\oplus$.
Un dg-foncteur $F$ de $\mathcal{C}$ vers $\mathcal{E}$ est {\em de
  Morita} s'il satisfait l'une des conditions $(M1)$ ou $(M2)$~:
\begin{itemize}
\item[(M1)]
\begin{itemize}
\item la dg-cat{\'e}gorie $\mathcal{C}$ est vide et tous les objects
  de la dg-cat{\'e}gorie $\mathcal{E}$ sont contractiles.
\end{itemize}
\item[(M2)]
\begin{itemize}
\item[-] pour tous objets $c_1$ et $c_2$ dans $\mathcal{C}$, le morphisme
  de complexes
\[
\mathrm{Hom}_{\mathcal{C}}(c_1, c_2) \to \mathrm{Hom}_{\mathcal{E}}(F(c_1),F(c_2))
\]
 est un quasi-isomorphisme et
\item[-] le foncteur $\mathrm{H}^0(\mbox{pre-tr}(F))^{\kar}$ de
  $\mathrm{H}^0(\mbox{pre-tr}(\mathcal{C}))^{\kar}$ vers
  $\mathrm{H}^0(\mbox{pre-tr}(\mathcal{E}))^{\kar}$ est essentiellement surjectif.
\end{itemize}

\end{itemize}
\begin{remarque}
Les dg-foncteurs $H$ v{\'e}rifiant $(M1)$ ne v{\'e}rifient pas $(M2)$ si
$\mathcal{E}$ est non vide puisque
$\mbox{pre-tr}(\emptyset)=\emptyset$.
La notion de dg-foncteur de Morita est importante car un dg-foncteur $F:\mathcal{C}
\rightarrow \mathcal{E}$ est de Morita si et seulement s'il induit une
{\'e}quivalence
$$F^*:\mathcal{D}(\mathcal{E})
\stackrel{\sim}{\rightarrow}\mathcal{D}(\mathcal{C})$$
dans les cat{\'e}gories d{\'e}riv{\'e}es, voir \cite{DerivingDG}.
\end{remarque}

On introduira une structure de cat{\'e}gorie de mod{\`e}les de
Quillen {\`a} engendrement cofibrant dans $\dgcat$ dont les
{\'e}quivalences faibles sont les dg-foncteurs de Morita. Pour cela, on se servira du th{\'e}or{\`e}me~\ref{thm}.

Soit $\mathrm{idem}$ la dg-cat{\'e}gorie avec un seul object $0$ et
dont les morphismes sont engendr{\'e}s par l'endomorphisme $e$ de degr{\'e}
$0$ tel que $de=0$ et $e^2=e$.
Pour un entier $n>0$ et des entiers $k_0, \ldots , k_n$, soit
$\mathrm{idem}_n(k_0, \ldots , k_n)$ la dg-cat{\'e}gorie obtenue {\`a} partir
de la dg-cat{\'e}gorie $\mathcal{M}_n(k_0, \ldots , k_n)$ en rajoutant de
nouveaux generateurs et relations~: on rajoute, pour tous $0 \leq i, j
\leq n$, des morphismes $e_{i, j}$ de source $j$ et de but $i$ de
degr{\'e} $k_i-k_j$ soumis aux relations
$$
\begin{array}{ccc}
d(E)=0 & \mbox{et} & E^2=E\,,
\end{array}
$$
o{\`u} $E$ est la matrice de coefficients $e_{i, j}$.

Pour $n\geq 0$, soit $\mathrm{fact}_n(k_0, \ldots , k_n)$ la sous-dg-cat{\'e}gorie pleine
de la dg-cat{\'e}gorie des dg-modules ({\`a} droite) sur
$\mathrm{idem}_n(k_0, \ldots , k_n)$ dont les objets sont les
$\hat{l}$, $0 \leq l \leq n$ et le facteur directe $X_n$ associ{\'e} {\`a}
l'idempotent $\widehat{E}$ du c{\^o}ne it{\'e}r{\'e} qui a le m{\^e}me
module gradu{\'e} sous-jacent que le dg-module
$$ X=\bigoplus^n_{l=0} \hat{l}\left[ k_l \right]$$
et dont la diff{\'e}rentielle est $d_X+\widehat{Q}$.

On note $L_n(k_0,\ldots, k_n)$ le dg-foncteur naturel et pleinement
fid{\`e}le, de $\mathrm{idem}_n(k_0,\ldots, k_n)$ dans
$\mathrm{fact}_n(k_0,\ldots, k_n)$. On consid{\`e}re $L_n(k_0,\ldots,
k_n)$ comme un objet de la cat{\'e}gorie des morphismes de
$\mathsf{dgcat}$ munie de sa structure de cat{\'e}gorie de mod{\`e}les
canonique (o{\`u} une {\'e}quivalence faible est une quasi-{\'e}quivalence en
chaque composante). On choisit, une fois pour toutes, un remplacement
cofibrant $L\!h_n(k_0,\ldots,k_n)$ de cet objet~:
$$
\xymatrix{
\mathrm{idem}_n(k_0,\ldots , k_n) \ar[d]_{L(k_0,\ldots ,k_n)} &
*+<1.5pc>{\mathrm{idemh}_n(k_0,\ldots , k_n)} \ar@{->>}[l]_{\sim}
\ar@{>->}[d]^{L\!h_n(k_0,\ldots , k_n)} \\
\mathrm{fact}_n(k_0,\ldots, k_n) & \mathrm{facth}_n(k_0,\ldots, k_n) \ar@{->>}[l]_{\sim}\,.
}
$$
\begin{notation}
On note par $\mathrm{idem}$, resp. $\mathrm{idemh}$, la dg-cat{\'e}gorie
$\mathrm{idem}_0(0)$, resp. $\mathrm{idemh}_0(0)$ et par $\mathrm{fact}$, resp. $\mathrm{facth}$, la dg-cat{\'e}gorie
$\mathrm{fact}_0(0)$, resp. $\mathrm{facth}_0(0)$.
\end{notation}

Nous d{\'e}finissons $\mathcal{B}$ comme la dg-cat{\'e}gorie avec un seul
object $7$ et dont les morphismes sont engrendr{\'e}s par l'endomorphisme
$\underline{h}$ de degr{\'e}e $-1$ soumis {\`a} la relation
$d(\underline{h})=\mathbf{1}_7$. Soit $C$ l'unique dg-foncteur de la
dg-cat{\'e}gorie vide $\mathcal{O}$ vers $\mathcal{B}$.

\begin{theoreme}\label{theorem2}
Si on consid{\`e}re pour cat{\'e}gorie $\mathcal{C}$ la cat{\'e}gorie $\dgcat$,
pour classe $W$ la sous-cat{\'e}gorie de $\dgcat$ des dg-foncteurs
de Morita, pour classe $J$ les dg-foncteurs $F$ et $R(n)$,
$n\in \mathbb{Z}$, (voir section~\ref{secdef}), $F(n)$, $n \in \mathbb{Z}$, $I_n(k_0,
 \ldots , k_n)$, $L\!h_n(k_0, \ldots , k_n)$ et $C$ et pour classe $I$ les dg-foncteurs $Q$ et $S(n)$,
$n\in \mathbb{Z}$, (voir section~\ref{secdef}), alors les conditions du
th{\'e}or{\`e}me~\ref{thm} sont satisfaites.
\end{theoreme}

{\`A} la proposition~\ref{nova4} apr{\`e}s la d{\'e}monstration du th{\'e}or{\`e}me,
nous allons donner une description explicite des objets fibrants de la
structure de cat{\'e}gorie de mod{\`e}les obtenue ainsi.

Une fois que tous les objets dans $\mathsf{dgcat}$ sont petits, voir
\cite[10.4.1]{Hirschhorn}, les conditions $\mbox{{\it (i)}}$,
$\mbox{{\it (ii)}}$ et $\mbox{{\it (iii)}}$ du th{\'e}or{\`e}me~\ref{thm} sont verifi{\'e}es.
Pour v{\'e}rifier les autres conditions, nous avons besoin d'un certain nombre
de lemmes.

\begin{lemme}\label{important1}
Soit $\ca$ une petite dg-cat{\'e}gorie.
\begin{itemize}
\item[a)] Soit $e$ un endomorphisme idempotent d'un objet $A$ de $H^0(\ca)$.
Alors il existe un isomorphisme $F:\ca \to \ca'$ dans $\mathsf{Heq}$ et
un idempotent ferm{\'e} $\tilde{e}$ dans $\ca'(FA,FA)$ tel que
l'image de $e$ par $H^0(F)$ est {\'e}gale {\`a} la classe de $\tilde{e}$.
\item[b)] Tout foncteur $\idem \to H^0(\ca)$ se rel{\`e}ve en un
morphisme $\idemh \to \ca$ de $\dgcat$.
\end{itemize}
\end{lemme}

\begin{proof}[D{\'e}monstration]
a) Dans la cat{\'e}gorie d{\'e}riv{\'e}e $\cd(\ca)$, les idempotents se scindent
puisque $\cd(\ca)$ est une cat{\'e}gorie triangul{\'e}e aux coproduits d{\'e}nombrables.
L'idem\-potent $e$ donne donc lieu {\`a} une d{\'e}composition
\[
\widehat{A} \iso P \oplus Q
\]
dans $\cd(\ca)$ telle que $e$ corresponde {\`a} l'endomorphisme $\id_P\oplus 0$.
On peut, et on va, supposer que $P$ et $Q$ sont des dg-modules cofibrants et
que nous avons un quasi-isomorphisme surjectif $p: P \oplus Q \to \widehat{A}$.
Notons $\cb$ la sous-cat{\'e}gorie pleine de $\cc_{dg}(\ca)$ dont les
objets sont $P\oplus Q$ et les foncteurs repr{\'e}sentables autres
que $\widehat{A}$. Notons $\ca'$ la sous-cat{\'e}gorie non pleine
de $\cb$ qui a les m{\^e}mes objets que $\cb$ et telle que pour tout $X$ dans $\ca'$,
on a
\[
\ca'(X,P\oplus Q)=\cb(X,P\oplus Q)
\]
et $\ca'(P\oplus Q,X) \subset \cb(P\oplus Q,X)$ est form{\'e} des morphismes
qui s'annulent sur le noyau de $p$. On a des dg-foncteurs canoniques
\[
\xymatrix{\cb & \ca' \ar[l] \ar[r] & \ca}
\]
et il n'est pas difficile de constater que ce sont des quasi-{\'e}quivalences
(bijectives sur les objets). Dans $\mathsf{Heq}$, ces dg-foncteurs donnent lieu a l'isomorphisme
annonc{\'e} $F: \ca\to\cb$. Pour $\tilde{e}$, on prend l'endomorphisme $\id_P\oplus 0$.

b) La donn{\'e}e d'un foncteur $\idem \to H^0(\ca)$ donne lieu {\`a} un
idempotent $e$ d'un objet $A$ de $\ca$. Gr{\^a}ce {\`a} a), on obtient
un isomorphisme $\cb\to \ca$ dans $\mathsf{Heq}$ et un morphisme
$\idem \to \cb$ dans $\dgcat$ tels que le foncteur compos{\'e}
\[
H^0(\idem) \to H^0(\cb) \to H^0(\ca)
\]
est le foncteur donn{\'e}. Consid{\'e}rons la composition
\[
\idemh \to \idem \to \cb \to \ca
\]
dans $\mathsf{Heq}$. Puisque $\idemh$ est cofibrant, cette composition
se rel{\`e}ve en un vrai morphisme $\idemh \to \ca$ dans $\dgcat$. La construction
montre qu'il a la propri{\'e}t{\'e} requise.
\end{proof}

\begin{lemme}\label{important3}
Soit $U:\ca\to\cb$ un dg-foncteur entre petites dg-cat{\'e}gories.
Supposons donn{\'e} un carr{\'e} commutatif de cat{\'e}gories
\[
\xymatrix{\idem \ar[r] \ar[d] & H^0(\ca) \ar[d]^{H^0(U)} \\
\fact \ar[r] & H^0(\cb).}
\]
Alors il existe un carr{\'e} commutatif de dg-cat{\'e}gories
\[
\xymatrix{\idemh \ar[r] \ar[d] & \ca \ar[d]^U \\
\facth \ar[r] & \cb}
\]
dont les fl{\`e}ches horizontales sont des relev{\'e}s (stricts) des fl{\`e}ches
du diagramme donn{\'e}.
\end{lemme}

\begin{proof}[D{\'e}monstration] Notons $A$ l'image de l'objet de $\idem$,
$B=UA$ l'image du `grand' objet de $\fact$ et $B'$ l'image de l'objet
`facteur direct' de $\fact$. Notons $e$ le morphisme de $H^0(\ca)$ qui
est l'image du morphisme $e_0$ de $\idem$. Comme dans le lemme ci-dessus,
nous choisissons une d{\'e}composition $\widehat{A}\iso P\oplus Q$ dans
la cat{\'e}gorie d{\'e}riv{\'e}e de $\ca$ telle que $P$ est une image pour $e$
et $P$ et $Q$ sont cofibrants.
Relevons l'isomorphisme $\widehat{A}\iso P\oplus Q$ en un morphisme
dans $\cc_{dg}(A)$ (qui sera donc une {\'e}quivalence d'homotopie).
Le foncteur $U$ s'{\'e}tend en un foncteur $U_!$ de $\cc_{dg}(\ca)$
vers $\cc_{dg}(B)$ (adjoint {\`a} gauche de la restriction le long de $U$)
et nous pouvons supposer que $U_!(\widehat{A})=\widehat{UA}=\widehat{B}$
et que, plus g{\'e}n{\'e}ralement, l'image par $U_!$ d'un module
repr{\'e}sentable $\widehat{A''}$ est le module repr{\'e}sentable $\widehat{UA''}$.
Nous obtenons ainsi un morphisme $U_!(P\oplus Q) \to \widehat{B}$ dans
$\cc_{dg}(\cb)$ (qui est une {\'e}quivalence d'homotopie) et nous
choisissons une {\'e}quivalence d'homotopie $U_!(P) \to \widehat{B'}$
compatible avec les donn{\'e}es provenant du carr{\'e} commutatif de d{\'e}part.
Nous d{\'e}finissons $\ca'$ comme la sous-dg-cat{\'e}gorie pleine
de la dg-cat{\'e}gorie $\cp(\cc_{dg}(\ca))$ (voir definition~\ref{1path1})
dont les objets sont les {\'e}quivalences d'homotopie
identiques
\[
\widehat{A''} \to \widehat{A''}
\]
pour tous objets $A''$ de $\ca$ distincts de $A$ et l'{\'e}quivalence
d'homotopie
\[
\widehat{A} \to P\oplus Q .
\]
Le foncteur naturel $\ca' \to \ca$ (projection sur la premi{\`e}re
composante) est alors bijectif au niveau des objets et induit
des quasi-isomorphismes surjectifs dans les espaces de morphismes.
De fa{\c c}on analogue, nous d{\'e}finissons $\cb'$ comme la
sous-dg-cat{\'e}gorie pleine de $\cp(\cc_{dg}(\cb))$ dont les
objets sont les {\'e}quivalences d'homotopie identiques
\[
\widehat{B''}\to \widehat{B''}
\]
pour tous objets $B''$ de $\cb$ distincts de $B$ et $B'$, et
les {\'e}quivalences d'homotopie construites
\[
\widehat{B} \to U_!(P\oplus Q)  \mbox{ et } \widehat{B'} \to U_!(P).
\]
Par construction, le dg-foncteur $U_!$ induit un dg-foncteur
$U':\ca'\to\cb'$.
Finalement, nous d{\'e}finissons $\ca''$ comme la sous-dg-cat{\'e}gorie
pleine de $\cc_{dg}(\ca)$ dont les objets sont les images de ceux
de $\ca'$ par la deuxi{\`e}me projection et de fa{\c c}on analogue pour $\cb''$.
Le foncteur $U'$ induit un foncteur $U'': \ca'' \to \cb''$. Les
deux projections d{\'e}finissent alors des morphismes de morphismes
de dg-cat{\'e}gories
\[
\xymatrix{ U & U' \ar[l] \ar[r] & U''}
\]
et ces morphismes ont pour composantes des quasi-{\'e}quivalences
(bijectives au niveau des objets). Notons $U_0$ le morphisme
de (dg-)cat{\'e}gories $\idem \to \fact$. Nous obtenons ainsi un morphisme
de $U_0=H^0(U_0)$ vers $H^0(U'')$. Notre construction de $\ca''$, $\cb''$ et
$U''$ nous permet alors de relever ce morphisme en un morphisme
$U_0 \to U''$, c'est-{\`a}-dire en un carr{\'e} commutatif de dg-foncteurs
\[
\xymatrix{\idem \ar[r] \ar[d] & \ca \ar[d]^{U} \\
\fact \ar[r] & \cb.}
\]
Notons $U\!h_0$ le dg-foncteur $\idemh \to \facth$. Par construction,
$U\!h_0 \to U_0$ est un remplacement cofibrant dans la cat{\'e}gorie
des morphismes de $\dgcat$. Les morphismes
\[
\xymatrix{ U & U' \ar[l] \ar[r] & U''}
\]
sont des {\'e}quivalences faibles dans cette cat{\'e}gorie. On peut
donc trouver un morphisme $U\!h_0 \to U$ dans la cat{\'e}gorie des
morphismes de $\dgcat$ tel que le diagramme
\[
\xymatrix{U\!h_0 \ar[d] \ar[rr] & & U_0 \ar[d] \\
U & U' \ar[l] \ar[r] & U''}
\]
commute dans la cat{\'e}gorie homotopique de la cat{\'e}gorie des
morphismes de $\dgcat$. On v{\'e}rifie facilement que le carr{\'e}
commutatif correspondant {\`a} $U\!h_0 \to U$
a les propri{\'e}t{\'e}s requises dans l'{\'e}nonc{\'e}.

\end{proof}

\begin{lemme}\label{J-inj-Surj2}
$\mbox{{\bf Surj}} = J-\mbox{inj} \cap W\,.$
\end{lemme}
\begin{proof}[D{\'e}monstration]
Montrons l'inclusion $\supseteq$.
Soit $H$ un dg-foncteur de $\mathcal{D}$ vers
$\mathcal{S}$ qui appartient {\`a} $J - \mbox{inj} \cap W$.
On montre maintenant que $H$ satisfait forc{\'e}ment la condition
$(M2)$. Supposons que $H$ satisfait la condition $(M1)$. Si $\cd$
est vide, alors $H$ satisfait $(M2)$. Supposons que $\cd$ est
non vide. Alors il
existe un object $X$ dans $\mathcal{S}$ et un morphisme $h$, de degr{\'e}e
$-1$ tel que $d(h)=\mathbf{1}_X$. On peut donc construire le carr{\'e}
naturel commutatif suivant~:
$$
\xymatrix{
\mathcal{O} \ar[d]_C \ar@{=}[r] & \mathcal{D} =\mathcal{O} \ar[d]^H \\
\mathcal{B} \ar[r] \ar@{-->}[ur] & \mathcal{S}\,.
}
$$
Puisque $H$ appartient {\`a} $J-\mbox{inj}$, on a un dg-foncteur de
$\mathcal{B}$ vers $\mathcal{O}$. On obtient donc une
contradition. Cela montre que $H$ satisfait la condition $(M2)$.

Maintenant on montre que $H$ est en effet un dg-foncteur
quasi-{\'e}quiconique. Le dg-foncteur $H$ a la propri{\'e}t{\'e}
de rel{\`e}vement {\`a} droite par rapport aux
dg-foncteurs $L\!h_n(k_0,\ldots, k_n), \, n \geq 0$, et donc
le dg-foncteur
$$ \mbox{pre-tr}(H): \mbox{pre-tr}(\mathcal{D}) \longrightarrow
\mbox{pre-tr}(S)$$
l'a par rapport au dg-foncteur
$\mathrm{idemh} \rightarrow \mathrm{facth}$. Puisque $H$
est par hypoth{\`e}se de Morita, le foncteur $\mathsf{H}^0(\mbox{pre-tr}(H))$
est pleinement fid{\`e}le et on dispose d'une {\'e}quivalence de cat{\'e}gories
$$\mathsf{H}^0(\mbox{pre-tr}(\mathcal{D}))^{\kar}
\stackrel{\sim}{\longrightarrow}
\mathsf{H}^0(\mbox{pre-tr}(\mathcal{S}))^{\kar}\,.$$
Soit maintenant $X \in \mbox{pre-tr}(\mathcal{S})$. L'objet $X$ est
donc isomorphe dans $\mathsf{H}^0(\mbox{pre-tr}(\mathcal{S}))$ {\`a} un
facteur direct d'un objet qui est dans l'image par
$\mathsf{H}^0(\mbox{pre-tr}(H))$. Cela nous permet de construire un
carr{\'e} commutatif de cat{\'e}gories
$$
\xymatrix{
\mathrm{idem} \ar[d] \ar[r] &
\mathsf{H}^0(\mbox{pre-tr}(\mathcal{D})) \ar[d]\\
\mathrm{fact} \ar[r]_-F & \mathsf{H}^0(\mbox{pre-tr}(S))
}
$$
tel que $F(C_0)=X$. Par le lemme~\ref{important3} on dispose d'un carr{\'e}
commutatif
$$
\xymatrix{
\mathrm{idemh} \ar[r] \ar[d] & \mbox{pre-tr}(\mathcal{D}) \ar[d] \\
\mathrm{facth} \ar[r]_-{F_h} \ar@{-->}[ur] & \mbox{pre-tr}(\mathcal{S}) \,,
}
$$
o{\`u} $F_h$ est un relev{\'e} strict de $F$. Cela implique que le foncteur
$\mbox{pre-tr}(H)$ est essentiellement surjectif et donc
$H$ est un dg-foncteur quasi-{\'e}quiconique. Puisque $H$ a la
propri{\'e}t{\'e} de rel{\`e}vement {\`a} droite par rapport aux cofibrations
triviales g{\'e}n{\'e}ratrices de la structure quasi-{\'e}quiconique,
le lemme~\ref{J-inj-Surj 1} implique que $H$ appartient {\`a} $\mbox{{\bf Surj}}$.

Montrons maintenant l'inclusion $\subseteq$. Soit $H$ un dg-foncteur
de $\mathcal{N}$ vers $\mathcal{E}$ dans la classe $\mbox{{\bf Surj}}$. Comme $H$ est surjectif au niveau des objets et un
quasi-isomorphisme au niveau des complexes de morphismes, on a $H \in
W$. La classe $R(n)-\mbox{drt}$ est form{\'e}e des dg-foncteurs surjectifs
au niveau des complexes de morphismes. Il suffit de montrer que $H$ a
la propri{\'e}t{\'e} de rel{\`e}vement {\`a} droite par rapport {\`a} $F$, $F(n)$, $n \in
\mathbb{Z}$, $I_n(k_0, \ldots , k_n)$ et $L\!h_n(k_0, \ldots , k_n)$ et $C$. On
consid{\`e}re ces trois cas~:
\begin{enumerate}
\item on sait par le th{\'e}or{\`e}me pr{\'e}c{\'e}dent que  $H \in F-\mbox{drt}$,
  $H \in F(n)-\mbox{drt}$ et $H \in I_n(k_0, \ldots , k_n)-\mbox{drt}$.
\item On consid{\`e}re le diagramme suivant~:
$$\xymatrix{
*+<1.5pc>{\mathrm{idemh}_n(k_0,\ldots, k_n)} \ar[r] \ar@{>->}[d]_{L\!h_n(k_0,
  \ldots , k_n)} & \mathcal{N} \ar[d]^H\\
\mathrm{facth}_n(k_0, \ldots , k_n) \ar[r]_-{T} & \mathcal{E}
}$$
Puisque $H$ est une fibration triviale  et $L\!h_n(k_0,\ldots,
k_n)$ une cofibration, le dg-foncteur $H$ a la propri{\'e}t{\'e} de rel{\`e}vement {\`a} droite par
rapport {\`a} $L\!h_n(k_0,\ldots, k_n)$.

\item On consid{\`e}re le diagramme suivant~:
$$
\xymatrix{
\mathcal{O} \ar[r] \ar[d]_C & \mathcal{N} \ar[d]^H \\
\mathcal{B} \ar[r]_T & \mathcal{E} \,.
}
$$
Soit $X$ l'image par $T$ de l'objet $7$ de $\mathcal{B}$. Puisque $H$
appartient {\`a} $\mbox{{\bf Surj}}$, il existe un object $Y$ dans
$\mathcal{N}$ tel que $H(Y)=X$. Maintenant puiqu'on dispose d'un
quasi-isomorphisme surjectif de dg-alg{\`e}bres de
$\mathsf{End}_{\mathcal{N}}(Y)$ vers $\mathsf{End}_{\mathcal{E}}(X)$,
un argument classique montre qu'on peut relever le long de $H$ une
contraction de $Y$ vers une contraction de $X$. Cela nous fournit le
rel{\`e}vement recherch{\'e}.
\end{enumerate}
\end{proof}

Soient $\cu$ une petite dg-cat{\'e}gorie et
\[
T:\idemh_n(k_0, \ldots, k_n) \to \cu
\]
un dg-foncteur. Soit la somme amalgam{\'e}e
\[
\xymatrix{\idemh_n \ar[r]^{T} \ar[d]_{L\! h_n} & \cu \ar[d]^{L'} \\
\facth_n \ar[r]_{T'} & \cv}
\]
dans $\dgcat$.

\begin{lemme}\label{pushout-homotopique} Le dg-foncteur $L'$ est de Morita. La restriction
$\cd(\cv) \to \cd(\cu)$ le long de $L'$ induit une {\'e}quivalence de
$H^0(\cv)$ sur la sous-cat{\'e}gorie pleine de $\cd(\cu)$ form{\'e}e
des foncteurs repr{\'e}sentables et de l'image du facteur direct
$X_n$ (voir la d{\'e}finition de $\facth_n$) par la composition
\[
\xymatrix{\cd(\idem_n) \ar[r] &  \cd(\idemh_n) \ar[r]^-{\mathbf{L}T_!} & \cd(\cu).}
\]
\end{lemme}

\begin{remarque} Le lemme montre que le carr{\'e} induit
\[
\xymatrix{\idem_n \ar[r] \ar[d] & H^0(\cu) \ar[d] \\
\fact_n \ar[r] & H^0(\cv)}
\]
est cocart{\'e}sien ({\`a} isomorphisme pr{\`e}s).
\end{remarque}

\begin{proof}[D{\'e}monstration]
Soit $\cw$ une petite dg-cat{\'e}gorie. Puisque le morphisme
$\idemh_n \to \facth_n$ est une cofibration entre dg-cat{\'e}gorie
cofibrantes, la somme amalgam{\'e}e est aussi une somme amalgam{\'e}e
homotopique. Donc le carr{\'e}
\begin{equation}\label{map-square}
\xymatrix{\Map(\idemh,\cw) & \Map(\cu,\cw) \ar[l]\\
\Map(\facth,\cw) \ar[u] & \Map(\cv, \cw). \ar[l] \ar[u]}
\end{equation}
est homotopiquement cart{\'e}sien. D'apr{\`e}s le th{\'e}or{\`e}me principal
de \cite{Toen}, les groupes $\pi_0$ dans ce diagramme
sont en bijection avec les classes d'isomorphisme et les groupes
d'homotopie sup{\'e}rieurs avec les espaces d'extensions sup{\'e}rieures dans les
cat{\'e}gories du carr{\'e} correspondant
\begin{equation} \label{rep-square}
\xymatrix{\rep(\idemh,\cw) & \rep(\cu,\cw) \ar[l]\\
\rep(\facth,\cw) \ar[u] & \rep(\cv, \cw). \ar[l] \ar[u]}
\end{equation}
Or, ici la fl{\`e}che verticale de gauche s'identifie au foncteur
\[
\rep(\fact, \cw) \to \rep(\idem, \cw)
\]
induit par la restriction de modules sur $\fact_n^{op}\ten \cw$
{\`a} $\idem_n^{op}\ten\cw$. Or la dg-cat{\'e}gorie $\fact_n^{op}\ten \cw$
est obtenue {\`a} partir de $\idem_n^{op}\ten\cw$ en rajoutant des
facteurs directs de foncteurs repr{\'e}sentables. La restriction est
donc pleinement fid{\`e}le et pr{\'e}serve le caract{\`e}re cofibrant
d'un module. Elle induit donc un foncteur pleinement fid{\`e}le
dans les cat{\'e}gorie d{\'e}riv{\'e}es et le foncteur
\[
\rep(\fact_n,\cw) \to \rep(\idem_n,\cw)
\]
induit une injection dans les classes d'isomorphisme et
des bijections dans le groupes d'extensions sup{\'e}rieurs.
Il s'ensuit que le morphisme
\[
\Map(\facth_n, \cw) \to \Map(\idemh_n,\cw)
\]
est un monomorphisme homotopique. Gr{\^a}ce au carr{\'e}
homotopiquement cart{\'e}sien (\ref{map-square}), le morphisme
\[
\Map(\cv,\cw) \to \Map(\cu,\cw)
\]
est {\'e}galement un monomorphisme homotopique. Ainsi, le dg-foncteur
$\cu\to\cv$ induit un foncteur pleinement fid{\`e}le
\[
\cd(\cv^{op}\ten \cw) \to \cd(\cu^{op}\ten \cw).
\]
En prenant pour $\cw$ la dg-cat{\'e}gorie finale, nous obtenons
que $\cu\to\cv$ est quasi-pleinement fid{\`e}le. Les classes
d'isomorphisme d'objets de $H^0(\cv)$ proviennent de ceux
de $\cu$ et de ceux de $H^0(\facth_n)$, qui est {\'e}quivalent
{\`a} $\fact_n$. Ceci nous donne que $\cd(\cv)\to\cd(\cu)$
est une {\'e}quivalence et la description de l'image de
$H^0(\cv)$ dans $\cd(\cu)$.
\end{proof}

\begin{lemme}\label{J-cell-dans-W2}
 On a $J-\mbox{cell}\subset W$.
\end{lemme}
\begin{proof}[D{\'e}monstration]
On sait d{\'e}j{\`a} par le th{\'e}or{\`e}me pr{\'e}c{\'e}dent que les classes
$F-\mbox{cell}$, $R(n)-\mbox{cell}$, $F(n)-\mbox{cell}$ et $I_n(k_0,
\ldots , k_n)-\mbox{cell}$ sont form{\'e}es de dg-foncteurs
quasi-{\'e}quiconiques donc contenues dans la classe $W$. Il est clair que
la classe $C$-$\mbox{cell}$ est contenue dans $W$.

Soit maintenant $T:\mathrm{idemh}_n(k_0,\ldots, k_n) \rightarrow \mathcal{U}$ un
dg-foncteur quelconque. On consid{\`e}re la somme amalgam{\'e}e suivante
$$\xymatrix{
*+<1.5pc>{\mathrm{idemh}_n(k_0,\ldots, k_n)} \ar@{>->}[d]_{L\!h_n(k_0, \ldots ,
  k_n)} \ar[r]^-{T}  \ar@{}[dr]|{\lrcorner} & \mathcal{U} \ar[d]^{inc} \\
\mathrm{facth}_n(k_0, \ldots , k_n) \ar[r] & \mathcal{V}
}$$
dans $\dgcat$. Il r{\'e}sulte du lemme~\ref{pushout-homotopique} que $inc \in W$.
\end{proof}

Nous avons v{\'e}rifi{\'e} que $J-\mbox{cell}\subseteq W$ (lemme~\ref{J-cell-dans-W2}) et
que $I-\mbox{inj}$ est {\'e}gal {\`a} $J-\mbox{inj}\cap W$ (lemmes~\ref{I-inj-Surj} et \ref{J-inj-Surj2}).
Ces conditions impliquent celles du th{\'e}or{\`e}me~\ref{thm}.

\begin{prop}\label{nova4}
Les dg-cat{\'e}gories Morita fibrantes, c'est-{\`a}-dire les objets fibrants pour la cat{\'e}gorie de mod{\`e}les de Quillen du
th{\'e}or{\`e}me~\ref{theorem2}, sont les dg-cat{\'e}gories $\mathcal{A}$
non vides telles que l'image essentielle du plongement
$\mathsf{H}^0(\mathcal{A}) \hookrightarrow \mathcal{D}(\mathcal{A})$
est stable par suspensions, c{\^o}nes et facteurs directs.
\end{prop}

\begin{proof}[D{\'e}monstration]
Soit $\mathcal{A}$ une dg-cat{\'e}gorie Morita fibrante. On montre
maintenant que $\mathcal{A}$ ne peut pas {\^e}tre la dg-cat{\'e}gorie
vide. Suposons que $\mathcal{A} = \emptyset$. Puisque
$\mathcal{A}$ est Morita fibrante il existe un dg-foncteur $E$
$$
\xymatrix{
\emptyset \ar@{=}[r] \ar[d]_C & \emptyset \\
\mathcal{B} \ar@{-->}[ur]_E & .
}
$$
Puisque $\mathcal{B}$ est non vide on obtient donc une
contradiction. Par la proposition~\ref{fibrants}, le plongement
$\mathsf{H}^0(\mathcal{A}) \hookrightarrow \mathcal{D}(\mathcal{A})$
est stable par suspensions et c{\^o}nes. Il reste {\`a} montrer qu'il est
aussi stable par facteurs directs. Soit $X \in
\mathsf{H}^0(\mbox{pre-tr}(\mathcal{A}))$ et $e$ un idempotent dans
$\mathsf{End}_{\mathcal{D}(\mathcal{A})}(X)$. Cela correspond {\`a} un
foncteur $\mathrm{idem} \rightarrow
\mathsf{H}^0(\mbox{pre-tr}(\mathcal{A})$ et par le
lemme~\ref{important1} on a un dg-foncteur $L:\mathrm{idemh}
\to \mbox{pre-tr}(\mathcal{A})$. Puisque $\mathcal{A}$ est une dg-cat{\'e}gorie Morita
fibrante on dispose du diagramme commutatif suivant
$$
\xymatrix{
\mbox{idemh} \ar[d] \ar[r]^L & \mathcal{A} \\
\mbox{facth} \ar@{-->}[ur]_{\overline{L}} & .
}
$$
Notons que le facteur direct associ{\'e} {\`a} $\overline{L}$ est isomorphe
dans $\mathcal{D}(\mathcal{A})$ au facteur direct associ{\'e} {\`a}
l'idempotent $e$.

Soit maintenant $\mathcal{A}$ une dg-cat{\'e}gorie non-vide telle que
l'image essentielle du plongement $\mathsf{H}^0(\mathcal{A}) \hookrightarrow \mathcal{D}(\mathcal{A})$
est stable par suspensions, c{\^o}nes et facteurs directes. Par
la proposition~\ref{fibrants}, le dg-foncteur $\mathcal{A} \rightarrow 0$
a la propri{\'e}t{\'e} de rel{\`e}vement {\`a} droite par rapport aux dg-foncteurs
$F(n), R(n)$ et $I_n(k_0, \ldots, k_n),\, n \in \mathbb{Z}$. On montre
maintenant qu'il a la propri{\'e}t{\'e} de rel{\`e}vement {\`a} droite par rapport {\`a}
$C$. Puisque $\mathcal{A}$ est non vide, il existe un objet $X
\in \mathcal{A}$. On consid{\`e}re le diagram
$$
\xymatrix{
\mbox{cone}_0(0) \ar[r]^L \ar[d]_{I_0(0)} & \mathcal{A} \\
\mbox{coneh}_0(0) \ar@{-->}[ur]_{\overline{L}} &
}
$$
o{\`u} le dg-foncteur $L$ correspond au morphisme identit{\'e} de
$X$. Alors $\overline{L}(X\!h\,_0)$ est un objet contractile dans
$\mathcal{A}$ et il existe donc un morphisme $R$ comme dans le diagramme
suivant.
$$
\xymatrix{
\emptyset \ar[d] \ar[r] & \mathcal{A} \\
\mathcal{B} \ar@{-->}[ur]_R & .
}
$$
On montre maintenant que le dg-foncteur $\mathcal{A} \rightarrow 0$ a
la propri{\'e}t{\'e} de relevement {\`a} droite par rapport aux dg-foncteurs
$L\!h_n(k_0,\ldots,k_n), n \in \mathbb{Z}$. On se donne un diagramme
$$
\xymatrix{
*+<1.5pc>{\mbox{idemh}_n(k_0,\ldots, k_n)} \ar[r]^-L \ar@{>->}[d]_{L_n(k_0,\ldots,k_n)}
& \mathcal{A} \\
\mbox{facth}_n(k_0, \ldots, k_n) &
}
$$
et on forme le carr{\'e} cocart{\'e}sien
$$
\xymatrix{
*+<1.5pc>{\mbox{idemh}_n(k_0,\ldots, k_n)} \ar[r]^-L
\ar@{>->}[d]_{L_n(k_0,\ldots,k_n)} \ar@{}[dr]|{\lrcorner}
& *+<1.5pc>{\mathcal{A}} \ar@{>->}[d] \\
\mbox{facth}_n(k_0, \ldots, k_n) \ar[r] & \cb\,.
}
$$
Par le lemme~\ref{pushout-homotopique}, le foncteur $\ca\to\cb$ est quasi-pleinement
fid{\`e}le et la cat{\'e}gorie $H^0(\cb)$ est obtenue
{\`a} partir de $H^0(\ca)$ en rajoutant un facteur direct d'une extension
it{\'e}r{\'e}e. Or, par hypoth{\`e}se, la cat{\'e}gorie $H^0(\ca)$ est stable par rajout
d'extensions it{\'e}r{\'e}es et de facteurs directs. Donc le foncteur $\ca\to\cb$
est une quasi-{\'e}quivalence. Mais il est aussi une cofibration pour la
structure dont les {\'e}quivalences faibles sont les quasi-{\'e}quivalences
(puisque le foncteur $\idemh_n\to \facth_n$ est une telle cofibration).
La dg-cat{\'e}gorie $\cb$ est fibrante pour cette structure (comme toute
petite dg-cat{\'e}gorie). Donc le foncteur $\ca\to\cb$ admet une r{\'e}traction.
Cela montre la proposition.
\end{proof}
On note $\mathcal{M}or$ la classe des dg-foncteurs de Morita et
$\mathsf{Hmo}$ la cat{\'e}gorie homotopique de $\dgcat$ par raport {\`a}
cette classe. 
\begin{prop}\label{Bousfil}
La structure de cat{\'e}gorie de mod{\`e}les de Quillen du th{\'e}or{\`e}me~\ref{theorem2} est
la localisation de Bousfield {\`a} gauche, voir \cite[3.3.1]{Hirschhorn},
de la structure quasi-{\'e}quiconique, par rapport au morphisme
$$ Lh_0 : \mathrm{idemh} \rightarrow \mathrm{facth}\,.$$
\end{prop}

\begin{proof}[D{\'e}monstration]
 On commence par montrer qu'une dg-cat{\'e}gorie $\mathcal{B}$ est
 $Lh_0$-locale, voir \cite[3.4]{Hirschhorn}, si et seulement si elle est Morita fibrante. Notons que puisque
 les deux structures de mod{\`e}les de Quillen sur $\dgcat$ ont les m{\^e}mes
 cofibrations, et donc les m{\^e}mes fibrations triviales, le foncteur de
 remplacement cofibrant simplicial $\Gamma^{\ast}$, voir
 \cite{Hirschhorn}, est le m{\^e}me dans les deux situations. Puisque
 $Lh_0$ est un dg-foncteur de Morita, cela
 implique que si $\mathcal{B}$ est Morita fibrante, le morphisme
 induit
$$ \mathsf{Map}(\mathrm{facth},\mathcal{B})
\stackrel{\sim}{\longrightarrow}
\mathsf{Map}(\mathrm{idemh},\mathcal{B})$$
est une {\'e}quivalence faible d'ensembles simpliciaux.

Soit maintenant $\mathcal{B}$ un objet $Lh_0$-local dans $\dgcat$.
Montrons qu'il est Morita fibrant. Par la proposition~\ref{nova4},
il s'agit de montrer que l'image essentielle du plongement
$\mathsf{H}^0(\mathcal{B}) \hookrightarrow \mathcal{D}(\mathcal{B})$
est stable par suspensions, c{\^o}nes et facteurs directs. Une fois que
$\mathcal{B}$ est fibrante pour la structure quasi-{\'e}quiconique
proposition~\ref{fibrants} implique que le plongement
$\mathsf{H}^0(\mathcal{B}) \hookrightarrow \mathcal{D}(\mathcal{B})$
est dej{\`a} stable par suspensions et c{\^o}nes. Maintenant la donn{\'e}e d'un
idempotent $e$ dans $\mathsf{H}^0(\mathcal{B})$ se rel{\`e}ve, par
lemme~\ref{important1}, en un morphisme $\mathrm{idemh} \rightarrow
\mathcal{B}$ et fournit donc une objet de $\rep(\mathrm{idemh},\cb)$.
Par la description de $\mathsf{Map}(-,-)$ dans \cite{Toen} et l'hypoth{\`e}se
sur $\cb$, le foncteur
\[
\rep(\mathrm{idemh}, \mathcal{B}) \to \rep(\mathrm{facth},\mathcal{B})
\]
est une {\'e}quivalence. Cela implique qu'il existe bien dans $\cb$ un objet
$F_e$ qui, dans $\cd\cb$, devient isomorphe au facteur directe associ{\'e}
{\`a} l'idempotent $e$.

On montre maintenant qu'un dg-foncteur est de Morita si et
seulement s'il est une $Lh_0$-{\'e}quivalence locale, voir
\cite[3.1.4]{Hirschhorn}. Un dg-foncteur $F$ est une
$Lh_0$-{\'e}quivalence locale ssi $\Map(F,\cc)$ est une {\'e}quivalence faible pour tout
$\cc$ $Lh_0$-locale, ssi $\Map(F,\cc)$ est une {\'e}quivalence faible pour toute
$\cc$ Morita fibrante. Une fois que les foncteurs $\mathsf{Map}(-,\mathcal{C})$
sont {\'e}quivalents pour les deux structures, avec $\mathcal{C}$ Morita
fibrante, $F$ est de Morita ssi c'est une $Lh_0$-{\'e}quivalence locale.
\end{proof}

Soient $\mathcal{A}$ et $\mathcal{B}$ des dg-cat{\'e}gories et $\mbox{can}_2: \mathsf{Heq} \rightarrow \mathsf{Hmo}$ le
foncteur canonique. La proposition~\ref{Bousfil} et le
corollaire~\ref{adjonct1} impliquent~:

\begin{corollaire} \label{adjonct2}
On a une adjonction
$$\mathrm{Hom}_{\mathsf{Hmo}}(\mbox{can}_2(\mathcal{A}),
\mathcal{B}) \stackrel{\sim}{\longrightarrow} \mathrm{Hom}_{\mathsf{Heq}}(\mathcal{A},
\mathcal{B}_{fib})\,,
$$
o{\`u} $\mathcal{B}_{fib}$ est un remplacement Morita fibrant
de la dg-cat{\'e}gorie $\mathcal{B}$.
\end{corollaire}

Soit $\mathsf{Heq}_{fib}$ la sous-cat{\'e}gorie pleine de $\mathsf{Heq}$
form{\'e}e des dg-cat{\'e}gories Morita fibrantes. On note
$\mbox{inc}:\mathsf{Heq}_{fib}\hookrightarrow \mathsf{Heq}$ l'inclusion
{\'e}vidente. Le corollaire~\ref{adjonct2} a comme consequence le

\begin{corollaire}\label{adjonct5}
On a une {\'e}quivalence de cat{\'e}gories
$$
\xymatrix{
\mathsf{Heq} \ar[d]_-{can_2}  &
\mathsf{Heq}_{fib} \ar@{^{(}->}[l]_-{inc}\\
\mathsf{Hmo} \ar[ur]^-{\sim}_{(-)_{fib}} &
}
$$
\end{corollaire}

Soit $\mathsf{Heq}_{kar}$ la sous-cat{\'e}gorie pleine de
$\mathsf{Heq}_{ex}$, voir notation~\ref{notexact}, form{\'e}e des dg-cat{\'e}gories $\mathcal{A}$ telles
que l'image du plongement $h:\mathrm{H}^0(\mathcal{A})
\hookrightarrow \mathcal{D}(\mathcal{A})$ est stable par facteurs
directs.

\begin{lemme}\label{lemme3}
Les cat{\'e}gories $\mathsf{Heq}_{kar}$ et $\mathsf{Heq}_{fib}$ sont {\'e}quivalentes.
\end{lemme}
\begin{proof}[D{\'e}monstration]
Soir $\mathcal{B} \in \mathsf{Heq}_{fib}$. Alors on a une
quasi-{\'e}quivalence
$$ \mathcal{B} \stackrel{\sim}{\longrightarrow}
\mbox{pre-tr}(\mathcal{B}) \,.$$
Par proposition~\ref{nova4}, $\mbox{pre-tr}(\mathcal{B})$ est Morita
fibrante et une fois qu'elle est aussi exacte elle appartient {\`a} $\mathsf{Heq}_{kar}$. 
\end{proof}

On note $\mathsf{rep}_{mor}(\mathcal{A}, \mathcal{B})$ la sous-cat{\'e}gorie pleine
de la cat{\'e}gorie d{\'e}riv{\'e} $\mathcal{D}(\mathcal{A}^{op} \otimes^{\mathbb{L}} \mathcal{B})$,
  voir \cite{Drinfeld} \cite{DerivingDG}, dont les objets sont les dg
  $\mathcal{A}\mbox{-}\mathcal{B}\,$-bimodules $X$ telles que $X(?, A)$ est
  isomorphe dans $\mathcal{D}(\mathcal{B})$ {\`a} un objet de l'image du
  dg-foncteur canonique $\mathrm{H}^0(\mbox{pre-tr}(\mathcal{B}))^{\kar}
  \rightarrow
  \mathrm{H}^0(\mathrm{Mod}\, \mathcal{B})$, pour tout $A \in \mathcal{A}$.

\begin{corollaire}\label{adjonct6}
On a une bijection
$$\mathrm{Hom}_{\mathsf{Hmo}}(\mathcal{A},
\mathcal{B}) \stackrel{\sim}{\longrightarrow} \mathrm{Iso}(\mathsf{rep}_{mor}(\mathcal{A}, \mathcal{B}))\,.
$$
\end{corollaire}

\begin{proof}[D{\'e}monstration]
D'apr{\`e}s le corollaire~\ref{adjonct2}, $\mathrm{Hom}_{\mathsf{Hmo}}(\mathcal{A},
\mathcal{B})$ s'identifie {\`a}  $\mathrm{Hom}_{\mathsf{Heq}}(\mathcal{A},
\mathcal{B}_{fib})$. Par \cite{Toen} on sait que
ce dernier ensemble s'identifie {\`a}
$\mathrm{Iso}(\mathsf{rep}(\mathcal{A},\mathcal{B}_{fib}))$.
On sait par le lemme~\ref{lemme3} que $\mathcal{B}_{fib}$ est
isomorphe dans $\mathsf{Heq}$ {\`a} une dg-cat{\'e}gorie qui appartient {\`a}
$\mathsf{Heq}_{kar}$. Cela implique que $\mathsf{rep}(\mathcal{A},\mathcal{B}_{fib})$
s'identifie {\`a} $\mathsf{rep}_{mor}(\mathcal{A},\mathcal{B})$, d'o{\`u} le resultat.
\end{proof}

\begin{remarque}\label{monoi}
On sait que les structures de cat{\'e}gorie de mod{\`e}les de Quillen dans
$\dgcat$ de \cite{cras} et du th{\'e}or{\`e}me~\ref{theorem2} ont les m{\^e}mes
cofibrations et les m{\^e}mes fibrations acycliques. Cela a pour
consequence que
\begin{itemize}
\item[-] Le foncteur de remplacement cofibrant simplicial $\Gamma^{\ast}$, voir
\cite{Hovey}, dans $\dgcat$ est le m{\^e}me dans les deux situations. Cela
implique que les corollaires~\ref{adjonct2}, ~\ref{adjonct5} et~\ref{adjonct6} sont encore vrais si on remplace $\mathrm{Hom}$ par
l'espace de morphismes, $\underline{\mathrm{Map}}$, voir \cite{Hovey}, et
$\mathrm{Iso}(\mathsf{rep}_{mor}(\mathcal{A}, \mathcal{B}))$ dans le
corollaire~\ref{adjonct6}, par le nerf de la sous-cat{\'e}gorie de
$\mathrm{Mod}(\mathcal{A}^{op}\otimes^{\mathbb{L}}\mathcal{B})$ dont les objets
sont les m{\^e}mes que ceux de $\mathsf{rep}_{mor}(\mathcal{A},
\mathcal{B})$ et dont les morphismes sont les quasi-isomorphismes.
\item[-] La cat{\'e}gorie mono{\"\i}dale sym{\'e}trique
  $(\mathsf{Hmo},\,-\otimes^{\mathbb{L}}-)$ est bien d{\'e}finie et ferm{\'e}e, voir \cite{Hovey}, et l'espace de
  morphismes interne de $\mathsf{Hmo}$, $\mathbb{R}\underline{\mathrm{Hom}}_{\mathsf{Hmo}}(\mathcal{A},
    \mathcal{B})$ s'identifie {\`a}
  $\mathbb{R}\underline{\mathrm{Hom}}_{\mathsf{Heq}}(\mathcal{A}, \mathcal{B}_{fib})$, voir \cite{Toen}.
\end{itemize}
\end{remarque}

\section{Invariants}
Soient $\mathcal{A}$, $\mathcal{B}$ et $\mathcal{C}$ des
dg-cat{\'e}gories. Soit $\mathsf{Hmo}_0$ la cat{\'e}gorie qui a pour objets
les petites dg-cat{\'e}gories et telle que $\mbox{Hom}_{\mathsf{Hmo}_0}(\mathcal{A}, \mathcal{B})$ est le groupe
de Grothendieck de la cat{\'e}gorie triangul{\'e}e
$\mathsf{rep}_{mor}(\mathcal{A}, \mathcal{B})$. Un argument
compl{\`e}tement analogue {\`a} la remarque~\ref{bitri} montre que l'op{\'e}ration
de composition dans $\mathsf{Hmo}_0$ est bien d{\'e}finie. On dispose donc
d'un foncteur canonique $\mbox{add}_2:\mathsf{Hmo} \rightarrow \mathsf{Hmo}_0$.
Des arguments similaires aux lemmes~\ref{addit} et \ref{tens} nous
permettent de montrer le lemme suivant.

\begin{lemme}\label{ad}
La cat{\'e}gorie $\mathsf{Hmo}_0$ est additive et le produit tensoriel
$-\otimes^{\mathbb{L}}-$ de $\mathsf{Hmo}$, voir remarque~\ref{monoi},
induit une structure mono{\"\i}dale sym{\'e}trique sur $\mathsf{Hmo}_0$.
\end{lemme}

Soit $\mathsf{Heq}_{{kar}_0}$ la sous-cat{\'e}gorie pleine de
$\mathsf{Hmo}_0$ dont les objets sont les dg-cat{\'e}gories exactes
$\mathcal{A}$ telles que l'image du plongement
$\widehat{ }:\mathrm{H}^0(\mathcal{A}) \hookrightarrow \mathcal{D}(\mathcal{A})$
est stable par facteurs directs.
\begin{remarque}
L'{\'e}quivalence entre $\mathsf{Hmo}$ et $\mathsf{Heq}_{kar}$ induit une
{\'e}quivalence entre $\mathsf{Hmo}_0$ et $\mathsf{Heq}_{{kar}_0}$.
\end{remarque}
On a donc le diagramme suivant~:
$$
\xymatrix{
\mathsf{Heq}_{{ex}_0} \ar@{^{(}->}[r]^-{\sim} & \mathsf{Hec}_0 \ar[d] \\
\mathsf{Heq}_{{kar}_0} \ar@{^{(}->}[u] \ar@{^{(}->}[r]^-{\sim} &
\mathsf{Hmo}_0\, .
}
$$
La d{\'e}monstration du th{\'e}or{\`e}me suivant est analogue {\`a} celle du th{\'e}or{\`e}me~\ref{semiloc}.
Soit $F:\mathsf{Hmo}\rightarrow \mathsf{C}$ un foncteur {\`a} valeurs
dans une cat{\'e}gorie additive $\mathsf{C}$.

\begin{theoreme}\label{Univ}
Les conditions suivantes sont {\'e}quivalentes~:
\begin{itemize}
\item[1)] Le foncteur $F$ est compos{\'e} d'un foncteur additif
  $\mathsf{Hmo}_0 \rightarrow \mathsf{C}$ et du foncteur canonique
  $\mathsf{Hmo}\rightarrow \mathsf{Hmo}_0$.
\item[2)] Pour toutes dg-cat{\'e}gories $\mathcal{A}$, $\mathcal{B}$, l'identit{\'e}
  $F(\left[X\right])+F(\left[Z\right])=F(\left[Y\right])$, est v{\'e}rifi{\'e}e
  dans $\mathrm{Hom}_{\mathsf{C}}(F(\mathcal{A}), F(\mathcal{B}))$ pour
  tout triangle $X \rightarrow Y \rightarrow Z \rightarrow X\left[1\right]$
  de $\mathsf{rep}_{mor}(\mathcal{A}, \mathcal{B})$.
\item[3)] Pour toute dg-cat{\'e}gorie $\mathcal{A}$, le morphisme
$$
\xymatrix{
F(\mathcal{A})\oplus F(\mathcal{A})
\ar[rr]^-{\left[F(i_1)
        \, , \,F(i_2)\right]} && F(\mbox{T}(\mathcal{A}))
}
$$
est un isomorphisme dans $\mathsf{C}$.

\item[4)] Pout toute dg-cat{\'e}gorie pr{\'e}triangul{\'e}e $\mathcal{A}$ munie de sous-dg-cat{\'e}gories pleines pr{\'e}triangul{\'e}es
  $\mathcal{B}$ et $\mathcal{C}$ qui donnent lieu {\`a} une d{\'e}composition
  semi-orthogonale
  $\mathrm{H}^0(\mathcal{A})=(\mathrm{H}^0(\mathcal{B}),
  \mathrm{H}^0(\mathcal{C}))$, voir
  \cite{Bondal}, le morphisme
$$
F(\mathcal{B})\oplus F(\mathcal{C}) \rightarrow F(\mathcal{A})
$$
induit par les inclusions est un isomorphisme dans $\mathsf{C}$.
\end{itemize}
\end{theoreme}

\begin{remarque}
Toute {\'e}quivalence d{\'e}riv{\'e}e, voir \cite{Rickard} \cite{Tilting}, donne un isomorphisme dans $\mathsf{Hmo}$ et donc dans
$\mathsf{Hmo}_0$. Cependant, il existe d'autres isomorphismes dans
$\mathsf{Hmo}_0$~: si $k$ est un corps alg{\'e}briquement clos et $A$ un
$k$-alg{\`e}bre (ordinaire) de dimension finie sur $k$ et de dimension
globale finie, alors $A$ est isomorphe {\`a} son plus grand quotient
semisimple $A/\mathrm{rad}(A)$ dans $\mathsf{Hmo}_0$
(voir \cite[2.5]{cyclicDG}), mais dans $\mathsf{Hmo}$, $A$ est isomorphe {\`a}
$A/\mathrm{rad}(A)$ seulement si $A$ est semisimple.
\end{remarque}

\subsection{Homologies de Hochschild, Cyclique, Negative, \ldots}
On note $\mathcal{D} \mbox{Mix}$ la cat{\'e}gorie d{\'e}riv{\'e}e de la cat{\'e}gorie
des complexes mixtes sur $k$, voir \cite{cyclichomology}. Soit $$M:\dgcat
\longrightarrow \mathcal{D}\mbox{Mix}$$ le foncteur qui, {\`a} une petite
dg-cat{\'e}gorie $\mathcal{C}$, associe son complexe mixte
$M(\mathcal{C})$, vu comme un objet dans $\mathcal{D}\mbox{Mix}$,
voir \cite{cyclichomology}. D'apr{\`e}s \cite{Kassel}, l'homologie cyclique de
$\mathcal{C}$ est l'homologie cyclique du complexe
mixte $M(\mathcal{C})$ et de m{\^e}me pour les autres variantes de la
th{\'e}orie (Hochschild, p{\'e}riodique, negative, \ldots). Comme dans
\cite{cyclicDG}, on montre que le foncteur
$M$ descend {\`a} un foncteur d{\'e}fini dans la localis{\'e}e $\mathsf{Hmo}$
et que celui-ci satisfait les conditions du th{\'e}or{\`e}me~\ref{Univ}. Il se factorise donc par le foncteur $\mathcal{U}_a:\dgcat
\longrightarrow \mathsf{Hmo}_0$ et donne lieu {\`a} un foncteur additif
$$
M:\mathsf{Hmo}_0 \rightarrow \mathcal{D}\mbox{Mix}\,.
$$

\subsection{$K$-th{\'e}orie alg{\'e}brique}
On note $\mathsf{Ho}(\mathsf{Spt})$ la cat{\'e}gorie homotopique des
spectres, voir \cite{Sspectra}. Soit $$K:\dgcat \longrightarrow
\mathsf{Ho}(\mathsf{Spt})$$ le foncteur qui, {\`a} une petite dg-cat{\'e}gorie
$\mathcal{C}$, associe le spectre de $K$-th{\'e}orie de Waldhausen
\cite{Waldhausen} associ{\'e} {\`a} la cat{\'e}gorie des dg-modules cofibrants et
parfaits dont les cofibrations sont les monomorphismes et les
{\'e}quivalences faibles les quasi-isomorphismes, voir \cite{Dugger-S}.
Comme il a {\'e}t{\'e} montr{\'e} dans \cite{Dugger-S}, les r{\'e}sultats de \cite{Waldhausen}
impliquent que le foncteur $K$ rend inversibles les dg-foncteurs de
Morita et descend donc en un foncteur $\mathsf{Hmo} \rightarrow
\mathsf{Ho}(\mathsf{Spt})$. Le th{\'e}or{\`e}me d'additivit{\'e} de Waldhausen
\cite[1.4]{Waldhausen} montre que ce foncteur v{\'e}rifie la condition $3)$ du
th{\'e}or{\`e}me~\ref{Univ}. Il induit donc un foncteur additif $$ K: \mathsf{Hmo}_0 \rightarrow \mathsf{Ho}(\mathsf{Spt})\, .$$

\subsection{Vision globale}
On dispose donc du diagramme suivant~:
$$
\xymatrix@!0 @R=3pc @C=7pc{
 \dgcat \ar[d]
\ar@/_-1pc/[ddddrr]^-{K} & &   \\
  \mathsf{Heq} \ar[d]^{can_1} \ar@/^-2pc/[dd]_-{can_2} & &  \\
 \mathsf{Hec} \ar[d] \ar[dr]^-{add_1} & &  \\
  \mathsf{Hmo} \ar[dr]_-{add_2} & \mathsf{Hec}_0 \ar[d] &  \\
  & \mathsf{Hmo}_0 \ar@{.>}[r]^-{K}  & \mathsf{Ho}(\mathsf{Spt})\,.
}
$$

\subsection{Charact{\`e}re de Chern}
On note $\mbox{K}_0(\mathcal{C})$ le groupe de Grothendieck d'une
dg-cat{\'e}gorie $\mathcal{C}$, c'est-{\`a}-dire le groupe de Grothendieck de
la cat{\'e}gorie des objets compacts dans $\mathcal{D}(\mathcal{C})$.
Soit $\mathcal{A}$ la dg-cat{\'e}gorie avec un seul objet $3$ et telle
que $\mathrm{Hom}_{\mathcal{A}}(3,3)=k$.
\begin{lemme}\label{jolie}
On a un isomorphisme naturel de groupes abeliens
$$\mathrm{Hom}_{\mathsf{Hmo}_0}(\mathcal{A}, \mathcal{C})
\stackrel{\sim}{\longrightarrow} \mbox{K}_0(\mathcal{C})\, .$$
\end{lemme}

\begin{proof}[D{\'e}monstration]
En effet le groupe abelien $\mathrm{Hom}_{\mathsf{Hmo}_0}(\mathcal{A},
\mathcal{B})$ est par d{\'e}finition {\'e}gal {\`a}
$\mbox{K}_0(\mathsf{rep}_{mor}(\mathcal{A}, \mathcal{C}))$. On sait
que la cat{\'e}gorie $\mathsf{rep}_{mor}(\mathcal{A}, \mathcal{C})$
s'identifie {\`a} $\mathsf{rep}(\mathcal{A}, \mathcal{C}_{fib})$, o{\`u}
$\mathcal{C}_{fib}$ est un remplacement Morita fibrant de la dg-cat{\'e}gorie
$\mathcal{C}$. On remarque que la cat{\'e}gorie $\mathsf{rep}(\mathcal{A}, \mathcal{C}_{fib})$ s'identifie {\`a}
la sous-cat{\'e}gorie pleine de $\mathcal{D}(\mathcal{C})$ form{\'e}e des
objets compacts.
\end{proof}
Soit $F:\mathsf{Hmo}_0 \rightarrow \mathrm{Mod}\,\mathbb{Z}$ un
foncteur. Le lemme de Yoneda et le lemme pr{\'e}c{\'e}dent impliquent le lemme suivant.

\begin{lemme}
On a un isomorphisme de groupes
$$ \mathrm{Nat}(\mbox{K}_0, F) \stackrel{\sim}{\longrightarrow}
F(\mathcal{A})\, ,$$
o{\`u} $\mathrm{Nat}(\mbox{K}_0, F)$ d{\'e}signe le groupe ab{\'e}lien des
transformations naturelles de $\mbox{K}_0$ vers $F$.
\end{lemme}

A titre d'exemple, pour $n \in \mathbb{N}$, soit
$\mathrm{HC}_n(\mathcal{C})$ le $n$-i{\`e}me groupe d'homologie
cyclique. On sait d{\'e}j{\`a} que le foncteur
$$\mathrm{HC}_n: \dgcat \rightarrow \mathrm{Mod}\, \mathbb{Z}$$ se
factorise par $\dgcat \rightarrow \mathsf{Hmo}_0$. A partir de
l'isomorphisme
$$ \mathrm{HC}_*(k)\simeq k \left[ u \right], \, \, \left| u \right|
=2\, , $$
le lemme nous fournit les caract{\`e}res de Chern
$$ch_{2n}: \mbox{K}_0 \rightarrow \mathrm{HC}_{2n} \, . $$

\section{DG-cat{\'e}gories compactes et lisses}
\begin{remarque}\label{nul}
On observe facilement que la cat{\'e}gorie aditive $(\mathsf{Hec}_0,
\oplus)$ poss{\`e}de des sommes infinies. Cela implique que son groupe de
Grothendieck $\mbox{K}_0(\mathsf{Hec}_0, \oplus)$ est nul.
\end{remarque}
Suivant \cite{IHP}, une dg-cat{\'e}gorie $\mathcal{A}$ est {\em compacte} si
\begin{itemize}
\item[1)] son ensemble d'objets est fini et
\item[2)] pour tous objets $X$ et $Y$ dans $\mathcal{A}$, le complexe
  $\mathrm{Hom}_{\mathcal{C}}(X, Y)$ est un objet parfait dans la
  cat{\'e}gorie d{\'e}riv{\'e}e des $k$-modules. 
\end{itemize}
La condition $2)$ correspond {\`a} la notion de perfection locale d{\'e}finie dans \cite[2.4]{Toen-Vaq}.
La dg-cat{\'e}gorie $\mathcal{A}$ est {\em lisse} si le $\mathcal{A}$-$\mathcal{A}$-bimodule associ{\'e} {\`a} l'identit{\'e}
  de $\mathcal{A}$, qu'on note $_{\mathcal{A}}\mathcal{A}_{\mathcal{A}}$, voir
  remarque~\ref{bimod}, est un objet compact dans
  $\mathcal{D}(\mathcal{A}^{op}\otimes \mathcal{A})$.
On note $\mathsf{Hec}_0^{cl}$ la sous-cat{\'e}gorie pleine de
$\mathsf{Hec}_0$ dont les objets sont les dg-cat{\'e}gories compactes et lisses.

\begin{lemme}
La cat{\'e}gorie $\mathsf{Hec}_0^{cl}$ est une sous-cat{\'e}gorie additive
mono{\"\i}dale de $\mathsf{Hec}_0$.
\end{lemme}
\begin{proof}[D{\'e}monstration]
Pour des dg-cat{\'e}gories compactes et lisses $\mathcal{A}$ et $\mathcal{B}$, on
remarque que
$$
\mathcal{D}((\mathcal{A}\amalg\mathcal{B})^{op}\otimes(\mathcal{A}\amalg\mathcal{B}))\backsimeq
\mathcal{D}(\mathcal{A}^{op}\otimes \mathcal{A})\times
\mathcal{D}(\mathcal{A}^{op}\otimes \mathcal{B})\times
\mathcal{D}(\mathcal{B}^{op}\otimes \mathcal{A})\times \mathcal{D}(\mathcal{B}^{op}\otimes \mathcal{B})\,.$$
Puisque le bimodule $_{\mathcal{A}\amalg\mathcal{B}}\mathcal{A}\amalg\mathcal{B}_{\mathcal{A}\amalg\mathcal{B}}$
  s'identifie {\`a} $_{\mathcal{A}}\mathcal{A}_{\mathcal{A}}\times 0 \times 0 \times
_{\mathcal{B}}\mathcal{B}_{\mathcal{B}}$, il est compact. Clairement, $\mathcal{A}\amalg\mathcal{B}$ est encore
compacte. Le fait que $\mathcal{A}\otimes\mathcal{B}$ soit lisse
r{\'e}sulte de l'isomorphisme
$$
\mathbb{R}\mathrm{Hom}_{{(\mathcal{A}\otimes\mathcal{B})}^{op}\otimes(\mathcal{A}\otimes\mathcal{B})}(\mathcal{A}\otimes\mathcal{B},
?) \backsimeq
\mathbb{R}\mathrm{Hom}_{\mathcal{A}^{op}\otimes\mathcal{A}}(\mathcal{A},
\mathbb{R}\mathrm{Hom}_{\mathcal{B}^{op}\otimes\mathcal{B}}(\mathcal{B}, ?))\,.
$$
La dg-cat{\'e}gorie $\mathcal{A}\otimes\mathcal{B}$ est encore compacte
une fois que les objets parfaits dans la cat{\'e}gorie d{\'e}riv{\'e}e des
$k$-modules sont stables par $\otimes$.
\end{proof}

On note $\mathrm{HH}(\mathcal{A})$ le complexe de Hochschild d'une
dg-cat{\'e}gorie $\mathcal{A}$, consid{\'e}r{\'e} comme un objet dans $\mathcal{D}(k)$. On dispose donc d'un
foncteur additif $$\mathrm{HH}: \mathsf{Hec}_0 \longrightarrow
\mathcal{D}(k)\,.$$ En outre, si la dg-cat{\'e}gorie
$\mathcal{A}$ est compacte et lisse, alors
$\mathrm{HH}(\mathcal{A})$ est un complexe parfait sur $k$. On dispose du diagramme~:
$$
\xymatrix{
(\mathsf{Hec}_0, \oplus) \ar[r]^-{\mathrm{HH}} &
\mathcal{D}(k) \\
*+<1pc>{(\mathsf{Hec}_0^{cl}, \oplus)} \ar@{^{(}->}[u]
\ar[r]^-{\mathrm{HH}} & *+<1pc>{\mathrm{per}(k)}
\ar@{^{(}->}[u]
}
$$
o{\`u} $\mathrm{per}(k)$ d{\'e}signe la plus petite cat{\'e}gorie triangul{\'e}e de
$\mathcal{D}(k)$ qui contient $k$ et est stable par facteurs directs. En applicant le foncteur $\mbox{K}_0$ on
obtient
$$
\xymatrix{
 \mathrm{K}_0(\mathsf{Hec}_0^{cl}, \oplus)
 \ar[rr]^-{\mathrm{K}_0(\mathrm{HH})} & &
\mathrm{K}_0(\mathrm{per}(k)) \,.
}
$$
En particulier, ces anneaux commutatifs, avec la structure
multiplicative induite par le produit tensoriel, sont non nuls.

Soit $\mathcal{P}\mathcal{T}$ le groupe de Grothendieck d{\'e}fini dans
\cite{Bondal}. Rappelons que les auteurs de \cite{Bon-Kap},
\cite{Bondal} appelent {\em pr{\'e}triangul{\'e}es} les petites dg-cat{\'e}gories
$\mathcal{A}$ telles que $\mathcal{A} \hookrightarrow
\mbox{pre-tr}(\mathcal{A})$ est une quasi-{\'e}quivalence (c'est-{\`a}-dire
que $\mathcal{A}$ est fibrante pour la structure de cat{\'e}gorie de
mod{\`e}les de Quillen quasi-{\'e}quiconique d{\'e}crite dans le
th{\'e}or{\`e}me~\ref{theorem}). Le groupe $\mathcal{P}\mathcal{T}$ est par
d{\'e}finition le groupe ab{\'e}lien engendr{\'e} par les classes de
quasi-{\'e}quivalence $\left[\mathcal{A}\right]$ de petites dg-cat{\'e}gories
pr{\'e}triangul{\'e}es soumises aux relations qui proviennent des
d{\'e}compositions semi-orthogonales.
\par
L'argument de la remarque~\ref{nul}
montre qu'en fait le groupe $\mathcal{P}\mathcal{T}$ s'annule. Pour
obtenir un groupe non nul, il convient d'imposer des conditions de
finitude~: soit ${\mathcal{P}\mathcal{T}}^{cl}$ le groupe ab{\'e}lien
engendr{\'e} par les classes de quasi-{\'e}quivalences $\left[ \mathcal{A}
\right]$ de petites dg-cat{\'e}gories $\mathcal{A}$ compactes, lisses et
pr{\'e}triangul{\'e}es soumises aux relations qui proviennent des
d{\'e}compositions semi-orthogonales. Le produit tensoriel munit ${\mathcal{P}\mathcal{T}}^{cl}$ d'une structure d'anneau commutatif.

\par
Par exemple, si $X$ est une vari{\'e}t{\'e} projective lisse, alors la
cat{\'e}gorie d{\'e}riv{\'e}e $\mathcal{D}^b(\mathsf{coh}X)$ est {\'e}quivalent {\`a}
$\mathrm{H}^0(\mathcal{D}^b(\mathsf{coh}X)_{dg})$ pour une
dg-cat{\'e}gorie compacte, lisse, pr{\'e}triangul{\'e}e
$\mathcal{D}^b(\mathsf{coh}X)_{dg}$ canonique {\`a} isomorphisme pr{\`e}s dans
$\mathsf{Hmo}$. Par exemple, pour $\mathcal{D}^b(\mathsf{coh}X)_{dg}$,
on peut prendre la dg-cat{\'e}gorie des complexes born{\'e}s {\`a} gauche de
$\mathcal{O}_X$-modules injectifs dont l'homologie est born{\'e}e et
coh{\'e}rente.

\begin{prop}\label{surj}
On a un morphisme surjectif d'anneaux commutatifs
$$
{\mathcal{P}\mathcal{T}}^{cl} \rightarrow \mbox{K}_0(\mathsf{Hec}_0^{cl}, \oplus)\, .
$$
\end{prop}

\begin{proof}[D{\'e}monstration]

On note $\left[\mathcal{A}\right]$ la classe de quasi-{\'e}quivalence
d'une dg-cat{\'e}gorie compacte, lisse et pr{\'e}triangul{\'e}e et $\left[\left[
    \mathcal{A} \right]\right]$ la classe d'isomorphisme de
$\mathcal{A}$ dans $\mathsf{Hec}_0^{cl}$. On consid{\`e}re l'application qui, {\`a} $\left[\mathcal{A}\right]$, associe
$\left[\left[ \mathcal{A} \right]\right]$. Elle est clairement bien
d{\'e}finie. Pour qu'elle induise un morphisme de
${\mathcal{P}\mathcal{T}}^{cl}$ vers $\mbox{K}_0(\mathsf{Hec}_0^{cl}, \oplus)$, il reste
{\`a} verifier l'{\'e}galit{\'e}
$$
\left[\left[ \mathcal{B}\oplus \mathcal{C} \right]\right]=\left[\left[\mathcal{A}\right]\right]\, ,
$$
pour toutes dg-cat{\'e}gories $\mathcal{A}$, $\mathcal{B}$ et
$\mathcal{C}$ qui satisfont les relations dans la d{\'e}finition de
${\mathcal{P}\mathcal{T}}^{cl}$. On consid{\`e}re le diagramme
$$
\xymatrix{
 &&&& *+<1pc>{\mathcal{A}} \ar@{^{(}->}[d]^-{inc} \\
\mathcal{B}\oplus\mathcal{C}=\mbox{pre-tr}(\mathcal{B}\coprod \mathcal{C})
\ar[rrrr]^-{pre-tr(\left[ inc(\mathcal{B}) \, \,
    inc(\mathcal{C})\right])} &&&& \mbox{pre-tr}(\mathcal{A})\,.
}
$$
Le dg-foncteur $inc$ est une quasi-{\'e}quivalence et il reste {\`a} montrer que l'image du dg-foncteur
$\mbox{pre-tr}(\left[ inc(\mathcal{B}) \,\, inc(\mathcal{C})\right])$
dans $\mathsf{Hec}_0^{cl}$ est un isomorphisme. Par le lemme de
Yoneda, il s'agit de montrer que, pour toute dg-cat{\'e}gorie
$\mathcal{U}$, l'application
$$
\mbox{K}_0(\mathsf{rep}_{ec}(\mathcal{U},
\mbox{pre-tr}(\mathcal{B}\amalg \mathcal{C}))) \longrightarrow \mbox{K}_0(\mathsf{rep}_{ec}(\mathcal{U},
\mbox{pre-tr}(\mathcal{A})))
$$
est bijective. En effet, la d{\'e}composition semi-orthogonale
$\mbox{H}^0(\mathcal{A})=(\mbox{H}^0(\mathcal{B}),
\mbox{H}^0(\mathcal{C}))$, induit une
d{\'e}composition semi-orthogonale
$$
\mathsf{rep}_{ec}(\mathcal{U},
\mbox{pre-tr}(\mathcal{A}))=(\mathsf{rep}_{ec}(\mathcal{U},
\mbox{pre-tr}(\mathcal{B})), \mathsf{rep}_{ec}(\mathcal{U},
\mbox{pre-tr}(\mathcal{C}))) \, .
$$
La surjectivit{\'e} du morphisme
${\mathcal{P}\mathcal{T}}^{cl}\rightarrow
\mbox{K}_0(\mathsf{Hec}_0^{cl}, \oplus)$ est
claire par construction.
\end{proof}

\begin{remarque}
En particulier, la proposition~\ref{surj} montre que
${\mathcal{P}\mathcal{T}}^{cl}$ est non nul. Si on d{\'e}finit
${\mathcal{P}\mathcal{T}}^{cl}_{kar}$ de fa{\c c}on analogue {\`a} partir des
dg-cat{\'e}gories compactes, lisses, pr{\'e}triangul{\'e}es $\mathcal{A}$ telles
que $\mathrm{H}^0(\mathcal{A})$ est Karoubienne, on obtient la
surjection
$$
{\mathcal{P}\mathcal{T}}^{cl}_{kar} \rightarrow \mbox{K}_0(\mathsf{Hmo}_0^{cl})
$$
mentionn{\'e}e dans l'introduction.
\end{remarque}

\chapter{Higher $K$-theory via universal invariants}

{\it Ce chapitre correspond au article \cite{univ}.}

\section{Introduction}
Differential graded categories (=dg categories) enhance our
understanding of triangulated categories appearing in algebra and
geometry, see \cite{ICM}.

They are considered as non-commutative schemes by Drinfeld
\cite{Drinfeld} \cite{Chicagotalk} and Kontsevich \cite{IHP} \cite{ENS} in their program
of non-commutative algebraic geometry, i.e. the study of dg categories
and their homological invariants.

In this article, using the formalism of Grothendieck's derivators, we construct `the
universal localizing invariant of dg categories' \emph{cf.}~\cite{Keller2002}. By this, we mean a
morphism $\mathcal{U}_l$ from the pointed derivator $\mathsf{HO}(\mathsf{dgcat})$ associated
with the Morita homotopy theory of dg categories, see \cite{cras}
\cite{IMRN} \cite{addendum}, to a triangulated
strong derivator $\mathcal{M}_{dg}^{loc}$ such that $\mathcal{U}_l$ commutes with filtered homotopy
colimits, preserves the point, sends each exact sequence of dg categories to a triangle and is universal for these properties.
Because of its universality property reminiscent of motives,
\emph{cf.} section $4.1$ of \cite{Konnew}, we call $\mathcal{M}_{dg}^{loc}$
the (stable) {\em localizing motivator} of dg categories.

Similary, we construct `the universal additive invariant of dg categories', i.e. the
universal morphism of derivators $\mathcal{U}_a$ from
$\mathsf{HO}(\mathsf{dgcat})$ to a strong triangulated derivator
$\mathcal{M}_{dg}^{add}$ which satisfies the first two properties but
the third one only for split exact sequences. We call
$\mathcal{M}_{dg}^{add}$ the {\em additive motivator} of dg
categories.

We prove that Waldhausen $K$-theory spectrum appears as a spectrum of morphisms in the
base category $\mathcal{M}_{dg}^{add}(e)$ of the additive
motivator. To the best of the author's knowledge, this is the first
conceptual characterization of Quillen-Waldhausen $K$-theory
\cite{Quillen} \cite{Waldhausen}.

The co-representation of $K$-theory as a spectrum of morphisms extends our
results in \cite{IMRN} \cite{addendum}, where we co-represented $K_0$
using functors with values in {\em additive categories} rather than
morphisms of derivators with values in strong {\em triangulated derivators}.

For example, the mixed complex construction \cite{cyclichomology}, from which
all variants of cyclic homology are deduced, and the
non-connective algebraic $\mbox{K}$-theory \cite{Marco} are localizing invariants and
factor through $\mathcal{U}_l$ and through $\mathcal{U}_a$. The
connective $K$-theory \cite{Waldhausen} is an example of an additive
invariant which is not a localizing one. We prove that it becomes
co-representable in $\mathcal{M}^{add}_{dg}$, see theorem~\ref{corep}.

Our construction is similar in spirit to those of Meyer-Nest
\cite{Nest}, Corti{\~n}as-Thom \cite{Cortinas} and Garkusha
\cite{Garkusha}. It splits into several general steps and also offers
some insight into the relationship between the theory of derivators
\cite{Heller} \cite{Grothendieck} \cite{KellerUniv} \cite{Malt}
\cite{Cis-Nee} and the classical theory
of Quillen model categories \cite{Quillen}. Derivators allow us to
state and prove precise universal properties and to dispense with many
of the technical problems one faces in using model categories.

In chapter~\ref{pre} we point out the connection between the
notions of Grothendieck derivator and that of small homotopy theory in
the sense of Heller \cite{Heller}. In chapter~\ref{extension}, we recall Cisinski's
theory of derived Kan extensions \cite{Cisinski} and in chapter~\ref{localisation}, we
develop his ideas on the Bousfield localization of derivators
\cite{Letter}. In particular, we characterize the derivator associated
with a left Bousfield localization of a Quillen model category by a
universal property, see theorem~\ref{Cisinsk}. This is based on a
constructive description of the local weak equivalences.

In chapter~\ref{homotopy}, starting from a Quillen model category
$\mathcal{M}$ satisfying some compactness conditions, we construct a morphism of prederivators
$$ \mathsf{HO}(\mathcal{M})
\stackrel{\mathbb{R}\underline{h}}{\longrightarrow}
\mathsf{L}_{\Sigma}\mathsf{Hot}_{\mathcal{M}_f}$$
which commutes with filtered homotopy colimits, has a derivator as
target and is universal for these properties. In chapter~\ref{chappoint} we
study morphisms of pointed derivators and in chapter~\ref{small} we prove a
general result which garantees that small weak generators are
preserved under left Bousfield localizations. In chapter~\ref{spectra}, we
recall Heller's stabilization construction \cite{Heller} and we prove
that this construction takes `finitely generated' unstable theories to
compactly generated stable ones. We establish the connection between
Heller's stabilization and Hovey/Schwede's stabilization \cite{Spectra}
\cite{Schwede} by proving that if we start with a pointed Quillen
model category which satisfies some mild `generation' hypotheses, then
the two stabilization procedures yield equivalent results. This allows us to
characterize Hovey/Schwede's construction by a universal property and
in particular to give a very simple characterisation of the classical
category of spectra in the sense of Bousfield-Friedlander \cite{Bos-Fri}. In
chapter~\ref{chapquotient}, by applying the general arguments of the previous
chapters to the Morita homotopy theory of dg categories \cite{cras}
\cite{IMRN} \cite{addendum}, we construct the universal morphism of
derivators
$$ \mathcal{U}_t: \, \mathsf{HO}(\mathsf{dgcat}) \longrightarrow
\mathsf{St}(\mathsf{L}_{\Sigma,P}\mathsf{Hot}_{\mathsf{dgcat}_f})$$
which commutes with filtered homotopy colimits, preserves the
point and has a strong triangulated derivator as target. For every inclusion $\mathcal{A} \stackrel{K}{\hookrightarrow}
\mathcal{B}$ of a full dg subcategory, we have an induced morphism
$$ S_K : \, \mathsf{cone}(\mathcal{U}_t(\mathcal{A}
\stackrel{K}{\hookrightarrow} \mathcal{B})) \rightarrow
\mathcal{U}_t(\mathcal{B}/\mathcal{A})\,,$$
where $\mathcal{B}/\mathcal{A}$ denotes Drinfeld's dg quotient. By
applying the localization techniques of section~\ref{localisation} and using the fact
that the derivator
$\mathsf{St}(\mathsf{L}_{\Sigma,P}\mathsf{Hot}_{\mathsf{dgcat}_f})$
admits a stable Quillen model, we inverse the morphisms $S_K$ and
finally obtain the universal derived invariant
of dg categories
$$ \mathcal{U}_l:\, \mathsf{HO}(\mathsf{dgcat}) \longrightarrow
\mathcal{M}_{dg}^{loc}\,.$$
We establish a connection between the triangulated category
$\mathcal{M}_{dg}^{loc}(e)$ and Waldhausen's $K$-theory by showing that
Waldhausen's $S_{\bullet}$-construction corresponds to the suspension
functor in $\mathcal{M}_{dg}^{loc}(e)$.
In section~\ref{chapspectra}, we prove that the derivator $\mathcal{M}_{dg}^{loc}$
admits a stable Quillen model given by a left Bousfield localization
of a category of presheaves of spectra.
In section~\ref{trimat}, we introduce the concept of upper triangular
dg category and construct a Quillen model structure on this class of dg
categories, which satisfies strong compactness conditions. In
section~\ref{splitex}, we establish the connection between upper
triangular dg categories and split short exact sequences and use the
Quillen model structure of section~\ref{trimat} to prove an
`approximation result', see proposition~\ref{aproxsplit}. In
section~\ref{quasi}, by applying the techniques of
section~\ref{localisation}, we construct the universal morphism of
derivators
$$\mathcal{U}_u: \mathsf{HO}(\mathsf{dgcat}) \longrightarrow \mathcal{M}_{dg}^{unst}$$
which commutes with filtered homotopy colimits, preserves the point and
sends each split short exact sequence to a homotopy cofiber
sequence. We prove that Waldhausen's $K$-theory space construction
appears as a fibrant object, which
allows us to obtain the weak equivalence
$$ \mathsf{Map}(\mathcal{U}_u(k),S^1\wedge \mathcal{U}_u(\mathcal{A})[1])
\stackrel{\sim}{\rightarrow} |N.wS_{\bullet}\mathcal{A}_f|\,.$$
in $\mathcal{M}^{unst}_{dg}$, see proposition~\ref{fibrant}.
This implies the isomorphisms
$$\pi_{i+1}
\mathsf{Map}(\mathcal{U}_u(k),S^1 \wedge \mathcal{U}_u(\mathcal{A})[1])
\stackrel{\sim}{\rightarrow} K_i(\mathcal{A}), \, \forall i \geq
0\,.$$

In section~\ref{univadit} we stabilize the derivator $\mathcal{M}_{dg}^{unst}$, using the fact that it
admits a Quillen model, and finally obtain the universal additive
invariant of dg categories
$$ \mathcal{U}_a : \mathsf{HO}(\mathsf{dgcat}) \longrightarrow
\mathcal{M}^{add}_{dg} \,.$$
Connective algebraic $K$-theory is additive and so factors through
$\mathcal{U}_a$. We prove that it corresponds to a fibrant resolution
of $\mathcal{U}_a(\mathcal{A})[1]$, see theorem~\ref{fibres}. Using the fact that the Quillen model for $\mathcal{M}_{dg}^{add}$ is
enriched over spectra, we prove the main co-representability theorem.
\begin{theo*}[\ref{corep}]
We have the following weak equivalence of spectra
$$ \mathsf{Hom}^{\mathsf{Sp}^{\mathbb{N}}}(\mathcal{U}_a(\mathcal{A}),\mathcal{U}_a(\mathcal{B})[1])
\stackrel{\sim}{\rightarrow}
K^c(\mathsf{rep}_{mor}(\mathcal{A},\mathcal{B}))\,,$$
where $K^c(\mathsf{rep}_{mor}(\mathcal{A},\mathcal{B}))$ denotes Waldhausen's connective $K$-theory spectrum of
$\mathsf{rep}_{mor}(\mathcal{A},\mathcal{B})$.

In particular, we
have the following weak equivalence of simplicial sets
$$\mathsf{Map}(\mathcal{U}_a(\mathcal{A}),\mathcal{U}_a(\mathcal{B})[1])
\stackrel{\sim}{\rightarrow}
|N.wS_{\bullet}\mathsf{rep}_{mor}(\mathcal{A},\mathcal{B})|$$
and so the isomorphisms
$$\pi_{i+1}
\mathsf{Map}(\mathcal{U}_a(A),\mathcal{U}_a(\mathcal{B})[1])
\stackrel{\sim}{\rightarrow} K_i(\mathsf{rep}_{mor}(\mathcal{A},\mathcal{B})), \,\,\, \forall i \geq 0\,.$$
\end{theo*}

Remark that if in the above theorem, we consider $\mathcal{A}=k$, we
have
$$ \mathsf{Hom}^{\mathsf{Sp}^{\mathbb{N}}}(\mathcal{U}_a(k),\mathcal{U}_a(\mathcal{B})[1])
\stackrel{\sim}{\rightarrow}
K^c(\mathcal{B})\,.$$
This shows that Waldhausen's connective $K$-theory spectrum becomes
co-representable in $\mathcal{M}_{dg}^{add}$.

In section~\ref{vistas}, we point out some questions that deserve
further investigation.

\section{Preliminaries}\label{pre}
Our reference for the language of derivators is \cite{Cis-Nee}.
The derivators in this article are always considered over the category
$cat$ of small categories.
Let $\mathbb{D}$ and $\mathbb{D}'$ be derivators.
We denote by $\underline{\mathsf{Hom}}(\mathbb{D},\mathbb{D}')$ the
category of all morphisms of derivators, by
$\underline{\mathsf{Hom}}_{!}(\mathbb{D},\mathbb{D}')$ the
category of morphisms of derivators which commute with homotopy
colimits and by
$\underline{\mathsf{Hom}}_{flt}(\mathbb{D},\mathbb{D}')$ the
category of morphisms of derivators which commute with filtered homotopy
colimits.
We denote the terminal category by $e$ and by $i:\ul \rightarrow \square$ the inclusion of
categories
$$
\xymatrix{
(0,0)  & \ar[l] (0,1) & & (0,0) & (0,1) \ar[l] \\
(1,0) \ar[u] & & & (1,0) \ar[u] & (1,1) \ar[l] \ar[u] \,.
}
$$
Associated with a small category $L$ and a point $x:e\rightarrow L$ in $L$ we have the
following adjunction:
$$
\xymatrix{
\mathbb{D}(L) \ar@<1ex>[d]^{x^*} \\
\mathbb{D}(e) \ar@<1ex>[u]^{x_{!}}\,.
}
$$
\begin{definition}
A derivator $\mathbb{D}$ is {\em strong} if for every
finite free category $A$ and every small category $B$, the natural
functor
$$ \mathbb{D}(A \times B) \longrightarrow
\mathsf{Fun}(A^o,\mathbb{D}(B))$$
is full and essentially surjective.
\end{definition}
Notice that a strong derivator is the same thing as a small homotopy
theory as defined in \cite{Heller}.

\begin{definition}
A derivator $\mathbb{D}$ is {\em regular} if in $\mathbb{D}$, sequential homotopy
colimits commute with finite products and homotopy pullbacks.
\end{definition}

Let $\mathcal{M}$ be a Quillen model category. By \cite{Cisinski1}, it
is known that $\mathcal{M}$ gives rise to a derivator which we denote by
$\mathsf{HO}(\mathcal{M})$. Observe that proposition $2.15$ in
\cite{catder} implies that $\mathsf{HO}(\mathcal{M})$ is a strong derivator.

We denote by $\mathsf{Ho}(\mathcal{M})$
the homotopy category of $\mathcal{M}$. By definition, it equals $\mathsf{HO}(\mathcal{M})(e)$.

Finally, for a small category $A$, we denote by $\underline{A}$ the
pre-derivator naturally associated to $A$.

\section{Derived Kan extensions}\label{extension}
Let $A$ be a small category and $\mathsf{Fun}(A^o,Sset)$ the
Quillen model category of simplicial pre-sheaves on $A$, endowed with
the projective model structure, see \cite{HAG}.
We have at our disposal the functor
$$
\begin{array}{ccc}
A & \stackrel{h}{\longrightarrow} & \mathsf{Fun}(A^o,Sset) \\
X & \longmapsto & \mathsf{Hom}_A(?,X)\,,
\end{array}
$$
where $\mathsf{Hom}_A(?,X)$ is considered as a constant simplicial set.

The functor $h$ gives rise to a morphism of pre-derivators
$$ \underline{A} \stackrel{h}{\longrightarrow}
\mathsf{HO}(\mathsf{Fun}(A^o,Sset))\,.$$

Using the notation of \cite{Cisinski}, we denote by $\mathsf{Hot}_A$
the derivator $\mathsf{HO}(\mathsf{Fun}(A^o,Sset))$. The following
results are proven in \cite{Cisinski}.

Let $\mathbb{D}$ be a derivator.

\begin{theorem}\label{Cin}
The morphism $h$ induces an
equivalence of categories
$$
\xymatrix{
\underline{\mathsf{Hom}}(\underline{A},\mathbb{D})
\ar@<-1ex>[d]_{\varphi} \\
\underline{\mathsf{Hom}}_!(\mathsf{Hot}_A, \mathbb{D})
\ar@<-1ex>[u]_{h^{\ast}}\,.
}
$$
\end{theorem}

\begin{proof}
This theorem is equivalent to corollary $3.26$ in \cite{Cisinski},
since we have
$$ \underline{\mathsf{Hom}}(\underline{A},\mathbb{D}) \simeq \mathbb{D}(A^o)\,.$$
\end{proof}

\begin{lemma}\label{ad1}
We have an adjunction
$$
\xymatrix{
\underline{\mathsf{Hom}}(\mathsf{Hot}_A,\mathbb{D})
\ar@<2ex>[d]^{\Psi} \\
*+<1pc>{\underline{\mathsf{Hom}}_!(\mathsf{Hot}_A, \mathbb{D})}
\ar@{^{(}->}[u]^{inc} \,,
}
$$
where
$$\Psi(F):=\varphi(F \circ h)\,.$$
\end{lemma}

\begin{proof}
We construct a universal $2$-morphism of functors
$$\epsilon : \, inc \circ \Psi \longrightarrow Id\,.$$
Let $F$ be a morphism of derivators belonging to
$\underline{\mathsf{Hom}}(\mathsf{Hot}_A,\mathbb{D})$. Let $L$ be a
small category and $X$ an object of $\mathsf{Hot}_A(L)$. Recall from
\cite{Cisinski} that we have the diagram
$$
\xymatrix{
 & \nabla \int X \ar[dl]_{\pi} \ar[dr]^{\varpi} & \\
L^{op} & & A\,.
}
$$
Now, let $p$ be the functor $\pi^{op}$ and $q$ the functor $\varpi^{op}$. By
the dual of proposition $1.15$ in \cite{Cisinski}, we have the
following functorial isomorphism
$$ p_! q^{\ast}(h) \stackrel{\sim}{\longrightarrow} X\,.$$
Now let $\epsilon_L(X)$ be the composed morphism.
$$ \epsilon_L(X) : \, \Psi(F)(X) = p_!q^{\ast}F(h) = p_!F(q^{\ast}h)
\rightarrow F(p_!q^{\ast}h) \stackrel{\sim}{\rightarrow}F(X)\,.$$
Now, notice using theorem~\ref{Cin} that $\epsilon$ induces an adjunction.
\end{proof}

\section{Localization: model categories versus derivators}\label{localisation}
Let $\mathcal{M}$ be a left proper, cellular Quillen model
category, see \cite{Hirschhorn}.

We fix a frame on $\mathcal{M}$, see
definition $16.6.21$ in \cite{Hirschhorn}. Let $D$ be a
small category and $F$ a functor from $D$ to
$\mathcal{M}$. We denote by $\mbox{hocolim} \,F$ the object of
$\mathcal{M}$, as in definition $19.1.2$ of \cite{Hirschhorn}.
Let $S$ be a set of morphisms in $\mathcal{M}$ and denote by $\mathsf{L}_S\mathcal{M}$ the
  left Bousfield localization of $\mathcal{M}$ by $S$.

Notice that the Quillen adjunction
$$
\xymatrix{
\mathcal{M} \ar@<-1ex>[d]_{Id} \\
\mathsf{L}_S \mathcal{M} \ar@<-1ex>[u]_{Id} \,,
}
$$
induces a morphism of derivators
$$ \gamma: \mathsf{HO}(\mathcal{M})
\stackrel{\mathbb{L}Id}{\longrightarrow}
\mathsf{HO}(\mathsf{L}_S \mathcal{M})$$
which commutes with homotopy colimits.

\begin{proposition}\label{Cisin}
Let $\mathcal{W}_S$ be the smallest class of morphisms in
$\mathcal{M}$ satisfying the following properties:
\begin{itemize}
\item[a)] Every element in $S$ belongs to $\mathcal{W}_S$.
\item[b)] Every weak equivalence of $\mathcal{M}$ belongs to
  $\mathcal{W}_S$.
\item[c)] If in a commutative triangle, two out of three morphisms
    belong to $\mathcal{W}_S$, then so does the third one. The class
    $\mathcal{W}_S$ is stable under retractions.
\item[d)] Let $D$ be a small category and $F$ and $G$ functors from
  $D$ to $\mathcal{M}$. If $\eta$ is a morphism of functors from $F$
  to $G$ such that for every object $d$ in $D$, $F(d)$ and $G(d)$ are cofibrant
  objects and the morphism $\eta(d)$
  belongs to $\mathcal{W}_S$, then so does the morphism
$$ \mbox{hocolim}\, F \longrightarrow \mbox{hocolim} \,G\,.$$
\end{itemize}
Then the class $\mathcal{W}_S$ equals the class of $S$-local
equivalences in $\mathcal{M}$, see \cite{Hirschhorn}.
\end{proposition}

\begin{proof}
The class of $S$-local equivalences satisfies properties $a)$,
$b)$, $c)$ and $d)$: Properties $a)$ and $b)$ are satisfied by
definition, proposition $3.2.3$ and $3.2.4$ in \cite{Hirschhorn} imply property
$c)$ and proposition $3.2.5$ in \cite{Hirschhorn} implies property $d)$.

Let us now show that conversely, each $S$-local equivalence is in $\mathcal{W}_S$.
Let
$$X \stackrel{g}{\rightarrow} Y$$
 be an $S$-local equivalence in $\mathcal{M}$. Without loss of generality, we can suppose that $X$ is
 cofibrant. Indeed, let $Q(X)$ be a cofibrant resolution of $X$ and consider
 the diagram
$$
\xymatrix{
Q(X) \ar[d]_{\pi}^{\sim} \ar[dr]^{g\circ \pi} & \\
X \ar[r]_g & Y\,.
}
$$
Notice that since $\pi$ is a weak equivalence, $g$ is an $S$-local
equivalence if and only if $g \circ \pi$ is one.

By theorem $4.3.6$ in \cite{Hirschhorn}, $g$ is an $S$-local equivalence
if and only if the morphism $\mathsf{L}_S(g)$ appering in the diagram
$$
\xymatrix{
X \ar[rr]^g \ar[d]_{j(X)} & &  Y \ar[d]^{j(Y)} \\
\mathsf{L}_SX \ar[rr]_{\mathsf{L}_S(g)} & &  \mathsf{L}_SY
}
$$
is a weak equivalence in $\mathcal{M}$.
This shows that it is enough to prove that $j(X)$ and $j(Y)$ belong to
$\mathcal{W}_S$. Apply the small object argument to the morphism
$$ X \longrightarrow \ast $$
using the set $\widetilde{\Lambda(S)}$, see proposition $4.2.5$ in
\cite{Hirschhorn}. We have the factorization
$$
\xymatrix{
X \ar[rr] \ar[dr]_{j(X)} & &  \ast \\
 & \mathsf{L}_S(X) \ar[ur] & ,
}
$$
where $j(X)$ is a relative $\widetilde{\Lambda(S)}$-cell complex.

We will now prove two stability conditions concerning the class
$\mathcal{W}_S$.
\begin{itemize}
\item[S1)] Consider the following push-out
$$
\xymatrix{
*+<1pc>{W_0} \ar[r] \ar@{>->}[d]_f  \ar@{}[dr]|{\lrcorner} & W_2 \ar[d]^{f_{\ast}} \\
W_1 \ar[r] & W_3 \,,
}
$$
where $W_0$, $W_1$ and $W_2$ are cofibrant objects in $\mathcal{M}$
and $f$ is a cofibration which belongs to $\mathcal{W}_S$. Observe that
$f_{\ast}$ corresponds to the colimit of the morphism of diagrams
$$
\xymatrix{
*+<1pc>{W_0} \ar@{=}[r] \ar@{>->}[d]_f & W_0 \ar@{=}[d] \ar[r] & W_2 \ar@{=}[d]
\\
W_1 & *+<1pc>{W_0} \ar@{>->}[l]^f \ar[r] & W_2 \,.
}
$$
Now, proposition $19.9.4$ in \cite{Hirschhorn} and property $d)$
imply that $f_{\ast}$ belongs to $\mathcal{W}_S$.

\item[S2)] Consider the following diagram
$$
\xymatrix{
\mathbf{X} :X_0 \,\,\ar@{>->}[r]^{f_0} & *+<1pc>{X_1} \ar@{>->}[r]^{f_1} &  *+<1pc>{X_2} \ar@{>->}[r]^{f_2} &  *+<1pc>{X_3} \ar@{>->}[r] & \ldots
}
$$
in $\mathcal{M}$, where the objects are cofibrant and the morphisms
are cofibrations which belong to the class $\mathcal{W}_S$. Observe that
the transfinite composition of $X$ corresponds to the colimit of the
morphism of diagrams
$$
\xymatrix{
X_0 \ar@{=}[r] \ar@{=}[d] & *+<1pc>{X_0} \ar@{=}[r] \ar@{>->}[d]^{f_0} & *+<1pc>{X_0}
\ar@{=}[r] \ar@{>->}[d]^{f_1 \circ f_0} & \ldots \\
*+<1pc>{X_0} \ar@{>->}[r]_{f_0} & *+<1pc>{X_1} \ar@{>->}[r]_{f_1} & *+<1pc>{X_2} \ar@{>->}[r] &
\ldots \,\,.
}
$$
Now, since $\mathbf{X}$ is a Reedy cofibrant diagram on a category with
fibrant constants, see definition $15.10.1$ in \cite{Hirschhorn},
theorem $19.9.1$ from \cite{Hirschhorn} and property $d)$ imply that the transfinite
composition of $\mathbf{X}$ belongs to $\mathcal{W}_S$. Notice that
the above argument can be immediatly generalized to a transfinite
composition of a $\lambda$-sequence, where $\lambda$ denotes an
ordinal, see section $10.2$ in \cite{Hirschhorn}.
\end{itemize}

Now, the construction of the morphism $j(X)$ and the stability
conditions $S1)$ and $S2)$ show us that it is enough to prove that the elements of
$\widetilde{\Lambda(S)}$ belong to $\mathcal{W}_S$.
By proposition $4.2.5$ in \cite{Hirschhorn}, it is sufficient to show
that the set
$$ \Lambda(S)= \{ \tilde{{\bf A}}\otimes \Delta[n]
\underset{\tilde{{\bf A}}\otimes
  \partial \Delta[n]}{\amalg} \tilde{{\bf B}}\otimes \Delta[n]
\stackrel{\Lambda(g)}{\longrightarrow} \tilde{{\bf B}}\otimes \Delta[n] \,|\,
(A \stackrel{g}{\rightarrow} B) \in S, \, n \geq 0\} \,,$$
of horns in $S$ is contained in $\mathcal{W}_S$. Recall from definition $4.2.1$ in
\cite{Hirschhorn} that $\tilde{g}: \tilde{{\bf A}} \rightarrow
\tilde{{\bf B}}$
denotes a cosimplicial resolution of $g: A \rightarrow B$ and $\tilde{g}$
is a Reedy cofibration.
We have the diagram
$$
\xymatrix{
\tilde{{\bf A}} \otimes \partial \Delta[n] \ar[d]_{1 \otimes i}
\ar[rr]^{\tilde{g} \otimes 1}  & &   \tilde{{\bf B}} \otimes \partial
\Delta[n] \ar[d]^{1 \otimes i} \\
\tilde{{\bf A}} \otimes \Delta[n] \ar[rr]_{\tilde{g} \otimes 1}  & &
\tilde{{\bf B}} \otimes \Delta[n] \,.
}
$$
Observe that the morphism
$$ \tilde{{\bf A}} \otimes \Delta[n] \stackrel{\tilde{g} \otimes 1}{\longrightarrow}
\tilde{{\bf B}} \otimes \Delta[n]$$
identifies with
$$ \tilde{{\bf A}}^n \stackrel{\tilde{g}^n}{\longrightarrow}
\tilde{{\bf B}}^n\,,$$
and so belongs to $\mathcal{W}_S$. Now, the morphism
$$ \tilde{{\bf A}}\otimes \partial \Delta[n]
\stackrel{\tilde{g}\otimes  1_{\partial\Delta[n]}}{\longrightarrow}
\tilde{{\bf B}}\otimes \partial \Delta[n]$$
corresponds to the induced map of latching objects
$$\mathsf{L}_n \tilde{{\bf A}} \longrightarrow  \mathsf{L}_n \tilde{{\bf B}}\,,$$
which is a cofibration in $\mathcal{M}$ by proposition $15.3.11$ in
\cite{Hirschhorn}.

Now by propositions $15.10.4$, $16.3.12$ and theorem $19.9.1$ of
\cite{Hirschhorn}, we have the following commutative diagram
$$
\xymatrix{
\underset{\partial(\stackrel{\rightarrow}{\Delta}\downarrow
  [n])}{\mbox{hocolim}} \, \tilde{{\bf A}} \ar[r] \ar[d]_{\sim} & \underset{\partial(\stackrel{\rightarrow}{\Delta}\downarrow
  [n])}{\mbox{hocolim}} \, \tilde{{\bf B}} \ar[d]^{\sim} \\
\mathsf{L}_n \tilde{{\bf A}} \ar[r] & \mathsf{L}_n \tilde{{\bf B}}\,,
}
$$
where the vertival arrows are weak equivalences and $\partial(\stackrel{\rightarrow}{\Delta}\downarrow
  [n])$ denotes the category of strictly increasing maps with target $[n]$. By property $d)$ of
the class $\mathcal{W}_S$, we conclude that $\tilde{g} \otimes 1_{\partial\Delta[n]}$
belongs to $\mathcal{W}_S$.

We have the following diagram
$$
\xymatrix{
\tilde{{\bf A}}\otimes \Delta[n] \underset{\tilde{{\bf A}}\otimes
  \partial \Delta[n]}{\amalg} \tilde{{\bf B}}\otimes \Delta[n]
\ar[drr]^{\Lambda(g)} & & \\
\tilde{{\bf A}} \otimes \Delta[n] \ar[u]^I \ar[rr]_{\tilde{g}\otimes
  1} & & \tilde{{\bf B}} \otimes \Delta[n]\,.
}
$$
Notice that the morphism $I$ belongs to $\mathcal{W}_S$ by the
stability condition $S1)$ applied to the morphism
$$ \tilde{{\bf A}} \otimes \partial \Delta[n]  \stackrel{\tilde{g}
  \otimes 1}{\longrightarrow} \tilde{{\bf B}} \otimes \partial \Delta[n]\,,$$
which is a cofibration and belongs to $\mathcal{W}_S$. Since the
morphism $\tilde{g}\otimes 1$ belongs to $\mathcal{W}_S$ so does $\Lambda(g)$.

This proves the proposition.
\end{proof}
Let $\mathbb{D}$ be derivator and $S$ a class of
morphisms in $\mathbb{D}(e)$.

\begin{definition}[Cisinski \cite{Letter}]\label{defCis}
The derivator $\mathbb{D}$ admits a {\it left Bousfield localization}
by $S$ if there exists a morphism of derivators
$$ \gamma : \mathbb{D} \rightarrow \mathsf{L}_S\mathbb{D}\,,$$
which commutes with homotopy colimits, sends the elements of $S$ to
isomorphisms in $\mathsf{L}_S\mathbb{D}(e)$ and satisfies the
universal property: for every derivator $\mathbb{D}'$ the morphism
$\gamma$ induces an equivalence of categories
$$\underline{\mathsf{Hom}}_!(\mathsf{L}_S\mathbb{D},\mathbb{D}')
  \stackrel{\gamma^{\ast}}{\longrightarrow}
  \underline{\mathsf{Hom}}_{!,S}(\mathbb{D},\mathbb{D}')\,,$$
where $\underline{\mathsf{Hom}}_{!,S}(\mathbb{D},\mathbb{D}')$ denotes
the category of morphisms of derivators which commute with homotopy
colimits and send the elements of $S$ to isomorphisms in $\mathbb{D}'(e)$.
\end{definition}

\begin{lemma}\label{lettri}
Suppose that $\mathbb{D}$ is a triangulated derivator, $S$ is stable
under the loop space functor $\Omega(e) : \mathbb{D}(e) \rightarrow
\mathbb{D}(e)$, see \cite{Cis-Nee}, and $\mathbb{D}$ admits a left
Bousfield localization $\mathsf{L}_S\mathbb{D}$ by $S$.

Then $\mathsf{L}_S\mathbb{D}$ is also a triangulated derivator.
\end{lemma}

\begin{proof}
Recall from \cite{Cis-Nee} that since $\mathbb{D}$ is a triangulated
derivator, we have the following equivalence
$$
\xymatrix{
\mathbb{D} \ar@<1ex>[d]^{\Omega} \\
\mathbb{D} \ar@<1ex>[u]^{\Sigma}\,.
}
$$
Notice that both morphisms of derivators, $\Sigma$ and $\Omega$,
commute with homotopy colimits. Since $S$ is stable under the functor
$\Omega(e): \mathbb{D}(e) \rightarrow \mathbb{D}(e)$ and $\mathbb{D}$
admits a left Bousfield localization $\mathsf{L}_S\mathbb{D}$ by $S$,
we have an induced morphism
$$ \Omega: \mathsf{L}_S\mathbb{D} \rightarrow
\mathsf{L}_S\mathbb{D}\,.$$
Let $s$ be an element of $S$. We now show that the image of $s$ by the
functor $\gamma \circ \Sigma$ is an isomorphism in
$\mathsf{L}_S\mathbb{D}(e)$. For this consider the category $\ul$, see
section~\ref{pre}, and the functors
$$
\begin{array}{rcl}
(0,0): e \rightarrow \ul & \mbox{and} & p:\ul \rightarrow e\,.
\end{array}
$$
Now recall from section $7$ from \cite{Heller} that
$$ \Sigma(e) := p_! \circ (0,0)_{\ast}\,.$$
This description show us that the image of $s$ under the functor
$\gamma \circ \Sigma$ is an isomorphism in $\mathsf{L}_S\mathbb{D}(e)$
because $\gamma$ commutes with homotopy colimits. In conclusion, we
have an induced adjunction
$$
\xymatrix{
\mathsf{L}_S\mathbb{D} \ar@<1ex>[d]^{\Omega}\\
\mathsf{L}_S\mathbb{D} \ar@<1ex>[u]^{\Sigma}
}
$$
which is clearly an equivalence. This proves the lemma.
\end{proof}

\begin{theorem}[Cisinski \cite{Letter}]\label{Cisinsk}
The morphism of derivators
$$\gamma: \mathsf{HO}(\mathcal{M})
\stackrel{\mathbb{L}Id}{\longrightarrow} \mathsf{HO}(\mathsf{L}_S
\mathcal{M})$$
is a left Bousfield localization of $\mathsf{HO}(\mathcal{M})$ by the
image of the set $S$ in $\mathsf{Ho}(\mathcal{M})$.
\end{theorem}

\begin{proof}
Let $\mathbb{D}$ be a derivator.

The morphism $\gamma $ admits a fully faithful right adjoint
$$\sigma : \, \mathsf{HO}(\mathsf{L}_S \mathcal{M})
\longrightarrow \mathsf{HO}(\mathcal{M})\,.$$
Therefore, the induced functor
$$ \gamma^{\ast} : \underline{\mathsf{Hom}}_!(\mathsf{HO}(\mathsf{L}_S
\mathcal{M}),
\mathbb{D}) \longrightarrow
\underline{\mathsf{Hom}}_{!,S}(\mathsf{HO}(\mathcal{M}),
\mathbb{D})\,,$$
admits a left adjoint $\sigma^{\ast}$ and $\sigma^{\ast}\gamma^{\ast}=
(\gamma \sigma)^{\ast}$ is isomorphic to the identity. Therefore
$\gamma^{\ast}$ is fully faithful. We now show that
$\gamma^*$ is essentially surjective. Let $F$ be an object of
$\underline{\mathsf{Hom}}_{!,S}(\mathsf{HO}(\mathcal{M}),
\mathbb{D})$. Notice that since $\mathbb{D}$ satisfies the conservativity
axiom, it is sufficient to show that the functor
$$ F(e): \mathsf{Ho}(\mathcal{M}) \rightarrow \mathbb{D}(e)$$
sends the images in $\mathsf{Ho}(\mathcal{M})$ of the $S$-local
equivalences of $\mathcal{M}$ to isomorphisms in $\mathbb{D}(e)$.
The morphism $F$ then becomes naturally a morphism of derivators
$$ \overline{F}: \mathsf{HO}(\mathsf{L}_S\mathcal{M}) \rightarrow
\mathbb{D}$$
such that $\gamma^*(\overline{F})=F$. Now, since $F$ commutes with
homotopy colimits, the functor
$$ \mathcal{M} \rightarrow \mathsf{Ho}(\mathcal{M}) \rightarrow
\mathbb{D}(e)$$
sends the elements of $\mathcal{W}_S$ to isomorphisms. This proves the
theorem since by proposition~\ref{Cisin} the class $\mathcal{W}_S$ equals the class of $S$-local equivalences in $\mathcal{M}$.
\end{proof}

\section{Filtered homotopy colimits}\label{homotopy}
Let $\mathcal{M}$ be a cellular Quillen model category, with $I$ the
set of generating cofibrations. Suppose that the domains and codomains
of the elements of $I$ are cofibrant, $\aleph_0$-compact,
$\aleph_0$-small and homotopically finitely presented, see definition
$2.1.1$ in \cite{Toen-Vaq}.

\begin{example}\label{mori}
Consider the Quillen model structures on $\mathsf{dgcat}$ constructed
in theorems~\ref{mal}, \ref{theorem} and \ref{theorem2}.
Recall that a dg functor $F$ from $\mathcal{C}$ to $\mathcal{E}$ is a
quasi-equivalence, resp. quasi-equiconic,
resp. a Morita dg functor,  if it satisfies one of the following conditions $\mathsf{C}1)$
or $\mathsf{C}2)$:
\begin{itemize}
\item[C1)] The dg category $\mathcal{C}$ is empty and all the objects
  of $\mathcal{E}$ are contractible.
\item[C2)] For every object $c_1$ and $c_2$ in $\mathcal{C}$, the
  morphism of complexes from $\mathsf{Hom}_{\mathcal{C}}(c_1,c_2)$ to
  $\mathsf{Hom}_{\mathcal{E}}(F(c_1)),F(c_2))$ is a quasi-isomorphism
  and the functor $\mathsf{H}^0(F)$,
  resp. $\mathsf{H}^0(\mbox{pre-tr}(F))$, resp. $\mathsf{H}^0(\mbox{pre-tr}(F))^{\kar}$, is essentially surjective.
\end{itemize}
Observe that the domains and codomains of the set $I$ of generating
cofibrations in $\mathsf{dgcat}$ satisfy the conditions above for all
the Quillen model structures.
\end{example}

The following proposition is a simplification of proposition $2.2$ in
\cite{Toen-Vaq}.

\begin{proposition}\label{prop}
Let $\mathcal{M}$ be a Quillen model category which satisfies the
conditions above. Then
\begin{itemize}
\item[1)] A filtered colimit of trivial fibrations is a trivial
  fibration.
\item[2)] For any filtered diagram $X_i$ in $\mathcal{M}$, the natural
  morphism
$$ \underset{i \in I}{\mathsf{hocolim}}\, X_i \longrightarrow
  \underset{i \in I}{\mathsf{colim}}\, X_i$$
is an isomorphism in $\mathsf{Ho}(\mathcal{M})$.
\item[3)] Any object $X$ in $\mathcal{M}$ is equivalent to a filtered
  colimit of strict finite $I$-cell objects.
\item[4)] An object $X$ in $\mathcal{M}$ is homotopically finitely
  presented if and only if it is equivalent to a rectract of a strict
  finite $I$-cell object.
\end{itemize}
\end{proposition}

\begin{proof}
The proof of $1)$, $2)$ and $3)$ is exactly the same as that of proposition
$2.2$ in \cite{Toen-Vaq}. The proof of $4)$ is also the same once we
observe that the domains and codomains of the elements of the set $I$
are already homotopically finitely presented by hypothesis.
\end{proof}

In everything that follows, we fix:
\begin{itemize}
\item[-] A co-simplicial resolution functor
$$(\Gamma (-) : \mathcal{M} \rightarrow \mathcal{M}^{\Delta}, \, i)$$
in the model category
  $\mathcal{M}$, see definition $16.1.8$ in \cite{Hirschhorn}. This
  means that for every object $X$ in $\mathcal{M}$, $\Gamma(X)$ is
  cofibrant in the Reedy model structure on $\mathcal{M}^{\Delta}$ and
$$ i(X) : \Gamma(X) \stackrel{\sim}{\longrightarrow} c^{\ast}(X)$$
is a weak equivalence on $\mathcal{M}^{\Delta}$, where $c^{\ast}(X)$
denotes the constant co-simplicial object associated with $X$.
\item[-] A fibrant resolution functor
$$( (-)_f : \mathcal{M} \rightarrow \mathcal{M}, \, \epsilon )$$
in the model category $\mathcal{M}$, see \cite{Hirschhorn}.
\end{itemize}

\begin{definition}
Let $\mathcal{M}_f$ be the smallest full subcategory of $\mathcal{M}$
such that
\begin{itemize}
\item[-]  $\mathcal{M}_f$ contains (a representative of the
  isomorphism class of) each strictly finite $I$-cell object of
  $\mathcal{M}$ and
\item[-] the category $\mathcal{M}_f$ is stable under the functors $(-)_f$ and $\Gamma(-)^n,\, n \geq 0$.
\end{itemize}
\end{definition}

\begin{remark}
Notice that $\mathcal{M}_f$ is a small category and that every object
in $\mathcal{M}_f$ is weakly equivalent to a strict finite $I$-cell.
\end{remark}

We have the inclusion
$$ \mathcal{M}_f \stackrel{i}{\hookrightarrow} \mathcal{M}\,. $$

\begin{Notation}
Let $S$ be the set of pre-images of the weak equivalences in
$\mathcal{M}$ under the functor $i$.
\end{Notation}

\begin{lemma}\label{fulfat}
The induced functor
$$ \mathcal{M}_f[S^{-1}] \stackrel{\mathsf{Ho}(i)}{\longrightarrow} \mathsf{Ho}(\mathcal{M})$$
is fully faithful, where $\mathcal{M}_f[S^{-1}]$ denotes the
localization of $\mathcal{M}$ by the set $S$.
\end{lemma}

\begin{proof}
Let $X$, $Y$ be objects of $\mathcal{M}_f$. Notice that $(Y)_f$ is a fibrant
resolution of $Y$ in $\mathcal{M}$ which belongs to $\mathcal{M}_f$ and
$$
\xymatrix{
\Gamma(X)^0 \coprod \Gamma(X)^0 \ar[d]_{d^0 \coprod
  d^1} \ar[r] &  \Gamma(X)^0\\
\Gamma(X)^1 \ar[ur]_{s^0}
}
$$
is a cylinder object for $\Gamma(X)^0$, see proposition $16.1.6.$ from
\cite{Hirschhorn}. Since this cylinder object also belongs to
$\mathcal{M}_f$, the localized category $\mathcal{M}_f[S^{-1}]$ can be
constructed inside $\mathcal{M}$.

This implies the lemma.
\end{proof}

We denote by $\mathsf{Fun}(\mathcal{M}_f^o,Sset)$ the Quillen model
category of simplicial pre-sheaves on $\mathcal{M}_f$ endowed with the
projective model structure, see section~\ref{extension}.
Let $\Sigma$ be the image in $\mathsf{Fun}(\mathcal{M}_f^o,Sset)$ by the functor $h$, see
section~\ref{extension}, of the set $S$ in $\mathcal{M}_f$. Since the
category $\mathsf{Fun}(\mathcal{M}_f^o,Sset)$ is cellular and left
proper, its left Bousfield localization by the set $\Sigma$ exists,
see \cite{Hirschhorn}. We denote it by
$\mathsf{L}_{\Sigma}\mathsf{Fun}(\mathcal{M}_f^o,Sset)$. We have a
composed functor that we still denote by $h$
$$h: \mathcal{M}_f \rightarrow \mathsf{Fun}(\mathcal{M}_f^o,Sset)
\stackrel{Id}{\rightarrow} \mathsf{L}_{\Sigma}\mathsf{Fun}(\mathcal{M}_f^o,Sset)\,.$$

Now, consider the functor
$$
\begin{array}{ccc}
\underline{h}: \mathcal{M} & \longrightarrow &
\mathsf{Fun}(\mathcal{M}_f^o,Sset)\\
X & \longmapsto & \mathsf{Hom}(\Gamma(-),X)_{|\mathcal{M}_f}\,.
\end{array}
$$
We also have a composed functor that we still denote by $\underline{h}$

$$\underline{h}: \mathcal{M} \rightarrow \mathsf{Fun}(\mathcal{M}_f^o,Sset)
\stackrel{Id}{\rightarrow} \mathsf{L}_{\Sigma}\mathsf{Fun}(\mathcal{M}_f^o,Sset)\,.$$

Now, observe that the natural equivalence
$$ i(-) : \Gamma(-) \longrightarrow c^{\ast}(-)\,,$$
induces, for every object $X$ in $\mathcal{M}_f$, a morphism $\Psi(X)$
in $\mathsf{L}_{\Sigma}\mathsf{Fun}(\mathcal{M}_f^o,Sset)$
$$\Psi(X): \, h(X) = \mathsf{Hom}(c^{\ast}(-),X) \longrightarrow
\mathsf{Hom}(\Gamma(-),X) =: (\underline{h} \circ i)(X)\,,$$
which is functorial in $X$.

\begin{lemma}
The functor $\underline{h}$ preserves weak equivalences between
fibrant objects.
\end{lemma}

\begin{proof}
Let $X$ be a fibrant object in $\mathcal{M}$. We have an equivalence
$$ \mathsf{Hom}(\Gamma(Y),X) \stackrel{\sim}{\longrightarrow}
\mathsf{Map}_{\mathcal{M}}(Y,X)\,,$$
see \cite{Hirschhorn}. This implies the lemma.
\end{proof}

\begin{remark}
The previous lemma implies that the functor $\underline{h}$ admits a right
derived functor
$$
\begin{array}{ccc}
\mathbb{R}\underline{h}: \mathsf{Ho}(\mathcal{M}) & \longrightarrow &
\mathsf{Ho}(\mathsf{L}_{\Sigma}\mathsf{Fun}(\mathcal{M}_f^o,Sset))\\
X & \longmapsto & \mathsf{Hom}(\Gamma(-),X_f)_{|\mathcal{M}_f}\,.
\end{array}
$$
\end{remark}

Since the functor
$$h: \mathcal{M}_f \rightarrow \mathsf{L}_{\Sigma}\mathsf{Fun}(\mathcal{M}_f^o,Sset)\,,$$
sends, by definition, the elements of $S$ to weak equivalences, we have an induced morphism
$$ \mathsf{Ho}(h): \mathcal{M}_f[S^{-1}] \rightarrow \mathsf{Ho}(\mathsf{L}_{\Sigma}\mathsf{Fun}(\mathcal{M}_f^o,Sset))\,.$$

\begin{remark}\label{funcpont}
Notice that lemma $4.2.2$ from \cite{HAG} implies that for every $X$ in
$\mathcal{M}_f$, the morphism $\Psi(X)$
$$ \Psi(X):\, \mathsf{Ho}(h)(X) \longrightarrow
(\mathbb{R}\underline{h} \circ \mathsf{Ho}(i))(X)$$
is an isomorphism in
$\mathsf{Ho}(\mathsf{L}_{\Sigma}\mathsf{Fun}(\mathcal{M}_f^o,Sset))$.
\end{remark}

This shows that the functors
$$ \mathsf{Ho}(h), \, \mathbb{R}\underline{h} \circ \mathsf{Ho}(I) : \,
\mathcal{M}_f[S^{-1}] \rightarrow \mathsf{Ho}(\mathsf{L}_{\Sigma}\mathsf{Fun}(\mathcal{M}_f^o,Sset)$$
are canonically isomorphic and so we have the following diagram
$$
\xymatrix{
\mathcal{M}_f[S^{-1}] \ar[rr]^{\mathsf{Ho}(i)}
\ar[d]_{\mathsf{Ho}(h)} & &  \mathsf{Ho}(\mathcal{M})
\ar[dll]^{\mathbb{R}\underline{h}}\,, \\
\mathsf{Ho}(\mathsf{L}_{\Sigma}\mathsf{Fun}(\mathcal{M}_f^o,Sset)) & &
}
$$
which is commutative up to isomorphism.

\begin{lemma}\label{filtered}
The functor $\mathbb{R}\underline{h}$ commutes with filtered homotopy colimits.
\end{lemma}

\begin{proof}
Let $\{Y_i\}_{i \in I}$ be a filtered diagram in $\mathcal{M}$. We
can suppose, without loss of generality, that $Y_i$ is fibrant in
$\mathcal{M}$. By proposition~\ref{prop}, the natural
morphism
$$ \underset{i \in I}{\mbox{hocolim}} \,Y_i \longrightarrow \underset{i \in
  I}{\mbox{colim}} \,Y_i$$
is an isomorphism in $\mathsf{Ho}(\mathcal{M})$ and $\underset{i \in I}{\mbox{colim}}
\,Y_i$ is also fibrant. Since the functor
$$ \mathsf{Ho}(\mathsf{Fun}(\mathcal{M}_f^o,Sset))
\stackrel{\mathbb{L}Id}{\longrightarrow}
\mathsf{Ho}(\mathsf{L}_{\Sigma}\mathsf{Fun}(\mathcal{M}_f^o,Sset))\,,$$
commutes with homotopy colimits and in
$\mathsf{Ho}(\mathsf{Fun}(\mathcal{M}_f^o,Sset))$ they are calculated objectwise, it is
sufficient to show that the morphism
$$ \underset{i \in I}{\mbox{hocolim}} \,\mathbb{R}\underline{h}(Y_i)(X)
\longrightarrow \mathbb{R}\underline{h}(\underset{i \in I}{\mbox{colim}}
\,Y_i)(X)$$
is an isomorphism in $\mathsf{Ho}(Sset)$, for every object $X$ in
$\mathcal{M}_f$. Now, since every object $X$ in $\mathcal{M}_f$ is homotopically finitely presented,
see proposition~\ref{prop}, we have the following equivalences:
$$
\begin{array}{rcl}
\mathbb{R}\underline{h}(\underset{i \in I}{\mbox{colim}} \,Y_i)(X) & =
&
\mathsf{Hom}(\Gamma (X), \underset{i \in I}{\mbox{colim}} \,Y_i)\\
 & \simeq  & \mathsf{Map}(\Gamma (X), \underset{i \in I}{\mbox{colim}} \,Y_i)\\
& \simeq & \underset{i \in I}{\mbox{colim}}\, \mathsf{Map}(X,Y_i)\\
& \simeq & \underset{i \in I}{\mbox{hocolim}} \,\mathbb{R}\underline{h}(Y_i)(X)
\end{array}
$$
This proves the lemma.
\end{proof}

We now denote by $\mathsf{L}_{\Sigma}\mathsf{Hot}_{\mathcal{M}_f}$ the derivator associated with
$\mathsf{L}_{\Sigma}\mathsf{Fun}(\mathcal{M}_f^o,Sset)$ and by
$\underline{\mathcal{M}_f}[S^{-1}]$ the pre-derivator
$\underline{\mathcal{M}_f}$ localized at the set $S$, see
\cite{Cisinski}.

Observe that the morphism of functors
$$ \Psi : h \longrightarrow \underline{h}\circ i$$
induces a $2$-morphism of derivators
$$ \overline{\Psi}: \, \mathsf{Ho}(h) \longrightarrow \mathbb{R}\underline{h}\circ
\mathsf{Ho}(i)\,.$$

\begin{lemma}\label{2-mor}
The $2$-morphism $\overline{\Psi}$ is an isomorphism.
\end{lemma}

\begin{proof}
For the terminal category $e$, the $2$-morphism $\overline{\Psi}$
coincides with the morphism of functors of
remark~\ref{funcpont}. Since this one is an isomorphism, so is
$\overline{\Psi}$ by conservativity. This proves the lemma.
\end{proof}

As before, we have the following diagram
$$
\xymatrix{
\underline{\mathcal{M}_f}[S^{-1}] \ar[rr]^{\mathsf{Ho}(i)}
\ar[d]_{\mathsf{Ho}(h)} & &  \mathsf{HO}(\mathcal{M})
\ar[dll]^{\mathbb{R}\underline{h}}\,, \\
\mathsf{L}_{\Sigma}\mathsf{Hot}_{\mathcal{M}_f} & &
}
$$
which is commutative up to isomorphism in the $2$-category of
pre-derivators. Notice that by lemma~\ref{filtered},
$\mathbb{R}\underline{h}$ commutes with filtered homotopy colimits.

Let $\mathbb{D}$ be a derivator.

\begin{lemma}\label{colim}
The morphism of pre-derivators
$$ \underline{\mathcal{M}_f}[S^{-1}]
\stackrel{\mathsf{Ho}(h)}{\longrightarrow}
\mathsf{L}_{\Sigma}\mathsf{Hot}_{\mathcal{M}_f} \,,$$
induces an equivalence of categories
$$ \underline{\mathsf{Hom}}_!(\mathsf{L}_{\Sigma}\mathsf{Hot}_{\mathcal{M}_f},
\mathbb{D}) \stackrel{\mathsf{Ho}(h)^{\ast}}{\longrightarrow}
\underline{\mathsf{Hom}}(\underline{\mathcal{M}_f}[S^{-1}], \mathbb{D})\,.$$
\end{lemma}

\begin{proof}
The category
$\underline{\mathsf{Hom}}_!(\mathsf{L}_{\Sigma}\mathsf{Hot}_{\mathcal{M}_f},\mathbb{D})$
is equivalent, by theorem~\ref{Cisinsk}, to the category $\underline{\mathsf{Hom}}_{!,\Sigma}(\mathsf{Hot}_{\mathcal{M}_f},
\mathbb{D})$. This last
category identifies, under the equivalence
$$ \underline{\mathsf{Hom}}_!(\mathsf{Hot}_{\mathcal{M}_f}, \mathbb{D})
\rightarrow
\underline{\mathsf{Hom}}(\underline{\mathcal{M}_f},\mathbb{D})$$
given by theorem~\ref{Cin}, with the full subcategory of
$\underline{\mathsf{Hom}}(\underline{\mathcal{M}_f},\mathbb{D})$
consisting of the morphisms of pre-derivators which send the elements
of $S$ to isomorphisms in $\mathbb{D}(e)$. Now observe that this last
category identifies with
$\underline{\mathsf{Hom}}(\underline{\mathcal{M}_f}[S^{-1}],
\mathbb{D})$, by definition of the localized pre-derivator $\underline{\mathcal{M}_f}[S^{-1}]$.
This proves the lemma.
\end{proof}

Recall from section $9.5$ in \cite{Dugger} that the
co-simplicial resolution functor $\Gamma(-)$ that we have fixed in the
beginning of this section allows us to construct a Quillen adjunction:
$$
\xymatrix{
\mathcal{M} \ar@<1ex>[d]^{\underline{h}=sing} \\
\mathsf{Fun}(\mathcal{M}_f^o,Sset) \ar@<1ex>[u]^{Re} \,.
}
$$
Since the functor $Re$ sends the elements of $\Sigma$ to weak
equivalences in $\mathcal{M}$, we have the following Quillen adjunction
$$
\xymatrix{
\mathcal{M} \ar@<1ex>[d]^{\underline{h}} \\
\mathsf{L}_{\Sigma}\mathsf{Fun}(\mathcal{M}_f^o,Sset)
\ar@<1ex>[u]^{Re}\,,
}
$$
and a natural weak equivalence
$$ \eta : Re \circ h \stackrel{\sim}{\longrightarrow} i\,,$$
see \cite{Dugger}.

This implies that we have the following diagram
$$
\xymatrix{
\mathcal{M}_f[S^{-1}] \ar[rr]^{\mathsf{Ho}(i)}
\ar[d]_{\mathsf{Ho}(h)} & &  \mathsf{Ho}(\mathcal{M})
\ar@<1ex>[dll]^{\mathbb{R}\underline{h}}\,, \\
\mathsf{Ho}(\mathsf{L}_{\Sigma}\mathsf{Fun}(\mathcal{M}_f^o,Sset))
\ar@<1ex>[urr]^{\mathbb{L}Re} & &
}
$$
which is commutative up to isomorphism.

We now claim that $\mathbb{L}Re \circ \mathbb{R}\underline{h}$
is naturally isomorphic to the identity. Indeed, by
proposition~\ref{prop}, each object of $\mathcal{M}$ is isomorphic in
$\mathsf{Ho}(\mathcal{M})$, to a filtered colimit of strict finite
$I$-cell objects. Since
$\mathbb{R}\underline{h}$ and $\mathbb{L}Re$ commute with filtered
homotopy colimits and $\mathbb{L}Re \circ
\mathsf{Ho}(h) \simeq \mbox{Id}$, we conclude that
$\mathbb{L}Re \circ \mathbb{R}\underline{h}$ is naturally
isomorphic to the identity. This implies that the morphism
$\mathbb{R}\underline{h}$ is fully faithful.

Now, observe that the natural weak
equivalence $\eta$ induces a $2$-isomorphism and so we obtain the
following diagram
$$
\xymatrix{
\underline{\mathcal{M}_f}[S^{-1}] \ar[rr]^{\mathsf{Ho}(i)}
\ar[d]_{\mathsf{Ho}(h)} & &  \mathsf{HO}(\mathcal{M})
\ar@<1ex>[dll]^{\mathbb{R}\underline{h}}\,, \\
\mathsf{L}_{\Sigma}\mathsf{Hot}_{\mathcal{M}_f}
\ar@<1ex>[urr]^{\mathbb{L}Re} & &
}
$$
which is commutative up to isomorphism in the $2$-category of pre-derivators. Notice that $\mathbb{L}Re \circ
\mathbb{R}\underline{h}$ is naturally isomorphic to the identity (by
conservativity) and
so the morphism of derivators $\mathbb{R}\underline{h}$ is fully faithful.

Let $\mathbb{D}$ be a derivator.

\begin{theorem}\label{flt}
The morphism of derivators
$$ \mathsf{HO}(\mathcal{M})
\stackrel{\mathbb{R}\underline{h}}{\longrightarrow}
\mathsf{L}_{\Sigma}\mathsf{Hot}_{\mathcal{M}_f}\,,$$
induces an equivalence of categories
$$ \underline{\mathsf{Hom}}_!(\mathsf{L}_{\Sigma}\mathsf{Hot}_{\mathcal{M}_f},
\mathbb{D}) \stackrel{\mathbb{R}\underline{h}^{\ast}}{\longrightarrow}
\underline{\mathsf{Hom}}_{flt}(\mathsf{HO}(\mathcal{M}), \mathbb{D})\,,$$
where $\underline{\mathsf{Hom}}_{flt}(\mathsf{HO}(\mathcal{M}),
\mathbb{D})$ denotes the category of morphisms of derivators which
commute with filtered homotopy colimits.
\end{theorem}

\begin{proof}
We have the following adjunction
$$
\xymatrix{
\underline{\mathsf{Hom}}(\mathsf{HO}(\mathcal{M}), \mathbb{D})
\ar@<1ex>[d]^{\mathbb{L}{Re}^{\ast}} \\
\underline{\mathsf{Hom}}(\mathsf{L}_{\Sigma}\mathsf{Hot}_{\mathcal{M}_f},
\mathbb{D}) \ar@<1ex>[u]^{\mathbb{R} \underline{h}^{\ast}} \,,
}
$$
with ${\mathbb{R}\underline{h}}^{\ast}$ a fully faithful functor.

Now notice that the adjunction of lemma~\ref{ad1} induces naturally an
adjunction
$$
\xymatrix{
\underline{\mathsf{Hom}}(\mathsf{L}_{\Sigma}\mathsf{Hot}_{\mathcal{M}_f},\mathbb{D})
\ar@<2ex>[d]^{\Psi} \\
\underline{\mathsf{Hom}}_!(\mathsf{L}_{\Sigma}\mathsf{Hot}_{\mathcal{M}_f}, \mathbb{D})
\ar@{^{(}->}[u] \,.
}
$$
This implies that the composed functor
$$ {\mathbb{R}\underline{h}}^{\ast}: \, \underline{\mathsf{Hom}}_!(\mathsf{L}_{\Sigma}\mathsf{Hot}_{\mathcal{M}_f},
\mathbb{D}) \longrightarrow
\underline{\mathsf{Hom}}_{flt}(\mathsf{HO}(\mathcal{M}), \mathbb{D})$$
is fully faithful.

We now show that this functor is essentially surjective.

Let $F$ be an object of
$\underline{\mathsf{Hom}}_{flt}(\mathsf{HO}(\mathcal{M}),\mathbb{D})$.
Consider the morphism
$$ \mathbb{L}{Re}^{\ast}(F) := F \circ \mathbb{L}Re \,.$$
Notice that this morphism does not necessarily commute with homotopy
colimits.
Now, by the above adjunction, we have a universal $2$-morphism
$$ \varphi : \, \Psi({\mathbb{L}Re}^{\ast}(F)) \longrightarrow
{\mathbb{L}Re}^{\ast}(F)\,.$$
Now, consider the $2$-morphism
$${\mathbb{R}\underline{h}}^{\ast}:\,
{\mathbb{R}\underline{h}}^{\ast}((\Psi \circ
{\mathbb{L}Re}^{\ast})(F)) \longrightarrow
({\mathbb{R}\underline{h}}^{\ast} \circ {\mathbb{L}Re}^{\ast})(F)
\simeq F\,.$$

We will now show that this $2$-morphism is a $2$-isomorphism. By
conservativity, it is
sufficient to show this for the case of the terminal
category $e$. For this, observe that ${\mathbb{R}\underline{h}}^{\ast}(\varphi)$ induces an isomorphism
$$ \Psi({\mathbb{L}Re}^{\ast}(F)) \circ \mathbb{R}\underline{h} \circ
\mathsf{Ho}(i) \longrightarrow F \circ \mathsf{Ho}(i)\,.$$
Now each object of $\mathcal{M}$ is isomorphic, in
$\mathsf{Ho}(\mathcal{M})$, to a filtered colimit of strict finite
$I$-cell objects. Since $F$ and
$\Psi({\mathbb{L}Re}^{\ast}(F))$ commute with filtered homotopy
colimits, ${\mathbb{R}\underline{h}}^{\ast}(\varphi)$ induces an
isomorphism.
This shows that the functor $\mathbb{R}\underline{h}^{\ast}$ is
essentially surjective.

This proves the theorem.
\end{proof}

\section{Pointed derivators}\label{chappoint}
Recall from the previous section that we have constructed a derivator
$\mathsf{L}_{\Sigma}\mathsf{Hot}_{\mathcal{M}_f}$ associated with a
Quillen model category $\mathcal{M}$ satisfying suitable compactness assumptions.

Now suppose that $\mathsf{Ho}(\mathcal{M})$ is pointed, i.e. that the
morphism
$$ \emptyset \longrightarrow \ast \,,$$
in $\mathcal{M}$, where $\emptyset$ denotes the initial object and $\ast$
the terminal one, is a weak equivalence.
Consider the morphism
$$ P: \widetilde{\emptyset} \longrightarrow h(\emptyset)\,,$$
where $\widetilde{\emptyset}$ denotes the initial object in $\mathsf{L}_{\Sigma}\mathsf{Fun}(\mathcal{M}_f^o,Sset)$.

Observe that, since $\mathbb{R}\underline{h}$ admits a left adjoint,
$h(\emptyset)$ identifies with the terminal object in
$$ \mathsf{Ho}(\mathsf{L}_{\Sigma}\mathsf{Fun}(\mathcal{M}_f^o,Sset))\,,$$
because
$$h(\emptyset)=  \mathsf{Ho}(h)(\emptyset) \stackrel{\sim}{\rightarrow}
\mathbb{R}\underline{h}\circ \mathsf{Ho}(i)(\emptyset) \stackrel{\sim}{\rightarrow} \mathbb{R}\underline{h}(\ast)\,.$$
We denote by
$$ \mathsf{L}_{\Sigma,P}\mathsf{Fun}(\mathcal{M}_f^o,Sset)\,,$$
the left Bousfield localization of
$\mathsf{L}_{\Sigma}\mathsf{Fun}(\mathcal{M}_f^o,Sset)$ at the
morphism $P$.

Notice that the category
$$
\mathsf{Ho}(\mathsf{L}_{\Sigma,P}\mathsf{Fun}(\mathcal{M}_f^o,Sset))\,,$$
is now a pointed one.

We have the following morphisms of derivators
$$
\xymatrix{
\mathsf{Ho}(\mathcal{M}) \ar@<1ex>[d]^{\mathbb{R}\underline{h}} \\
\mathsf{L}_{\Sigma}\mathsf{Hot}_{\mathcal{M}_f}
\ar@<1ex>^{\mathbb{L}Re}[u] \ar[d]_{\Phi} \\
\mathsf{L}_{\Sigma,P}\mathsf{Hot}_{\mathcal{M}_f}.
}
$$
By construction, we have a pointed morphism of derivators
$$ \mathsf{HO}(\mathcal{M}) \stackrel{\Phi \circ
  \mathbb{R}\underline{h}}{\longrightarrow}
\mathsf{L}_{\Sigma,P}\mathsf{Hot}_{\mathcal{M}_f}\,,$$
which commutes with filtered homotopy colimits and preserves the point.

Let $\mathbb{D}$ be a pointed derivator.

\begin{proposition}\label{ext}
The morphism of derivators $\Phi \circ \mathbb{R}\underline{h}$
induces an equivalence of categories
$$\underline{\mathsf{Hom}}_!(\mathsf{L}_{\Sigma,P}\mathsf{Hot}_{\mathcal{M}_f},\mathbb{D})
\stackrel{(\Phi \circ
  \mathbb{R}\underline{h})^{\ast}}{\longrightarrow}
\underline{\mathsf{Hom}}_{flt,p}(\mathsf{HO}(\mathcal{M}),\mathbb{D})\,,$$
where
$\underline{\mathsf{Hom}}_{flt,p}(\mathsf{HO}(\mathcal{M}),\mathbb{D})$
denotes the category of morphisms of derivators which commute with
filtered homotopy colimits and preserve the point.
\end{proposition}

\begin{proof}
By theorem~\ref{Cisinsk}, we have an equivalence of categories
$$\underline{\mathsf{Hom}}_!(\mathsf{L}_{\Sigma,P}\mathsf{Hot}_{\mathcal{M}_f},\mathbb{D})
\stackrel{\Phi^{\ast}}{\longrightarrow}
\underline{\mathsf{Hom}}_{!,P}(\mathsf{L}_{\Sigma}\mathsf{Hot}_{\mathcal{M}_f},\mathbb{D})\,.$$
By theorem~\ref{flt}, we have an equivalence of categories
$$\underline{\mathsf{Hom}}_!(\mathsf{L}_{\Sigma}\mathsf{Hot}_{\mathcal{M}_f},\mathbb{D})
\stackrel{\mathbb{R}\underline{h}^{\ast}}{\longrightarrow}
\underline{\mathsf{Hom}}_{flt}(\mathsf{HO}(\mathcal{M}),\mathbb{D})\,.$$
We now show that under this last equivalence, the category
$\underline{\mathsf{Hom}}_{!,P}(\mathsf{L}_{\Sigma}\mathsf{Hot}_{\mathcal{M}_f},\mathbb{D})$
identifies with
$\underline{\mathsf{Hom}}_{flt,p}(\mathsf{HO}(\mathcal{M}),\mathbb{D})$.
Let $F$ be an object of
$\underline{\mathsf{Hom}}_{!,P}(\mathsf{L}_{\Sigma}\mathsf{Hot}_{\mathcal{M}_f},\mathbb{D})$.
Since $F$ commutes with homotopy colimits, it preserves the initial
object. This implies that $F \circ \mathbb{R}\underline{h}$ belongs to $\underline{\mathsf{Hom}}_{flt,p}(\mathsf{HO}(\mathcal{M},\mathbb{D})\,.$

Let now $G$ be an object of
$\underline{\mathsf{Hom}}_{flt,p}(\mathsf{HO}(\mathcal{M}),\mathbb{D})$.
Consider, as in the proof of theorem~\ref{flt}, the morphism
$$ \Psi({\mathbb{L}Re}^{\ast}(G)):
\mathsf{L}_{\Sigma}\mathsf{Hot}_{\mathcal{M}_f} \longrightarrow
\mathbb{D}\,.$$

Since $\Psi({\mathbb{L}Re}^{\ast}(G))$ commutes with homotopy
colimits, by construction, it sends $\widetilde{\emptyset}$ to the
point of $\mathbb{D}$. Observe also that $h(\emptyset)$ is also sent to
the point of $\mathbb{D}$ because
$$\Psi({\mathbb{L}Re}^{\ast}(G))(h(\emptyset)) \simeq
G(\emptyset)\,.$$

This proves the proposition.

\end{proof}

\section{Small weak generators}\label{small}

Let $\mathcal{N}$ be a pointed, left proper, compactly generated
Quillen model category as in definition $2.1$ of \cite{Toen-Vaq}.
Observe that in particular this implies that
$\mathcal{N}$ is finitely generated, as in section $7.4$ in \cite{Hovey}.
We denote by $\mathcal{G}$ the set of cofibers of the generating
cofibrations $I$ in $\mathcal{N}$. By corollary $7.4.4$ in
\cite{Hovey}, the set $\mathcal{G}$ is a set of small weak generators
for $\mathsf{Ho}(\mathcal{N})$, see definitions $7.2.1$ and $7.2.2$ in
\cite{Hovey}.
Let $S$ be a set of morphisms in
$\mathcal{N}$ between objects which are homotopically finitely
presented, see \cite{Toen-Vaq}, and $\mathsf{L}_S\mathcal{N}$ the left Bousfield localization
of $\mathcal{N}$ by $S$. We have an adjunction
$$
\xymatrix{
\mathsf{Ho}(\mathcal{N}) \ar@<-1ex>[d]_{\mathbb{L}Id} \\
\mathsf{Ho}(\mathsf{L}_S\mathcal{N}) \ar@<-1ex>[u]_{\mathbb{R}Id}\,.
}
$$
\begin{lemma}\label{gener}
The image of the set $\mathcal{G}$ under the functor $\mathbb{L}Id$ is a set of small weak generators in
$\mathsf{Ho}(\mathsf{L}_S\mathcal{N})$.
\end{lemma}

\begin{proof}
The previous adjunction is equivalent to
$$
\xymatrix{
\mathsf{Ho}(\mathcal{N}) \ar@<-2ex>[d]_{(-)_f} \\
*+<1pc>{\mathsf{Ho}(\mathcal{N})_S} \ar@{_{(}->}[u]\,,
}
$$
where $\mathsf{Ho}(\mathcal{N})_S$ denotes the full subcategory of
$\mathsf{Ho}(\mathcal{M})$ formed by the $S$-local objects of
$\mathcal{N}$ and $(-)_f$ denotes a fibrant resolution functor in
$\mathsf{L}_S\mathcal{N}$, see \cite{Hirschhorn}. Clearly, this
implies that the image of the set $\mathcal{G}$ under the functor $(-)_f$ is a set of weak generators in
$\mathsf{Ho}(\mathsf{L}_S\mathcal{N})$.

We now show that the $S$-local
objects in $\mathcal{N}$ are stable under filtered homotopy colimits.
Let ${\{ X_i \}}_{i \in I}$ be a filtered diagram of $S$-local
objects.
By proposition~\ref{prop}, we have an
isomorphism
$$ \underset{i \in I}{\mbox{hocolim}}\,X_i \stackrel{\sim}{\longrightarrow}
\underset{i \in I}{\mbox{colim}}\,X_i$$
in $\mathsf{Ho}(\mathcal{N})$.
We now show that $\underset{i \in I}{\mbox{colim}}\,X_i$ is an $S$-local
object.
Let $g:A \rightarrow B$ be an element of $S$. We have at our disposal the
following commutative diagram
$$
\xymatrix{
\mathsf{Map}(B,\underset{i \in I}{\mbox{colim}}\,X_i) \ar[rr]^{g^*} & &
\mathsf{Map}(A,\underset{i \in I}{\mbox{colim}}\,X_i) \\
\underset{i \in I}{\mbox{colim}}\,\mathsf{Map}(B,X_i) \ar[u]^{\sim}   \ar[rr]_{\underset{i
    \in I}{\mbox{\mbox{colim}}} \, g_i^*} & & \underset{i \in
  I}{\mbox{colim}}\,\mathsf{Map}(A,X_i) \ar[u]_{\sim}\,.
}
$$
Now observe that since $A$ and $B$ are homotopically finitely
presented objects, the vertical arrows in the diagram are isomorphisms
in $\mathsf{Ho}(Sset)$. Since each object $X_i$ is $S$-local, the
morphism $g_i^*$ is an isomorphism in $\mathsf{Ho}(Sset)$ and so is $\underset{i
    \in I}{\mbox{colim}} \, g_i^*$. This implies that $\underset{i \in
    I}{\mbox{colim}}\,X_i$ is an $S$-local object. This shows that the
  inclusion
$$ \mathsf{Ho}(\mathcal{M})_S \hookrightarrow
\mathsf{Ho}(\mathcal{M})\,,$$
commutes with filtered homotopy colimits and so the image of the
set $\mathcal{G}$ under the functor $(-)_f$ consists of small objects of
$\mathsf{Ho}(\mathsf{L}_S\mathcal{N})$.

This proves the lemma.
\end{proof}

Recall from the previous chapter that we have constructed a pointed
derivator $\mathsf{L}_{\Sigma,P}\mathsf{Hot}_{\mathcal{M}_f}$. We will
now construct a strictly pointed Quillen model catgory whose
associated derivator is equivalent to $\mathsf{L}_{\Sigma,P}\mathsf{Hot}_{\mathcal{M}_f}$.
Consider the pointed Quillen model category
$$ \ast \downarrow \mathsf{Fun}(\mathcal{M}^o_f,Sset) =
\mathsf{Fun}(\mathcal{M}_f^o, Sset_{\bullet})\,.$$
We have the following Quillen adjunction
$$
\xymatrix{
\mathsf{Fun}(\mathcal{M}^o_f,Sset_{\bullet}) \ar@<1ex>[d]^U \\
\mathsf{Fun}(\mathcal{M}_f^o,Sset) \ar@<1ex>[u]^{(-)_{+}}\,,
}
$$
where $U$ denotes the forgetful functor.

We denote by
$\mathsf{L}_{\Sigma,P}\mathsf{Fun}(\mathcal{M}_f^o,Sset_{\bullet})$
the left Bousfield localization of
$\mathsf{Fun}(\mathcal{M}_f^o,Sset_{\bullet})$ by the image of the set
$\Sigma \cup \{ P\}$ under the functor $(-)_{+}$.
We denote by ${\mathsf{L}_{\Sigma,P}\mathsf{Hot}_{\mathcal{M}_f}}_{\bullet}$ the
derivator associated with $\mathsf{L}_{\Sigma,P}\mathsf{Fun}(\mathcal{M}_f^o,Sset_{\bullet})$.

\begin{remark}\label{remar}
Since the derivators associated with
$\mathsf{L}_{\Sigma,P}\mathsf{Fun}(\mathcal{M}_f^o,Sset)$ and
$\mathsf{L}_{\Sigma,P}\mathsf{Fun}(\mathcal{M}_f^o,Sset_{\bullet})$
are caracterized by the same universal property we have a canonical equivalence of pointed derivators
$$ \mathsf{L}_{\Sigma,P}\mathsf{Hot}_{\mathcal{M}_f}
\stackrel{\sim}{\longrightarrow}
{\mathsf{L}_{\Sigma,P}\mathsf{Hot}_{\mathcal{M}_f}}_{\bullet}$$
Notice also that the category
$\mathsf{Fun}(\mathcal{M}_f^o,Sset_{\bullet})$ endowed with the
projective model structure is pointed, left proper, compactly
generated and that the domains and codomains of the elements of the
set $(\Sigma \cup \{P\})_+$ are homotopically finitely presented
objects.
Therefore by lemma~\ref{gener}, the set
$$ \mathcal{G}= \{
\mathbf{F}^X_{\Delta[n]_+ / \partial \Delta[n]_+ } |
\, X \in
\mathcal{M}_f\,, n\geq 0 \} \,,$$
of cofibers of the generating cofibrations in
$\mathsf{Fun}(\mathcal{M}_f^o,Sset_{\bullet})$ is a set of small weak
generators in \newline $\mathsf{Ho}(\mathsf{L}_{\Sigma,P}\mathsf{Fun}(\mathcal{M}_f^o,Sset_{\bullet}))$.
\end{remark}

\section{Stabilization}\label{spectra}

Let $\mathbb{D}$ be a regular pointed strong derivator.

In \cite{Heller}, Heller constructs a universal morphism to a triangulated
strong derivator
$$ \mathbb{D} \stackrel{stab}{\longrightarrow}
\mathsf{St}(\mathbb{D})\,,$$
which commutes with homotopy colimits.

This construction is done in two steps.
First consider the following ordered set
$$ \mathbf{V} := \{ (i,j) |\, |i -j|\leq 1\} \subset \mathbb{Z} \times \mathbb{Z}$$
naturally as a small category. We denote by
$$ \overset{\cdot}{\mathbf{V}}:= \{ (i,j) |\, |i-j|=1\} \subset
\mathbb{V}\,,$$
the full subcategory of `points on the boundary'.

Now, let $\mathsf{Spec}(\mathbb{D})$ be the full subderivator of
$\mathbb{D}_{\mathbf{V}}$, see definition $3.4$ in \cite{Cis-Nee}, formed by the objects $X$ in
$\mathbb{D}_{\mathbf{V}}(L)$, whose image under the functor
$$ \mathbb{D}_{\mathbf{V}}(L)=\mathbb{D}(\mathbf{V} \times L)
\longrightarrow \mathsf{Fun}(\mathbf{V}^{op},\mathbb{D}(L))$$
is of the form
$$
\xymatrix{
 & \ast \ar[r] & X_{(1,1)}\, \cdots \\
\ast \ar[r] & X_{(0,0)} \ar[r] \ar[u] & \ast \ar[u] \\
\cdots\,X_{(-1,-1)} \ar[u] \ar[r] & \ast \ar[u] & \,,
}
$$
see section $8$ in \cite{Heller}.
We have an evaluation functor $ev_{(0,0)}:\mathsf{Spec}(\mathbb{D})
\rightarrow \mathbb{D}$, which admits a left adjoint $L[0,0]$.

Finally, let $\mathsf{St}(\mathbb{D})$ be the full reflexive subderivator of
$\mathsf{Spec}(\mathbb{D})$ consisting of the $\Omega$-spectra, as
defined in \cite{Heller}.

We have the following adjunctions
$$
\xymatrix{
\mathbb{D} \ar@<-1ex>[d]_{L[0,0]} \ar@/_4pc/[dd]_{stab} \\
\mathsf{Spec}(\mathbb{D}) \ar@<-1ex>[u]_{ev_{(0,0)}}
\ar@<-1ex>[d]_{loc} \\
*+<1pc>{\mathsf{St}(\mathbb{D})} \ar@{_{(}->}[u] \,,
}
$$
in the $2$-category of derivators.

Let $\mathbb{T}$ be a triangulated strong derivator.
The following theorem is proved in \cite{Heller}.

\begin{theorem}\label{HellerT}
The morphism $stab$ induces an equivalence of categories
$$ \underline{\mathsf{Hom}}_!(\mathsf{St}(\mathbb{D}),\mathbb{T})
\stackrel{stab^*}{\longrightarrow}
\underline{\mathsf{Hom}}_!(\mathbb{D},\mathbb{T})\,.$$
\end{theorem}

\begin{lemma}\label{gener1}
Let $\mathcal{G}$ be a set of objects in $\mathbb{D}(e)$
which satisfies the following conditions:
\begin{itemize}
\item[A1)] If, for each $g$ in $\mathcal{G}$, we have
$$ \mathsf{Hom}_{\mathbb{D}(e)}(g,X)=\{*\}\,,$$
then $X$ is isomorphic to $*$, where $*$ denotes
 the terminal and initial object in $\mathbb{D}(e)$.

\item[A2)] For every set $K$ and each $g$ in $\mathcal{G}$  the
  canonical map
$$\underset{\stackrel{S \subseteq K}{S finite}}{\mbox{colim}}\,
\mathsf{Hom}_{\mathbb{D}(e)}(g, \underset{\alpha \in S}{\coprod}
X_{\alpha}) \stackrel{\sim}{\rightarrow} \mathsf{Hom}_{\mathbb{D}(e)}(g, \underset{\alpha \in S}{\coprod}
X_{\alpha})$$
is bijective.
\end{itemize}
Then the set
$$ {\{ \Sigma^n stab(g) \,|\, g \in \mathcal{G}, \, n \in \mathbb{Z} \}}$$
of objects in $\mathsf{St}(\mathbb{D})(e)$, where $\Sigma$ denotes the
suspension functor in $\mathsf{St}(\mathbb{D})(e)$, satisfies $A1)$
and $A2)$.
\end{lemma}

\begin{proof}
Let $\underline{X}$ be an object of $\mathsf{St}(\mathbb{D})(e)$.
Suppose that for each $g$ in $\mathcal{G}$ and $n$ in $\mathbb{Z}$, we have
$$ \mathsf{Hom}_{\mathsf{St}(\mathbb{D})(e)}(\Sigma^n stab(g),
\underline{X}) = \{\ast\}\,.$$
Then by the following isomorphisms
$$
\begin{array}{rcl}
 \mathsf{Hom}_{\mathsf{St}(\mathbb{D})(e)}(\Sigma^n stab(g),
\underline{X}) & \simeq  & \mathsf{Hom}_{\mathsf{St}(\mathbb{D})(e)}(stab(g),
\Omega^n \underline{X})\\
& \simeq & \mathsf{Hom}_{\mathbb{D}(e)}(g,
ev_{(0,0)} \Omega^n \underline{X})\\
& \simeq & \mathsf{Hom}_{\mathbb{D}(e)}(g,
ev_{(n,n)}\underline{X})\,,
\end{array}
$$
we conclude that, for all $n$ in $\mathbb{Z}$, we have
$$ ev_{(n,n)} \underline{X} = \ast\,.$$
By the conservativity axiom, $\underline{X}$ is isomorphic to $\ast$ in $\mathsf{St}(\mathbb{D})(e)$. This shows condition $A1)$. Now
observe that condition $A2)$ follows from the following isomorphisms
$$
\begin{array}{rcl}
\mathsf{Hom}_{\mathsf{St}(\mathbb{D})(e)}(\Sigma^n stab(g),
\underset{\alpha \in K}{\bigoplus} \underline{X_{\alpha}})&
\simeq & \mathsf{Hom}_{\mathsf{St}(\mathbb{D})(e)}(stab(g),
 \Omega^n \underset{\alpha \in K}{\bigoplus} \underline{X_{\alpha}})\\

& \simeq & \mathsf{Hom}_{\mathsf{St}(\mathbb{D})(e)}(stab(g), \underset{\alpha
  \in K}{\bigoplus} \Omega^n \underline{X_{\alpha}})\\

& \simeq & \mathsf{Hom}_{\mathbb{D}(e)}(g, ev_{(0,0)} \underset{\alpha
  \in K}{\coprod} \Omega^n \underline{X_{\alpha}})\\

&\simeq &\mathsf{Hom}_{\mathbb{D}(e)}(g, \underset{\alpha
  \in K}{\coprod} ev_{(0,0)} \Omega^n \underline{X_{\alpha}})\\

&\simeq &\underset{\stackrel{S \subseteq K}{S finite}}{\mbox{colim}} \, \mathsf{Hom}_{\mathbb{D}(e)}(g, \underset{\alpha
  \in S}{\coprod} ev_{(0,0)} \Omega^n \underline{X_{\alpha}})\\

&\simeq &\underset{\stackrel{S \subseteq K}{S finite}}{\mbox{colim}}
\,\mathsf{Hom}_{\mathbb{D}(e)}(g, ev_{(0,0)} \underset{\alpha
  \in S}{\coprod} \Omega^n \underline{X_{\alpha}})\\

&\simeq &\underset{\stackrel{S \subseteq K}{S finite}}{\mbox{colim}}
\,\mathsf{Hom}_{\mathsf{St}(\mathbb{D})(e)}(stab(g),\underset{\alpha
  \in S}{\bigoplus} \Omega^n \underline{X_{\alpha}})\\

&\simeq &\underset{\stackrel{S \subseteq K}{S finite}}{\mbox{colim}}
\,\mathsf{Hom}_{\mathsf{St}(\mathbb{D})(e)}(\Sigma^n stab(g),\underset{\alpha
  \in S}{\bigoplus} \underline{X_{\alpha}})

\end{array}
$$

\end{proof}

\begin{lemma}\label{gen1}
Let $\mathbb{T}$ be a triangulated derivator and $\mathcal{G}$ a set
of objects in $\mathbb{T}(e)$ which satisfies conditions $A1)$ and $A2)$ of lemma~\ref{gener1}.

Then for every small category $L$ and every point $x: e \rightarrow L$
in $L$, the set
$$ \{x_!(g)\,|\,g \in \mathcal{G}, \, x:e\rightarrow L\}$$ satisfies conditions
$A1)$ and $A2)$ in the category $\mathbb{T}(L)$.
\end{lemma}

\begin{proof}
Suppose that
$$ \mathsf{Hom}_{\mathbb{T}(L)}(x_!(g),M) = \{\ast \}\,,$$
for every $g \in \mathcal{G}$ and every point $x$ in $L$. Then by adjunction
$x^{\ast}M$ is isomorphic to $\ast$ in $\mathbb{T}(e)$ and so by the
conservativity axiom, $M$ is isomorphic to $\ast$ in
$\mathbb{T}(L)$. This shows condition $A1)$. Condition $A2)$ follows
from the following isomorphisms
$$
\begin{array}{rcl}
\mathsf{Hom}_{\mathbb{T}(L)}(x_!(g),\underset{\alpha \in
  K}{\bigoplus} M_{\alpha})&
\simeq &\mathsf{Hom}_{\mathbb{T}(e)}(g,x^{\ast} \underset{\alpha \in
  K}{\bigoplus} M_{\alpha})\\

&\simeq &\mathsf{Hom}_{\mathbb{T}(e)}(g, \underset{\alpha \in
  K}{\bigoplus} x^{\ast} M_{\alpha})\\

&\simeq & \underset{\alpha \in K}{\bigoplus}   \mathsf{Hom}_{\mathbb{T}(e)}(g, x^{\ast} M_{\alpha})\\

&\simeq & \underset{\alpha \in K}{\bigoplus}
\mathsf{Hom}_{\mathbb{T}(L)}(x_!(g), M_{\alpha})\,.

\end{array}
$$
\end{proof}

\begin{remark}\label{important}
Notice that if $\mathbb{D}$ is a regular pointed strong derivator and we
have at our disposal of a set $\mathcal{G}$ of objects in $\mathbb{D}(e)$
which satisfies conditions $A1)$ and $A2)$, then lemma~\ref{gener} and lemma~\ref{gen1} imply
that $\mathsf{St}(\mathbb{D})(L)$ is a compactly generated triangulated
category, for every small category $L$.
\end{remark}

\subsection*{Relation with Hovey/Schwede's stabilization}

We will now relate Heller's construction with the construction of
spectra as it is done by Hovey in \cite{Spectra} and Schwede in \cite{Schwede}.

Let $\mathcal{M}$ be a pointed, simplicial, left proper, cellular,
almost finitely generated Quillen model category, see definition $4.1$ in
\cite{Spectra}, where sequential colimits commute with finite products
and homotopy pullbacks. This implies that the associated derivator
$\mathsf{HO}(\mathcal{M})$ will be regular.

\begin{example}\label{exem}
Consider the category
$\mathsf{L}_{\Sigma,P}\mathsf{Fun}(\mathcal{M}_f^o,Sset_{\bullet})$
defined in section~\ref{small}.
Notice that the category of pointed simplicial pre-sheaves
$\mathsf{Fun}(\mathcal{M}_f^o,Sset_{\bullet})$ is pointed, simplicial,
left proper, cellular and even finitely generated, see definition
$4.1$ in \cite{Spectra}. Since limits and colimits in
$\mathsf{Fun}(\mathcal{M}_f^o,Sset_{\bullet})$ are calculated
objectwise, we conclude that sequential colimits commute with finite
products.
Now, by theorem $4.1.1$ in \cite{Hirschhorn} the category
$\mathsf{L}_{\Sigma,P}\mathsf{Fun}(\mathcal{M}_f^o,Sset_{\bullet})$ is
also pointed, simplicial, left proper and cellular.

Now observe that the domains and codomains of each morphism in
$\Lambda((\Sigma \cup \{P\})_{+})$, see definition $4.2.1$ in
\cite{Hirschhorn}, are finitely presented, since the forgetful functor
$$ \mathsf{Fun}(\mathcal{M}_f^o, Sset_{\bullet}) \rightarrow \mathsf{Fun}(\mathcal{M}^o,Sset)$$
commutes with filtered colimits and homotopy pullbacks. Now, by proposition $4.2.4$
in \cite{Hirschhorn}, we conclude that a morphism $A
\stackrel{f}{\rightarrow} B$ in
$\mathsf{L}_{\Sigma,P}\mathsf{Fun}(\mathcal{M}_f^o,Sset_{\bullet})$,
with $B$ a local object, is a local fibration if and
only if it has the right lifting property with respect to the set
$$ J \cup \Lambda((\Sigma \cup \{P\})_{+})\,,$$
where $J$ denotes the set of generating acyclic cofibrations in
$\mathsf{Fun}(\mathcal{M}_f^o,Sset_{\bullet})$. This shows that
$\mathsf{L}_{\Sigma,P}\mathsf{Fun}(\mathcal{M}_f^o,Sset_{\bullet})$ is
almost finitely generated.

\end{example}

Recall from section $1.2$ in \cite{Schwede} that since $\mathcal{M}$
is a pointed, simplicial model category, we have a Quillen adjunction
$$
\xymatrix{
\mathcal{M} \ar@<-1ex>[d]_{\Sigma(-)} \\
\mathcal{M} \ar@<-1ex>[u]_{\Omega(-)}\,,
}
$$
where $\Sigma(X)$ denotes the suspension of an object $X$ , i.e. the
pushout of the diagram
$$
\xymatrix{
X \otimes \partial \Delta^1 \ar[r] \ar[d] & X \otimes \Delta^1 \\
\ast & .
}
$$
Recall also that in \cite{Spectra} and \cite{Schwede} the autors
construct a stable Quillen model category
$\mathsf{Sp}^{\mathbb{N}}(\mathcal{M})$ of spectra associated with
$\mathcal{M}$ and the left Quillen functor $\Sigma(-)$.
We have the following Quillen adjunction, see~\cite{Spectra},
$$
\xymatrix{
\mathcal{M} \ar@<-1ex>[d]_{\Sigma^{\infty}} \\
\mathsf{Sp}^{\mathbb{N}}(\mathcal{M}) \ar@<-1ex>[u]_{ev_0}
}
$$
and thus a morphism to a strong triangulated derivator
$$ \mathsf{HO}(\mathcal{M})
\stackrel{\mathbb{L}\Sigma^{\infty}}{\longrightarrow}
  \mathsf{HO}(\mathsf{Sp}^{\mathbb{N}}(\mathcal{M}))$$
which commutes with homotopy colimits.

By theorem~\ref{HellerT}, we have at our disposal a diagram
$$
\xymatrix{
\mathsf{HO}(\mathcal{M}) \ar[d]_{stab}
\ar[dr]^{\mathbb{L}\Sigma^{\infty}} & \\
\mathsf{St}(\mathsf{HO}(\mathcal{M})) \ar[r]_{\varphi} & \mathsf{HO}(\mathsf{Sp}^{\mathbb{N}}(\mathcal{M}))\,,
}
$$
which is commutative up to isomorphism in the $2$-category of derivators.

Now suppose also that we have a set $\mathcal{G}$ of small weak generators in $\mathsf{Ho}(\mathcal{M})$, as in
definitions $7.2.1$ and $7.2.2$ in \cite{Spectra}.
Suppose also that each object of $\mathcal{G}$ considered in $\mathcal{M}$ is cofibrant,
finitely presented, homotopy finitely presented and has a finitely
presented cylinder object.

\begin{example}\label{ex2}
Observe that the category
$\mathsf{Fun}(\mathcal{M}_f^o,Sset_{\bullet})$ is pointed and finitely
generated. By corollary $7.4.4$ in \cite{Hovey}, the set
$$ \mathcal{G}= \{
\mathbf{F}^X_{\Delta[n]_{+} / \partial \Delta[n]_{+} } |
\, X \in
\mathcal{M}_f\,, n\geq 0 \} \,,$$
is a set of small weak generators in
$\mathsf{Ho}(\mathsf{Fun}(\mathcal{M}^o_f,Sset_{\bullet}))$. Since the
domains and codomains of the set
$$ (\Sigma \cup \{ P\} )_{+}$$
are homotopically finitely presented objects, lemma~\ref{gener}
implies that $\mathcal{G}$ is a set of small weak generators
in $\mathsf{Ho}(\mathsf{L}_{\Sigma,P}
\mathsf{Fun}(\mathcal{M}_f^o,Sset_{\bullet}))$. Clearly the elements
of $\mathcal{G}$ are cofibrant, finitely presented and have a finitely
presented cylinder object. They are also homotopically finitely presented.
\end{example}

Under the hypotheses above on the category $\mathcal{M}$, we have the following comparison theorem

\begin{theorem}\label{repre}
The induced morphism of triangulated derivators
$$ \varphi: \mathsf{St}(\mathsf{HO}(\mathcal{M})) \longrightarrow
\mathsf{HO}(\mathsf{Sp}^{\mathbb{N}}(\mathcal{M}))$$
is an equivalence.
\end{theorem}

The proof of theorem~\ref{repre} will consist in verifying the
conditions of the following general proposition.

\begin{proposition}\label{aux}
Let $F: \mathbb{T}_1 \rightarrow \mathbb{T}_2$ be a morphism of strong
triangulated derivators. Suppose that the triangulated categories
$\mathbb{T}_1(e)$ and $\mathbb{T}_2(e)$ are compactly generated and
that there is a set $\mathcal{G} \subset \mathbb{T}_1(e)$ of compact generators,
which is stable under suspensions and satisfies the following conditions~:
\begin{itemize}
\item[a)] $F(e)$ induces bijections
$$
  \mathsf{Hom}_{\mathbb{T}_1(e)}(g_1,g_2) \rightarrow
  \mathsf{Hom}_{\mathbb{T}_2(e)}(Fg_1, Fg_2), \, \forall g_1, g_2 \in
  \mathcal{G}$$
and
\item[b)] the set of objects $\{ Fg \, | \, g \in \mathcal{G} \}$ is a
  set of compact generators in $\mathbb{T}_2(e)$.
\end{itemize}
Then the morphism $F$ is an equivalence of derivators.
\end{proposition}

\begin{proof}
Conditions $a)$ and $b)$ imply that $F(e)$ is an equivalence of
triangulated category, see~\cite{Neeman}.

Now, let $L$ be a small category. We show that conditions $a)$ and
$b)$ are also verified by $F(L)$, $\mathbb{T}_1(L)$ and $\mathbb{T}_2(L)$. By lemma~\ref{gen1} the sets
$$
\begin{array}{rcl}
\{ x_!(g) |\, g \in \mathcal{G}, \, x:e \rightarrow L\} & \mbox{and} & \{ x_!(Fg) |\, g \in \mathcal{G}, \, x:e \rightarrow L\}
\end{array}
$$
consist of compact generators for $\mathbb{T}_1(L)$, resp. $\mathbb{T}_2(L)$, which are stable under suspensions. Since $F$
commutes with homotopy colimits $F(x_!(g)) = x_!(Fg)$ and so the following isomorphisms
$$
\begin{array}{rcl}
\mathsf{Hom}_{\mathbb{T}_1(L)}(x_!(g_1), x_!(g_2)) & \simeq &
\mathsf{Hom}_{\mathbb{T}_1(e)}(g_1, x^{\ast} x_!(g_2)) \\
 & \simeq & \mathsf{Hom}_{\mathbb{T}_2(e)}(F(g_1), F(x^{\ast} x_!(g_2)))
 \\
 & \simeq & \mathsf{Hom}_{\mathbb{T}_2(e)}(Fg_1, x^{\ast} F(x_!(g_2)))
 \\
 & \simeq &  \mathsf{Hom}_{\mathbb{T}_2(L)}(x_! F(g_1), x_!F(g_2))
\end{array}
$$
imply the proposition.
\end{proof}

Let us now prove theorem~\ref{repre}~:
\begin{proof}

Let us first prove condition $b)$ of proposition~\ref{aux}. Since the
set $\mathcal{G}$ of small generators in $\mathsf{Ho}(\mathcal{M})$
satisfies the conditions of lemma~\ref{gener1}, we have a set
$$ {\{ \Sigma^n stab(g) \,|\, g \in \mathcal{G}, \, n \in \mathbb{Z} \}}$$
of compact generators in $\mathsf{St}(\mathsf{HO}(\mathcal{M}))(e)$,
which is stable under suspensions. We now show that the set
$$ {\{ \Sigma^n \mathbb{L} \Sigma^{\infty}(g) \,|\, g \in \mathcal{G}, \, n \in \mathbb{Z} \}}$$
is a set of compact generators in
$\mathsf{Ho}(\mathsf{Sp}^{\mathbb{N}}(\mathcal{M}))$. These objects
are compact because the functor $\mathbb{R}ev_0$ in the adjunction
$$
\xymatrix{
\mathsf{Ho}(\mathcal{M}) \ar@<-1ex>[d]_{\mathbb{L}\Sigma^{\infty}} \\
\mathsf{Ho}(\mathsf{Sp}^{\mathbb{N}}(\mathcal{M})) \ar@<-1ex>[u]_{\mathbb{R}ev_0}
}
$$
commutes with filtered homotopy colimits.
We now show that they form a set of generators. Let $Y$ be an object in
$\mathsf{Ho}(\mathsf{Sp}^{\mathbb{N}}(\mathcal{M}))$, that we can
suppose, without loss of generality, to be an $\Omega$-spectrum,
see \cite{Spectra}.
Suppose that
$$
\mathsf{Hom}(\Sigma^n \mathbb{L}\Sigma^{\infty}(g_i),Y)
\simeq \underset{m}{\mbox{colim}}\, \mathsf{Hom}(g_i, \Omega^m Y_{m+p}) =
\{ \ast \},\,\, n \geq 0 \,.
$$
Since $Y$ is an $\Omega$-spectrum, we have
$$ Y_p = \ast\,,\forall p\geq 0\,.$$
This implies that $Y$ is isomorphic to $\ast$ in
$\mathsf{Ho}(\mathsf{Sp}^{\mathbb{N}}(\mathcal{M}))$.

We now show condition $a)$.
Let $g_1$ and $g_2$ be objects in $\mathcal{G}$. Observe that we have the following isomorphisms, see \cite{Heller}
$$
\begin{array}{l}
\mathsf{Hom}_{\mathsf{St}(\mathsf{HO}(\mathcal{M}))(e)}(stab(g_1),stab(g_2))
\\
\simeq \mathsf{Hom}_{\mathsf{Ho}(\mathcal{M})}(g_1, (ev_{(0,0)} \circ loc
\circ L[0,0])(g_2))\\
\simeq \mathsf{Hom}_{\mathsf{Ho}(\mathcal{M})}(g_1,
ev_{(0,0)}(\mbox{hocolim}\,( L[0,0](g_2) \rightarrow
\Omega\sigma L[0,0](g_2) \rightarrow \ldots )))\\
\simeq \mathsf{Hom}_{\mathsf{Ho}(\mathcal{M})}(g_1,
\mbox{hocolim}\, ev_{(0,0)} (L[0,0](g_2) \rightarrow
\Omega\sigma L[0,0](g_2) \rightarrow \cdots ))\\
\simeq \underset{j}{\mbox{colim}}\,
\mathsf{Hom}_{\mathsf{Ho}(\mathcal{M})}(g_1, \Omega^j \Sigma^j(g_2))\,.
\end{array}
$$
Now, by corollary $4.13$ in \cite{Spectra}, we have
$$
\begin{array}{ccc}
\mathsf{Hom}_{\mathsf{Ho}(\mathsf{Sp}^{\mathbb{N}}(\mathcal{M}))}(\mathbb{L}\Sigma^{\infty}(g_1),
\mathbb{L}\Sigma^{\infty}(g_2)) & \simeq & \mathsf{Hom}_{\mathsf{Ho}(\mathsf{Sp}^{\mathbb{N}}(\mathcal{M}))}(\Sigma^{\infty}(g_1),
(\Sigma^{\infty}(g_2))_f) \\
& \simeq & \underset{m}{\mbox{colim}}\,
\mathsf{Hom}_{\mathsf{Ho}(\mathcal{M})}(g_1, \Omega^m (\Sigma^m(g_2))_f)\,,
\end{array}
$$
where $(\Sigma^{\infty}(g_2))_f$ denotes a levelwise fibrant
resolution of $\Sigma^{\infty}(g_2)$ in the category $\mathsf{Sp}^{\mathbb{N}}(\mathcal{M})$.

Now, notice that since $g_2$ is cofibrant, so is $\Sigma^m(g_2)$ and
so we have the following isomorphism
$$ \Omega^m (\Sigma^m(g_2))_f \stackrel{\sim}{\longrightarrow}
(\mathbb{R}\Omega)^m \circ (\mathbb{L}\Sigma)^m(g_2)$$
in  $\mathsf{Ho}(\mathsf{SP}^{\mathbb{N}}(\mathcal{M}))$.
This implies that for $j\geq 0$, we have an isomorphism
$$ \Omega^j \Sigma^j(g_2) \stackrel{\sim}{\longrightarrow} \Omega^j
(\Sigma^j(g_2))_f$$
in $\mathsf{Ho}(\mathsf{Sp}^{\mathbb{N}}(\mathcal{M}))$ and so
$$ \mathsf{Hom}_{\mathsf{St}(\mathsf{HO}(\mathcal{M}))(e)}(stab(g_1), stab(g_2)) =
\mathsf{Hom}_{\mathsf{Ho}(\mathsf{Sp}^{\mathbb{N}}(\mathcal{M}))}(\mathbb{L}\Sigma^{\infty}(g_1),
\mathbb{L}\Sigma^{\infty}(g_2))\,.$$
Let now $p$ be an integer.
Notice that
$$ \mathsf{Hom}_{\mathsf{St}(\mathsf{HO}(\mathcal{M}))(e)}(stab(g_1),
\Sigma^p stab(g_2)) = \underset{j}{\mbox{colim}}\,
\mathsf{Hom}_{\mathsf{Ho}(\mathcal{M})}(g_1, \Omega^j
\Sigma^{j+p}(g_2))$$
and that
$$
\mathsf{Hom}_{\mathsf{Ho}(\mathsf{Sp}^{\mathbb{N}}(\mathcal{M}))}(\mathbb{L}
\Sigma^{\infty} (g_1),
\Sigma^p\mathbb{L}\Sigma^{\infty}(g_2))= \underset{m}{\mbox{colim}}\,
\mathsf{Hom}_{\mathsf{Ho}(\mathcal{M})}(g_1,
\Omega^m(\Sigma^{m+p}(g_2))_f)\,.$$

This proves condition $a)$ and so the theorem is proven.
\end{proof}

\begin{remark}\label{Tr}
If we consider for $\mathcal{M}$ the category
$\mathsf{L}_{\Sigma,P}\mathsf{Fun}(\mathcal{M}_f^o, Sset_{\bullet})$, we
have equivalences of derivators
$$ \varphi :
\mathsf{St}({\mathsf{L}_{\Sigma,P}\mathsf{Hot}_{\mathcal{M}_f}}_{\bullet})
\stackrel{\sim}{\longrightarrow}
\mathsf{HO}(\mathsf{Sp}^{\mathbb{N}}(\mathsf{L}_{\Sigma,P}\mathsf{Fun}(\mathcal{M}_f^o,
Sset_{\bullet}))) \stackrel{\sim}{\leftarrow} \mathsf{St}(\mathsf{L}_{\Sigma,P}\mathsf{Hot}_{\mathcal{M}_f}) \,.$$
\end{remark}
Let $\mathbb{D}$ be a strong triangulated derivator.

Now, by theorem~\ref{HellerT} and proposition~\ref{ext}, we have the
following proposition

\begin{proposition}\label{Trin}
We have an equivalence of categories
$$
\underline{\mathsf{Hom}}_!(\mathsf{St}(\mathsf{L}_{\Sigma,P} \mathsf{Hot}_{\mathsf{dgcat}_f}),\mathbb{D})
\stackrel{(stab\circ \Phi \circ
  \mathbb{R}\underline{h})^{\ast}}{\longrightarrow} \underline{\mathsf{Hom}}_{flt,p}(\mathsf{HO}(\mathsf{dgcat}),\mathbb{D})\,.$$
\end{proposition}

Since the category $Sset_{\bullet}$ satisfies all the conditions of
theorem~\ref{repre}, we have the following characterization of the
classical category of spectra, after Bousfield-Friedlander~\cite{Bos-Fri}, by a
universal property.

\begin{proposition}
We have an equivalence of categories
$$\underline{\mathsf{Hom}}_!(\mathsf{HO}(\mathsf{Sp}^{\mathbb{N}}(Sset_{\bullet})),\mathbb{D})
\stackrel{\sim}{\longrightarrow} \mathbb{D}(e) \,.$$
\end{proposition}

\begin{proof}
By theorems~\ref{repre} and \ref{Cin}, we have the following
equivalences
$$
\begin{array}{rcl}
\underline{\mathsf{Hom}}_!(\mathsf{HO}(\mathsf{Sp}^{\mathbb{N}}(Sset_{\bullet})),\mathbb{D})
&
\simeq &
\underline{\mathsf{Hom}}_!(\mathsf{HO}(Sset_{\bullet}),\mathbb{D})\\
& = & \underline{\mathsf{Hom}}_!(\mathsf{Hot}_{\bullet},\mathbb{D}) \\
& \simeq & \underline{\mathsf{Hom}}_!(\mathsf{Hot},\mathbb{D}) \\
& \simeq & \mathbb{D}(e) \,.
\end{array}
$$
This proves the proposition.
\end{proof}

\begin{remark}
An analoguos caracterization of the category of spectra, but in the
context of stable $\infty$-categories is proved in \cite[17.6]{Lurie}.
\end{remark}

\section{DG quotients}\label{chapquotient}
Recall from theorem~\ref{theorem2} that we have at our disposal a Morita Quillen model structure on the category of small dg
categories $\mathsf{dgcat}$. The homotopy category $\mathsf{Ho}(\mathsf{dgcat})$ is pointed.
In the following, we will be considering this Quillen model structure. We
denote by $I$ the set of generating cofibrations.

\begin{Notation}
We denote by $\mathcal{E}$ the set of inclusions of full dg subcategories
$$ \mathcal{G} \hookrightarrow \mathcal{H}\,,$$
where $\mathcal{H}$ is a strict finite $I$-cell.
\end{Notation}

Recall that we have a morphism of derivators
$$ \mathcal{U}_t:= stab \circ \Phi \circ \mathbb{R}\underline{h} : \mathsf{HO}(\mathsf{dgcat})
\rightarrow
\mathsf{St}(\mathsf{L}_{\Sigma,P}\mathsf{Hot}_{\mathsf{dgcat}_f})$$
which commutes with filtered homotopy colimits and preserves the point.

Let us now make some general arguments.

Let $\mathbb{D}$ be a pointed derivator. We denote by $I$ the category
associated to the graph
$$ 0 \leftarrow 1 \,.$$
Consider the functor $t=1 : e \rightarrow I$. Since the functor
$t$ is an open immersion and the derivator $\mathbb{D}$ is pointed,
the functor
$$ t_! : \mathbb{D}(e) \rightarrow \mathbb{D}(I)$$
has a left adjoint
$$ t^? : \mathbb{D}(I) \rightarrow \mathbb{D}(e)\,,$$
see \cite{Cis-Nee}.
We denote it by
$$ \mathsf{cone}: \mathbb{D}(I) \rightarrow \mathbb{D}(e)\,.$$
Let $F:\mathbb{D} \rightarrow \mathbb{D}'$ be a morphism of pointed
derivators. Notice that we have a natural transformation of functors.
$$S: \mathsf{cone} \circ F(I) \rightarrow F(e) \circ
\mathsf{cone}\,.$$

\begin{proposition}\label{cons}
Let $\mathcal{A} \stackrel{R}{\hookrightarrow} \mathcal{B}$ be an
inclusion of a full dg subcategory and $\ul_R$
$$
\xymatrix{
\mathcal{A} \ar@{^{(}->}[r]^R \ar[d] & \mathcal{B} \,,\\
0 &
}
$$
the associated object in
$\mathsf{HO}(\mathsf{dgcat})(\ul)$, where $0$ denotes the
terminal object in $\mathsf{Ho}(\mathsf{dgcat})$. Then there exists a filtered
category $J$ and an object $D_R$ in
$\mathsf{HO}(\mathsf{dgcat})(\ul \times J)$, such that
$$ p_!(D_R) \stackrel{\sim}{\longrightarrow} \ul_R\,,$$
where $p:\ul \times J \rightarrow \ul$ denotes the
projection functor. Moreover, for every point $j:e \rightarrow J$ in
$J$ the object $(1 \times j)^{\ast}$ in
$\mathsf{HO}(\mathsf{dgcat})(\ul)$ is of the form
$$ 0 \leftarrow Y_j \stackrel{L_j}{\rightarrow} X_j \,,$$
where $Y_j \stackrel{L_j}{\rightarrow} X_j$, belongs to
the set $\mathcal{E}$.
\end{proposition}

\begin{proof}
Apply the small object argument to the morphism
$$ \emptyset \longrightarrow \mathcal{B}$$
using the set of generating cofibrations $I$ and obtain the
factorization
$$
\xymatrix{
*+<1pc>{\emptyset} \ar[rr] \ar@{>->}[dr]_i & & \mathcal{B} \\
 & Q(\mathcal{B}) \ar@{->>}[ur]^{\sim}_p & \,,
}
$$
where $i$ is an $I$-cell.
Now consider the following fiber product
$$
\xymatrix{
p^{-1}(\mathcal{A}) \ar@{->>}[d]_{\sim} \ar@{^{(}->}[r]
\ar@{}[dr]|{\ulcorner} &
  Q(\mathcal{B}) \ar@{->>}[d]^p_{\sim} \\
\mathcal{A} \ar@{^{(}->}[r]^J & \mathcal{B}\,.
}
$$
Notice that $p^{-1}(\mathcal{A})$ is a full dg subcategory of
  $Q(\mathcal{B})$.

Now, by proposition~\ref{prop}, we have an
isomorphism
$$ \underset{j \in J}{\mbox{colim}} \, X_j \stackrel{\sim}{\longrightarrow}
Q(\mathcal{B})\,,$$
where $J$ is the filtered category of inclusions of strict finite
sub-$I$-cells $X_j$ into $Q(\mathcal{B})$.

For each $j \in J$, consider the fiber product
$$
\xymatrix{
Y_j \ar[d] \ar@{^{(}->}[r] \ar@{}[dr]|{\ulcorner} & X_j \ar[d] \\
p^{-1}(\mathcal{A}) \ar[r] \ar@{^{(}->}[r] & Q(\mathcal{B})\,.
}
$$
In this way, we obtain a morphism of diagrams
$$ \{ Y_j \}_{j \in J} \hookrightarrow \{ X_j \}_{j \in
  J} \,,$$
such that for each $j$ in $J$, the inclusion
$$ Y_j \hookrightarrow X_j $$
belongs to the set $\mathcal{E}$ and $J$ is filtered.

Consider now the diagram $D_I$
$$ {\{0 \leftarrow Y_j \hookrightarrow X_j \}}_{j \in
  J}$$
in the category $\mathsf{Fun}(\ul \times J, \mathsf{dgcat})$.
Now, notice that we have the isomorphism
$$ \underset{j \in J}{\mathsf{colim}} \, \{0 \leftarrow
Y_j \hookrightarrow X_j \} \stackrel{\sim}{\longrightarrow}
\{ 0
 \leftarrow p^{-1}(\mathcal{A}) \hookrightarrow Q(\mathcal{B}) \}$$
in $\mathsf{Fun}(\ul, \mathsf{dgcat})$ and the weak equivalence
$$
\xymatrix{
\{ 0 \ar@{=}[d] & p^{-1}(\mathcal{A}) \ar[l] \ar[d]^{\sim}
\ar@{^{(}->}[r] & Q(\mathcal{B})\} \ar[d]^{\sim} \\
\{ 0 & \mathcal{A} \ar[l] \ar@{^{(}->}[r] & \mathcal{B} \}
}
$$
in $\mathsf{Fun}(\ul, \mathsf{dgcat})$, when endowed with the
projective model structure, see \cite{Hirschhorn}. Since
$\mathsf{Fun}(\ul, \mathsf{dgcat})$ is clearly also compactly
generated, we have an isomorphism
$$ \underset{j \in J}{\mathsf{hocolim}} \, ( 0 \leftarrow
Y_j \rightarrow X_j) \stackrel{\sim}{\longrightarrow} \underset{j \in J}{\mathsf{colim}} \, ( 0 \leftarrow Y_j \rightarrow X_j)\,.$$

Finally, notice that $D_R$ is an object of
$\mathsf{HO}(\mathsf{dgcat})(\ul \times J)$ and that $p_!(D_R)$,
where $p:\ul \times J \rightarrow J$ denotes the projection
functor, identifies with
$$ \underset{i \in J}{\mathsf{hocolim}}\, ( 0 \leftarrow Y_j \rightarrow X_j )\,.$$
This proves the proposition.
\end{proof}

\begin{Notation}
We denote by $\mathcal{E}_{st}$ the set of morhisms $S_L$, where $L$
belongs to the set $\mathcal{E}$.
\end{Notation}

Let $\mathbb{D}$ be a strong triangulated derivator.
\begin{theorem}\label{invert}
If $$G:
\mathsf{St}(\mathsf{L}_{\Sigma,P}\mathsf{Hot}_{\mathsf{dgcat}_f})
\rightarrow \mathbb{D}$$ is a morphism of triangulated derivators
commuting with arbitrary homotopy colimits and such that $G(e)(S_L)$
is invertible for each $L$ in $\mathcal{E}$, then $G(e)(S_K)$ is
invertible for each inclusion $K:\mathcal{A} \hookrightarrow
\mathcal{B}$ of a full dg subcategory.
\end{theorem}

\begin{proof}
Let $\mathcal{A} \stackrel{K}{\hookrightarrow} \mathcal{B}$ be an
inclusion of a full dg subcategory. Consider the morphism
$$ \varphi_K := \varphi(\ulcorner_K):(i_! \circ
\mathcal{U}_T)(\ulcorner_K) \longrightarrow (\mathcal{U}_T \circ
i_!)(\ulcorner_K)$$
in
$\mathsf{St}(\mathsf{L}_{\Sigma,P}\mathsf{Hot}_{\mathsf{dgcat}_f})(\square)$.

Let $D_K$ be the object of
$\mathsf{HO}(\mathsf{dgcat})(\ul \times J)$ constructed in
proposition~\ref{cons}. In particular $p'_!(D_K)
\stackrel{\sim}{\rightarrow} \ul_K$, where $p':\ul \times
J \rightarrow \ul$ denotes the projection functor.

The inclusion $i : \ul \hookrightarrow \square$, induces a
commutative square
$$
\xymatrix{
\mathsf{HO}(\mathsf{dgcat})(\square \times J) \ar[d]^{(i\times
  1)^{\ast}} \ar[rr]^{\mathcal{U}_T(\square \times J)} & &
\mathsf{St}(\mathsf{L}_{\Sigma,P}\mathsf{Hot}_{\mathsf{dgcat}_f})(\square
\times J) \ar[d]^{(i\times 1)^{\ast}} \\
\mathsf{HO}(\mathsf{dgcat})(\ul \times J)
\ar[rr]_{\mathcal{U}_T(\square \times J)} & & \mathsf{St}(\mathsf{L}_{\Sigma,P}\mathsf{Hot}_{\mathsf{dgcat}_f})(\ul
\times J)
}
$$
and a morphism
$$ \Psi : ((i\times 1)_! \circ \mathcal{U}_T(\ul \times J))(D_K)
\longrightarrow (\mathcal{U}_T(\square \times J) \circ (i \times
1)_!)(D_K)\,.$$
We will now show that
$$ p_! \Psi \stackrel{\sim}{\longrightarrow} \varphi_K \,,$$
where $p:\square \times J \rightarrow \square$, denotes the projection
functor.

The fact that we have the following commutative square
$$
\xymatrix{
\square & \square \times J \ar[l]_p \\
\ul \ar[u]^i & \ul \times J \ar[u]_{i\times 1} \ar[l]^{p'}
}
$$
and that the morphism of derivators $\mathcal{U}_T$ commutes with
filtered homotopy colimits implies the following equivalences
$$
\begin{array}{lcl}
p_!\Psi & & \\
= p_! \circ (i\times 1)_! \circ \mathcal{U}_T(\ul \times j)(D_K)
& \rightarrow & p_! \circ \mathcal{U}_T(\square\times J)\circ (i\times
1)_! (D_k) \\
\simeq i_! \circ p'_! \circ \mathcal{U}_T(\ul \times J)(D_k) &
\rightarrow & \mathcal{U}_T(\square \times J) \circ p_! \circ (i\times
1)_! (D_K)\\
\simeq i_! \circ \mathcal{U}_T(\ul) \circ p'_!(D_K) &
\rightarrow & \mathcal{U}_T(\square) \circ i_! \circ p'_!(D_K) \\
\simeq (i_! \circ \mathcal{U}_T(\ul))(\ul_K) & \rightarrow
& (\mathcal{U}_T(\square) \circ i_!)(\ul_K)\\
= \varphi_J
\end{array}
$$
This shows that
$$ p_!(\Psi) \stackrel{\sim}{\longrightarrow} \varphi_K\,.$$
We now show that $\Psi$ is an isomorphism. For this, by conservativity, it
is enough to show that for every object $j:e \rightarrow J$ in $J$,
the morphism
$$ (1\times j)^{\ast}(\Psi)\,,$$
is an isomorphism in
$\mathsf{St}(\mathsf{L}_{\Sigma,P}\mathsf{Hot}_{\mathsf{dgcat}_f})(\square)$.
Recall from proposition~\ref{cons} that $(1\times j)^{\ast} (D_K)$
identifies with
$$ \{ \, 0 \leftarrow Y_j
\stackrel{L_j}{\hookrightarrow} X_j \}\,,$$
where $L_j$ belongs to $\mathcal{E}$. We now show that $(1\times
j)^{\ast}(\Psi)$ identifies with $\varphi_{L_j}$, which by hypotheses,
is an isomorphism.

Now, the following commutative diagram
$$
\xymatrix{
\square \ar[r]^{1\times j} & \square \times J \\
\ul \ar[u]^i \ar[r]_{1\times j} & \ul \times J
\ar[u]_{i\times i}
}
$$
and the dual of proposition $2.8$ in \cite{Cisinski} imply that we
have the following equivalences
$$
\begin{array}{lcl}
(1\times j)^{\ast} \Psi & & \\
= ((1\times j)^{\ast} \circ (i \times 1)_! \circ
\mathcal{U}_T(\ul \times J))(D_K) & \rightarrow &
((1\times j)^{\ast} \circ \mathcal{U}_T(\square \times J) \circ (i
\times 1)_!)(D_K)\\
\simeq (i_! \circ (1\times j)^{\ast} \circ \mathcal{U}_T(\ul
\times J))(D_K) & \rightarrow & (\mathcal{U}_T(\square \times J) \circ
(1\times j)^{\ast} \circ (i \times 1)_!)(D_K)\\
\simeq (i_! \circ \mathcal{U}_T(\ul) \circ (1\times
j)^{\ast})(D_K) & \rightarrow & (\mathcal{U}_T(\square) \circ i_!
\circ (1\times j)^{\ast})(D_K)\\
\simeq i_! \circ \mathcal{U}_T(\ul)(\ul_{L_j}) &
\rightarrow & \mathcal{U}_T(\square) \circ i_!(\ul_{L_j})\\
= \varphi_{L_j} & &
\end{array}
$$
Since by hypotheses $\varphi_{L_j}$ is an isomorphism and the morphism
$G$ commutes with homotopy colimits the theorem is proven.
\end{proof}

\section{The universal localizing invariant}\label{labuniv}

Recall from theorem~\ref{repre} and remark~\ref{Tr} that if we
consider for the category $\mathcal{M}$ the category
$\mathsf{L}_{\Sigma,P}\mathsf{Fun}(\mathsf{dgcat}_f^o,Sset_{\bullet})$,
see example~\ref{ex2}, we have an equivalence of triangulated
derivators
$$\varphi:
\mathsf{St}(\mathsf{L}_{\Sigma,P}\mathsf{Hot}_{\mathsf{dgcat}_f}) \stackrel{\sim}{\longrightarrow}  \mathsf{HO}(\mathsf{Sp}^{\mathbb{N}}(\mathsf{L}_{\Sigma,P}\mathsf{Fun}(\mathsf{dgcat}_f^o,
Sset_{\bullet})))\,.$$
Now, stabilize the set $\mathcal{E}_{st}$ defined in the previous
section under the functor loop space and choose for each
element of this stabilized set a representative in the category $\mathsf{Sp}^{\mathbb{N}}(\mathsf{L}_{\Sigma,P}\mathsf{Fun}(\mathsf{dgcat}_f^o,
Sset_{\bullet}))$. We denote the set of these representatives by
$\widetilde{\mathcal{E}_{st}}$. Since $\mathsf{Sp}^{\mathbb{N}}(\mathsf{L}_{\Sigma,P}\mathsf{Fun}(\mathsf{dgcat}_f^o,
Sset_{\bullet}))$ is a left proper, cellular Quillen model category,
see \cite{Spectra}, its left Bousfield localization by
$\widetilde{\mathcal{E}_{st}}$ exists. We denote it by $\mathsf{L}_{\widetilde{\mathcal{E}_{st}}}\mathsf{Sp}^{\mathbb{N}}(\mathsf{L}_{\Sigma,P}\mathsf{Fun}(\mathsf{dgcat}_f^o,
Sset_{\bullet}))$. By lemma~\ref{lettri} it is a stable Quillen
model category.

\begin{remark}
Since the localization morphism
$$ \gamma:\,
\mathsf{St}(\mathsf{L}_{\Sigma,P}\mathsf{Hot}_{\mathsf{dgcat}_f})
\stackrel{\mathbb{L}Id}{\longrightarrow}
\mathsf{HO}(\mathsf{L}_{\widetilde{\mathcal{E}_{st}}} \mathsf{Sp}^{\mathbb{N}}(\mathsf{L}_{\Sigma,P}\mathsf{Fun}(\mathsf{dgcat}_f^o,
Sset_{\bullet})))$$
commutes with homotopy colimits and inverts the set of morphisms
$\mathcal{E}_{st}$, theorem~\ref{invert} allows us to conclude that it
inverts all morphisms $S_K$ for each inclusion $\mathcal{A}
\hookrightarrow \mathcal{B}$ of a full dg subcategory.
\end{remark}

\begin{definition}\label{defloc}
\begin{itemize}
\item[-] The {\em Localizing motivator of dg categories} $\mathcal{M}_{dg}^{loc}$ is the triangulated
derivator associated with the stable Quillen model category
$$ \mathsf{L}_{\widetilde{\mathcal{E}_{st}}} \mathsf{Sp}^{\mathbb{N}}(\mathsf{L}_{\Sigma,P}\mathsf{Fun}(\mathsf{dgcat}_f^o,
Sset_{\bullet}))\,.$$
\item[-] The {\em Universal localizing invariant of dg categories} is the canonical morphism of
  derivators $$ \mathcal{U}_l : \mathsf{HO}(\mathsf{dgcat}) \rightarrow \mathcal{M}_{dg}^{loc}\,.$$
\end{itemize}
\end{definition}

We sum up the construction of $\mathcal{M}_{dg}^{loc}$ in the following diagram

$$
\xymatrix{
\underline{\mathsf{dgcat}_f}[S^{-1}] \ar[r] \ar[d]_{\mathsf{Ho}(h)} &
\mathsf{HO}(\mathsf{dgcat}) \ar[dl]^{\mathbb{R}\underline{h}} \ar@/^2pc/[ddddl]^{\mathcal{U}_l}
 \\
\mathsf{L}_{\Sigma}\mathsf{Hot}_{\mathsf{dgcat}_f}
\ar[d]_{\Phi}  \ar@<1ex>[ur]^{\mathbb{L}Re}  & \\
\mathsf{L}_{\Sigma,P}\mathsf{Hot}_{\mathsf{dgcat}_f}
\ar[d]_{stab} & \\
\mathsf{St}({\mathsf{L}_{\Sigma,P}\mathsf{Hot}_{\mathsf{dgcat}_f}})
\ar[d]_{\gamma} & \\
\mathcal{M}_{dg}^{loc} &
}
$$
Observe that the morphism of derivators $\mathcal{U}_l$ is pointed, commutes with filtered
homotopy colimits and satisfies the following condition:

\begin{itemize}
\item[Dr)] For every inclusion $\mathcal{A} \stackrel{K}{\hookrightarrow}
  \mathcal{B}$ of a full dg subcategory the canonical morphism
$$ S_K:\, \mathsf{cone}(\mathcal{U}_l(\mathcal{A} \stackrel{K}{\hookrightarrow} \mathcal{B}))
\rightarrow \mathcal{U}_l(\mathcal{B}/\mathcal{A})$$
is invertible in $\mathcal{M}_{dg}^{loc}(e)$.
\end{itemize}

We now give a conceptual characterization of condition $Dr)$. Let $I$
be the category associated with the graph $0 \leftarrow 1$.

\begin{lemma}\label{mono}
The isomorphism classes in $\mathsf{HO}(\mathsf{dgcat})(I)$ associated
with the inclusions $\mathcal{A} \stackrel{K}{\hookrightarrow}
\mathcal{B}$ of full dg subcategories coincide with the classe of
homotopy monomorphims in $\mathsf{dgcat}$, see section $2$ in \cite{Toen}.
\end{lemma}

\begin{proof}
Recall from section $2$ in \cite{Toen} that in a model category
$\mathcal{M}$ a morphism $X \stackrel{f}{\rightarrow} Y$ is a homotopy
monomorphism if for every object $Z$ in $\mathcal{M}$, the induced
morphism of simplicial sets
$$ \mathsf{Map}(Z,X) \stackrel{f_{\ast}}{\rightarrow}
\mathsf{Map}(Z,Y)$$
induces an injection on $\pi_0$ and isomorphisms on all $\pi_i$ for
$i>0$ (for all base points).

Now, by lemma $2.4$ of \cite{Toen} a dg functor $\mathcal{A}
\stackrel{F}{\rightarrow} \mathcal{B}$ is an homotopy monomorphism on
the quasi-equivalent Quillen model category on $\mathsf{dgcat}$ if and
only if it is quasi-fully faithful, i.e. for any two objects $X$ and
$Y$ in $\mathcal{A}$ the morphism of complexes
$\mathsf{Hom}_{\mathcal{A}}(X,Y) \rightarrow
\mathsf{Hom}_{\mathcal{B}}(FX,FY)$ is a quasi-isomorphism.

Recall that by remark~\ref{monoi} the mapping space functor $\mathsf{Map}(\mathcal{A},\mathcal{B})$ in the
Morita Quillen model category identifies with the mapping space
$\mathsf{Map}(\mathcal{A},\mathcal{B}_{f})$ in the quasi-equivalent
Quillen model category, where $\mathcal{B}_f$ denotes a Morita fibrant
resolution of $\mathcal{B}$.
This implies that a dg functor $\mathcal{A} \stackrel{F}{\rightarrow}
\mathcal{B}$ is a homotopy monomorphism if and only if $\mathcal{A}_f
\stackrel{F_f}{\rightarrow} \mathcal{B}_f$ is a quasi-fully faithful
dg functor.

Now, notice that an inclusion $\mathcal{A} \hookrightarrow
\mathcal{B}$ of a full dg subcategory is a homotopy
monomorphism. Conversely, let $\mathcal{A} \stackrel{F}{\rightarrow}
\mathcal{B}$ be a homotopy monomorphism. Consider the diagram
$$
\xymatrix{
\widetilde{\mathcal{A}_f} \ar@{^{(}->}[r] & \mathcal{B}_f \ar@{=}[d]  \\
\mathcal{A}_f \ar[u]^{\pi} \ar[r]^{F_f} & \mathcal{B}_f \\
\mathcal{A} \ar[u]^{\sim} \ar[r]_F & \mathcal{B} \ar[u]_{\sim} \,,
}
$$
where $\widetilde{\mathcal{A}_f}$ denotes the full dg subcategory of
$\mathcal{B}_f$ whose objects are those in the image by the dg
functor $F_f$. Since $F_f$ is a quasi-fully faithful dg functor, the dg
functor $\pi$ is a quasi-equivalence.
This proves the lemma.
\end{proof}

\begin{remark}
Lemma~\ref{mono} shows that condition $Dr)$ is equivalent to
\begin{itemize}
\item[Dr')] For every homotopy monomorphism $\mathcal{A}
  \stackrel{F}{\rightarrow} \mathcal{B}$ in
  $\mathsf{HO}(\mathsf{dgcat})(I)$ the canonical morphism
$$ \mathsf{cone}(\mathcal{U}_l (\mathcal{A} \stackrel{F}{\rightarrow}
\mathcal{B})) \rightarrow \mathcal{U}_l (\mathsf{cone}(F))$$
is invertible in $\mathcal{M}_{dg}^{loc}(e)$.
\end{itemize}
\end{remark}

Let $\mathbb{D}$ be a strong triangulated derivator.
\begin{theorem}\label{principal}
The morphism $\mathcal{U}_l$ induces an equivalence of categories
$$
\underline{\mathsf{Hom}}_!(\mathcal{M}_{dg}^{loc},
\mathbb{D}) \stackrel{\mathcal{U}_l^{\ast}}{\longrightarrow}
\underline{\mathsf{Hom}}_{flt,\,Dr,\,p}(\mathsf{HO}(\mathsf{dgcat}),\mathbb{D})\,,$$
where
$\underline{\mathsf{Hom}}_{flt,\,Dr\,, p}(\mathsf{HO}(\mathsf{dgcat}),\mathbb{D})$
denotes the category of morphisms of derivators which commute with filtered
homotopy colimits, satisfy condition Dr) and preserve the point.
\end{theorem}

\begin{proof}
By theorem~\ref{Cisinsk}, we have the following equivalence of
categories
$$ \underline{\mathsf{Hom}}_!(\mathcal{M}_{dg}^{loc}, \mathbb{D}) \stackrel{\gamma^{\ast}}{\longrightarrow}
\underline{\mathsf{Hom}}_{!,\widetilde{\mathcal{E}_{st}}}( \mathsf{St}(\mathsf{L}_{\Sigma,P}\mathsf{Fun}(\mathsf{dgcat}_f^o,
Sset_{\bullet})), \mathbb{D})\,.$$
We now show that we have the following equivalence of categories
$$ \underline{\mathsf{Hom}}_{!,\widetilde{\mathcal{E}_{st}}}(\mathsf{St}(\mathsf{L}_{\Sigma,P}\mathsf{Fun}(\mathsf{dgcat}_f^o,
Sset_{\bullet})), \mathbb{D}) \stackrel{\sim}{\rightarrow} \underline{\mathsf{Hom}}_{!,\mathcal{E}_{st}}(\mathsf{St}(\mathsf{L}_{\Sigma,P}\mathsf{Fun}(\mathsf{dgcat}_f^o,
Sset_{\bullet})), \mathbb{D})\,.$$
Let $G$ be an element of  $\underline{\mathsf{Hom}}_{!,\mathcal{E}_{st}}(\mathsf{St}(\mathsf{L}_{\Sigma,P}\mathsf{Fun}(\mathsf{dgcat}_f^o,
Sset_{\bullet})), \mathbb{D})$ and $s$ an element of
$\mathcal{E}_{st}$. We show that the image of $s$ under the functor
$G(e) \circ \Omega(e)$ is an isomorphism in $\mathbb{D}(e)$. Recall from the proof of lemma~\ref{lettri} that the functor $G(e)$
commutes with $\Sigma(e)$. Since the suspension and loop space
functors in $\mathbb{D}(e)$ are inverse of each other we conclude that
the image of $s$ under the functor $G(e) \circ \Omega(e)$ is an
isomorphism in $\mathbb{D}(e)$.
Now, simply observe that the category on the right hand side of the
above equivalence identifies with
$\underline{\mathsf{Hom}}_{flt,\,Dr\,,
  p}(\mathsf{HO}(\mathsf{dgcat}),\mathbb{D})$ under the equivalence
$$
\underline{\mathsf{Hom}}_!(\mathsf{St}(\mathsf{L}_{\Sigma,P} \mathsf{Hot}_{\mathsf{dgcat}_f}),\mathbb{D})
\stackrel{(stab\circ \Phi \circ
  \mathbb{R}\underline{h})^{\ast}}{\longrightarrow} \underline{\mathsf{Hom}}_{flt,p}(\mathsf{HO}(\mathsf{dgcat}),\mathbb{D})\,,$$
of proposition~\ref{Trin}.

This proves the theorem.
\end{proof}
\begin{terminology}
We call an object of the right hand side category of
theorem~\ref{principal} a {\em localizing invariant of dg categories}.
\end{terminology}
We now present some examples.

\subsection*{Hochschild and cyclic homology}
Let $\mathcal{A}$ be a small $k$-flat $k$-category. The {\it
  Hochschild chain complex} of $\mathcal{A}$ is the complex
concentrated in homological degrees $p\geq 0$ whose $p$th component is
the sum of the
$$ \mathcal{A}(X_p,X_0) \otimes \mathcal{A}(X_p,X_{p-1})\otimes
\mathcal{A}(X_{p-1}, X_{p-2})\otimes \cdots \otimes
\mathcal{A}(X_0,X_1)\,,$$
where $X_0, \ldots, X_p$ range through the objects of $\mathcal{A}$,
endowed with the differential
$$ d(f_p \otimes \ldots \otimes f_0) = f_{p-1} \otimes \cdots \otimes
f_0 f_p + \sum_{i=1}^p (-1)^i f_p\otimes \cdots \otimes f_i f_{i-1} \otimes
\cdots \otimes f_0\,.$$
Via the cyclic permutations
$$ t_p(f_{p-1} \otimes \cdots \otimes f_0) = (-1)^pf_0\otimes f_{p-1}
\otimes \cdots \otimes f_1$$
this complex becomes a precyclic chain complex and thus gives rise to
a {\it mixed complex} $C(\mathcal{A})$, i.e. a dg module over the dg
algebra $\Lambda = k[B]/(B^2)$, where $B$ is of degree $-1$ and
$dB=0$. All variants of cyclic homology only depend on
$C(\mathcal{A})$ considered in $\mathcal{D}(\Lambda)$. For example,
the cyclic homology of $\mathcal{A}$ is the homology of the complex
$C(\mathcal{A})\overset{\mathbb{L}}{\otimes}_{\Lambda}k$, \emph{cf.}~\cite{Kassel}.

If $\mathcal{A}$ is a $k$-flat differential graded category, its mixed
complex is the sum-total complex of the bicomplex obtained as the
natural re-interpretation of the above complex. If $\mathcal{A}$ is an
arbitrary dg $k$-category, its Hochschild chain complex is defined as
the one of a $k$-flat (e.g. a cofibrant) resolution of $\mathcal{A}$.
The following theorem is proved in \cite{cyclichomology}.

\begin{theorem}
 The map $\mathcal{A} \mapsto C(\mathcal{A})$ yields a
  morphism of derivators
 $$\mathsf{HO}(\mathsf{dgcat}) \rightarrow
  \mathsf{HO}(\Lambda-Mod)\,,$$
 which commutes with filtered homotopy colimits, preserves the point
 and satisfies condition Dr).
\end{theorem}

\begin{remark}
By theorem~\ref{principal} the morphism of derivators $C$ factors
through $\mathcal{U}_l$ and so gives rise to a morphism
$$ C: \mathcal{M}_{dg}^{loc} \rightarrow \mathsf{HO}(\Lambda-Mod)\,.$$
\end{remark}

\subsection*{Non-connective $K$-theory}
Let $\mathcal{A}$ be a small dg category. Its non-connective $K$-theory
spectrum $K(A)$ is defined by applying Schlichting's construction
\cite{Marco} to the Frobenius pair associated with the category of
cofibrant perfect $\mathcal{A}$-modules (to the empty dg category we
associate $0$).
Recall that the conflations in the Frobenius category of cofibrant
perfect $\mathcal{A}$-modules are the short exact sequences which
split in the category of graded $\mathcal{A}$-modules.

\begin{theorem}
The map $\mathcal{A} \mapsto K(\mathcal{A})$ yields a
  morphism of derivators
 $$\mathsf{HO}(\mathsf{dgcat}) \rightarrow \mathsf{HO}(Spt)\,,$$
  to the derivator associated with the category of spectra, which
  commutes with filtered homotopy
  colimits, preserves the point and satisfies condition Dr).
\end{theorem}

\begin{proof}
Proposition $11.15$ in \cite{Marco}, which an adaption of theorem
$1.9.8$ in \cite{Thomason}, implies that we have a well defined
morphism of derivators
$$ \mathsf{HO}(\mathsf{dgcat}) \rightarrow \mathsf{HO}(Spt)\,.$$
Lemma $6.3$ in \cite{Marco} implies that this morphism commutes with
filtered homotopy colimits and theorem $11.10$ implies that condition
Dr) is satisfied.
\end{proof}

\begin{remark}
By theorem~\ref{principal}, the morphism of derivators $K$ factors
through $\mathcal{U}_l$ and so gives rise to a morphism
$$ K :\mathcal{M}_{dg}^{loc} \rightarrow \mathsf{HO}(Spt)\,.$$
\end{remark}

We now establish a connection between Waldhausen's
$S_{\bullet}$-construction, see~\cite{Waldhausen} and the suspension functor
in the triangulated category $\mathcal{M}_{dg}^{loc}(e)$. Let $\mathcal{A}$
be a Morita fibrant dg category, see proposition~\ref{nova4}. Notice that $\mathsf{Z}^0(\mathcal{A})$ carries a natural
exact category structure obtained by pulling back the graded-split structure on $\mathcal{C}_{dg}(\mathcal{A})$ along the Yoneda functor
$$
\begin{array}{rcl}
h: \mathsf{Z}^0(\mathcal{A}) & \longrightarrow &
\mathcal{C}_{dg}(\mathcal{A})\\
A & \mapsto & \mathsf{Hom}^{\bullet}(?,A)\,.
\end{array}
$$

\begin{Notation}\label{notacao}
Remark that the simplicial category $S_{\bullet}\mathcal{A}$, obtained
by applying Waldhausen's $S_{\bullet}$-construction, see
\cite{Waldhausen} to $\mathsf{Z}^0(\mathcal{A})$, admits a natural
enrichissement over the complexes. We denote by
$S_{\bullet}\mathcal{A}$ this simplicial Morita fibrant dg category obtained.
\end{Notation}

We denote by $S_{\bullet}\mathcal{A}$ the Waldhausen
$S_{\bullet}$-construction of
$\mathcal{A}$, see \cite{Waldhausen}. Observe that we obtain in this way a simplicial Morita
fibrant dg category. We denote by $\Delta$ the simplicial category,
see~\cite{MacCarthy}, and by $p:\Delta \rightarrow e$ the projection
functor.

\begin{proposition}\label{real}
There is a canonical isomorphism in $\mathcal{M}_{dg}^{loc}(e)$
$$ p_!\mathcal{U}_l(S_{\bullet}\mathcal{A})
\stackrel{\sim}{\rightarrow} \mathcal{U}_l(\mathcal{A})[1]\,.$$
\end{proposition}

\begin{proof}
As in \cite[3.3]{MacCarthy}, we consider the sequence in $\mathsf{HO}(\mathsf{dgcat})(\Delta)$
$$ 0 \rightarrow \mathcal{A}_{\bullet} \rightarrow
PS_{\bullet}\mathcal{A} \rightarrow S_{\bullet}\mathcal{A} \rightarrow
0\,,$$
where $\mathcal{A}_{\bullet}$ denotes the constant simplicial dg
category with value $\mathcal{A}$ and $PS_{\bullet}\mathcal{A}$ the
path object of $S_{\bullet}\mathcal{A}$. For each
point $n: e \rightarrow \Delta$, the $n$th component of the above
sequence is the following short exact sequence in
$\mathsf{Ho}(\mathsf{dgcat})$
$$ 0 \rightarrow \mathcal{A} \stackrel{I}{\rightarrow}
PS_n\mathcal{A}=S_{n+1}\mathcal{A} \stackrel{Q}{\rightarrow}
S_n\mathcal{A} \rightarrow 0\,,$$
where $I$ maps $A \in \mathcal{A}$ to the constant sequence
$$
0 \rightarrow A \stackrel{Id}{\rightarrow} A \stackrel{Id}{\rightarrow}
\cdots \stackrel{Id}{\rightarrow} A $$
and $Q$ maps a sequence
$$ 0 \rightarrow A_0 \rightarrow A_1 \rightarrow \cdots \rightarrow
A_n$$
to
$$ A_1/A_0 \rightarrow \cdots \rightarrow A_n/A_0\,.$$
Since the morphism of derivators $\mathcal{U}_l$ satisfies condition
Dr), the conservativity axiom implies that we obtain a triangle
$$ \mathcal{U}_l(\mathcal{A}_{\bullet}) \rightarrow
\mathcal{U}_l(PS_{\bullet}) \rightarrow \mathcal{U}_l(S_{\bullet})
\rightarrow \mathcal{U}_l(\mathcal{A}_{\bullet})[1]$$
in $\mathcal{M}_{dg}^{loc}(\Delta)$. By applying the functor $p_!$, we obtain
the following triangle
$$ p_! \mathcal{U}_l(\mathcal{A}_{\bullet}) \rightarrow
p_!\mathcal{U}_l(PS_{\bullet}\mathcal{A}) \rightarrow p_!\mathcal{U}_l(S_{\bullet}\mathcal{A}) \rightarrow p_!\mathcal{U}_l(\mathcal{A}_{\bullet})[1]$$
in $\mathcal{M}_{dg}^{loc}(e)$.
We now show that we have natural isomorphisms
$$ p_!\mathcal{U}_l(\mathcal{A}_{\bullet})
\stackrel{\sim}{\rightarrow} \mathcal{U}_l(\mathcal{A})$$
and
$$ p_! \mathcal{U}_l(PS_{\bullet}\mathcal{A})
\stackrel{\sim}{\rightarrow} 0\,,$$
in $\mathcal{M}_{dg}^{loc}(e)$, where $0$ denotes the zero object in the
  triangulated category $\mathcal{M}_{dg}^{loc}(e)$. This clearly implies
  the proposition.
Since the morphisms of derivators $\Phi$, $stab$ and $\gamma$ commute
with homotopy colimits it is enough to show that we have isomorphisms
$$ p_! \mathbb{R}\underline{h}(\mathcal{A}_{\bullet})
\stackrel{\sim}{\rightarrow} \mathbb{R}\underline{h}(\mathcal{A})$$
and
$$ p_!\mathbb{R}\underline{h}(PS_{\bullet}\mathcal{A})
\stackrel{\sim}{\rightarrow} \ast$$
in $\mathsf{Hot}_{\mathsf{dgcat}_f}(e)$, where $\ast$ denotes the terminal
object in $\mathsf{Hot}_{\mathsf{dgcat}_f}(e)$. Notice that since
$\mathcal{A}$ and $PS_n\mathcal{A}$, $n \geq 0$ are Morita fibrant dg
categories, we have natural isomorphisms
$$ \underline{h}(\mathcal{A}_{\bullet})
\stackrel{\sim}{\rightarrow}\mathbb{R}\underline{h}(\mathcal{A}_{\bullet})$$
and
$$ \underline{h}(PS_{\bullet}\mathcal{A}) \stackrel{\sim}{\rightarrow}
\mathbb{R}\underline{h}(PS_{\bullet}\mathcal{A})$$
in $\mathsf{Hot}_{\mathsf{dgcat}_f}(\Delta)$.

Now, since homotopy colimits in $\mathsf{Fun}(\mathsf{dgcat}_f,Sset)$
are calculated objectwise and since $h(\mathcal{A}_{\bullet})$ is a
constant simplicial object in $\mathsf{Fun}(\mathsf{dgcat}_f,Sset)$, corollary $18.7.7$ in \cite{Hirschhorn} implies that we have an
isomorphism
$$ p_!\mathbb{R}\underline{h}(\mathcal{A}_{\bullet})
\stackrel{\sim}{\rightarrow} \mathbb{R}\underline{h}(\mathcal{A})$$
in $\mathcal{M}_{dg}^{loc}(e)$.

Notice also that since $PS_{\bullet}\mathcal{A}$ is a contractible
simplicial object, see \cite{MacCarthy}, so is
$\underline{h}(PS_{\bullet}\mathcal{A})$. Since homotopy colimits in
$\mathsf{Fun}(\mathsf{dgcat}_f,Sset)$ are calculated objectwise, we
have an isomorphism
$$ p_!\mathbb{R}\underline{h}(PS_{\bullet}\mathcal{A})
\stackrel{\sim}{\rightarrow} \ast$$
in $\mathsf{Hot}_{\mathsf{dgcat}_f}(e)$.

This proves the proposition.

\end{proof}

\section{A Quillen model in terms of presheaves of spectra}\label{chapspectra}
In this section, we construct another Quillen model category whose
associated derivator is $\mathcal{M}_{dg}^{loc}$.

Consider the Quillen adjunction
$$
\xymatrix{
\mathsf{Fun}(\mathsf{dgcat}_f^o,Sset_{\bullet})\ar@<-1ex>[d]_{\Sigma^{\infty}}\\
\mathsf{Sp}^{\mathbb{N}}(\mathsf{Fun}(\mathsf{dgcat}_f^o,
Sset_{\bullet})) \ar@<-1ex>[u]_{ev_0}\,.
}
$$
Recall from section~\ref{small} that we have a set of morphisms
$(\Sigma \cup \{P \})_+$ in the category
$\mathsf{Fun}(\mathsf{dgcat}_f^o, Sset_{\bullet})$. Now stabilize the
image of this set by the derived functor $\mathbb{L}\Sigma^{\infty}$, under the functor loop space in
$\mathsf{Ho}(\mathsf{Sp}^{\mathbb{N}}(\mathsf{Fun}(\mathsf{dgcat}^o_f,Sset_{\bullet})))$.
For each one of the morphims thus obtained, choose a representative
in the model category
$\mathsf{Sp}^{\mathbb{N}}(\mathsf{Fun}(\mathsf{dgcat}^o_f,Sset_{\bullet}))$.
\begin{Notation}
Let us denote this set by $G$ and by
$\mathsf{L}_G\mathsf{Sp}^{\mathbb{N}}(\mathsf{Fun}(\mathsf{dgcat}_f^o,Sset_{\bullet}))$
the associated left Bousfield localization.
\end{Notation}

\begin{proposition}
We have an equivalence of triangulated strong derivators
$$
\mathsf{HO}(\mathsf{Sp}^{\mathbb{N}}(\mathsf{L}_{\Sigma,P}\mathsf{Fun}(\mathcal{M}_f^o,
Sset_{\bullet}))) \stackrel{\sim}{\longrightarrow}
\mathsf{HO}(\mathsf{L}_G \mathsf{Sp}^{\mathbb{N}}(\mathsf{Fun}(\mathcal{M}_f^o,Sset_{\bullet})))\,.$$
\end{proposition}

\begin{proof}

Observe that theorems \ref{Cisinsk} and \ref{repre} imply that both
derivators have the same universal property.

This proves the proposition.
\end{proof}

\begin{remark}

Notice that the stable Quillen model category
$$
\mathsf{Sp}^{\mathbb{N}}(\mathsf{Fun}(\mathcal{M}_f^o,Sset_{\bullet}))$$
identifies with
$$\mathsf{Fun}(\mathcal{M}_f^o,\mathsf{Sp}^{\mathbb{N}}(Sset_{\bullet}))$$
endowed with the projective model structure.
\end{remark}
The above considerations imply the following proposition.

\begin{proposition}\label{prespectres}
We have an equivalence of derivators
$$ \mathsf{HO}(\mathsf{L}_{\widetilde{\mathcal{E}_{st}},G}\mathsf{Fun}(\mathsf{dgcat}_f^o,\mathsf{Sp}^{\mathbb{N}}(Sset_{\bullet}))) \stackrel{\sim}{\longrightarrow} \mathcal{M}_{dg}^{loc} \,.$$
\end{proposition}

\section{Upper  triangular  DG categories}\label{trimat}

In this section we study upper triangular dg categories using the
formalism of Quillen's homotopical algebra. In the next section, we
will relate this important class of dg categories with split short
exact sequences in $\mathsf{Ho}(\mathsf{dgcat})$.

\begin{definition}
An {\em upper triangular} dg category $\underline{\mathcal{B}}$ is
given by an upper triangular matrix
$$
\begin{array}{rcl}
\underline{\mathcal{B}} & := & \begin{pmatrix} \mathcal{A} & X \\ 0 &
  \mathcal{C} \end{pmatrix}\,,
\end{array}
$$
where $\mathcal{A}$ and $\mathcal{C}$ are small dg categories and $X$
is a $\mathcal{A}$-$\mathcal{C}$-bimodule.

A morphism $\underline{F} : \underline{\mathcal{B}} \rightarrow
\underline{\mathcal{B}'}$ of upper triangular dg categories is given by a triple
$\underline{F}:=(F_{\mathcal{A}}, F_{\mathcal{C}}, F_X)$, where
$F_{\mathcal{A}}$, resp. $F_{\mathcal{C}}$, is a dg functor from
  $\mathcal{A}$ to $\mathcal{A}'$, resp. from $\mathcal{C}$ to
  $\mathcal{C}'$, and $F_X$ is a morphism of
  $\mathcal{A}$-$\mathcal{C}$-bimodules from $X$ to $X'$ (we consider
  $X'$ endowed with
  the action induced by $F_{\mathcal{A}}$ and $F_{\mathcal{C}}$). The
  composition is the natural one.
\end{definition}

\begin{Notation}
We denote by $\mathsf{dgcat}^{tr}$ the category of upper triangular dg categories.
\end{Notation}

Let $\underline{\mathcal{B}} \in \mathsf{dgcat}^{tr}$.
\begin{definition}
Let $|\underline{\mathcal{B}}|$ be the {\em totalization} of
$\underline{\mathcal{B}}$, i.e. the small dg category whose set of
objects is the disjoint union of the set of objects of $\mathcal{A}$
and $\mathcal{C}$ and whose morphisms are given by
$$
\begin{array}{rcl}
\mathsf{Hom}_{|\underline{\mathcal{B}}|}(x,x') & := &
\left\{
\begin{array}{ccl}
\mathsf{Hom}_{\mathcal{A}}(x,x') & \mbox{if} & x,\, x' \in \mathcal{A} \\
\mathsf{Hom}_{\mathcal{C}}(x,x') & \mbox{if} & x,\, x' \in \mathcal{C} \\
X(x,x') & \mbox{if} & x \in \mathcal{A},\, x' \in \mathcal{C} \\
0 & \mbox{if} &  x \in \mathcal{C},\, x' \in \mathcal{A}
\end{array}\right.\,.
\end{array}
$$
\end{definition}
We have the following adjunction
$$
\xymatrix{
\mathsf{dgcat}^{tr} \ar@<-1ex>[d]_{|-|} \\
\mathsf{dgcat} \ar@<-1ex>[u]_I \,,
}
$$
where
$$
\begin{array}{rcl}
I(\mathcal{B}') & := & \begin{pmatrix} \mathcal{B}' & \mathsf{Hom}_{\mathcal{B}'}(-,-) \\
0 & \mathcal{B}' \end{pmatrix}\,.
\end{array}
$$
\begin{lemma}\label{triancomp}
The category $\mathsf{dgcat}^{tr}$ is complete and cocomplete.
\end{lemma}
\begin{proof}
Let $\{\underline{\mathcal{B}_j}\}_{j \in J}$ be a diagram in
$\mathsf{dgcat}^{tr}$. Observe that the upper triangular dg category
$$
\begin{pmatrix}
\underset{j \in J}{\mathsf{colim}\,\mathcal{A}_j} & \underset{j \in
  J}{\mathsf{colim}}\,|\underline{\mathcal{B}_j}|(-,-)  \\
0 & \underset{j \in J}{\mathsf{colim}}\,\mathcal{C}_j
\end{pmatrix}\,,
$$
where $\underset{j \in
  J}{\mathsf{colim}}\,|\underline{\mathcal{B}_j}|(-,-)$ is the $\underset{j \in
  J}{\mathsf{colim}}\,\mathcal{A}_j$-$\underset{j \in
  J}{\mathsf{colim}}\,\mathcal{C}_j$-bimodule naturally associated
with the dg category $\underset{j \in
  J}{\mathsf{colim}}\,|\underline{\mathcal{B}_j}|$, corresponds to $\underset{j \in
  J}{\mathsf{colim}}\,\underline{\mathcal{B}_j}$. Observe also that the
upper triangular dg category
$$
\begin{pmatrix}
\underset{j \in J}{\mathsf{lim}}\, \mathcal{A}_j & \underset{j \in J}{\mathsf{lim}}\,X_j \\
0 & \underset{j \in J}{\mathsf{lim}}\,\mathcal{C}_j
\end{pmatrix}\,,
$$
corresponds to $\underset{j \in
  J}{\mathsf{lim}}\,\underline{\mathcal{B}_j}$.
This proves the lemma.
\end{proof}

\begin{Notation}
Let $p_1(\underline{\mathcal{B}}) := \mathcal{A}$ and $p_2(\underline{\mathcal{B}}) := \mathcal{C}$.
\end{Notation}
We have at our disposal the following adjunction
$$
\xymatrix{
\mathsf{dgcat}^{tr} \ar@<1ex>[d]^{p_1 \times p_2} \\
\mathsf{dgcat}\times\mathsf{dgcat} \ar@<1ex>[u]^E \,,
}
$$
where
$$
\begin{array}{rcl}
E(\mathcal{B}', \mathcal{B}'') & := & \begin{pmatrix} \mathcal{B}' & 0 \\
0 & \mathcal{B}'' \end{pmatrix}\,.
\end{array}
$$
Recall from theorem~\ref{theorem2} that $\mathsf{dgcat}$ admits a structure of cofibrantly
generated Quillen model category whose weak equivalences are the
Morita dg functors. This structure clearly induces a componentwise
model structure on $\mathsf{dgcat} \times \mathsf{dgcat}$ which is
also cofibrantly generated.

\begin{proposition}\label{prop7}
The category $\mathsf{dgcat}^{tr}$ admits a structure of cofibrantly generated
Quillen model category whose weak equivalences, resp. fibration, are
the morphisms $\underline{F}: \underline{\mathcal{B}} \rightarrow
\underline{\mathcal{B}'}$ such that $(p_1 \times p_2) (\underline{F})$
are weak equivalences, resp. fibrations, in $\mathsf{dgcat}\times \mathsf{dgcat}$.
\end{proposition}

\begin{proof}
We show that
the previous adjunction $(E,p_1 \times p_2)$ verifies conditions $(1)$ and $(2)$ of theorem
$11.3.2$ from \cite{Hirschhorn}.
\begin{itemize}
\item[(1)] Since the functor $E$ is also a right adjoint to $p_1
  \times p_2$, the functor $p_1 \times p_2$ commutes with colimits and
  so condition $(1)$ is verified.
\item[(2)] Let $J$, resp. $J \times J$, be the set of generating trivial
cofibrations in $\mathsf{dgcat}$, resp. in $\mathsf{dgcat} \times \mathsf{dgcat}$. Since the
functor $p_1 \times p_2$ commutes with filtered colimits it is enough
to prove the following: let $\underline{G}: \underline{\mathcal{B}'}
\rightarrow \underline{\mathcal{B}''}$ be an element of the set
$E(J \times J)$, $\underline{\mathcal{B}}$ an object in
$\mathsf{dgcat}^{tr}$ and $\underline{\mathcal{B}'} \rightarrow
\underline{\mathcal{B}}$ a morphism in $\mathsf{dgcat}^{tr}$. Consider
the following push-out in $\mathsf{dgcat}^{tr}$~:
$$
\xymatrix{
\underline{\mathcal{B}'} \ar[r] \ar[d]_{\underline{G}}
\ar@{}[dr]|{\lrcorner} &
\underline{\mathcal{B}} \ar[d]^{\underline{G}_{\ast}} \\
\underline{\mathcal{B}''} \ar[r] &
\underline{\mathcal{B}''}\underset{\underline{\mathcal{B}'}}{\coprod}\underline{\mathcal{B}}\,.
}
$$
We now prove that $(p_1 \times p_2)(\underline{G}_{\ast})$ is a
weak-equivalence in $\mathsf{dgcat}\times \mathsf{dgcat}$. Observe
that the image of the previous push-out under the functors $p_1$ and
$p_2$ correspond to the following two push-outs in $\mathsf{dgcat}$~:
$$
\xymatrix{
*+<1pc>{\mathcal{A}'} \ar[r] \ar@{>->}[d]_{G_{\mathcal{A}'}}^{\sim} \ar@{}[dr]|{\lrcorner}  & \mathcal{A}
  \ar[d]^{G_{\mathcal{A}'_{\ast}}}  &  *+<1pc>{\mathcal{C}'} \ar[r]
  \ar@{>->}[d]_{G_{\mathcal{C}'}}^{\sim}  \ar@{}[dr]|{\lrcorner}  & \mathcal{C}
  \ar[d]^{G_{\mathcal{C}'_{\ast}}} \\
\mathcal{A}'' \ar[r] & \mathcal{A}'' \underset{\mathcal{A}'}{\coprod}
\mathcal{A} & \mathcal{C}'' \ar[r] & \mathcal{C}'' \underset{\mathcal{C}'}{\coprod}
\mathcal{C}\,.
}
$$
Since $G_{\mathcal{A}'_{\ast}}$ and $G_{\mathcal{C}'_{\ast}}$
belong to $J$ the morphism
$$(p_1\times p_2)(\underline{G}_{\ast})=(G_{\mathcal{A}'_{\ast}},
G_{\mathcal{C}'_{\ast}})$$
is a weak-equivalence in $\mathsf{dgcat} \times \mathsf{dgcat}$. This proves condition $(2)$.
\end{itemize}
The proposition is proven.
\end{proof}

Let $\underline{\mathcal{B}}$, $\underline{\mathcal{B}'} \in \mathsf{dgcat}^{tr}$.

\begin{definition}
A morphism $\underline{F}: \underline{\mathcal{B}} \rightarrow
\underline{\mathcal{B}'}$ is a {\em total Morita dg functor} if
$F_{\mathcal{A}}$ and $F_{\mathcal{C}}$ are Morita dg functors, see section~\ref{secmor}, and $F_X$
    is a quasi-isomorphism of $\mathcal{A}$-$\mathcal{C}$-bimodules.
\end{definition}

\begin{remark}
Notice that if $\underline{F}$ is a total Morita dg functor then
$|\underline{F}|$ is a Morita dg functor in $\mathsf{dgcat}$ but the converse is not true.
\end{remark}

\begin{theorem}\label{theo1}
The category $\mathsf{dgcat}^{tr}$ admits a structure of cofibrantly generated
Quillen model category whose weak equivalences $\mathcal{W}$ are the
total Morita dg functors and whose fibrations are the morphisms
$\underline{F}:\underline{\mathcal{B}} \rightarrow
\underline{\mathcal{B}'}$ such that $F_{\mathcal{A}}$ and
  $F_{\mathcal{C}}$ are Morita fibrations, see theorem~\ref{theorem2}, and $F_X$ is a componentwise surjective morphism of bimodules.
\end{theorem}

\begin{proof}
The proof is based on enlarging the set $E(I \times I)$,
resp. $E(J \times J)$, of
generating cofibrations, resp. generating trivial cofibrations, of the
Quillen model structure of proposition~\ref{prop7}.

Let $\tilde{I}$ be the set of morphisms in $\mathsf{dgcat}^{tr}$
$$
\begin{array}{rcl}
\begin{pmatrix} k & S^{n-1} \\ 0 & k \end{pmatrix} & \hookrightarrow &
\begin{pmatrix} k & D^{n} \\ 0 & k \end{pmatrix}\,, n \in \mathbb{Z}\,,
\end{array}
$$
where $S^{n-1}$ is the complex $k[n-1]$ and $D^n$ the mapping cone on
the identity of $S^{n-1}$. The $k$-$k$-bimodule $S^{n-1}$ is sent to
$D^n$ by the identity on $k$ in degree $n-1$.

Consider also the set $\tilde{J}$ of morphisms in $\mathsf{dgcat}^{tr}$
$$
\begin{array}{rcl}
\begin{pmatrix} k & 0 \\ 0 & k \end{pmatrix} & \hookrightarrow &
\begin{pmatrix} k & D^{n} \\ 0 & k \end{pmatrix}\,, n \in \mathbb{Z}\,.
\end{array}
$$
Observe that a morphism $\underline{F} :\underline{\mathcal{B}}
\rightarrow \underline{\mathcal{B}'}$ in $\mathsf{dgcat}^{tr}$ has the
right lifting property (=R.L.P.) with respect to the set $\tilde{J}$,
resp. $\tilde{I}$, if and only if $F_X$ is a componentwise surjective morphism,
resp. surjective quasi-isomorphism, of $\mathcal{A}$-$\mathcal{C}$-bimodules.

Define $I:= E(I\times I) \cup \tilde{I}$ as the set of {\em generating
cofibrations} in $\mathsf{dgcat}^{tr}$ and $J:=E(J\times J)\cup
\tilde{J}$ as the set of {\em generating trivial cofibrations}. We now prove
that conditions $(1)$-$(6)$ of theorem $2.1.19$ from \cite{Hovey} are
satisfied. This is clearly the case for conditions $(1)$-$(3)$.
\begin{itemize}

\item[(4)] We now prove that $J-cell \subset \mathcal{W}$, see
\cite{Hirschhorn}. Since by proposition~\ref{prop7} we have $E(J\times J)-cell
\subset \mathcal{W}'$, where $\mathcal{W}'$ denotes the weak equivalences of
proposition~\ref{prop7}, it is enough to prove that pushouts with respect
to any morphism in $\tilde{J}$ belong to $\mathcal{W}$.
Let $n$ be an integer and $\underline{\mathcal{B}}$ an object in
$\mathsf{dgcat}^{tr}$. Consider the following push-out in $\mathsf{dgcat}^{tr}$~:
$$
\xymatrix{
{\begin{pmatrix} k & 0 \\ 0 & k \end{pmatrix}} \ar[r]^{\underline{T}}
\ar@{^{(}->}[d] \ar@{}[dr]|{\lrcorner}
& \underline{\mathcal{B}} \ar[d]^{\underline{R}} \\
{\begin{pmatrix} k & D^{n} \\ 0 & k \end{pmatrix}} \ar[r] & \underline{\mathcal{B}'}\,.
}
$$
Notice that the morphism $\underline{T}$ corresponds to specifying
an object $A$ in $\mathcal{A}$ and an object $C$ in $\mathcal{C}$. The
upper triangular dg category $\underline{\mathcal{B}'}$ is then
obtained from $\underline{\mathcal{B}}$ by gluing a new morphism of
degree $n$ from $A$ to $C$. Observe that $R_{\mathcal{A}}$ and
  $R_{\mathcal{C}}$ are the identity dg functors and that $R_X$ is a
  quasi-isomorphism of bimodules. This shows that $\underline{R}$
  belongs to $\mathcal{W}$ and so condition $(4)$ is proved.

\item[(5-6)] We now show that R.L.P.($I$)=R.L.P.($J$) $\cap$ $\mathcal{W}$.
The proof of proposition~\ref{prop7} implies that
R.L.P.($E(I\times I)$)=R.L.P.($E(J \times J)$) $\cap$ $\mathcal{W}'$. Let $\underline{F}:\underline{\mathcal{B}} \rightarrow
  \underline{\mathcal{B}'}$ be a morphism in
  R.L.P.($\tilde{I}$). Clearly $\underline{F}$ belongs to
  R.L.P.($\tilde{J}$) and $F_X$ is a quasi-isomorphism of
  bimodules. This shows that R.L.P.($I$) $\subset$ R.L.P.($J$) $\cap$
  $\mathcal{W}$.

Let now $\underline{F}:\underline{\mathcal{B}} \rightarrow
  \underline{\mathcal{B}'}$ be a morphism in
  R.L.P.($\tilde{J}$) $\cap$ $\mathcal{W}$. Clearly $\underline{F}$
  belongs to R.L.P.($\tilde{I}$) and so R.L.P.($J$) $\cap$
  $\mathcal{W}$ $\subset$ R.L.P($I$). This proves conditions $(5)$ and $(6)$.

\end{itemize}
This proves the theorem.
\end{proof}

\begin{remark}\label{filtri}
Notice that the Quillen model structure of theorem~\ref{theo1} is
cellular, see \cite{Hirschhorn} and that the domains and codomains of
$I$ (the set of generating cofibrations) are cofibrant,
$\aleph_0$-compact, $\aleph_0$-small and homotopically finitely
presented, see definition $2.1.1.$ from \cite{Toen-Vaq}. This implies
that we are in the conditions of proposition~\ref{prop} and so any object
$\underline{\mathcal{B}}$ in $\mathsf{dgcat}^{tr}$ is weakly equivalent
to a filtered colimit of strict finite $I$-cell objects.
\end{remark}

\begin{proposition}\label{comspl}
If $\underline{\mathcal{B}}$ be a strict finite $I$-cell object in
$\mathsf{dgcat}^{tr}$, then $p_1(\underline{\mathcal{B}})$,
$p_2(\underline{\mathcal{B}})$ and $|\underline{\mathcal{B}}|$ are
strict finite $I$-cell objects in $\mathsf{dgcat}$.
\end{proposition}

\begin{proof}
We consider the following inductive argument~:
\begin{itemize}
\item[-] notice that the initial object in $\mathsf{dgcat}^{tr}$ is
$$ \begin{pmatrix} \emptyset & 0 \\
0 & \emptyset \end{pmatrix}$$
and it is sent to $\emptyset$ (the initial object in
$\mathsf{dgcat}$) by the functors $p_1$, $p_2$ and $|-|$.
\item[-] suppose that $\underline{\mathcal{B}}$ is an upper triangular
    dg category such that $p_1(\underline{\mathcal{B}})$,
    $p_2(\underline{\mathcal{B}})$ and $|\underline{\mathcal{B}}|$ are
    strict finite $I$-cell objects in $\mathsf{dgcat}$. Let
    $\underline{G}: \underline{\mathcal{B}'} \rightarrow
    \underline{\mathcal{B}''}$ be an element of the set $I$ in
    $\mathsf{dgcat}^{tr}$, see the proof of theorem~\ref{theo1}, and $\underline{\mathcal{B}'} \rightarrow
    \underline{\mathcal{B}}$ a morphism. Consider the following
    push-out in $\mathsf{dgcat}^{tr}$~:
$$
\xymatrix{
\underline{\mathcal{B}'} \ar[r] \ar[d]_{\underline{G}}
\ar@{}[dr]|{\lrcorner} &
\underline{\mathcal{B}} \ar[d]^{\underline{G}_{\ast}} \\
\underline{\mathcal{B}''} \ar[r] &
\mathsf{PO}\,.
}
$$
We now prove that
$p_1(\mathsf{PO})$,
$p_2(\mathsf{PO})$ and
$|\mathsf{PO}|$
are strict finite $I$-cell objects in $\mathsf{dgcat}$. We consider the following two cases~:
\begin{itemize}
\item[1)] $\underline{G}$ belongs to
  $E(I \times I)$: observe that $p_1(\mathsf{PO})$,
$p_2(\mathsf{PO})$ and
$|\mathsf{PO}|$
correspond exactly to the following push-outs in $\mathsf{dgcat}$~:
$$
\xymatrix{
*+<1pc>{\mathcal{A}'} \ar[r] \ar@{>->}[d]_{G_{\mathcal{A}'}}^{\sim} \ar@{}[dr]|{\lrcorner}  & \mathcal{A}
  \ar[d]  &   *+<1pc>{\mathcal{C}'} \ar[r]
  \ar@{>->}[d]_{G_{\mathcal{C}'}}^{\sim}  \ar@{}[dr]|{\lrcorner}  & \mathcal{C}
  \ar[d]  & \mathcal{A}' \coprod \mathcal{C}' \ar[r]
  \ar@{}[dr]|{\lrcorner} \ar[d]_{G_{\mathcal{A}'}\amalg
    G_{\mathcal{C}'}}  & |\underline{\mathcal{B}}| \ar[d] \\
\mathcal{A}'' \ar[r] & p_1(\mathsf{PO}) & \mathcal{C}'' \ar[r] & p_2(\mathsf{PO}) & \mathcal{A}'' \coprod \mathcal{C}'' \ar[r] & |\mathsf{PO}| \,.
}
$$
Since $G_{\mathcal{A}'}$ and $G_{\mathcal{C}'}$ belong to $I$ this
case is proved.
\end{itemize}
\item[2)] $\underline{G}$ belongs to $\tilde{I}$: observe
  that
  $p_1(\mathsf{PO})$ identifies with $\mathcal{A}$,
$p_2(\mathsf{PO})$ identifies with $\mathcal{C}$ and
$|\mathsf{PO}|$
corresponds to the following push-out in $\mathsf{dgcat}$~:
$$
\xymatrix{
\mathcal{C}(n) \ar[r] \ar[d]_{S(n)} \ar@{}[dr]|{\lrcorner} &
|\underline{\mathcal{B}}| \ar[d] \\
\mathcal{P}(n) \ar[r] &
|\mathsf{PO}|,
}
$$
where $S(n)$ is a generating cofibration in $\mathsf{dgcat}$, see theorem~\ref{mal}. This proves this case.
\end{itemize}
The proposition is proven.
\end{proof}

\section{Split short exact sequences}\label{splitex}
In this section, we establish the connection between split short exact
sequences of dg categories and upper triangular dg categories.
\begin{definition}\label{adjsplt}
A split short exact sequence of dg categories is a short exact
sequence of dg categories, see \cite{ICM}, which is equivalent in $\mathsf{Hmo}$
to one of the form
$$
\xymatrix{
0 \ar[r] & \mathcal{A} \ar[r]_{i_{\mathcal{A}}} & \mathcal{B} \ar@<-1ex>[l]_R
\ar[r]_P & \mathcal{C} \ar@<-1ex>[l]_{i_{\mathcal{C}}} \ar[r] & 0 \,,
}
$$
where we have $P\circ i_{\mathcal{A}} =0$, $R$ is a dg functor right adjoint to $i_{\mathcal{A}}$,
$i_{\mathcal{C}}$ is a dg functor right adjoint to $P$ and we have $P\circ i_{\mathcal{C}}=
Id_{\mathcal{C}}$ and $R\circ i_{\mathcal{A}}=Id_{\mathcal{A}}$ via
the adjunction morphisms.

\end{definition}
To a split short exact sequence, we can naturally associate the upper triangular
dg category
$$
\begin{array}{rcl}
\underline{\mathcal{B}} & := & {\begin{pmatrix} \mathcal{A} &
    \mathsf{Hom}_{\mathcal{B}}(i_{\mathcal{C}}(-), i_{\mathcal{A}}(-)) \\
0 & \mathcal{C} \end{pmatrix}}\,.
\end{array}
$$

Conversely to an upper triangular dg category
$\underline{\mathcal{B}}$ such that $\mathcal{A}$ and $\mathcal{C}$
admit a zero object (for instance if they are Morita fibrant), we can
associate a split short exact sequence
$$
\xymatrix{
0 \ar[r] & \mathcal{A} \ar[r]_{i_{\mathcal{A}}} & |\underline{\mathcal{B}}| \ar@<-1ex>[l]_R
\ar[r]_P & \mathcal{C} \ar@<-1ex>[l]_{i_{\mathcal{C}}} \ar[r] & 0 \,,
}
$$
where $P$ and $R$ are the projection dg functors.
Moreover, this construction is functorial in $\underline{\mathcal{B}}$
and sends total Morita equivalences to Morita equivalent split short
exact sequences. Notice also that by lemma~\ref{triancomp} this
functor preserves colimits.

\begin{proposition}\label{aproxsplit}
Every split short exact sequence of dg categories is weakly equivalent
to a filtered homotopy colimit of split short exact sequences whose
components are strict finite $I$-cell objects in $\mathsf{dgcat}$.
\end{proposition}

\begin{proof}
Let
$$
\xymatrix{
0 \ar[r] & \mathcal{A} \ar[r]_{i_{\mathcal{A}}} & \mathcal{B} \ar@<-1ex>[l]_R
\ar[r]_P & \mathcal{C} \ar@<-1ex>[l]_{i_{\mathcal{C}}} \ar[r] & 0 \,,
}
$$
be a split short exact sequence of dg categories. We can supose that
$\mathcal{A}$, $\mathcal{B}$ and $\mathcal{C}$ are Morita fibrant dg
categories, see proposition~\ref{nova4}. Consider the upper
triangular dg category
$$
\begin{array}{rcl}
\underline{\mathcal{B}} & := & {\begin{pmatrix} \mathcal{A} & \mathsf{Hom}_{\mathcal{B}}(i_{\mathcal{C}}(-),i_{\mathcal{A}}(-)) \\
0 & \mathcal{C} \end{pmatrix}}\,.
\end{array}
$$
Now by remark~\ref{filtri}, $\underline{\mathcal{B}}$ is
equivalent to a filtered colimit of strict finite $I$-cell objects in
$\mathsf{dgcat}^{tr}$. Consider the image of this diagram by the
functor, described above,  which sends an upper triangular dg category to a split
short exact sequence. By proposition~\ref{comspl} the
components of each split short exact sequence of this diagram are
strict finite $I$-cell objects in $\mathsf{dgcat}$. Since the category
$\mathsf{dgcat}$ satisfies the conditions of proposition~\ref{prop}, filtered homotopy colimits are equivalent to filtered colimits and so the proposition is proven.
\end{proof}

\section{Quasi-Additivity}\label{quasi}
Recall from section~\ref{labuniv} that we have at our disposal the
Quillen model category
$\mathsf{L}_{\Sigma,P} \mathsf{Fun}(\mathsf{dgcat}_f^o,Sset)$ which is
  {\em homotopically pointed}, i.e. the morphism $\emptyset
  \rightarrow \ast$, from the initial object $\emptyset$ to the
  terminal one $\ast$, is a weak equivalence. We now consider a
  strictly pointed Quillen model.

\begin{proposition}\label{point}
We have a Quillen equivalence
$$
\xymatrix{
\ast \downarrow \mathsf{L}_{\Sigma,P}\mathsf{Fun}(\mathcal{M}^o_f,Sset) \ar@<1ex>[d]^U \\
\mathsf{L}_{\Sigma,P} \mathsf{Fun}(\mathcal{M}_f^o,Sset) \ar@<1ex>[u]^{(-)_{+}}\,,
}
$$
where $U$ denotes the forgetful functor.
\end{proposition}
This follows from the fact that the category $\mathsf{L}_{\Sigma,P} \mathsf{Fun}(\mathsf{dgcat}_f^o,Sset)$ is
  homotopically pointed and from the following general argument.

\begin{proposition}
Let $\mathcal{M}$ be a homotopically pointed Quillen model
category. We have a Quillen equivalence.
$$
\xymatrix{
\ast \downarrow \mathcal{M} \ar@<1ex>[d]^U \\
\mathcal{M} \ar@<1ex>[u]^{(-)_{+}}\,,
}
$$
where $U$ denotes the forgetful functor.
\end{proposition}

\begin{proof}
Clearly the functor $U$ preserves cofibrations, fibrations and weak
equivalences, by construction. Let now $N\in \mathcal{M}$ and $M \in
\ast \downarrow \mathcal{M}$. Consider the following commutative
diagram in $\mathcal{M}$
$$
\xymatrix{
N \simeq \emptyset\coprod N \ar[rr]^f \ar[dr]^{\sim}_{i\amalg \mathbf{1}} & & U(M)
\\
 & \ast \coprod N \ar[ur]_{f^{\sharp}} & \,,
}
$$
where $f^{\sharp}$ is the morphism which corresponds to $f$, considered
as a morphism in $\mathcal{M}$, under the adjunction and $i:\emptyset
\stackrel{\sim}{\rightarrow} \ast$. Since the morphism $i\amalg \mathbf{1}$
corresponds to the homotopy colimit of $i$ and $1$, which are both
weak equivalences, proposition~\ref{Cisin} implies that  $i\amalg \mathbf{1}$  is a
weak equivalence. Now, by the property `$2$ out of $3$', the morphism
$f$ is a weak equivalence if and only if $f^{\sharp}$ is one. This
proves the proposition.
\end{proof}

\begin{Notation}
Let $\mathcal{A}$ and $\mathcal{B}$ be small dg categories. We denote
by $\mathsf{rep}_{mor}(\mathcal{A},\mathcal{B})$ the
full-dg-subcategory of $\mathcal{C}_{dg}(\mathcal{A}^{op}_c\otimes
\mathcal{B})$, where $\mathcal{A}_c$ denotes a cofibrant resolution of
$\mathcal{A}$, whose objects are the bimodules $X$ such that $X(?,A)$
is a compact object in $\mathcal{D}(\mathcal{B})$ for all $A\in
\mathcal{A}_c$ and which are cofibrant as bimodules.
We denote by $w \mathcal{A}$ the category of homotopy equivalences of $\mathcal{A}$ and by $N.w\mathcal{A}$ its nerve.
\end{Notation}

Now, consider the morphism
$$
\begin{array}{ccl}
\mathsf{Ho}(\mathsf{dgcat}) & \rightarrow & \mathsf{Ho}(\ast
\downarrow \mathsf{L}_{\Sigma,P}\mathsf{Fun}(\mathcal{M}^o_f,Sset)) \\
\mathcal{A} & \mapsto & \left\{\begin{array}{l}
    \mathsf{Hom}_{\mathsf{dgcat}}(\Gamma(?),\mathcal{A}_f)_+\\
\simeq \mathsf{Map}_{\mathsf{dgcat}}(?,\mathcal{A})_+\\
\simeq N.w \mathsf{rep}_{mor}(?,\mathcal{A})_+ \end{array} \right.
\end{array}
$$
which by sections~\ref{homotopy}, \ref{chappoint} and
proposition~\ref{point} corresponds to the component $(\Phi\circ
\mathbb{R}\underline{h})(e)$ of the morphism of derivators
$$ \Phi \circ \mathbb{R}\underline{h}: \mathsf{HO}(\mathsf{dgcat})\longrightarrow
\mathsf{L}_{\Sigma,P}\mathsf{Hot}_{\mathsf{dgcat}_f}\,,$$
see proposition~\ref{ext}. Observe that the simplicial presheaf
$N.w\mathsf{rep}_{mor}(?,\mathcal{A})$ is already canonically pointed.

\begin{proposition}
The canonical morphism
$$\Psi : N.w\mathsf{rep}_{mor}(?,\mathcal{A})_+ \rightarrow N.w \mathsf{rep}_{mor}(?,\mathcal{A})$$
is a weak equivalence in $\ast \downarrow \mathsf{L}_{\Sigma,P}\mathsf{Fun}(\mathcal{M}^o_f,Sset)$.
\end{proposition}

\begin{proof}
Observe that $N.w\mathsf{rep}_{mor}(?,\mathcal{A})$ is a fibrant object
in $\ast
  \downarrow \mathsf{L}_{\Sigma,P}\mathsf{Fun}(\mathcal{M}^o_f,Sset)$
  and that the canonical morphism $\Psi$ corresponds to the co-unit
  of the adjunction of proposition~\ref{point}. Since this adjunction
  is a Quillen equivalence, the proposition is proved.
\end{proof}

Recall now from remark~\ref{remar} that we have a canonical equivalence of
pointed derivators
$$ \mathsf{L}_{\Sigma,P}\mathsf{Hot}_{\mathsf{dgcat}_f}
\stackrel{\sim}{\longrightarrow}
{\mathsf{L}_{\Sigma,P}\mathsf{Hot}_{\mathsf{dgcat}_f}}_{\bullet}\,,$$
where ${\mathsf{L}_{\Sigma,P}\mathsf{Hot}_{\mathsf{dgcat}_f}}_{\bullet}$
is the derivator associated with the Quillen model category
$\mathsf{L}_{\Sigma,P}\mathsf{Fun}(\mathsf{dgcat}_f^o,Sset_{\bullet})$.
From now on, we will consider this Quillen model. We have the
following morphism of derivators
$$\Phi \circ \mathbb{R}\underline{h} : \mathsf{HO}(\mathsf{dgcat}) \longrightarrow
{\mathsf{L}_{\Sigma,P}\mathsf{Hot}_{\mathsf{dgcat}_f}}_{\bullet}\,,$$
which commutes with filtered homotopy colimits and preserves the point.

\begin{Notation}
\begin{itemize}
\item[-] We denote by $\mathcal{E}^s$ the set of retractions of dg categories
$$ \xymatrix{
\mathcal{G} \ar[r]_{i_{\mathcal{G}}} & \mathcal{H} \ar@<-1ex>[l]_R \,,
}
$$
where $\mathcal{G}$ and $\mathcal{H}$ are strict finite $I$-cell
objects in $\mathsf{dgcat}$, $i_{\mathcal{G}}$ is a fully faithful dg
functor, $R$ is a right adjoint to $i_{\mathcal{G}}$ and $R\circ i_{\mathcal{G}} =Id_{\mathcal{G}}$.
\item[-] We denote by $\mathcal{E}^s_{un}$ the set of morphisms $S_L$
  in
  ${\mathsf{L}_{\Sigma,P}\mathsf{Hot}_{\mathsf{dgcat}_f}}_{\bullet}(e)$, see section~\ref{chapquotient}, where $L$ belongs to the set $\mathcal{E}^s$.
\end{itemize}
\end{Notation}

Now choose for each element of the set $\mathcal{E}^s_{un}$ a
representative in the category
$\mathsf{L}_{\Sigma,P}\mathsf{Fun}(\mathsf{dgcat}_f^o,Sset_{\bullet})$.
We denote this set of representatives by
$\widetilde{\mathcal{E}^s_{un}}$. Since $\mathsf{L}_{\Sigma,P}\mathsf{Fun}(\mathsf{dgcat}_f^o,Sset_{\bullet})$
is a left proper, cellular Quillen model category, see
\cite{Hirschhorn}, its left Bousfield localization by
$\widetilde{\mathcal{E}^s_{un}}$ exists. We denote it by
$\mathsf{L}_{\widetilde{\mathcal{E}^s_{un}}}
\mathsf{L}_{\Sigma,P}\mathsf{Fun}(\mathsf{dgcat}_f^o,Sset_{\bullet})$
and by ${\mathsf{L}_{\widetilde{\mathcal{E}^s_{un}}}
\mathsf{L}_{\Sigma,P}\mathsf{Hot}_{\mathsf{dgcat}_f}}_{\bullet}$ the
associated derivator. We have the following morphism of derivators
$$
\Psi: {\mathsf{L}_{\Sigma,P}\mathsf{Hot}_{\mathsf{dgcat}_f}}_{\bullet} \rightarrow
 {\mathsf{L}_{\widetilde{\mathcal{E}^s_{un}}}
\mathsf{L}_{\Sigma,P}\mathsf{Hot}_{\mathsf{dgcat}_f}}_{\bullet}\,.
$$
\begin{remark}\label{fib}
\begin{itemize}

\item[-] Notice that by construction the domains and codomains of the
  set $\widetilde{\mathcal{E}^s_{un}}$ are homotopically finitely
  presented objects. Therefore by lemma~\ref{gener} the set
$$ \mathcal{G}= \{\mathbf{F}^X_{\Delta[n]_+ / \partial \Delta[n]_+ } |
\, X \in
\mathcal{M}_f\,, n\geq 0 \} \,,$$
of cofibers of the generating cofibrations in
$\mathsf{Fun}(\mathcal{M}_f^o,Sset_{\bullet})$ is a set of small weak
generators in $\mathsf{Ho}(\mathsf{L}_{\widetilde{\mathcal{E}^s_{un}}}
\mathsf{L}_{\Sigma,P}\mathsf{Fun}(\mathcal{M}_f^o,Sset_{\bullet}))$.

\item[-] Notice also that proposition~\ref{aproxsplit} implies that variants of
proposition~\ref{cons} and theorem~\ref{invert} are also verified:
simply consider the set $\mathcal{E}^s$ instead of $\mathcal{E}$ and a
retraction of dg categories instead of an inclusion of a full dg
subcategory. The proofs are exactly the same.
\end{itemize}
\end{remark}
Let $\mathcal{M}$ be a left proper cellular model category, $S$ a
set of maps in $\mathcal{M}$ and $\mathsf{L}_S\mathcal{M}$ the left
Bousfield localization of $\mathcal{M}$ with respect to $S$, see \cite{Hirschhorn}. Recall
from \cite[4.1.1.]{Hirschhorn} that an object $X$ in
$\mathsf{L}_S\mathcal{M}$ is fibrant if $X$ is fibrant in
$\mathcal{M}$ and for every element $f:A \rightarrow B$ of $S$ the
induced map of homotopy function complexes $f^{\ast}:\mathsf{Map}(B,X)
\rightarrow \mathsf{Map}(A,X)$ is a weak equivalence.
\begin{proposition}
An object $F \in \mathsf{L}_{\widetilde{\mathcal{E}^s_{un}}}
\mathsf{L}_{\Sigma,P}\mathsf{Fun}(\mathsf{dgcat}_f^o,Sset_{\bullet})$
is fibrant if and only if it satisfies the following conditions
\begin{itemize}
\item[1)] $F(\mathcal{B}) \in Sset_{\bullet}$ is fibrant, for all
  $\mathcal{B} \in \mathsf{dgcat}_f$.
\item[2)] $F(\emptyset) \in Sset_{\bullet}$ is contractible.
\item[3)] For every Morita equivalence $\mathcal{B}
  \stackrel{\sim}{\rightarrow} \mathcal{B}'$ in $\mathsf{dgcat}_f$ the
  morphism $F(\mathcal{B}') \stackrel{\sim}{\rightarrow}
  F(\mathcal{B})$ is a weak equivalence in $Sset_{\bullet}$.
\item[4)] Every split short exact sequence
$$
\xymatrix{
0 \ar[r] & \mathcal{B}' \ar[r]_{i_{\mathcal{B}'}} & \mathcal{B} \ar@<-1ex>[l]_R
\ar[r]_P & \mathcal{B}'' \ar@<-1ex>[l]_{i_{\mathcal{B}''}} \ar[r] & 0
}
$$
in $\mathsf{dgcat}_f$ induces a homotopy fiber sequence
$$ F(\mathcal{B}'') \rightarrow F(\mathcal{B}) \rightarrow
F(\mathcal{B}')$$
in $Sset$.
\end{itemize}
\end{proposition}

\begin{proof}
Clearly condition $1)$ corresponds to the fact that $F$ is fibrant in
$\mathsf{Fun}(\mathsf{dgcat}_f^o,Sset_{\bullet})$. Now observe that $\mathsf{Fun}(\mathsf{dgcat}_f^o,Sset_{\bullet})$ is a
simplicial Quillen model category with the simplicial action given by
$$
\begin{array}{ccc}
Sset \times \mathsf{Fun}(\mathsf{dgcat}_f^o,Sset_{\bullet}) &
\rightarrow & \mathsf{Fun}(\mathsf{dgcat}_f^o,Sset_{\bullet}) \\
(K,F) & \mapsto & K_+\wedge F\,,
\end{array}
$$
where $K_+ \wedge F$ denotes the componentwise smash product. This simplicial structure and the construction of the
localized Quillen model category \newline
$\mathsf{L}_{\widetilde{\mathcal{E}^s_{un}}}
\mathsf{L}_{\Sigma,P}\mathsf{Fun}(\mathsf{dgcat}_f^o,Sset_{\bullet})$,
see section~\ref{small}, allow us to recover conditions $2)$ and
$3)$. Condition $4)$ follows from the construction of the set
$\widetilde{\mathcal{E}^s_{un}}$ and from the fact that the functor
$$ \mathsf{Map}(?,F) :
\mathsf{Fun}(\mathsf{dgcat}_f^o,Sset_{\bullet})^{op} \rightarrow
Sset$$
transforms homotopy cofiber sequences into homotopy fiber sequences.

This proves the proposition.
\end{proof}

Let $\mathcal{A}$ be a Morita fibrant dg
  category. Recall from notation~\ref{notacao} that
  $S_{\bullet}\mathcal{A}$ denotes the simplicial Morita fibrant dg
  category obtained by applying Waldhausen's
  $S_{\bullet}$-construction to the exact category $\mathsf{Z}^0(\mathcal{A})$
  and remembering the enrichment in complexes.

\begin{Notation}
We denote by $K(\mathcal{A}) \in \mathsf{Fun}(\mathsf{dgcat}_f^o,Sset_{\bullet})$ the following simplicial presheaf
$$
\begin{array}{rcl}
\mathcal{B} & \mapsto &
|N.wS_{\bullet}\mathsf{rep}_{mor}(\mathcal{B},\mathcal{A})|\,,
\end{array}
$$
where $|-|$ denotes the fibrant realization functor of bisimplicial sets.
\end{Notation}

\begin{proposition}\label{fibrant}
The simplicial presheaf $K(\mathcal{A})$ is fibrant in
$\mathsf{L}_{\widetilde{\mathcal{E}^s_{un}}} \mathsf{L}_{\Sigma,P}\mathsf{Fun}(\mathsf{dgcat}_f^o,Sset_{\bullet})$.
\end{proposition}

\begin{proof}
Observe that $K(\mathcal{A})$ satisfies conditions $(1)$-$(3)$. We now
prove that Waldhausen's fibration theorem \cite[1.6.4]{Waldhausen} implies condition
$(4)$. Apply the contravariant functor
$\mathsf{rep}_{mor}(?,\mathcal{A})$ to the split short exact sequence
$$
\xymatrix{
0 \ar[r] & \mathcal{B}' \ar[r]_{i_{\mathcal{B}'}} & \mathcal{B} \ar@<-1ex>[l]_R
\ar[r]_P & \mathcal{B}'' \ar@<-1ex>[l]_{i_{\mathcal{B}''}} \ar[r] & 0
}
$$
and obtain a split short exact sequence
$$
\xymatrix{
0 \ar[r] & \mathsf{rep}_{mor}(\mathcal{B}'',\mathcal{A}) \ar[r] & \mathsf{rep}_{mor}(\mathcal{B},\mathcal{A}) \ar@<-1ex>[l]
\ar[r] & \mathsf{rep}_{mor}(\mathcal{B}',\mathcal{A}) \ar@<-1ex>[l] \ar[r] & 0 \,.
}
$$
Now consider the Waldhausen category
$v\mathsf{rep}_{mor}(\mathcal{B},\mathcal{A}):=\mathsf{Z}^0(\mathsf{rep}_{mor}(\mathcal{B},\mathcal{A}))$, where the weak equivalences
are the morphisms $f$ such that $\mathsf{cone}(f)$ is
contractible. Consider also the Waldhausen category $w\mathsf{rep}_{mor}(\mathcal{B},\mathcal{A})$,
which has the same cofibrations as $v\mathsf{rep}_{mor}(\mathcal{B},\mathcal{A})$ but the weak
equivalences are the morphisms $f$ such that $\mathsf{cone}(f)$
belongs to $\mathsf{rep}_{mor}(\mathcal{B}'',\mathcal{A})$. Observe that we have the inclusion
$v\mathsf{rep}_{mor}(\mathcal{B},\mathcal{A}) \subset w\mathsf{rep}_{mor}(\mathcal{B},\mathcal{A})$ and an equivalence $\mathsf{rep}_{mor}(\mathcal{B},\mathcal{A})^w
\simeq \mathsf{Z}^0(\mathsf{rep}_{mor}(\mathcal{B}',\mathcal{A}))$, see section $1.6$ from
\cite{Waldhausen}. The conditions of theorem $1.6.4$ from \cite{Waldhausen}
are satisfied and so we have a homotopy fiber sequence
$$
|N.wS_{\bullet}\mathsf{rep}_{mor}(\mathcal{B}'',\mathcal{A})| \rightarrow
|N.wS_{\bullet}\mathsf{rep}_{mor}(\mathcal{B},\mathcal{A})|
\rightarrow |N.wS_{\bullet} \mathsf{rep}_{mor}(\mathcal{B}',\mathcal{A})| $$
in $Sset$. This proves the proposition.
\end{proof}

\begin{definition}
\begin{itemize}
\item[-] The {\em Unstable motivator of dg categories}
  $\mathcal{M}^{unst}_{dg}$ is the derivator associated with the Quillen model category
$$\mathsf{L}_{\widetilde{\mathcal{E}^s_{un}}} \mathsf{L}_{\Sigma,P} \mathsf{Fun}(\mathsf{dgcat}_f^o,Sset_{\bullet})\,.$$
\item[-] The {\em Universal unstable invariant of dg categories} is the canonical morphism of
  derivators $$ \mathcal{U}_u : \mathsf{HO}(\mathsf{dgcat}) \rightarrow \mathcal{M}^{unst}_{dg}\,.$$
\end{itemize}
\end{definition}

Let $p:\Delta \rightarrow e$ be the projection functor.

\begin{proposition}\label{clef}
The objects
$$ 
\begin{array}{ccc}
S^1 \wedge N.w\mathsf{rep}_{mor}(?,\mathcal{A}) & \mbox{and} &
|N.wS_{\bullet} \mathsf{rep}_{mor}(?,\mathcal{A})|=K(\mathcal{A})
\end{array}
$$
are canonically isomorphic in $\mathsf{Ho}(\mathsf{L}_{\widetilde{\mathcal{E}^s_{un}}} \mathsf{L}_{\Sigma,P}\mathsf{Fun}(\mathsf{dgcat}_f^o,Sset_{\bullet}))$.
\end{proposition}

\begin{proof}
As in \cite[3.3]{MacCarthy}, we consider the sequence in $\mathsf{HO}(\mathsf{dgcat})(\Delta)$
$$ \xymatrix{0 \ar[r] & \mathcal{A}_{\bullet} \ar[r]^-I &
PS_{\bullet}\mathcal{A} \ar[r]^Q &  S_{\bullet}\mathcal{A} \ar[r]& 0} \ko
$$
where $\mathcal{A}_{\bullet}$ denotes the constant simplicial dg
category with value $\mathcal{A}$ and $PS_{\bullet}\mathcal{A}$ the
path object of $S_{\bullet}\mathcal{A}$. By applying the morphism of
derivators $\mathcal{U}_u$ to this sequence, we obtain the canonical
morphism
$$ S_I: \mbox{cone} (\mathcal{U}_u(\mathcal{A}_{\bullet}
\stackrel{I}{\rightarrow} PS_{\bullet}\mathcal{A})) \rightarrow
\mathcal{U}_u(S_{\bullet}\mathcal{A})$$
in $\mathcal{M}^{unst}_{dg}(\Delta)$. We now prove that for each point
$n:e \rightarrow \Delta$, the $n$th component of $S_I$ is an
isomorphism in $\mathsf{L}_{\widetilde{\mathcal{E}^s_{un}}}
\mathsf{L}_{\Sigma,P}
\mathsf{Fun}(\mathsf{dgcat}_f^o,Sset_{\bullet})$. For each point $n:e
\rightarrow \Delta$, we have a split short exact sequence in
$\mathsf{Ho}(\mathsf{dgcat})$~:
$$
\xymatrix{
0 \ar[r] & \mathcal{A} \ar[r]_-{I_n} & PS_n\mathcal{A}=S_{n+1}\mathcal{A} \ar@<-1ex>[l]_-{R_n}
\ar[r]_-{Q_n} & S_n\mathcal{A} \ar@<-1ex>[l]_-{S_n} \ar[r] & 0 \,,
}
$$
where $I_n$ maps $A \in \mathcal{A}$ to the constant sequence
$$
0 \rightarrow A \stackrel{Id}{\rightarrow} A \stackrel{Id}{\rightarrow}
\cdots \stackrel{Id}{\rightarrow} A \,,$$
$Q$ maps a sequence
$$ 0 \rightarrow A_0 \rightarrow A_1 \rightarrow \cdots \rightarrow
A_n$$
to
$$ A_1/A_0 \rightarrow \cdots \rightarrow A_n/A_0\,,$$
$S_n$ maps a sequence
$$ 0 \rightarrow A_0 \rightarrow A_1 \rightarrow \cdots \rightarrow
A_{n-1}$$
to
$$ 0 \rightarrow 0 \rightarrow A_0 \rightarrow \cdots \rightarrow
A_{n-1}$$
and $R_n$ maps a sequence
$$ 0 \rightarrow A_0 \rightarrow A_1 \rightarrow \cdots \rightarrow
A_{n-1}$$
to $A_0$. Now, by construction of
$\mathsf{L}_{\widetilde{\mathcal{E}^s_{un}}} \mathsf{L}_{\Sigma,P}
\mathsf{Fun}(\mathsf{dgcat}_f^o,Sset_{\bullet})$ the canonical
morphisms
$$ S_{I_n}: \mbox{cone}(\mathcal{U}_u (\mathcal{A}
\stackrel{I_n}{\rightarrow} PS_n\mathcal{A})) \rightarrow
\mathcal{U}_u(S_n\mathcal{A}),\,\, n \in \mathbb{N} \ko$$
are isomorphisms in $\mathcal{M}_{dg}^{unst}(e)$. Since homotopy
colimits in $\mathsf{L}_{\widetilde{\mathcal{E}^s_{un}}} \mathsf{L}_{\Sigma,P}
\mathsf{Fun}(\mathsf{dgcat}_f^o,Sset_{\bullet})$ are calculated
objectwise the $n$th component of $S_I$ identifies with $S_{I_n}$ and
so by the conservativity axiom $S_I$ is an isomorphism in
$\mathcal{M}_{dg}^{unst}(\Delta)$. This implies that we obtain the
homotopy cocartesian square
$$
\xymatrix{
p_!\mathcal{U}_u(\mathcal{A}_{\bullet}) \ar[d] \ar[r] &
p_!(\mathcal{U}_u(PS_{\bullet}\mathcal{A})) \ar[d] \\
\ast \ar[r] & p_!(\mathcal{U}_u(S_{\bullet}\mathcal{A}))\,.
}
$$
As in the proof of proposition~\ref{real}, we show that
$p_!\mathcal{U}_u(\mathcal{A}_{\bullet})$ identifies with $N.w
\mathsf{rep}_{mor}(?,\mathcal{A})=\mathcal{U}_u(\mathcal{A})$ and that
$p_!(\mathcal{U}_u(PS_{\bullet}\mathcal{A}))$ is contractible. Since
we have the equivalence
$$p_!(\mathcal{U}_u(S_{\bullet}\mathcal{A})) = p_!(N.w
\mathsf{rep}_{mor}(?,S_{\bullet}\mathcal{A}))
\stackrel{\sim}{\rightarrow} |N.wS_{\bullet} \mathsf{rep}_{mor}(?,\mathcal{A})|$$
and $N.w\mathsf{rep}_{mor}(?,\mathcal{A})$ is cofibrant in  $\mathsf{L}_{\widetilde{\mathcal{E}^s_{un}}} \mathsf{L}_{\Sigma,P}
\mathsf{Fun}(\mathsf{dgcat}_f^o,Sset_{\bullet})$ the proposition is proven.
\end{proof}

\begin{proposition}\label{kth}
We have the following weak equivalence of simplicial sets
$$ \mathsf{Map}(\mathcal{U}_u(k),S^1 \wedge \mathcal{U}_u(\mathcal{A}))
\stackrel{\sim}{\rightarrow} |N.wS_{\bullet}\mathcal{A}_f|$$
in $\mathsf{L}_{\widetilde{\mathcal{E}^s_{un}}} \mathsf{L}_{\Sigma,P}
\mathsf{Fun}(\mathsf{dgcat}_f^o, Sset_{\bullet})$. In particular, we
have the following isomorphism
$$\pi_{i+1}
\mathsf{Map}(\mathcal{U}_u(k),S^1 \wedge \mathcal{U}_u(\mathcal{A}))
\stackrel{\sim}{\rightarrow} K_i(\mathcal{A}), \, \forall i \geq
0\,.$$
\end{proposition}

\begin{proof}
This follows from propositions \ref{fibrant}, \ref{clef} and from the
following weak equivalences
$$
\begin{array}{rcl}
\mathsf{Map}(\mathcal{U}_u(k),S^1 \wedge \mathcal{U}_u(\mathcal{A})) & \simeq
& \mathsf{Map}(\mathbb{R}\underline{h}(k), K(\mathcal{A}))\\
 & \simeq & (K(\mathcal{A}))(k)\\
 & \simeq & |N.wS_{\bullet}\mathcal{A}_f|\,.
\end{array}
$$
\end{proof}
\section{The universal additive invariant}\label{univadit}
Consider the Quillen model category
$\mathsf{L}_{\widetilde{\mathcal{E}^s_{un}}} \mathsf{L}_{\Sigma,P}
\mathsf{Fun}(\mathsf{dgcat}_f^o, Sset_{\bullet})$ constructed in the
previous section. The definition of the
  set $\widetilde{\mathcal{E}^s_{un}}$ and the same arguments as those
  of example~\ref{exem} and example~\ref{ex2} allows us to conclude
  that $\mathsf{L}_{\widetilde{\mathcal{E}^s_{un}}} \mathsf{L}_{\Sigma,P}
\mathsf{Fun}(\mathsf{dgcat}_f^o,Sset_{\bullet})$ satisfies the
  conditions of theorem~\ref{repre}. In particular we have an
  equivalence of triangulated derivators
$$ \mathsf{St}(\mathcal{M}_{dg}^{unst}) \stackrel{\sim}{\rightarrow}
\mathsf{HO}(\mathsf{Sp}^{\mathbb{N}}(\mathsf{L}_{\widetilde{\mathcal{E}^s_{un}}}
\mathsf{L}_{\Sigma,P}
\mathsf{Fun}(\mathsf{dgcat}_f^o, Sset_{\bullet})))\,.
$$

\begin{definition}\label{condadit}
\begin{itemize}
\item[-] The {\em Additive motivator of dg categories} $\mathcal{M}^s_{dg}$ is the triangulated
derivator associated with the stable Quillen model category
$$ \mathsf{Sp}^{\mathbb{N}}(\mathsf{L}_{\widetilde{\mathcal{E}^s_{{un}}}}\mathsf{L}_{\Sigma,P}
\mathsf{Fun}(\mathsf{dgcat_f^o,Sset_{\bullet}}))\,.$$
\item[-] The {\em Universal additive invariant of dg categories} is the canonical morphism of
  derivators $$ \mathcal{U}_a : \mathsf{HO}(\mathsf{dgcat}) \rightarrow \mathcal{M}^{add}_{dg}\,.$$
\end{itemize}
\end{definition}

\begin{remark}
Observe that remark~\ref{fib} and remark~\ref{important} imply that
$\mathcal{M}^{add}_{dg}$ is a compactly generated triangulated derivator.
\end{remark}

We sum up the construction of $\mathcal{M}^{add}_{dg}$ in the following diagram

$$
\xymatrix{
\underline{\mathsf{dgcat}_f}[S^{-1}] \ar[r] \ar[d]_{\mathsf{Ho}(h)} &
\mathsf{HO}(\mathsf{dgcat}) \ar[dl]^{\mathbb{R}\underline{h}}
\ar@/^3pc/[ddddl]^{\mathcal{U}_a}  \ar@/^2pc/[dddl]_{\mathcal{U}_u}
 \\
\mathsf{L}_{\Sigma}\mathsf{Hot}_{\mathsf{dgcat}_f}
\ar[d]_{\Phi}  \ar@<1ex>[ur]^{\mathbb{L}Re}  & \\
{\mathsf{L}_{\Sigma,P}\mathsf{Hot}_{\mathsf{dgcat}_f}}_{\bullet}
\ar[d]_{\Psi} & \\
\mathcal{M}_{dg}^{unst}
\ar[d]_{\varphi} & \\
\mathcal{M}^{add}_{dg} &
}
$$
Observe that the morphism of derivators $\mathcal{U}_a$ is pointed, commutes with filtered
homotopy colimits and satisfies the following condition :

\begin{itemize}
\item[A)] For every split short exact sequence
$$
\xymatrix{
0 \ar[r] & \mathcal{A} \ar[r]_{i_{\mathcal{A}}} & \mathcal{B} \ar@<-1ex>[l]_R
\ar[r]_P & \mathcal{C} \ar@<-1ex>[l]_{i_{\mathcal{C}}} \ar[r] & 0
}
$$
in $\mathsf{Ho}(\mathsf{dgcat})$, we have a split triangle
$$
\xymatrix{
\mathcal{U}_a(\mathcal{A}) \ar[r]_{i_{\mathcal{A}}} & \mathcal{U}_a(\mathcal{B}) \ar@<-1ex>[l]_R
\ar[r]_P & \mathcal{U}_a(\mathcal{C}) \ar@<-1ex>[l]_{i_{\mathcal{C}}} \ar[r] & \mathcal{U}_a(\mathcal{A})[1]
}
$$
in $\mathcal{M}^{add}_{dg}(e)$. (This implies that the dg functors
$i_{\mathcal{A}}$ and $i_{\mathcal{C}}$ induce an isomorphism
$$ \mathcal{U}_a(\mathcal{A}) \oplus \mathcal{U}_a(\mathcal{C})
\stackrel{\sim}{\rightarrow} \mathcal{U}_a(\mathcal{B})$$
in $\mathcal{M}^{add}_{dg}(e)$).
\end{itemize}

\begin{remark}
Since the dg category $\mathcal{B}$ in the above split short exact
sequence is Morita equivalent to the dg category
$E(\mathcal{A},\mathcal{B},\mathcal{C})$, see section $1.1$ from
\cite{Waldhausen}, condition $A)$ is equivalent to the additivity property
stated by Waldhausen in \cite[1.3.2]{Waldhausen}.
\end{remark}

Let $\mathbb{D}$ be a strong triangulated derivator.
\begin{theorem}\label{principal1}
The morphism $\mathcal{U}_a$ induces an equivalence of categories
$$
\underline{\mathsf{Hom}}_!(\mathcal{M}^{add}_{dg},
\mathbb{D}) \stackrel{\mathcal{U}_a^{\ast}}{\longrightarrow}
\underline{\mathsf{Hom}}_{flt,\,A),\,p}(\mathsf{HO}(\mathsf{dgcat}),\mathbb{D})\,,$$
where
$\underline{\mathsf{Hom}}_{flt,\,A)\,, p}(\mathsf{HO}(\mathsf{dgcat}),\mathbb{D})$
denotes the category of morphisms of derivators which commute with filtered
homotopy colimits, satisfy condition $A)$ and preserve the point.
\end{theorem}

\begin{proof}
By theorem~\ref{HellerT}, we have an equivalence of categories
$$
\underline{\mathsf{Hom}}_!(\mathcal{M}^{add}_{dg},
\mathbb{D}) \stackrel{\sim}{\longrightarrow}
\underline{\mathsf{Hom}}_!(\mathcal{M}_{dg}^{unst},\mathbb{D})\,.$$
By theorem~\ref{Cisinsk}, we have an equivalence of categories
$$  \underline{\mathsf{Hom}}_!(\mathcal{M}_{dg}^{unst},\mathbb{D})
\stackrel{\sim}{\longrightarrow} \underline{\mathsf{Hom}}_{!,\mathcal{E}^s_{un}}({\mathsf{L}_{\Sigma,P}
\mathsf{Hot}_{\mathsf{dgcat}_f}}_{\bullet}, \mathbb{D})\,.$$
Now, we observe that since $\mathbb{D}$ is a strong triangulated derivator, the category  $\underline{\mathsf{Hom}}_{!,\mathcal{E}^s_{un}}({\mathsf{L}_{\Sigma,P}
\mathsf{Hot}_{\mathsf{dgcat}_f}}_{\bullet}, \mathbb{D})$ identifies
with $\underline{\mathsf{Hom}}_{flt,\,A),\,p}(\mathsf{HO}(\mathsf{dgcat}),\mathbb{D})$.
This proves the theorem.
\end{proof}

\begin{terminology}
We call an object of the right hand side category of
theorem~\ref{principal1} an {\em additive invariant} of dg categories.
\end{terminology}

\begin{example}
\begin{itemize}
\item[-] The Hochschild and cyclic homology and the non-connective
  $K$-theory defined in section~\ref{labuniv} are examples of additive
  invariants.
\item[-] Another example is given by Waldhausen's connective
  $K$-theory spectrum
$$K^c: \mathsf{HO}(\mathsf{dgcat}) \rightarrow \mathsf{HO}(Spt)\,,$$
see \cite{Waldhausen}.
\end{itemize}
\end{example}

\begin{remark}
By theorem~\ref{principal1}, the morphism of derivators $K^c$ factors
through $\mathcal{U}_a$ and so gives rise to a morphism
$$ K^c:\mathcal{M}^{add}_{dg} \rightarrow \mathsf{HO}(Spt)\,.$$
\end{remark}
We now will prove that this morphism of derivators is co-representable in
$\mathcal{M}^{add}_{dg}$.

Let $\mathcal{A}$ be a small dg category.
\begin{Notation}\label{structural}
We denote by $K(\mathcal{A})^c \in \mathsf{Sp}^{\mathbb{N}}(\mathsf{L}_{\widetilde{\mathcal{E}^s_{un}}} \mathsf{L}_{\Sigma,P}
\mathsf{Fun}(\mathsf{dgcat}_f^o, Sset_{\bullet}))$ the spectrum such
that
$$ K(\mathcal{A})^c_n:=
|N.wS_{\bullet}^{(n+1)}\mathsf{rep}_{mor}(?,\mathcal{A})|, \, n\geq 0\,,$$
endowed with the natural structure morphisms 
$$ \beta_n: S^1 \wedge |N.wS_{\bullet}^{(n+1)}\mathsf{rep}_{mor}(?,\mathcal{A})|
\stackrel{\sim}{\longrightarrow}
|N.wS_{\bullet}^{(n+2)}\mathsf{rep}_{mor}(?,\mathcal{A})|,\, n\geq 0\,,$$
see \cite{Waldhausen}.
\end{Notation}

Notice that $\mathcal{U}_a(\mathcal{A})$ identifies in
$\mathsf{Ho}(\mathcal{M}^{add}_{dg})$ with the suspension spectrum given
by
$$ (\Sigma^{\infty}|N.w\mathsf{rep}_{mor}(?,\mathcal{A})|)_n := S^n
\wedge |N.w\mathsf{rep}_{mor}(?,\mathcal{A})|\,.$$
Now proposition~\ref{clef} and the fact that the morphism of derivators $\varphi$ commutes with homotopy
colimits implies that we have an isomorphism
$$\mathcal{U}_a(\mathcal{A})[1] \stackrel{\sim}{\rightarrow}
p_!\mathcal{U}_a(S_{\bullet}\mathcal{A})\,.$$
In particular, we have a natural morphism
$$\eta : \mathcal{U}_a(\mathcal{A})[1] \rightarrow
K(\mathcal{A})^c$$
in $\mathsf{Sp}^{\mathbb{N}}(\mathsf{L}_{\widetilde{\mathcal{E}^s_{un}}} \mathsf{L}_{\Sigma,P}
\mathsf{Fun}(\mathsf{dgcat}_f^o, Sset_{\bullet}))$ induced by the
identity in degree $0$.

\begin{theorem}\label{fibres}
The morphism $\eta$ is a fibrant resolution of $\mathcal{U}_a(\mathcal{A})[1]$.
\end{theorem}

\begin{proof}
We prove first that $K(\mathcal{A})^c$ is a fibrant object in $\mathsf{Sp}^{\mathbb{N}}(\mathsf{L}_{\widetilde{\mathcal{E}^s_{un}}} \mathsf{L}_{\Sigma,P}
\mathsf{Fun}(\mathsf{dgcat}_f^o, Sset_{\bullet}))$. By \cite{Hovey}\cite{Schwede} we need to show that $K(\mathcal{A})^c$ is an
$\Omega$-spectrum, i.e. that $K(\mathcal{A})^c_n$ is a fibrant object $\mathsf{L}_{\widetilde{\mathcal{E}^s_{un}}} \mathsf{L}_{\Sigma,P}
\mathsf{Fun}(\mathsf{dgcat}_f^o, Sset_{\bullet})$ and that the induced
map
$$ K(\mathcal{A})^c_n \rightarrow \Omega K(\mathcal{A})^c_{n+1}$$
is a weak equivalence. By Waldhausen's
additivity theorem, see \cite{Waldhausen}, we have weak equivalences
$$ K(\mathcal{A})^c_n \stackrel{\sim}{\rightarrow}
\Omega K(\mathcal{A})^c_{n+1}$$
in $\mathsf{Fun}(\mathsf{dgcat}_f^o,Sset_{\bullet})$. Now observe that
for every integer $n$, $K^c(\mathcal{A})_n$ satisfies conditions
$(1)$-$(3)$ of proposition~\ref{fibrant}. Condition $(4)$ follows from
Waldhausen's fibration theorem, as in the proof of
proposition~\ref{clef}, applied to the $S_{\bullet}$-construction. This shows
that $K(\mathcal{A})^c$ is an $\Omega$-spectrum.

We now prove that $\eta$ is a (componentwise) weak equivalence in \newline
$\mathsf{Sp}^{\mathbb{N}}(\mathsf{L}_{\widetilde{\mathcal{E}^s_{un}}} \mathsf{L}_{\Sigma,P}
\mathsf{Fun}(\mathsf{dgcat}_f^o, Sset_{\bullet}))$. For this, we prove
first that the structural morphisms
$$ \beta_n: S^1 \wedge N.wS_{\bullet}^{(n+1)}\mathsf{rep}_{mor}(?,\mathcal{A})
\stackrel{\sim}{\longrightarrow}
|N.wS_{\bullet}^{(n+2)}\mathsf{rep}_{mor}(?,\mathcal{A})|,\, n\geq
0\,,$$ 
see notation~\ref{structural}, are weak equivalences in $\mathsf{L}_{\widetilde{\mathcal{E}^s_{un}}} \mathsf{L}_{\Sigma,P}
\mathsf{Fun}(\mathsf{dgcat}_f^o, Sset_{\bullet})$. By considering the same argument as in the proof of
proposition~\ref{clef}, using $S_{\bullet}^{(n+1)}\mathcal{A}$ instead
of $\mathcal{A}$, we obtain the following homotopy cocartesian square
$$
\xymatrix{
K(\mathcal{A})^c_n \ar[d] \ar[r] &
p_!(\mathcal{U}_u(PS_{\bullet}^{(n+2)}\mathcal{A})) \ar[d] \\
\ast \ar[r] & K(\mathcal{A})^c_{n+1}
}
$$
in $\mathsf{L}_{\widetilde{\mathcal{E}^s_{un}}} \mathsf{L}_{\Sigma,P}
\mathsf{Fun}(\mathsf{dgcat}_f^o, Sset_{\bullet})$ with
$p_!(\mathcal{U}_u(PS_{\bullet}^{(n+2)}\mathcal{A}))$
contractible. Since $K(\mathcal{A})^c_{n+1}$ is fibrant, proposition
$1.5.3$ in \cite{Waldhausen}, implies that the previous square is also
homotopy cartesian and so the canonical morphism
$$ \beta_n^{\sharp}: K(\mathcal{A})^c_n \rightarrow
\Omega K(\mathcal{A})^c_{n+1}$$
is a weak equivalence in $\mathsf{L}_{\widetilde{\mathcal{E}^s_{un}}} \mathsf{L}_{\Sigma,P}
\mathsf{Fun}(\mathsf{dgcat}_f^o, Sset_{\bullet})$. We now show that
the structure morphism $\beta_n$, which corresponds to
$\beta_n^{\sharp}$ by adjunction, see \cite{Waldhausen}, is also a
weak equivalence. The derived adjunction $(S^1\wedge-, \mathbb{R} \Omega(-))$
induces the following commutative diagram
$$
\xymatrix{
S^1 \wedge K(\mathcal{A})^c_n \ar[dr]_{S^1 \wedge^{\mathbb{L}}
  \beta_n^{\sharp}} \ar[r] &
K(\mathcal{A})^c_{n+1} \\
 & S^1 \wedge^{\mathbb{L}} \Omega K(\mathcal{A})_{n+1}^c \ar[u]^{\sim}
}
$$
in $\mathsf{Ho}(\mathsf{L}_{\widetilde{\mathcal{E}^s_{un}}} \mathsf{L}_{\Sigma,P}
\mathsf{Fun}(\mathsf{dgcat}_f^o, Sset_{\bullet}))$ where the vertical
arrow is an isomorphism since the previous square is homotopy
bicartesian. This shows that the induced morphism
$$S^1 \wedge K(\mathcal{A})_n^c  \longrightarrow
K(\mathcal{A})^c_{n+1}$$
is an isomorphism in $\mathsf{Ho}(\mathsf{L}_{\widetilde{\mathcal{E}^s_{un}}} \mathsf{L}_{\Sigma,P}
\mathsf{Fun}(\mathsf{dgcat}_f^o, Sset_{\bullet}))$ and so $\beta_n$ is
a weak equivalence.

Now to prove that $\eta$ is a componentwise weak equivalence, we
proceed by induction~: observe that the zero component of the morphism $\eta$ is the
identity. Now suppose that the $n$-component of $\eta$ is a weak
equivalence. The $n+1$-component of $\eta$ is the composition
of $\beta_{n+1}$, which is a weak equivalence, with the suspension of the
$n$-component of $\eta$, which by
proposition~\ref{Cisin} is also a weak equivalence.

This proves the theorem.
\end{proof}

Let $\mathcal{A}$ and $\mathcal{B}$ be small dg categories.
We denote by
$\mathsf{Hom}^{\mathsf{Sp}^{\mathbb{N}}}(-,-)$ the
spectrum of morphisms in $\mathsf{Sp}^{\mathbb{N}}(\mathsf{L}_{\widetilde{\mathcal{E}^s_{un}}} \mathsf{L}_{\Sigma,P}
\mathsf{Fun}(\mathsf{dgcat}_f^o, Sset_{\bullet}))$.

\begin{theorem}\label{corep}
We have the following weak equivalence of spectra
$$ \mathsf{Hom}^{\mathsf{Sp}^{\mathbb{N}}}(\mathcal{U}_a(\mathcal{A}),\mathcal{U}_a(\mathcal{B})[1])
\stackrel{\sim}{\rightarrow}
K^c(\mathsf{rep}_{mor}(\mathcal{A},\mathcal{B}))\,,$$
where $K^c(\mathsf{rep}_{mor}(\mathcal{A},\mathcal{B}))$ denotes Waldhausen's connective $K$-theory spectrum of
$\mathsf{rep}_{mor}(\mathcal{A},\mathcal{B})$.

In particular, we
have the following weak equivalence of simplicial sets
$$\mathsf{Map}(\mathcal{U}_a(\mathcal{A}),\mathcal{U}_a(\mathcal{B})[1])
\stackrel{\sim}{\rightarrow}
|N.wS_{\bullet}\mathsf{rep}_{mor}(\mathcal{A},\mathcal{B})|$$
and so the isomorphisms
$$\pi_{i+1}
\mathsf{Map}(\mathcal{U}_a(A),\mathcal{U}_a(\mathcal{B})[1])
\stackrel{\sim}{\rightarrow} K_i(\mathsf{rep}_{mor}(\mathcal{A},\mathcal{B})), \,\,\, \forall i \geq 0\,.$$
\end{theorem}

\begin{proof}
Suppose first that $\mathcal{A}$ belongs to $\mathsf{dgcat}_f$. In
this case $\mathcal{U}_a(\mathcal{A})$ identifies with the suspension spectrum
$$\Sigma^{\infty}|N.w\mathsf{rep}_{mor}(?,\mathcal{A})|$$
which is cofibrant in  $\mathsf{Sp}^{\mathbb{N}}(\mathsf{L}_{\widetilde{\mathcal{E}^s_{un}}} \mathsf{L}_{\Sigma,P}
\mathsf{Fun}(\mathsf{dgcat}_f^o, Sset_{\bullet}))$. Since
$\mathcal{A}$ belongs to $\mathsf{dgcat}_f$ and by
theorem~\ref{fibres} we have the following equivalences
$$
\begin{array}{rcl}
\mathsf{Hom}^{\mathsf{Sp}^{\mathbb{N}}}(\mathcal{U}_a(\mathcal{A}),\mathcal{U}_a(\mathcal{B})[1]) & \simeq
& \mathsf{Hom}^{\mathsf{Sp}^{\mathbb{N}}}(\mathcal{U}_a(\mathcal{A}), K^c(\mathcal{B}))\\
 & \simeq & K^c(\mathcal{B})(\mathcal{A})\\
 & \simeq & K^c(\mathsf{rep}_{mor}(\mathcal{A},\mathcal{B}))\,.
\end{array}
$$
Now, for a dg category $\mathcal{A}$ we have an equivalence
$$ \underset{i \in I}{\mbox{hocolim}} \, \mathcal{A}_i
\stackrel{\sim}{\longrightarrow} \mathcal{A}\,,$$
where $I$ is a filtered category and $\mathcal{A}_i$ belongs to
$\mathsf{dgcat}_f$. Since $\mathcal{U}_a$ commutes with filtered
homotopy colimits, we have
$$
\begin{array}{rcl}
\mathsf{Hom}^{\mathsf{Sp}^{\mathbb{N}}}(\mathcal{U}_a(A),\mathcal{U}_a(\mathcal{B})[1]) & \simeq
& \mathsf{Hom}^{\mathsf{Sp}^{\mathbb{N}}}(\underset{i \in
  I}{\mbox{hocolim}}\, \mathcal{U}_a(\mathcal{A}_i), \mathcal{U}_a(\mathcal{B})[1])\\
 & \simeq & \underset{i \in I}{\mbox{holim}} \, \mathsf{Hom}^{\mathsf{Sp}^{\mathbb{N}}}(\mathcal{U}_a(A),\mathcal{U}_a(\mathcal{B})[1])\\
 & \simeq & \underset{i \in I}{\mbox{holim}} \, K^c(\mathsf{rep}_{mor}(\mathcal{A}_i,\mathcal{B}))\\
 & \simeq & K^c(\mathsf{rep}_{mor}(\underset{i \in
   I}{\mbox{hocolim}} \, \mathcal{A}_i,\mathcal{B}))\\
 & \simeq & K^c(\mathsf{rep}_{mor}(\mathcal{A},\mathcal{B}))\,.
\end{array}
$$
This proves the theorem.
\end{proof}

\begin{remark}
Remark that if in the above theorem, we consider $\mathcal{A}=k$, we
have
$$ \mathsf{Hom}^{\mathsf{Sp}^{\mathbb{N}}}(\mathcal{U}_a(k),\mathcal{U}_a(\mathcal{B})[1])
\stackrel{\sim}{\rightarrow}
K^c(\mathcal{B})\,.$$
This shows that Waldhausen's connective $K$-theory spectrum becomes
co-representable in $\mathcal{M}_{dg}^{add}$.
\end{remark}

\section{Concluding remarks}\label{vistas}

By the universal properties of $\mathcal{U}_u$, $\mathcal{U}_a$ and $\mathcal{U}_l$, we
obtain the following diagram~:
$$
\xymatrix{
\mathsf{HO}(\mathsf{dgcat}) \ar[r]^-{\mathcal{U}_u}
\ar[dr]^-{\mathcal{U}_a} \ar[ddr]_-{\mathcal{U}_l}
& \mathcal{M}^{unst}_{dg} \ar@{.>}[d]\\
& \mathcal{M}^{add}_{dg}
\ar@{.>}[d]^{\Phi} \\
& \mathcal{M}^{loc}_{dg} \,.
}
$$
Notice that Waldhausen's connective $K$-theory is an example of an
additive invariant which is NOT a localizing one, see
\cite{ICM}. Waldhausen's connective $K$-theory becomes co-representable
in $\mathcal{M}_{dg}^{add}$ by theorem~\ref{corep}.

To the best of the author's knowledge, this is the first
conceptual characterization of Quillen-Waldhausen $K$-theory
\cite{Quillen} \cite{Waldhausen}.

An analoguous result should be
true for non-connective $K$-theory and the morphism $\Phi$ should be
thought of as co-representing {\em `the passage from additivity to localization'}.

\part{Ramifications de la th{\'e}orie homotopique des DG-cat{\'e}gories}
\chapter{The $Q$-model for the Morita homotopy theory of DG categories}

\textit{\small{Ce chapitre correspond {\`a} l'article \cite{dgquot}.}}

\section{Introduction}
A differential graded (=dg) category is a category enriched in the
category of complexes of modules over some commutative base ring
$k$. Dg categories (and their close cousins:
$A_{\infty}$-categories) provide a framework for `homological
geometry' and for `non commutative algebraic geometry' in the sense of
Drinfeld, Manin, Kontsevich, $\ldots$. This has motivated much
foundational work (\emph{cf.} \cite{ICM} for a survey) and notably
To{\"e}n's construction \cite{Toen} of an internal $\mathsf{Hom}$-functor
for the homotopy category of dg categories with respect to the class
of quasi-equivalences (i.e. the dg functors $F:\mathcal{A} \rightarrow
\mathcal{B}$ which induce quasi-isomorphisms in the morphism complexes
and equivalences $\mathsf{H}^0(\mathcal{A}) \rightarrow
\mathsf{H}^0(\mathcal{B})$). This construction was extended in \cite{IMRN} \cite{addendum} to the class of Morita dg
functors (i.e. the functors $F:\mathcal{A} \rightarrow
\mathcal{B}$ which induce an equivalence $\mathcal{D}(\mathcal{A})
\stackrel{\sim}{\rightarrow} \mathcal{D}(\mathcal{B})$
between derived categories).

One of the main problems which To{\"e}n had to overcome in \cite{Toen} was
the fact that the natural model structure \cite{cras} on the category
of dg categories is not compatible with the tensor product~: In fact,
the tensor product of two cofibrant dg categories is not cofibrant in
general, in the same way as the tensor product of two noncommutative
free algebras is not noncommutative free, in general. As a
consequence, the category of dg categories is not a monoidal model
category in the sense of \cite{Hovey} and its $\mathsf{Hom}$-functor
cannot be derived to yield the internal $\mathsf{Hom}$-functor for
the homotopy category.

In this paper, we propose a new model for the Morita homotopy category
of dg categories (for the case where the ground ring $k$ is a
field). Its objects are the localization pairs \cite{cyclichomology}, i.e. the
inclusions $\mathcal{A}_0 \subset \mathcal{A}_1$ of full dg
subcategories. Morphisms are defined in the natural way. By
definition, a morphism $\mathcal{A} \rightarrow \mathcal{A}'$ is a
weak equivalence iff it induces a Morita dg functor $\mathcal{A}_1/\mathcal{A}_0 \rightarrow \mathcal{A}_1'/\mathcal{A}_0'$
in the Drinfeld {\em dg quotients} \cite{Drinfeld}. We therefore call
this model category the {\em $Q$-model}. We show that it admits a
natural closed monoidal structure which induces the correct tensor
product in the homotopy category. The $Q$-model is not a monoidal
model category, either. Our main theorem (\ref{rep}) states that
nevertheless, its internal $\mathsf{Hom}$-functor admits a derived
functor which yields the correct closed structure in the homotopy
category. We thus obtain a new description of To{\"e}n's internal
$\mathsf{Hom}$-objects. It is expected that it will help to clarify
their links with the $A_{\infty}$-formalism, {\em cf.}
\cite{keller}, and in a future `derived deformation theory' {\em cf.}
\cite{lowen}, analogous to the deformation theory for abelian
categories as developed in \cite{lowen-Vdb}.

\section{Preliminaries}
In what follows, $k$ will denote a field. The
tensor product $\otimes$ will denote the tensor product over $k$. Let
$\mathsf{Ch}(k)$ denote the category of complexes over $k$. By a
{\it dg category}, we mean a differential graded $k$ category, see definition~\ref{dgcategorie}. For a dg category
$\mathcal{A}$, we denote by $\mathcal{C}_{dg}(\mathcal{A})$ the dg
category of right $\mathcal{A}$ dg modules and by $\, \widehat{ } : \mathcal{A} \rightarrow
\mathcal{C}_{dg}(\mathcal{A})$ the Yoneda dg functor. We write $\mathsf{dgcat}$
for the
category of small dg categories.
It is proven in theorem~\ref{theorem2}, that the
category $\mathsf{dgcat}$ admits a structure of cofibrantly generated
model category whose weak equivalences are the Morita dg functors
(i.e. the dg functors $F:\mathcal{A} \rightarrow \mathcal{B}$ which
induce an equivalence $\mathcal{D}(\mathcal{B})
\stackrel{\sim}{\rightarrow} \mathcal{D}(\mathcal{A})$ between derived
categories). We have
at our disposal an explicit set $I=\{Q, S(n)\}$ of
generating cofibrations and  an explicit set $J = \{R(n), F(n), I_n(k_0, \ldots, k_n),
  Lh_n(k_0, \ldots, k_n), C\}$ of generating trivial cofibrations.

\section{Homotopy of DG functors}

Let $\mathcal{B}$ be a dg category.

\begin{definition}\label{1path1}
Let $P(\mathcal{B})$ be the dg category, see~\cite[2.9]{Drinfeld}, whose
objects are the closed morphisms of degree zero in $\mathcal{B}$
$$ X \stackrel{f}{\longrightarrow} Y\,,$$
that become invertible in $\mathsf{H}^0(\mathcal{B})$.
We define the complex of morphisms
$$ \mathsf{Hom}_{P(\mathcal{B})}(X\stackrel{f}{\rightarrow}Y,
X'\stackrel{f'}{\rightarrow}Y')$$
as the homotopy pull-back in $\mathsf{Ch}(k)$ of the diagram
$$
\xymatrix{
& \mathsf{Hom}_{\mathcal{B}}(Y,Y') \ar[d]^{f^*} \\
\mathsf{Hom}_{\mathcal{B}}(X,X') \ar[r]^{f'_*} & \mathsf{Hom}_{\mathcal{B}}(X,Y')\,.
}
$$
By this homotopy pull-back we mean the complex
$$ \{ u \in
\mathsf{Hom}_{\mathcal{C}_{dg}(\mathcal{B})}(\mbox{cone}(\widehat{f}),
\mbox{cone}(\widehat{f})) | \,\, \pi' u i =0\}\,,$$
where $i: \widehat{Y} \rightarrow \mbox{cone}(\widehat{f})$ and
$\pi':\mbox{cone}(\widehat{f'}) \rightarrow X'[1]$ are the natural
morphisms in $\mathcal{C}_{dg}(\mathcal{B})$, see
\cite[2.9]{Drinfeld}.
\end{definition}

\begin{remark}
We have a natural commutative diagram in $\mathsf{dgcat}$
$$
\xymatrix{
\mathcal{B} \ar[rr]^{\Delta} \ar[dr]_i & & \mathcal{B}\times
\mathcal{B} \\
& P(\mathcal{B}) \ar[ur]_{p_0\times p_1} & \,,
}
$$
where $i$ is the dg functor that associates to an object $B$ of
$\mathcal{B}$ the morphism $B \stackrel{Id}{\rightarrow}B$ and $p_0$,
resp. $p_1$, is the dg functor that sends a closed morphism $X
\stackrel{f}{\rightarrow} Y$ to $X$, resp. $Y$.
\end{remark}

\begin{lemma}\label{lempath}
The dg category $P(\mathcal{B})$ is a path object for $\mathcal{B}$,
see~\cite{Hirschhorn}, in the Quillen model
structure described in theorem~\ref{mal}.
\end{lemma}

\begin{proof}
We prove that the dg functor $i$ is a quasi-equivalence. Clearly
the dg functor $i$ induces a quasi-isomorphism in $\mathsf{Ch}(k)$
$$ \mathsf{Hom}_{\mathcal{B}}(X,Y) \stackrel{\sim}{\longrightarrow}
\mathsf{Hom}_{P(\mathcal{B})}(i(X), i(Y))\,,$$
for every object $X,Y \in \mathcal{B}$.
Remark that the functor $\mathsf{H}^0(i)$ is also essentially surjective.
In fact, let $X \stackrel{f}{\rightarrow}Y$ be an object of
$P(\mathcal{B})$. Consider the following morphism in $P(\mathcal{B})$
from $i(X)$ to $X \stackrel{f}{\rightarrow}Y$,
$$
\xymatrix{
X \ar[r]^{Id} \ar@{=}[d]  \ar@{}[dr]|{h=0} & X \ar[d]^f \\
X \ar[r]^f & Y\,,
}
$$
where $h$ denotes the zero homotopy. Remark that it becomes an
isomorphism in $\mathsf{H}^0(P(\mathcal{B}))$ simply because $f$
becomes an isomorphism in $\mathsf{H}^0(\mathcal{B})$. This proves
that the dg functor $i$ is a quasi-equivalence.
We will now show that the dg functor $p_0 \times p_1$ is a fibration
in the Quillen model structure of theorem~\ref{mal}.
Remark first, that by definition of $P(\mathcal{B})$ the dg functor
$p_0 \times p_1$ induces a surjective morphism in $\mathsf{Ch}(k)$
$$
\xymatrix{
\mathsf{Hom}_{P(\mathcal{B})}(X \stackrel{f}{\rightarrow}Y, X'
\stackrel{f'}{\rightarrow}Y') \ar@{->>}[rr]^{p_0\times
  p_1} & & \mathsf{Hom}_{\mathcal{B}}(X,X')\times
\mathsf{Hom}_{\mathcal{B}}(Y,Y')
}
$$
for every object $X \stackrel{f}{\rightarrow}Y$ and
$X'\stackrel{f'}{\rightarrow} Y$ in $P(\mathcal{B})$.
We will now show that contractions lift along the dg functor $P(\mathcal{B})
\stackrel{p_0\times p_1}{\longrightarrow} \mathcal{B}\times
\mathcal{B}$. Let $X \stackrel{f}{\rightarrow}Y$
be an object of $P(\mathcal{B})$. Remark that a contraction of $X
\stackrel{f}{\rightarrow}Y$ in $P(\mathcal{B})$ corresponds exactly to
the following morphisms in $\mathcal{B}$, $c_X \in
\mathsf{Hom}^{-1}_{\mathcal{B}}(X,X)$, $c_Y \in
\mathsf{Hom}^{-1}_{\mathcal{B}}(Y,Y)$ and $h \in
\mathsf{Hom}^{-2}_{\mathcal{B}}(X,Y)$ satisfying the relations
$d(c_X)=\mathbf{1}_X$, $d(c_Y)=\mathbf{1}_Y$ and $d(h)=c_Y\circ f +
f\circ c_X$.
Suppose now, that we have a contraction $(c_1,c_2)$ of $(X,Y)$
in $\mathcal{B}\times\mathcal{B}$. We can lift this contraction by
considering $c_X=c_1$, $c_Y=c_2$ and $h=c_2\circ f \circ c_1$. This
shows that contractions lift along the dg functor $P(\mathcal{B}) \stackrel{p_0 \times
  p_1}{\longrightarrow} \mathcal{B}\times\mathcal{B}$. We have the following equivalence of dg categories
$$ \mbox{pre-tr}(P(\mathcal{B})) \stackrel{\sim}{\longrightarrow}
P(\mbox{pre-tr}(\mathcal{B}))\,,$$
where $\mbox{pre-tr}$ denotes the pre-triangulated hull of a dg category,
see~\cite[2.4]{Drinfeld}. This implies that the dg functor $p_0 \times p_1$
is a fibration.
\end{proof}

Let $\mathcal{A}$ be a cofibrant dg category and $F,G:\mathcal{A}
\rightarrow \mathcal{B}$ dg functors. Since every dg category is
fibrant in the Quillen model structure of theorem~\ref{mal}, see remark~\ref{toutfibrant}, the dg functors $F$ and $G$ are
homotopic if and only if there exists a dg functor $H:\mathcal{A} \rightarrow
P(\mathcal{B})$ that makes the following diagram commute
$$
\xymatrix{
& \mathcal{B} \\
\mathcal{A} \ar[r]^-H \ar[ur]^F \ar[dr]_G & P(\mathcal{B}) \ar[u]_{P_0}
\ar[d]^{P_1} \\
& \mathcal{B}\,\,,
}
$$
see~\cite{Hirschhorn}.

\begin{remark}\label{cic}
Remark that a dg functor $H$ as above corresponds exactly,
see~\cite{cyclichomology}, to:
\begin{enumerate}
\item[-] a morphism $\eta A:F(A) \rightarrow G(A)$ of
  $Z^0(\mathcal{B})$ which becomes invertible in
  $\mathsf{H}^0(\mathcal{B})$ for all $A \in \mathcal{A}$ (but which
  will not be functorial in $A$, in general) and
\item[-] a morphism of graded $k$-modules homogeneous of degree $-1$
$$ h=h(A,B): \mathsf{Hom}_{\mathcal{A}}(A,B) \rightarrow
\mathsf{Hom}_{\mathcal{B}}(F(A),G(B))\,,$$
for all $A,B \in \mathcal{A}$ such that we have
$$(\eta B)(F(f)) -(G(f)(\eta A)=d(h(f))+h(d(f))$$
and
$$h(fg)=h(f)(F(g))+(-1)^n(G(f))h(g)\,$$
for all composable morphismes $f,g$ of $\mathcal{A}$, where $f$ is of
degree $n$.
\end{enumerate}
It is shown in~\cite{cyclichomology} that if we have a dg functor $H$
as above and the dg category $\mathcal{B}$ is stable under cones, we
can construct a sequence of dg functors
$$ F \rightarrow I \rightarrow G[1]\,,$$
where $I(A)$ is a contractible object of $\mathcal{B}$, for all $A \in \mathcal{B}$.
\end{remark}

\section{$Q$-model structure}\label{model}

\begin{definition}
A {\em localization pair} $\mathcal{A}$ is given by a small dg category
$\mathcal{A}_1$ and a full dg subcategory $\mathcal{A}_0 \subset
\mathcal{A}_1$. A {\em morphism} $F:\mathcal{A} \rightarrow
\mathcal{B}$ of localization pairs is given by a commutative square
$$
\xymatrix{
\mathcal{A}_0 \ar[d]_{F_0} \ar@{^{(}->}[r] & \mathcal{A}_1 \ar[d]^{F_1} \\
\mathcal{B}_0 \ar@{^{(}->}[r] & \mathcal{B}_1
}
$$
of dg functors.
\end{definition}
We denote by $\mathsf{Lp}$ the category of localization pairs.

Let $\mathcal{A}$ be a localization pair.

\begin{definition}
The {\em dg quotient} of $\mathcal{A}$, see~\cite{Drinfeld}, is the dg
category $\mathcal{A}_1/\mathcal{A}_0$
obtained from $\mathcal{A}_1$ by introducing a new morphism $h_X$ of
degree $-1$ for every object $X$ of $\mathcal{A}_0$ and by imposing the
relation $d(h_X)={\bf 1}_X$.
\end{definition}

\subsection{Morita model structure}\label{secMorita}

Let $L$ be the category with two objects $0$ and $1$ and with a unique
non identity morphism $0 \rightarrow 1$.

\begin{remark}
An immediate application of Theorem $11.6.1$ from \cite{Hirschhorn}
implies that the category $\mathsf{dgcat}^L$, i.e. the category of morphisms
in $\mathsf{dgcat}$, admits a structure of cofibrantly generated model
category whose weak equivalences $W$ are the componentwise Morita dg functors and with generating cofibrations $\mathbf{F}^L_{I}$ and
generating trivial cofibrations $\mathbf{F}^L_{J}$, where we use the
notation of \cite{Hirschhorn}:
\end{remark}
The functor $\mathbf{F}^i_?$,
$i=0,1$, from $\mathsf{dgcat}$ to $\mathsf{dgcat}^L$ is left adjoint
to the natural evaluation functor $Ev_i$, $i=0,1$, from
$\mathsf{dgcat}^L$ to $\mathsf{dgcat}$. By definition, we have $\mathbf{F}^L_{I} =
\mathbf{F}^0_{I} \cup \mathbf{F}^1_{I}$ and  $\mathbf{F}^L_{J} =
\mathbf{F}^0_{J} \cup \mathbf{F}^1_{J}$.

The inclusion functor $U :\mathsf{Lp} \rightarrow \mathsf{dgcat}^L$
admits a left adjoint $S$ which sends an object $G:\mathcal{B}_0
\rightarrow \mathcal{B}_1$ to the localization pair formed by
$\mathcal{B}_1$ and its full dg subcategory $\mathsf{Im}\,G$.

\begin{proposition}\label{naive}
The category $\mathsf{Lp}$ admits a structure of cofibrantly generated model
category whose weak equivalences $W$ are the componentwise Morita dg functors and with generating cofibrations $\mathbf{F}^L_{I}$ and
generating trivial cofibrations $\mathbf{F}^L_{J}$.
\end{proposition}

\begin{proof}
We first prove that $\mathsf{Lp}$ is complete and cocomplete.
Let $\{X_i\}_{i \in I}$ be a diagram in $\mathsf{Lp}$. We remark that
$$ \underset{i \in I}{\mbox{colim}} \,X_i
\stackrel{\sim}{\longrightarrow} S(\underset{i \in I}{\mbox{colim}}
\,U(X_i)) \,,$$
which implies that $\mathsf{L}p$ is cocomplete. The category
$\mathsf{Lp}$ is also complete, since it is stable under products and
equalizers in $\mathsf{dgcat}^L$.
We now prove that conditions $(1)$ and $(2)$ of Theorem $11.3.2$ from \cite{Hirschhorn} are satisfied~:
\begin{enumerate}
\item[(1)] Since $S(\mathbf{F}^L_{I})=\mathbf{F}^L_{I}$ and
  $S(\mathbf{F}^L_{J})=\mathbf{F}^L_{J}$ condition $(1)$ is verified.
\item[(2)] Since the functor $U$ clearly commutes with filtered
  colimits, it is enough to prove the following: let $Y
  \stackrel{G}{\rightarrow} Z$ be an element of the set
  $\mathbf{F}^L_{J}$, $X$ an object in $\mathsf{Lp}$ and $Y \rightarrow X$ a morphism in
  $\mathsf{Lp}$. Consider the following push-out in $\mathsf{Lp}$:
$$
\xymatrix{
Y \ar[r] \ar[d]_G  \ar@{}[dr]|{\lrcorner} & X \ar[d]^{G_{\ast}} \\
Z \ar[r] & Z \underset{Y}{\coprod} X
}
$$
We prove that $U(G_{\ast})$ is a weak equivalence in $\mathsf{dgcat}^L$. We consider two situations:
\begin{enumerate}
\item[-] if $G$ belongs to the set $\mathbf{F}^0_{J} \subset \mathbf{F}^L_{J}$,
  then $U(G_{\ast})$ is a weak-equivalence simply because
  $J-\mbox{cell} \subset W$ in $\mathsf{dgcat}$ , see lemma~\ref{J-cell-dans-W2}.
\item[-] if $G$ belongs to the set $\mathbf{F}^1_{J} \subset \mathbf{F}^L_{J}$,
  then $Ev_1(U(G_{\ast}))$ is a Morita dg functor. In particular it
  induces a quasi-isomorphism in the $\mathsf{Hom}$ spaces and since the
  $0$-component of $G_*$ is the identity on objects, the functor $Ev_0(U(G_*))$
  is also a Morita dg functor. This implies that $U(G_*)$ is a
  weak equivalence and so condition $(2)$ is proven.
\end{enumerate}
This proves the proposition.
\end{enumerate}
\end{proof}

We will now slightly modify the previous Quillen model
structure on $\mathsf{Lp}$.

Let $\sigma$ be the morphism of localization pairs:
$$
\xymatrix{
(\mbox{End}_{\mathcal{K}}(1) \ar@{^{(}->}[r] \ar[d]_{inc}  & \mathcal{K}) \ar@{=}[d]\\
(\mathcal{K} \ar@{=}[r] & \mathcal{K})\,,
}
$$
where $\mathsf{End}_{\mathcal{K}}(1)$ is the dg algebra of
endomorphisms of the object $1$ in $\mathcal{K}$, see section~\ref{secdef},
and $inc$ is the natural inclusion dg-functor. Clearly $\sigma$ is a
componentwise Morita dg functor.
We write $\tilde{\mathbf{F}}^L_{I}$ resp. $\tilde{\mathbf{F}}^L_{J}$ for
the union of $\{\sigma\}$ with $\mathbf{F}^L_{I}$ resp. $\mathbf{F}^L_{J}$.

\begin{proposition}\label{sat}
The category $\mathsf{Lp}$ admits a structure of cofibrantly generated model
category whose weak equivalences $W$ are the componentwise Morita dg functors and
with generating cofibrations $\tilde{\mathbf{F}}^L_{I}$ and
generating trivial cofibrations $\tilde{\mathbf{F}}^L_{J}$.
\end{proposition}

\begin{proof}
The proof will consist in verifying that conditions
$(1) -(6)$ of Theorem $2.1.19$ from \cite{Hovey}
are satisfied. Clearly the class $W$ has the two out of three property
and is closed under retracts. This shows condition $(1)$. Since the localization pair
$(\mbox{End}_{\mathcal{K}}(1) \subset \mathcal{K})$ is small in $\mathsf{Lp}$,
it is clear that the domains of $\tilde{\mathbf{F}}^L_{I}$,
resp. $\tilde{\mathbf{F}}^L_{J}$, are small relative to
$\tilde{\mathbf{F}}^L_{I} - \mbox{cell}$,
resp. $\tilde{\mathbf{F}}^L_{J} - \mbox{cell}$ and so
conditions $(2)$ and $(3)$ are also satisfied. We have
$$ \mathbf{F}^L_{I} - \mbox{inj} = \mathbf{F}^L_{J} - \mbox{inj} \cap W$$
and so by construction
$$ \tilde{\mathbf{F}}^L_{I} - \mbox{inj} = \tilde{\mathbf{F}}^L_{J} - \mbox{inj}
\cap W\,.$$
This shows conditions $(5)$ and $(6)$. We now prove that
$\tilde{\mathbf{F}}^L_{J} - \mbox{cell} \subset W$. Since
$\mathbf{F}^L_{J}-\mbox{cell} \subset W$ it is enough to prove that
pushouts with respect to $\sigma$ belong to $W$.
Let $\mathcal{A}$ be a localization pair and
$$T: (\mbox{End}_{\mathcal{K}}(1) \subset \mathcal{K})
\rightarrow (\mathcal{A}_0 \subset \mathcal{A}_1)$$ a morphism in
$\mathsf{Lp}$. Consider the following push-out in $\mathsf{Lp}$:
$$
\xymatrix{
(\mbox{End}_{\mathcal{K}}(1) \subset \mathcal{K}) \ar[r]^-{T}
\ar[d]_{\sigma} \ar@{}[dr]|{\lrcorner} & (\mathcal{A}_0 \subset
\mathcal{A}_1) \ar[d]^R \\
(\mathcal{K} = \mathcal{K}) \ar[r] & (\mathcal{U}_0 \subset
\mathcal{U}_1) \,.
}
$$
We remark that the morphism $T$ corresponds to specifying an homotopy
equivalence in $\mathcal{A}_1$ from an object $X$ to an object $Y$,
where the object $X$ belongs to $\mathcal{A}_0$. Clearly $\mathcal{U}_1 = \mathcal{A}_1$
and $\mathcal{U}_0$ identifies with the full dg-subcategory of
$\mathcal{U}_1$ whose objects are $Y$ and those of $\mathcal{A}_0$.
Since $X$ and $Y$ are homotopy equivalent, the natural dg-functor
$R_0: \mathcal{A}_0 \hookrightarrow \mathcal{U}_0$ is a
quasi-equivalence. This proves condition $(4)$.
The proposition is now proven.
\end{proof}

\begin{remark}
Remark that in this new Quillen model structure on $\mathsf{Lp}$ we
have more cofibrations and fewer fibrations than the Quillen model structure
of proposition~\ref{naive} since the weak equivalences are the same.
\end{remark}

From now on, by Quillen model structure on $\mathsf{Lp}$ we mean that of
proposition~\ref{sat}.

\begin{lemma}\label{fibrSat}
A localization pair $(\mathcal{A}_0 \subset \mathcal{A}_1)$ is fibrant
in $\mathsf{Lp}$ if and only if $\mathcal{A}_0$ and $\mathcal{A}_1$
are Morita fibrant dg categories and $\mathcal{A}_0$ is stable under
homotopy equivalences in $\mathcal{A}_1$.
\end{lemma}

\begin{proof}
A localization pair $(\mathcal{A}_0 \subset \mathcal{A}_1)$ is
fibrant in $\mathsf{Lp}$ if and only if for every morphism $F$ in
$\tilde{\mathbf{F}}^L_{J}$, the following extension problem in
$\mathsf{Lp}$ is solvable:
$$
\xymatrix{
X \ar[r] \ar[d]_F & (\mathcal{A}_0 \subset \mathcal{A}_1) \\
Y  \ar@{-->}[ur] & \,.
}
$$
If $F$ belongs to $\mathbf{F}^L_{J}$ this means that $\mathcal{A}_0$ and $\mathcal{A}_1$
are fibrant and if $F=\sigma$, remark that it corresponds
exactly to the
statement that $\mathcal{A}_0$ is stable under homotopy equivalences
in $\mathcal{A}_1$.
\end{proof}

\begin{lemma}\label{cofibrant}
If the localization pair $\mathcal{A}$ is
cofibrant in $\mathsf{Lp}$ then $\mathcal{A}_1$ is cofibrant in $\mathsf{dgcat}$.
\end{lemma}

\begin{proof}
We need to construct a lift to the following problem~:
$$
\xymatrix{
 & \mathcal{C} \ar@{->>}[d]^P_{\sim} \\
\mathcal{A}_1 \ar[r] & \mathcal{B} \,,
}
$$
where $P$ is a trivial fibration in $\mathsf{dgcat}$, see proposition~\ref{sat},
and $\mathcal{A}_1 \rightarrow \mathcal{B}$ is a dg-functor. Consider
the following diagram in $\mathsf{Lp}$:
$$
\xymatrix{
 & \mathbf{F}^0_{\mathcal{C}} \ar@{->>}[d]_{\sim}^{\mathbf{F}^0_P} \\
\mathcal{A} \ar[r] \ar@{.>}[ur] & \mathbf{F}^0_{\mathcal{B}}\,.
}
$$
where $\mathcal{A} \rightarrow \mathbf{F}^0_{\mathcal{B}}$ is the natural morphism of
localization pairs. Remark that $\mathbf{F}^0_P$ belongs to
$\sigma-\mbox{inj} \cap \mathbf{F}^L_{I} - \mbox{inj}$ and so is a trivial
fibration in $\mathsf{Lp}$. Since $\mathcal{A}$ is
cofibrant in $\mathsf{Lp}$ we have a lifting $\mathcal{A} \rightarrow  \mathbf{F}^0_{\mathcal{C}}$ that when
restricted to the $1$-component gives us the searched lift
$\mathcal{A}_1 \rightarrow \mathcal{C}$. This proves the lemma.
\end{proof}

\subsection{$Q$-model structure}

\begin{definition}
Let $Q : \mathsf{Lp} \rightarrow \mathsf{Lp}$ be the functor that
sends a localization pair $\mathcal{A}$ to the localization pair
$$ \overline{\mathcal{A}_0} \hookrightarrow \mathcal{A}_1/\mathcal{A}_0\,,$$
 where $\overline{\mathcal{A}_0}$ is the full dg-subcategory
of $\mathcal{A}_1/\mathcal{A}_0$ whose objets are those of
$\mathcal{A}_0$.
\end{definition}

\begin{remark}
Remark that we have natural morphisms
$$ \eta_{\mathcal{A}} : (\mathcal{A}_0
\subset \mathcal{A}_1) \rightarrow (\overline{\mathcal{A}_0} \subset
\mathcal{A}_1/\mathcal{A}_0)\,$$
in $\mathsf{Lp}$.
\end{remark}

\begin{definition}
A morphism of localization pairs $F: \mathcal{A} \rightarrow \mathcal{B}$ is a $Q$-{\it weak equivalence} if the induced
morphism $Q(F)$ is a weak equivalence in the Quillen model structure of
proposition~\ref{sat}.
\end{definition}

\begin{remark}
Remark that since the objects of $\overline{\mathcal{A}_0}$
and $\overline{\mathcal{B}_0}$ are all contractible, the dg-functor
$\overline{\mathcal{A}_0} \rightarrow \overline{\mathcal{B}_0}$ is
clearly a Morita dg functor and so the morphism $F$ is a $Q$-weak
equivalence if and only if the induced dg-functor
$\mathcal{A}_1/\mathcal{A}_0 \rightarrow \mathcal{B}_1/\mathcal{B}_0$
is a Morita dg functor.
\end{remark}

\begin{definition}
A morphism in $\mathsf{Lp}$ is a {\em cofibration} if it is
  one for the Quillen model structure of Proposition~\ref{sat} and it is a $Q$-{\em fibration} if it has the right lifting property
  with respect to all cofibrations of $\mathsf{Lp}$ which are $Q$-weak equivalences.
\end{definition}

\begin{theorem}\label{main}
The category $\mathsf{Lp}$ admits a structure of Quillen model
  category, the {\em $Q$-model}, whose weak equivalences are the $Q$-weak equivalences, whose
  cofibrations are the cofibrations of $\mathsf{Lp}$ and whose fibrations are
  the $Q$-fibrations.
\end{theorem}

The proof will consist of adapting the general arguments from
chapter~X from \cite{Jardine} to our situation. We start with some remarks:
\begin{enumerate}
\item[{\bf A1}] Since $k$ is a field, every complex $X \in
  \mathsf{Ch}(k)$ is $k$-flat and so by theorem $3.4$
  from \cite{Drinfeld} the functor $Q$ preserves weak equivalences.
\item[{\bf A2}] The morphisms of localization pairs:
$$
\xymatrix{
Q(\mathcal{A})
\ar@<1ex>[r]^{\eta_{Q(\mathcal{A})}}  \ar[r]_{Q(\eta_{\mathcal{A}})}  &  QQ(\mathcal{A})
}
$$
are weak equivalences in $\mathsf{Lp}$. This follows from the fact that in
both cases we are introducing contractions to objects that are already
contractible and that the functor $Q$ is the identity functor on objects.
\end{enumerate}

\begin{lemma}\label{weak fibration}
A morphism $F:\mathcal{A}
\rightarrow \mathcal{B}$ is a fibration and
a weak equivalence of $\mathsf{Lp}$ if and only if it is a $Q$-weak
equivalence and a $Q$-fibration.
\end{lemma}

\begin{proof}
Since condition {\bf A1} is verified (i.e. the functor $Q$ preserves
weak equivalences) we can use the proof of lemma
$4.3$ in chapter~X from \cite{Jardine}.
\end{proof}

\begin{nonexample}
Remark that the Quillen model structure of proposition~\ref{sat}
is not right proper, see \cite{Hirschhorn}.\newline
Let $\mathcal{B}$ be your favorite Morita fibrant dg category, whose
derived category $\mathcal{D}(\mathcal{B})$ is not trivial. In
particular the dg functor $P: \mathcal{B} \rightarrow 0$, where $0$
denotes the terminal object in $\mathsf{dgcat}$ is a fibration. Let $\mathcal{A}$ be the dg category with one object $1$
and whose dg algebra of endomorphisms of $1$ is $k$. Consider the
following diagram~:
$$
\xymatrix{
 & \mathcal{B} \ar@{->>}[d]^{i_0 \circ P} \\
\mathcal{A} \ar[r]_{j_{\mathcal{A}}} & 0 \coprod \mathcal{A}\,.
}
$$
Clearly $j_{\mathcal{A}}$ is a Morita dg functor and remark that the
dg functor $i_0 \circ P$ is a fibration, since the object $1$
in $\mathcal{A}$ is not contractible. This implies that in the fiber
product
$$
\xymatrix{
\emptyset \ar[d] \ar[r] \ar@{}[dr]|{\ulcorner} & \mathcal{B} \ar@{->>}[d]^{i_0 \circ P} \\
\mathcal{A} \ar[r]_{j_{\mathcal{A}}} & 0 \coprod \mathcal{A}\,,
}
$$
the dg functor $\emptyset \rightarrow \mathcal{B}$ is not a Morita dg functor and so this Quillen model structure is not right proper.
This implies that the Quillen model structure of proposition~\ref{sat}
is also not right proper. Apply the functor $\mathbf{F}^0_?$ from
$\mathsf{dgcat}$ to $\mathsf{Lp}$ to the previous fiber product~:
$$
\xymatrix{
\emptyset = \mathbf{F}^0_{\emptyset} \ar@{}[dr]|{\ulcorner}   \ar[r] \ar[d] &
\mathbf{F}^0_{\mathcal{B}} \ar@{->>}[d]^{\mathbf{F}^0_{i_0 \circ P}} \\
\mathbf{F}^0_{\mathcal{A}} \ar[r]_{\mathbf{F}^0_{j_{\mathcal{A}}}} &
\mathbf{F}^0_{0 \coprod \mathcal{A}}\,.
}
$$
We have a fiber product since the functor $\mathbf{F}^0_?$
preserves limits. Clearly $\mathbf{F}^0_{j_{\mathcal{A}}}$ is a weak
equivalence in $\mathsf{Lp}$ and remark that the morphism
$\mathbf{F}^0_{i_0\circ P}$ belongs to $\sigma-\mbox{inj} \cap
\mathbf{F}^L_{J}-\mbox{inj}$, which implies that it is a fibration in $\mathsf{Lp}$.
\end{nonexample}

Nevertheless we have the following lemma.

\begin{lemma}\label{fiber}
Let $\mathcal{A}$ be a localization pair such
that the natural morphism
$$ \eta_{\mathcal{A}}: \mathcal{A} \stackrel{\sim}{\longrightarrow} Q(\mathcal{A})$$
is a weak equivalence in $\mathsf{Lp}$. Let $F:\mathcal{W} \rightarrow
Q(\mathcal{A})$  be a fibration in $\mathsf{Lp}$. Then the
morphism
$$ \eta^*_{\mathcal{A}}: \mathcal{W} \underset{Q(\mathcal{A})}{\times} \mathcal{A} \stackrel{\sim}{\longrightarrow} \mathcal{W}$$
is a weak equivalence in $\mathsf{Lp}$.
\end{lemma}

\begin{proof}
We remark that each component of the morphism $\eta_{\mathcal{A}}$ is the identity functor on the objects of
the dg categories involved. Since fiber products in $\mathsf{Lp}$ are
calculated componentwise, we conclude that each component of the morphism
$\eta^*_{\mathcal{A}}$ is the identity
functor on the objects. Let $X$ and $Y$ be arbitrary objects of
$\mathcal{W}_1$. We remark that we have the following fiber
product in  $\mathsf{Ch}(k)$~:
$$
\xymatrix{
\mbox{Hom}_{\mathcal{W}_1
  \underset{\mathcal{A}_1/\mathcal{A}_0}{\times} \mathcal{A}_1}(X,Y)
\ar[rr] \ar[d]_{\eta^*(F_1X,F_1Y)} \ar@{}[drr]|{\ulcorner} & &
\mbox{Hom}_{\mathcal{A}_1}(F_1X,F_1Y)
\ar[d]_{\sim}^{\eta(F_1X,F_1Y)}\\
\mbox{Hom}_{\mathcal{W}_1}(X,Y) \ar@{->>}[rr]_-{F_1(X,Y)} & & \mbox{Hom}_{\mathcal{A}_1/\mathcal{A}_0}(F_1X,F_1Y)\,.
}
$$
Since $F$ is a fibration in $\mathsf{Lp}$, $F_1(X,Y)$ is a fibration in
 the projective model structure on $\mathsf{Ch}(k)$ and since this
 Quillen model structure on $\mathsf{Ch}(k)$ is right proper, $\eta^*(F_1X,F_1Y)$ is a quasi-isomorphism.
We could do the same argument for $X$ and $Y$ objects in $\mathcal{W}_0$ instead of
$\mathcal{W}_1$. This proves the lemma.
\end{proof}

\begin{lemma}\label{fibration}
Suppose that $F:\mathcal{A}_0 \rightarrow \mathcal{B}$ is a fibration in
$\mathsf{Lp}$ and that $\eta_{\mathcal{A}}$ and
$\eta_{\mathcal{B}}$ are weak equivalences
of $\mathsf{Lp}$. Then $F$ is a $Q$-fibration.
\end{lemma}

\begin{proof}
Consider exactly the same proof as for lemma $4.4$ in chapter~X from \cite{Jardine},
but use lemma~\ref{fiber} instead of the right properness assumption
on $\mathsf{Lp}$.
\end{proof}

\begin{lemma}\label{facto}
Any morphism $F: Q(\mathcal{A})
\rightarrow Q(\mathcal{B})$ has a factorization
$F=P \circ I$ where $P:\mathcal{Z} \rightarrow Q(\mathcal{B})$ is a
$Q$-fibration and $I: Q(\mathcal{A}) \rightarrow \mathcal{Z}$ is a cofibration and a $Q$-weak equivalence.
\end{lemma}

\begin{proof}
Since lemma~\ref{fibration} and conditions {\bf A1} and {\bf A2}
are satisfied, we consider the proof of lemma $4.5$ in chapter~X from
\cite{Jardine}.
\end{proof}

Let $\mathcal{A}$ be a localization pair. By
condition {\bf A2} we know that the natural morphism:
$$ \eta_{\mathcal{A}}:(\mathcal{A}_0 \subset
\mathcal{A}_1) \longrightarrow (\overline{\mathcal{A}_0} \subset
\mathcal{A}_1/\mathcal{A}_0)$$
is a $Q$-weak equivalence in $\mathsf{Lp}$.

\begin{lemma}\label{Q-weak}
Let $F:\mathcal{Z} \rightarrow Q(\mathcal{A})$ be a
fibration in $\mathsf{Lp}$. Then the induced morphism
$$ \eta^*_{\mathcal{A}}: \mathcal{Z} \underset{Q(\mathcal{A})} {\times} \mathcal{A} \longrightarrow \mathcal{Z}$$
is a $Q$-weak equivalence in $\mathsf{Lp}$.
\end{lemma}

\begin{proof}
We need to prove that $Q(\eta^*_{\mathcal{A}})$ is a weak equivalence in $\mathsf{Lp}$.
\begin{enumerate}
\item[(1)] We prove that the induced morphism:
$$
Q(\eta_{\mathcal{A}})^* : Q(\mathcal{Z})
\underset{QQ(\mathcal{A})}{\times} Q(\mathcal{A}) \longrightarrow Q(\mathcal{Z})
$$
is a weak equivalence in $\mathsf{Lp}$. Remark first that since $F$ is a
fibration in $\mathsf{Lp}$, the dg functors $F_0$ and $F_1$ are Morita fibrations and so they are surjective at the level of
$\mathsf{Hom}$-spaces. We now show that the dg functor
$F_0:\mathcal{Z}_0 \rightarrow \overline{\mathcal{A}_0}$ is surjective
on objects. If $\overline{\mathcal{A}_0}$ is the empty dg category
then so is $\mathcal{Z}_0$ and the claim is showed. If
$\overline{\mathcal{A}_0}$ is not empty, every object $X$ in
$\overline{\mathcal{A}_0}$ is contractible and since the dg functor $F_0$ belongs to
$C-\mbox{inj}$ there exists an object $Y$ in $\mathcal{Z}_0$ such
that $F_0(Y)=X$. This implies that each component of the morphism
$$ Q(F): Q(\mathcal{Z}) \longrightarrow
QQ(\mathcal{A})$$
is a dg functor that is surjective at the level of
$\mbox{Hom}$-spaces. Since by condition {\bf A2} the morphism
$$ (Q\eta_{\mathcal{A}}): Q(\mathcal{A}) \longrightarrow QQ(\mathcal{A})$$
is a weak equivalence an analogous argument to the proof of lemma~\ref{fiber} (we
have just proved that $F_1(X,Y)$ is a fibration in the projective
model structure on $\mathsf{Ch}(k)$),
  proves the condition $(1)$.
\item[(2)] We prove that the induced morphism:
$$ Q(\mathcal{Z}
\underset{Q(\mathcal{A})}{\times} \mathcal{A}) \longrightarrow   Q(\mathcal{Z})
\underset{QQ(\mathcal{A})}{\times} Q(\mathcal{A})$$
is an isomorphism in $\mathsf{Lp}$.
Since by construction the functor $Q$ is the identity functor on
objects, both components of the above morphism are also the identity
on objects. Let us consider the $1$-component of the above morphism. Let $X$ and
$Y$ be objects of $\mathcal{Z}_1/\mathcal{Z}_0$. We have the following fiber product in $\mathsf{Ch}(k)$~:

$$
\xymatrix{
\mbox{Hom}_{\mathcal{Z}_1/\mathcal{Z}_0
  \underset{(\mathcal{A}_1/\mathcal{A}_0)/\overline{\mathcal{A}_0}}{\times} \mathcal{A}_1/\mathcal{A}_0 }(X,Y) \ar[rr] \ar[d]  \ar@{}[drr]|{\ulcorner} & & \mbox{Hom}_{\mathcal{A}_1/\mathcal{A}_0}(F_1(X),F_1(Y))
\ar[d]^{Q \eta_{\mathcal{A}}} \\
\mbox{Hom}_{\mathcal{Z}_1/\mathcal{Z}_0}(X,Y) \ar@{->>}[rr]^-{QF_1} &&
\mbox{Hom}_{(\mathcal{A}_1/\mathcal{A}_0)/\overline{\mathcal{A}_0}}(F_1(X),F_1(Y))\,.
}
$$
Remark that the functor $Q \eta_{\mathcal{A}}$, resp. $QF_1$, sends the contractions in
$\mathcal{A}_1/\mathcal{A}_0$, resp. $\mathcal{Z}_1/\mathcal{Z}_0$, associated with the objects of
$\mathcal{A}_0$, resp. $\mathcal{Z}_0$, to the new contractions in
$(\mathcal{A}_1/\mathcal{A}_0)/\overline{\mathcal{A}_0}$ associated
with the objects of $\overline{\mathcal{A}_0}$.
Recall that we have the following fiber product in
$\mathsf{Ch}(k)$:
$$
\xymatrix{
\mbox{Hom}_{\mathcal{Z}_1
  \underset{\mathcal{A}_1/\mathcal{A}_0}{\times} \mathcal{A}_1}(X,Y)
\ar[r] \ar[d]  \ar@{}[dr]|{\ulcorner} & \mbox{Hom}_{\mathcal{A}_1}(F_1X,F_1Y) \ar[d]^{\eta} \\
\mbox{Hom}_{\mathcal{Z}_1}(X,Y) \ar@{->>}[r]^-{F_1} &
\mbox{Hom}_{\mathcal{A}_1/\mathcal{A}_0}(F_1X, F_1Y)\,.
}
$$
A analysis of the above fiber products shows that the induced morphism
$$ \mbox{Hom}_{(\mathcal{Z}_1
  \underset{\mathcal{A}_1/\mathcal{A}_0}{\times}
  \mathcal{A}_1)/(\mathcal{Z}_0
    \underset{\overline{\mathcal{A}_0}}{\times}
    \mathcal{A}_0)}(X,Y) \stackrel{\sim}{\longrightarrow}
\mbox{Hom}_{\mathcal{Z}_1/\mathcal{Z}_0
  \underset{(\mathcal{A}_1/\mathcal{A}_0)/\overline{\mathcal{A}_0}}{\times} \mathcal{A}_1/\mathcal{A}_0}(X,Y)$$
is an isomorphism in $\mathsf{Ch}(k)$. The same argument applies to the
$0$-component of the above morphism. This proves condition $(2)$.
\end{enumerate}
Now, conditions $1)$ and $2)$ imply that the morphism
$$ Q(\mathcal{Z}
\underset{Q(\mathcal{A})}{\times} \mathcal{A}) \stackrel{ (Q\eta)^*_{\mathcal{A}}}{\longrightarrow} Q(\mathcal{Z})$$
is a weak equivalence in $\mathsf{Lp}$, which is exactly the statement of the
lemma. The lemma is then proved.

\end{proof}

\begin{lemma}\label{Q-fibration}
Any morphism $F:\mathcal{A}
\rightarrow \mathcal{B}$ of $\mathsf{Lp}$ has a
factorization $F=Q \circ J$ where
$Q:\mathcal{Z} \rightarrow
\mathcal{B}$ is a $Q$-fibration and
$J:\mathcal{A} \rightarrow \mathcal{Z}$ is a cofibration and a $Q$-weak equivalence.
\end{lemma}

\begin{proof}
Consider exactly the same proof as for lemma $4.6$ in chapter~X from \cite{Jardine},
but use lemma~\ref{Q-weak} instead of condition {\bf A3}.
\end{proof}

We now prove theorem~\ref{main}.
\begin{proof}
We will prove that conditions $M1-M5$ of definition $7.1.3$ from \cite{Hirschhorn}
are satisfied.
By the proof of proposition~\ref{naive}, the category $\mathsf{Lp}$ is complete and
cocomplete and so condition $M1$ is verified.
By definition the $Q$-weak equivalences in $\mathsf{Lp}$ satisfy condition
$M2$, i.e. the two out of three condition.
Clearly the $Q$-weak equivalences and $Q$-fibrations in $\mathsf{Lp}$ are
stable under retractions. Since the cofibrations are those of
proposition~\ref{sat} condition $M3$ is verified. Finally
lemma~\ref{weak fibration} implies the lifting condition $M4$ and
lemmas~\ref{weak fibration} and \ref{Q-fibration} imply the
factorization condition $M5$.
\end{proof}

\subsection{$Q$-fibrant objects}

We denote by $\mathsf{Ho}(\mathsf{Lp})$ the homotopy category of
$\mathsf{Lp}$ given by theorem~\ref{main}.

Let $\mathcal{A}$ be a localization pair.

\begin{lemma}\label{fibrant1}
If $\mathcal{A}$ is fibrant, in the Quillen model structure of
proposition~\ref{sat}, and the morphism $\eta_{\mathcal{A}}:
\mathcal{A} \rightarrow Q(\mathcal{A})$ is a weak
equivalence in $\mathsf{Lp}$ then $\mathcal{A}$ is $Q$-fibrant.
\end{lemma}

\begin{proof}
We need to show that the morphism $\mathcal{A}
\stackrel{P}{\rightarrow} 0$ is a $Q$-fibration, where $0$ denotes the
terminal object in $\mathsf{Lp}$. Consider the
following diagram:
$$
\xymatrix{
\mathcal{A} \ar[d]_P \ar[rr]^-{\eta_{\mathcal{A}}} & & Q(\mathcal{A})
\ar[d]^{Q(P)} \\
0 \ar[rr]_-{\eta} & & Q(0) \,.
}
$$
Factorize the morphism $Q(P)$ as
$$
\xymatrix{
Q(\mathcal{A}) \ar[r]^i \ar[dr]_{Q(P)} & \mathcal{Z} \ar[d]^q\\
& Q(0) \,,
}
$$
where $i$ is a trivial cofibration and $q$ a fibration in
$\mathsf{Lp}$. By the proof of lemma~\ref{facto}, $q$ is a $Q$-fibration. Since the morphism $0
\rightarrow Q(0)$ is a weak equivalence, lemma~\ref{fiber} implies
that the induced morphism $0 \underset{Q(0)}{\times}
\mathcal{A} \rightarrow \mathcal{Z}$ is a weak
equivalence. Since $\eta_{\mathcal{A}}$ is a weak equivalence
the induced morphism
$$ \theta: \mathcal{A} \rightarrow 0 \underset{Q(0)}{\times}
\mathcal{Z}$$ is also a weak equivalence. Factorize the morphism $\theta$ as
$$
\xymatrix{
\mathcal{A} \ar[r]^j \ar[dr]_{\theta} & \mathcal{W}
\ar[d]^{\pi}\\
& 0 \underset{Q(0)}{\times} \mathcal{Z}\,,
}
$$
where $\pi$ is a trivial fibration of $\mathsf{Lp}$ and $j$ is a trivial
cofibration. Then $q_* \circ \pi$ is a $Q$-fibration and the lifting
exists in the diagram~:
$$
\xymatrix{
\mathcal{A} \ar@{=}[r] \ar[d]_j & \mathcal{A} \ar[d]^p \\
\mathcal{W} \ar@{.>}[ur] \ar[r]_{q_* \circ \pi} & 0\,.
}
$$
Thus $P$ is a rectract of a $Q$-fibration, and is therefore a
$Q$-fibration itself. This proves the lemma.
\end{proof}

\begin{lemma}\label{fibrant2}
If $\mathcal{A}$ is $Q$-fibrant, then $\mathcal{A}$ is
fibrant in $\mathsf{Lp}$ and the natural morphism
$$ \eta_{\mathcal{A}}: \mathcal{A} \rightarrow
Q(\mathcal{A})$$
is a weak equivalence.
\end{lemma}

\begin{proof}
Since the $Q$-model structure on $\mathsf{Lp}$ has fewer fibrations
than the Quillen model structure of proposition~\ref{sat}, the
localization pair $\mathcal{A}$ is fibrant in $\mathsf{Lp}$.
Consider the following diagram:
$$
\xymatrix{
\mathcal{A} \ar[rr]^-{\eta_{\mathcal{A}}} \ar[d]_p & & Q(\mathcal{A})
\ar[d]^{Q(p)}\\
0 \ar[rr]_{\eta} & &  Q(0) \,.
}
$$
Factorize $Q(p)=q\circ i$ as in the previous lemma. We have the
following diagram~:
$$
\xymatrix{
\mathcal{A} \ar[r]^-{\theta} \ar[d]_p &  0
\underset{Q(0)}{\times} \mathcal{Z} \ar[dl]^{q_*} \\
0 &
}
$$
Since $p$ and $q_*$ are $Q$-fibrations, $\mathcal{A}$ and
$\mathcal{Z}$ are $Q$-fibrant objects in $\mathsf{Lp}$ and $\theta$
is a $Q$-weak equivalence in $\mathsf{Lp}$. By application of lemma $7.7.1$ b)
from \cite{Hirschhorn} to $\theta$ and using lemma~\ref{weak
  fibration} we conclude that $\theta$ is a weak
equivalence. Since so is $i$, we conclude that $\eta_{\mathcal{A}}$ is also a
weak equivalence. This proves the lemma.
\end{proof}

\begin{remark}
By lemmas~\ref{fibrant1} and \ref{fibrant2} a localization pair
$\mathcal{A}$ is $Q$-fibrant if and only if
it is fibrant in $\mathsf{Lp}$ and the natural morphism
$$ \eta_{\mathcal{A}}: \mathcal{A} \longrightarrow Q(\mathcal{A})$$
is a weak equivalence.
\end{remark}

We now describe explicitly the $Q$-fibrant objects in $\mathsf{Lp}$.

\begin{proposition}\label{Q-fibrant}
A localization pair $\mathcal{A}$ is
$Q$-fibrant, i.e. fibrant in the model structure of Theorem~\ref{main}, if and
only if it is isomorphic in $\mathsf{Lp}$ to a localization pair of the
form~:
$$ (\mathcal{B}_{contr} \subset \mathcal{B})\,,$$
where $\mathcal{B}$ is a fibrant dg category and
$\mathcal{B}_{contr}$ is the full dg subcategory of contractible
objects in $\mathcal{B}$.
\end{proposition}

\begin{proof}
Suppose first that $\mathcal{A}$ is
$Q$-fibrant. Since it is also fibrant in $\mathsf{Lp}$ the dg category
$\mathcal{A}_1$ is fibrant in $\mathsf{dgcat}$. Since the morphism
$$ \eta_{\mathcal{A}}:(\mathcal{A}_0 \subset \mathcal{A}_1) \longrightarrow
(\overline{\mathcal{A}_0} \subset \mathcal{A}_1/\mathcal{A}_0)$$
is a weak equivalence all the objects of $\mathcal{A}_0$ are
contractible. Since $\mathcal{A}$ is fibrant
in $\mathsf{Lp}$ by lemma~\ref{fibrSat} $\mathcal{A}_0$ is stable
under homotopy equivalences in $\mathcal{A}_1$. This implies that $\mathcal{A}_0$ is in fact the
full dg subcategory of contractible objects of $\mathcal{A}_1$.
Consider now a localization pair $(\mathcal{B}_{contr} \subset
\mathcal{B})$ as in the statement of the proposition.
We remark that since $\mathcal{B}$ is fibrant in $\mathsf{dgcat}$,
then $\mathcal{B}_{contr}$ is also fibrant. Clearly $(\mathcal{B}_{contr}
\subset \mathcal{B})$ satisfies the extension condition with regard to $\sigma$ and the morphism
$$ \eta: (\mathcal{B}_{contr} \subset \mathcal{B}) \longrightarrow
(\overline{\mathcal{B}_{contr}} \subset
\mathcal{B}/\mathcal{B}_{contr})$$
is a weak equivalence in $\mathsf{Lp}$. This proves the proposition.
\end{proof}

\section{Closed symmetric monoidal structure}\label{secmon}
Let $\mathcal{A}$ and $\mathcal{B}$ be small dg categories. Recall
from \cite{ICM} \cite{Toen} that the tensor product
$\mathcal{A} \otimes \mathcal{B}$ of two dg categories has the class
of objects $\mbox{obj}(\mathcal{A}) \times \mbox{obj}(\mathcal{B})$ and the morphism
spaces
$$ \mathsf{Hom}_{\mathcal{A} \otimes \mathcal{B}}((X,Y),(X',Y')) =
\mathsf{Hom}_{\mathcal{A}}(X,X') \otimes
\mathsf{Hom}_{\mathcal{B}}(Y,Y')$$
with the natural compositions and units. Recall also from
\cite{ICM} \cite{Toen} that we have the  the dg category of
dg functors $\mathsf{Fun}_{dg}(\mathcal{A},\mathcal{B})$ from
$\mathcal{A}$ to $\mathcal{B}$. For two dg functors $F,G: \mathcal{A}
\rightarrow \mathcal{B}$, the {\em complex of graded morphisms}
$\mathsf{Hom}_{\mathsf{Fun}_{dg}(\mathcal{A},\mathcal{B})}(F,G)$ has as
its $n$th component the module formed by the families of morphisms
$$ \phi_X  \in \mathsf{Hom}_{\mathcal{B}}^n(FX,GX)$$
such that $(Gf)(\phi_X)=(\phi_Y)(Ff)$ for all $f \in
\mathsf{Hom}_{\mathcal{A}}(X,Y), \, X,Y \in \mathcal{A}$. The
differential is induced by that of $\mathsf{Hom}_{\mathcal{B}}(FX,GX)$.

\begin{definition}
The {\it internal $\mathsf{Hom}$ functor} in $\mathsf{Lp}$
$$ \mathsf{Hom}(-,-): \mathsf{Lp}^{op}\times \mathsf{Lp}
\longrightarrow \mathsf{Lp}\,,$$
associates to the localization pairs  $(\mathcal{A}_0 \subset \mathcal{A}_1)$, $(\mathcal{B}_0
\subset \mathcal{B}_1)$ the localization pair~:
$$ ( \mathsf{Fun}_{dg}(\mathcal{A}_1,\mathcal{B}_0) \subset
\mathsf{Fun}_{dg}(\mathcal{A}_0,\mathcal{B}_0) \underset{\mathsf{Fun}_{dg}(\mathcal{A}_0,\mathcal{B}_1)}{\times}\mathsf{Fun}_{dg}(\mathcal{A}_1,\mathsf{B}_1))\,.$$
\end{definition}

\begin{definition}
The {\it tensor product} functor in $\mathsf{Lp}$
$$ - \otimes- : \mathsf{Lp} \times \mathsf{Lp} \longrightarrow
\mathsf{Lp} $$
associates to the localization pairs $(\mathcal{A}_0 \subset
\mathcal{A}_1)$, $(\mathcal{B}_0 \subset \mathcal{B}_1)$ the
localization pair~:
$$( \mathcal{A}_0\otimes\mathcal{B}_1 \cup \mathcal{A}_1\otimes
\mathcal{B}_0 \subset \mathcal{A}_1\otimes \mathcal{B}_1)\,,$$
where  $\mathcal{A}_0\otimes\mathcal{B}_1 \cup \mathcal{A}_1\otimes
\mathcal{B}_0$ is the full dg subcategory of $\mathcal{A}_1\otimes
\mathcal{B}_1$ consisting of those objects $a\otimes b$ of
$\mathcal{A}_1 \otimes \mathcal{B}_1$ such that $a$ belongs to
$\mathcal{A}_0$ or $b$ belongs to $\mathcal{B}_0$.
\end{definition}

Let $\mathcal{A}=(\mathcal{A}_0 \subset \mathcal{A}_1)$,
$\mathcal{B}=(\mathcal{B}_0 \subset \mathcal{B}_1)$ and
$\mathcal{C}=(\mathcal{C}_0 \subset \mathcal{C}_1)$ be
localization pairs.

\begin{proposition}\label{monferme}
The category $\mathsf{Lp}$ endowed with the functors
$\mathsf{Hom}(-,-)$ and $-\otimes-$ is a closed symmetric monoidal
category. In particular we have a natural isomorphism in $\mathsf{Lp}$:
$$ \mathsf{Hom}_{\mathsf{Lp}}(\mathcal{A} \otimes
\mathcal{B}, \mathcal{C}) \stackrel{\sim}{\longrightarrow}
\mathsf{Hom}_{\mathsf{Lp}}(\mathcal{A},
\mathsf{Hom}(\mathcal{B}, \mathcal{C}))\,.$$
\end{proposition}

\begin{proof}
Consider the following commutative square in $\mathsf{dgcat}$~:
$$
\xymatrix{
\mathcal{A}_0 \ar[d] \ar[rr] & & \mathsf{Fun}_{dg}(\mathcal{B}_1,
\mathcal{C}_0) \ar[d] \\
\mathcal{A}_1 \ar[rr] & &
\mathsf{Fun}_{dg}(\mathcal{B}_0, \mathcal{C}_0)
\underset{\mathsf{Fun}_{dg}(\mathcal{B}_0,\mathcal{C}_1)}{\times}
\mathsf{Fun}_{dg}(\mathcal{B}_1,\mathcal{C}_1)\,,
}
$$
which corresponds exactly to an element of
$\mathsf{Hom}_{\mathsf{Lp}}(\mathcal{A},\mathsf{Hom}(\mathcal{B},\mathcal{C}))$.
Recall from \cite{ICM} that $\mathsf{dgcat}$ endowed with
  the functors $-\otimes-$ and $\mathsf{Hom}(-,-)$ is a closed
  symmetric monoidal category. This implies by adjunction that the
  commutative square above corresponds to the following commutative
  square in $\mathsf{dgcat}$:
$$
\xymatrix{
\mathcal{A}_0\otimes\mathcal{B}_1 \underset{\mathcal{A}_0 \otimes
  \mathcal{B}_0}{\times} \mathcal{A}_1\otimes \mathcal{B}_0 \ar[d]
\ar[r] & \mathcal{C}_0 \ar[d] \\
\mathcal{A}_1\otimes \mathcal{B}_1 \ar[r] & \mathcal{C}_1\,.
}
$$
This commutative square can be seen simply as a morphism in
$\mathsf{dgcat}^L$ from
$$\mathcal{A}_0\otimes\mathcal{B}_1 \underset{\mathcal{A}_0 \otimes
  \mathcal{B}_0}{\times} \mathcal{A}_1\otimes \mathcal{B}_0
\longrightarrow \mathcal{A}_1 \otimes \mathcal{B}_1$$
 to the localization pair $(\mathcal{C}_0
\subset \mathcal{C}_1)$. Remark that the morphism
$$\mathcal{A}_0\otimes\mathcal{B}_1 \underset{\mathcal{A}_0 \otimes
  \mathcal{B}_0}{\times} \mathcal{A}_1\otimes \mathcal{B}_0
\rightarrow \mathcal{A}_1 \otimes \mathcal{B}_1$$
of dg categories is
injective on objects and that its image consists of those objects
$a\otimes b$ of $\mathcal{A}_1 \otimes \mathcal{B}_1$ such that $a$
belongs to $\mathcal{A}_0$ or $b$ belongs to $\mathcal{B}_0$.
This implies that
$$ \mathsf{Im}\,(\mathcal{A}_0\otimes\mathcal{B}_1 \underset{\mathcal{A}_0 \otimes
  \mathcal{B}_0}{\times} \mathcal{A}_1\otimes \mathcal{B}_0
\rightarrow \mathcal{A}_1 \otimes \mathcal{B}_1) = \mathcal{A}
\otimes \mathcal{B}\,,$$
and by the adjunction $(S,U)$ from subsection~\ref{secMorita}, this
last commutative square in $\mathsf{dgcat}$ corresponds exactly to an
element of
$\mathsf{Hom}_{\mathsf{Lp}}(\mathcal{A}\otimes\mathcal{B},
\mathcal{C})$. This proves the proposition.
\end{proof}

\begin{remark}
Remark that the unit object is the localization pair $(\emptyset
\subset \mathcal{A})$, where $\mathcal{A}$ is the dg category with
one object and whose dg algebra of endomorphisms is $k$.
\end{remark}

\section{Derived internal \mbox{Hom}-functor}

Let $\mathcal{A}$ be a cofibrant dg category and $\lambda$ an infinite
cardinal whose size is greater than or equal to the cardinality of the set of isomorphism classes of objects in the category
$\mathsf{H}^0(\mathcal{A})$.
Let $\mathcal{B}$ be a Morita fibrant dg category. Recall that we
denote by $\,\widehat{ }: \mathcal{B} \rightarrow
\mathcal{C}_{dg}(\mathcal{B})$ the Yoneda dg functor.

\begin{definition}\label{lambda}
Let $\mathcal{B}_{\lambda}$ be the full dg subcategory of
$\mathcal{C}_{dg}(\mathcal{B})$, whose objects are:
\begin{itemize}
\item[-] the right $\mathcal{B}$ dg modules $M$ such that $M \oplus D$
  is representable for a contractible right $\mathcal{B}$ dg module
  $D$ and
\item[-] the right $\mathcal{B}$ dg modules of the form $\widehat{B}\oplus C$, where $B$
is an object of $\mathcal{B}$ and the right $\mathcal{B}$ dg module
$C$ is a direct factor of $\underset{i \in S}{\bigoplus}
\mathsf{cone}(\mathbf{1}_{\widehat{B_i}})$, with $B_i$ an object of
$\mathcal{B}$ and $S$ a set of cardinality bounded by $\lambda$.
\end{itemize}
\end{definition}

Let $\mathsf{rep}_{dg}(\mathcal{A},\mathcal{B})$ be the dg category as in
\cite{ICM} \cite{Toen}.

\begin{remark}
Remark that we have a quasi-equivalence $\mathcal{B}
\stackrel{h}{\rightarrow} \mathcal{B}_{\lambda}$ and that the objects
  of $\mathcal{B}_{\lambda}$ are cofibrant and quasi-representable as
  right dg $\mathcal{B}$ modules, see~\cite{Toen}. This implies
  that we have a natural dg functor:
$$
\overline{\mathsf{Fun}}_{dg}(\mathcal{A}, \mathcal{B}_{\lambda}) :=
\mathsf{Fun}_{dg}(\mathcal{A}, \mathcal{B}_{\lambda}) /
\mathsf{Fun}_{dg}(\mathcal{A}, (\mathcal{B}_{\lambda})_{contr}) \stackrel{\Phi}{\longrightarrow}
\mathsf{rep}_{dg}(\mathcal{A},\mathcal{B})\,.
$$
\end{remark}

\begin{theorem}\label{rep1}
For a cofibrant dg category $\mathcal{A}$, a Morita fibrant dg
category $\mathcal{B}$ and an infinite cardinal $\lambda$ as above, the natural induced dg functor:
$$
\mathsf{Fun}_{dg}(\mathcal{A},\mathcal{B}_{\lambda})/\mathsf{Fun}_{dg}(\mathcal{A},(\mathcal{B}_{\lambda})_{contr})
\stackrel{\Phi}{\longrightarrow}
\mathsf{rep}_{dg}(\mathcal{A},\mathcal{B})\,,$$
is a quasi-equivalence.
\end{theorem}

\begin{proof}
We prove first that $\mathsf{H}^0(\Phi)$ is essentially surjective. We
have the following composition of dg functors
$$
\mathsf{Fun}_{dg}(\mathcal{A},\mathcal{B}) \stackrel{I}{\longrightarrow}
\overline{\mathsf{Fun}}_{dg}(\mathcal{A}, \mathcal{B}_{\lambda})
\stackrel{\Phi}{\longrightarrow}
\mathsf{rep}_{dg}(\mathcal{A},\mathcal{B})\,.
$$
Since $\mathcal{A}$ is a cofibrant dg category, lemma $4.3$ and sub-lemma $4.4$
from \cite{Toen} imply that $\mathsf{H}^0(\Phi \circ I)$ is essentially
surjective and so we conclude that so is $\mathsf{H}^0(\Phi)$.

We now prove also that the functor $\mathsf{H}^0(I)$ is essentially surjective.
Let $F:\mathcal{A} \rightarrow \mathcal{B}_{\lambda}$ be a dg
functor. Since $\mathcal{A}$ is a cofibrant dg category and $h$ is a
quasi-equivalence, there exists a dg functor $F':\mathcal{A}
\rightarrow \mathcal{B}$ such that $F$ and $h\circ F'$ are homotopic
in the Quillen model structure constructed in theorem~\ref{mal}. Remark that
since $\mathcal{B}$ is a Morita fibrant dg category so is
$\mathcal{B}_{\lambda}$. In particular $\mathcal{B}_{\lambda}$ is
stable under cones up to homotopy,
see proposition~\ref{nova4}. Since a cone can be obtained from a
cone up to homotopy, by adding or factoring out contractible modules, we
conclude that by definition, $\mathcal{B}_{\lambda}$ is also stable
under cones. By remark~\ref{cic} we dipose of a sequence of dg functors
$$ F \longrightarrow I \longrightarrow h \circ F'[1]\,,$$
such that $I$ belongs to
$\mathsf{Fun}_{dg}(\mathcal{A},(\mathcal{B}_{\lambda})_{contr})$. This
implies that $F$ and $h \circ F'$ become isomorphic in
$\mathsf{H}^0(\overline{\mathsf{Fun}}_{dg}(\mathcal{A},
\mathcal{B}_{\lambda}))$. This proves that the functor
$\mathsf{H}^0(I)$ is essentially surjective.

Let us now prove that the functor $\mathsf{H}^0(\Phi)$ is fully faithful.
Let $F$ belong to
$\mathsf{Fun}_{dg}(\mathcal{A},\mathcal{B}_{\lambda})$. Since
$\mathsf{H}^0(I)$ is essentially surjective, we can consider $F$ as
belonging to $\mathsf{Fun}_{dg}(\mathcal{A},\mathcal{B})$. We will construct a morphism of dg functors
$$ F' \stackrel{\mu}{\longrightarrow} F\,,$$
where $\mu$ becomes invertible in $
\mathsf{H}^0(\overline{\mathsf{Fun}}_{dg}(\mathcal{A}, \mathcal{B}_{\lambda}))$
and $F'$ belongs to the left-orthogonal of the category
$\mathsf{H}^0(\mathsf{Fun}_{dg}(\mathcal{A},(\mathcal{B}_{\lambda})_{contr}))$,
i.e.
$$
\mathsf{Hom}_{\mathsf{H}^0(\mathsf{Fun}_{dg}(\mathcal{A},(\mathcal{B}_{\lambda})))}(F',G)=0\,,$$
for every $G \in \mathsf{H}^0(\mathsf{Fun}_{dg}(\mathcal{A},(\mathcal{B}_{\lambda})_{contr}))$.
Consider the $\mathcal{A}$-$\mathcal{B}$-bimodule $X_F$ naturally
associated to $F$. Consider $X_F$ as a left $\mathcal{A}$-module and
let $\mathbf{P}X_F$ denote the bar resolution of $X_F$. Remark that
$\mathbf{P}X_F$ is naturally a right $\mathcal{B}$-module and that it
is cofibrant in the projective model structure on the category of $\mathcal{A}$-$\mathcal{B}$-bimodules. Let $A$ be an object of
$\mathcal{A}$. Since the dg category $\mathcal{A}$ is cofibrant in
$\mathsf{dgcat}$, $(\mathbf{P}X_F)(?,A)$ is cofibrant as a
$\mathcal{B}$-module. We have the following homotopy
equivalence
$$
\xymatrix{
(\mathbf{P}X_F)(?,A) \ar@{->>}[r]_{\sim}^{\mu_A} & X_F(?,A)\,,
}
$$
since both $\mathcal{B}$-modules are cofibrant.
This implies that the $\mathcal{B}$-module $(\mathbf{P}X_F)(?,A)$ is
isomorphic to a direct sum $X_F(?,A) \oplus C$, where $C$ is a
contractible and cofibrant $\mathcal{B}$-module.
The $\mathcal{B}$-module $C$ is in fact isomorphic to a direct factor
of a $\mathcal{B}$-module
$$ \underset{i \in S}\bigoplus(\mathsf{cone}\mathbf{1}_{\widehat{B_i}})[n_i],$$
where $S$ is a set whose cardinality is bounded by $\lambda$, $B_i$, $i \in I$ is an object of $\mathcal{B}$ and
$n_i$, $i \in S$ is an integer, see \cite{DerivingDG}.

This implies, by definition of $\mathcal{B}_{\lambda}$, that the $\mathcal{B}$-module
$$ X_F(?,A)\oplus C$$ belongs to $\mathcal{B}_{\lambda}$ and so the $\mathcal{A}$-$\mathcal{B}$-bimodule $\mathbf{P}X_F$ is in fact
isomorphic to $X_{F'}$ for a dg functor $F': \mathcal{A} \rightarrow \mathcal{B}_{\lambda}$. Remark that the previous construction is
functorial in $A$ and so we have a morphism of dg functors
$$ F' \stackrel{\mu}{\longrightarrow} F\,.$$
Since for each $A$ in $\mathcal{A}$, the morphism $\mu_A: F'A
\rightarrow FA$ is a retraction with contractible kernel, the morphism
$\mu$ becomes invertible in
$$\mathsf{H}^0(\overline{\mathsf{Fun}}_{dg}(\mathcal{A}, \mathcal{B}_{\lambda}))\,.$$
Let now $G$ belong to
$\mathsf{Fun}_{dg}(\mathcal{A},(\mathcal{B}_{\lambda})_{contr})$. We remark that
$$
\mathsf{Hom}_{\mathsf{H}^0(\mathsf{Fun}_{dg}(\mathcal{A},\mathcal{B}_{\lambda}))}(F',G)
\stackrel{\sim}{\longrightarrow}
\mathsf{Hom}_{\mathcal{H}(\mathcal{A}^{op}\otimes
  \mathcal{B})}(\mathbf{P}X_F,X_G)\,,$$
where $\mathcal{H}(\mathcal{A}^{op}\otimes
  \mathcal{B})$ denotes the homotopy category of
  $\mathcal{A}$-$\mathcal{B}$ bimodules.
Since $\mathbf{P}X_F$ is a cofibrant $\mathcal{A}$-$\mathcal{B}$-bimodule
and $X_G(?,A)$ is a contractible $\mathcal{B}$-module, for every
object $A$ in $\mathcal{A}$, we conclude that the right hand side
vanishes and $F'$ belongs to the
left-orthogonal of
$\mathsf{H}^0(\mathsf{Fun}_{dg}(\mathcal{A},(\mathcal{B}_{\lambda})_{contr}))$.
This implies that the induced functor
$$
\mathsf{H}^0(\mathsf{Fun}_{dg}(\mathcal{A},\mathcal{B}_{\lambda})/\mathsf{Fun}_{dg}(\mathcal{A},(\mathcal{B}_{\lambda})_{contr}))
\rightarrow \mathsf{H}^0(\mathsf{rep}_{dg}(\mathcal{A},\mathcal{B}))$$
is fully faithful.
This proves the theorem.
\end{proof}

\begin{theorem}\label{rep}
The internal $\mathsf{Hom}$ functor
$$ \mathsf{Hom}(-,-): \mathsf{Lp}^{op} \times \mathsf{Lp} \rightarrow
  \mathsf{Lp} \,,$$
admits a total right derived functor
$$ \mathcal{R}\mathsf{Hom}(-,-):\mathsf{Ho}(\mathsf{Lp}^{op} \times
\mathsf{Lp}) \rightarrow \mathsf{Ho}(\mathsf{Lp})\,
$$
as in definition $8.4.7$ from~\cite{Hirschhorn}.
\end{theorem}

\begin{proof}
Let $\mathcal{A}$ and $\mathcal{B}$ be localization pairs. We are now
going to define $\mathcal{R}\mathsf{Hom}(\mathcal{A},\mathcal{B})$ and
the morphism $\epsilon$ as in definition $8.4.7$
from~\cite{Hirschhorn}. We denote by
$\mathcal{A}_c \stackrel{P}{\rightarrow} \mathcal{A}$ a functorial
cofibrant resolution of $\mathcal{A}$ in $\mathsf{Lp}$ and  by
$\mathcal{B} \stackrel{I}{\rightarrow} \mathcal{B}_f$ a functorial
$Q$-fibrant resolution of $\mathcal{B}$ in $\mathsf{Lp}$.
Remember, that by proposition~\ref{Q-fibrant}, $\mathcal{B}_f$ is of
the form
$$\mathcal{B}_f=((\mathcal{B}_f)_{contr} \subset \mathcal{B}_f)\,,$$
where $\mathcal{B}_f$ is a Morita fibrant dg category.
Let $\lambda$ be an infinite cardinal whose size is greater or equal to
the cardinality of the set of isomorphim classes in the category
$\mathsf{H}^0((\mathcal{A}_c)_1)$.
Consider now the following localization pair
$$ (\mathcal{B}_f)_{\lambda} := (((\mathcal{B}_f)_{\lambda})_{contr}
\subset (\mathcal{B}_f)_{\lambda})\,,$$
where $(\mathcal{B}_f)_{\lambda}$ is as in
definition~\ref{lambda}. Remark that we have a canonical weak equivalence in $\mathsf{Lp}$
$$ \mathcal{B}_f \stackrel{F}{\longrightarrow}
(\mathcal{B}_f)_{\lambda}\,.$$
We now define
$$\mathcal{R}\mathsf{Hom}(\mathcal{A},\mathcal{B}) :=
\mathsf{Hom}(\mathcal{A}_c, (\mathcal{B}_f)_{\lambda})$$
and we consider
for morphism $\epsilon$ the image in $\mathsf{H}^0(\mathsf{Lp})$ of the following $Q$-equivalence in
$\mathsf{Lp}$
$$ \eta : (\mathcal{A},\mathcal{B}) \stackrel{(P,I)}{\longrightarrow}
(\mathcal{A}_c, \mathcal{B}_f) \stackrel{(Id,F)}{\longrightarrow}
(\mathcal{A}_c, (\mathcal{B}_f)_{\lambda})$$
under the functor $\mathsf{Hom}(-,-)$.

We will now show that the dg category associated with the localization
pair $\mathcal{R}\mathsf{Hom}(\mathcal{A},\mathcal{B})$ is canonically
Morita equivalent to
$$ \mathsf{rep}_{dg}((\mathcal{A}_c)_1/(\mathcal{A}_c)_0
,\mathcal{B}_f)\,.$$
Remark that since $\mathcal{A}_c$ is a cofibrant object in
$\mathsf{Lp}$, by lemma~\ref{cofibrant}, $(\mathcal{A}_c)_1$ is
cofibrant in $\mathsf{dgcat}$ and so we have an exact sequence
in the Morita homotopy category of dg categories $\mathsf{Hmo}$, see~\cite{ICM}
$$ (\mathcal{A}_c)_0 \hookrightarrow (\mathcal{A}_c)_1 \rightarrow
(\mathcal{A}_c)_1 / (\mathcal{A}_c)_0\,.$$
Since the dg category $(\mathcal{B}_f)$ is Morita fibrant,
the application of the functor
$\mathsf{rep}_{dg}(-,\mathcal{B}_f)$ to the previous exact
sequence induces a new exact sequence in $\mathsf{Hmo}$
$$ \mathsf{rep}_{dg}((\mathcal{A}_c)_0, \mathcal{B}_f)
\leftarrow \mathsf{rep}_{dg}((\mathcal{A}_c)_1,
\mathcal{B}_f) \leftarrow
\mathsf{rep}_{dg}((\mathcal{A}_c)_1/ (\mathcal{A}_c)_0,
\mathcal{B}_f)\,.
$$
Remember that:
$$\mathsf{Hom}(\mathcal{A}_c, (\mathcal{B}_f)_{\lambda})_1 = \mathsf{Fun}_{dg}((\mathcal{A}_c)_0,((\mathcal{B}_f)_{\lambda})_{contr})
\underset{\mathsf{Fun}_{dg}((\mathcal{A}_c)_0,(\mathcal{B}_f)_{\lambda})}{\times}
\mathsf{Fun}_{dg}((\mathcal{A}_c)_1,(\mathcal{B}_f)_{\lambda})\,.$$
Now, since the dg categories $(\mathcal{A}_c)_1$ and
$(\mathcal{B}_f)_{\lambda}$ satisfy the conditions of
theorem~\ref{rep1}, we have a natural inclusion of dg categories
$$
\mathsf{Hom}(\mathcal{A}_c, (\mathcal{B}_f)_{\lambda})_1 /\mathsf{Fun}_{dg}((\mathcal{A}_c)_1,((\mathcal{B}_f)_{\lambda})_{contr})
\longrightarrow \mathsf{rep}_{dg}((\mathcal{A}_c)_1 ,\mathcal{B}_f)\,.
$$
Now remark that this inclusion induces the following Morita
equivalence
$$
\mathsf{Hom}(\mathcal{A}_c, (\mathcal{B}_f)_{\lambda})_1 /\mathsf{Fun}_{dg}((\mathcal{A}_c)_1,((\mathcal{B}_f)_{\lambda})_{contr})
\stackrel{\sim}{\longrightarrow} \mathsf{rep}_{dg}((\mathcal{A}_c)_1/(\mathcal{A}_c)_0 ,\mathcal{B}_f)\,.
$$

We now show that the functor $\mathcal{R} \mathsf{Hom}(-,-)$ preserves
$Q$-weak equivalences in $\mathsf{Lp}^{op}\times
\mathsf{Lp}$. Consider a $Q$-weak equivalence
$$ (\mathcal{A},\mathcal{B}) \rightarrow (\tilde{\mathcal{A}},\tilde{\mathcal{B}})\,,$$
in $\mathsf{Lp}^{op} \times \mathsf{Lp}$. By construction it will
induce a Morita dg functor
$$ (\tilde{\mathcal{A}}_c)_1 / (\tilde{\mathcal{A}}_c)_0
\stackrel{\sim}{\longrightarrow} (\mathcal{A}_c)_1 /
(\mathcal{A}_c)_0$$
and also a Morita dg functor
$$ \mathcal{B}_f \stackrel{\sim}{\longrightarrow}
\tilde{\mathcal{B}}_f\,.$$
This implies that the induced dg functor
$$ \mathsf{rep}_{dg}((\mathcal{A}_c)_1 /(\mathcal{A}_c)_0,
\mathcal{B}_f) \stackrel{\sim}{\longrightarrow} \mathsf{rep}_{dg}((\tilde{\mathcal{A}}_c)_1 /(\tilde{\mathcal{A}}_c)_0,\tilde{\mathcal{B}}_f)$$
is a Morita dg functor. Now, remark that we have the natural zig-zag of $Q$-weak equivalences in
$\mathsf{Lp}$:
\begin{eqnarray*}
& (\mathsf{Fun}_{dg}((\mathcal{A}_c)_1,
((\mathcal{B}_f)_{\lambda})_{contr}) \subset
\mathsf{Hom}(\mathcal{A}_c, (\mathcal{B}_f)_{\lambda})_1 & \\
& \downarrow & \\
& (\overline{\mathsf{Fun}_{dg}((\mathcal{A}_c)_1,
((\mathcal{B}_f)_{\lambda})_{contr})} \subset  \mathsf{Hom}(\mathcal{A}_c, (\mathcal{B}_f)_{\lambda})_1 / \mathsf{Fun}_{dg}((\mathcal{A}_c)_1,
((\mathcal{B}_f)_{\lambda})_{contr})) & \\
& \uparrow & \\
& (\emptyset \subset \mathsf{Hom}(\mathcal{A}_c, (\mathcal{B}_f)_{\lambda})_1 / \mathsf{Fun}_{dg}((\mathcal{A}_c)_1,
((\mathcal{B}_f)_{\lambda})_{contr})) &
\end{eqnarray*}
This allow us to conclude that the  the functor $\mathcal{R} \mathsf{Hom}(-,-)$ preserves
$Q$-weak equivalences in $\mathsf{Lp}^{op}\times
\mathsf{Lp}$.
This proves the proposition.
\end{proof}

\begin{proposition}\label{tensorprod}
Let $\mathcal{A}$ be a cofibrant object in $\mathsf{Lp}$. The induced
internal tensor product functor
$$ \mathcal{A}\otimes- : \mathsf{Lp} \longrightarrow \mathsf{Lp}\,,$$
preserves $Q$-weak equivalences.
\end{proposition}

\begin{proof}
Let $F:\mathcal{B} \rightarrow \mathcal{C}$ be a $Q$-weak equivalence
in $\mathsf{Lp}$ between cofibrant objects. We prove that the induced
morphism in $\mathsf{Lp}$
$$ \mathcal{A} \otimes \mathcal{B} \stackrel{F_*}{\longrightarrow}
\mathcal{A}\otimes \mathcal{C}\,,$$
is a $Q$-weak equivalence. By lemma~\ref{cofibrant}, $\mathcal{A}_1$,
$\mathcal{B}_1$ and $\mathcal{C}_1$ are cofibrant dg categories in
$\mathsf{dgcat}$ and so we have a morphism of exact sequences in
$\mathsf{Hmo}$:
$$
\xymatrix{
\mathcal{B}_0 \ar[d] \ar@{^{(}->}[r] & \mathcal{B}_1 \ar[d] \ar[r] &
\mathcal{B}_1/\mathcal{B}_0 \ar[d]^{\sim} \\
\mathcal{C}_0 \ar@{^{(}->}[r] & \mathcal{C}_1 \ar[r] & \mathcal{C}_1/\mathcal{C}_0\,,
}
$$
where the last column is a Morita dg functor. Since $\mathcal{A}_1$
is cofibrant in $\mathsf{dgcat}$ proposition $1.6.3$
in \cite{Drinfeld} implies that by applying the functor
$\mathcal{A}\otimes-$ we obtain the following morphism of exact
sequences in $\mathsf{Hmo}$:
$$
\xymatrix{
\mathcal{A}_1 \otimes \mathcal{B}_0 \ar[d] \ar[r] &
\mathcal{A}_1\otimes \mathcal{B}_1 \ar[d] \ar[r] &
\mathcal{A}_1 \otimes (\mathcal{B}_1/\mathcal{B}_0) \ar[d]^{\sim} \\
\mathcal{A}_1 \otimes \mathcal{C}_0 \ar[r] & \mathcal{A}_1
\otimes \mathcal{C}_1 \ar[r] & \mathcal{A}_1 \otimes (\mathcal{C}_1/\mathcal{C}_0)\,.
}
$$
This implies that we have the following Morita dg functor:
$$ (\mathcal{A}_1 \otimes \mathcal{B}_1)/(\mathcal{A}_1 \otimes
\mathcal{B}_0) \stackrel{\sim}{\longrightarrow} (\mathcal{A}_1 \otimes \mathcal{C}_1)/(\mathcal{A}_1 \otimes
\mathcal{C}_0)\,.$$
Let $\mathcal{H}$ be the full dg subcategory of $(\mathcal{A}_1 \otimes \mathcal{B}_1)/(\mathcal{A}_1 \otimes
\mathcal{B}_0)$, whose objects are $a \otimes b$, where $a$ belongs to
$\mathcal{A}_0$ and $\mathcal{P}$ the full dg subcategory of $(\mathcal{A}_1 \otimes \mathcal{C}_1)/(\mathcal{A}_1 \otimes
\mathcal{C}_0)$ whose objects are $a \otimes c$, where $a$ belongs to
$\mathcal{A}_0$.
We have the following diagram:
$$
\xymatrix{
\mathcal{H} \ar[d]^{\sim} \ar@{^{(}->}[r] &  (\mathcal{A}_1 \otimes \mathcal{B}_1)/(\mathcal{A}_1 \otimes
\mathcal{B}_0) \ar[d]^{\sim}\\
\mathcal{P} \ar@{^{(}->}[r] & (\mathcal{A}_1 \otimes \mathcal{C}_1)/(\mathcal{A}_1 \otimes
\mathcal{C}_0)\,.
}
$$
Remark that the dg categories $\mathcal{A}\otimes \mathcal{B}$ and
$\mathcal{A}\otimes \mathcal{C}$ are Morita equivalent dg
subcategories of $((\mathcal{A}_1 \otimes \mathcal{B}_1)/(\mathcal{A}_1 \otimes
\mathcal{B}_0))/\mathcal{H}$, resp. $((\mathcal{A}_1 \otimes \mathcal{C}_1)/(\mathcal{A}_1 \otimes
\mathcal{C}_0))/\mathcal{P}$ and so we have the commutative square:
$$
\xymatrix{
((\mathcal{A}_1 \otimes \mathcal{B}_1)/(\mathcal{A}_1 \otimes
\mathcal{B}_0))/\mathcal{H} \ar[d]^{\sim} &
\mathcal{A} \otimes \mathcal{B} \ar[l]^-{\sim} \ar[d]^{F^*}\\
((\mathcal{A}_1 \otimes \mathcal{C}_1)/(\mathcal{A}_1 \otimes
\mathcal{C}_0))/\mathcal{P} & \mathcal{A}\otimes \mathcal{C} \ar[l]^-{\sim}\,.
}
$$
By the two out of three property $F^{\ast}$ is a $Q$-weak equivalence.
This implies the lemma.
\end{proof}

\begin{remark}
Since the internal tensor product $-\otimes-$ is symmetric,
lemma~\ref{tensorprod} implies that the total left derived functor $-\otimes-$
$$ - \overset{\mathbb{L}}{\otimes}- : \mathsf{Ho}(\mathsf{Lp})\times
\mathsf{Ho}(\mathsf{Lp}) \rightarrow \mathsf{Ho}(\mathsf{Lp})$$
exists, see definition $8.4.7$ of \cite{Hirschhorn}.
\end{remark}

\section{Relation with $\mathsf{dgcat}$}

We have the following adjunction:
$$
\xymatrix{
\mathsf{Lp} \ar@<1ex>[d]^{Ev_1}\\
\mathsf{dgcat} \ar@<1ex>[u]^L \,,
}
$$
where $Ev_1$ associates to a localization pair $(\mathcal{A}_0 \subset
\mathcal{A}_1)$ the dg category $\mathcal{A}_1$ and $L$
associates to a dg category $\mathcal{A}$ the localization pair $(\emptyset
\subset \mathcal{A})$.
\begin{proposition}\label{relat}
If we consider on $\mathsf{dgcat}$ the Quillen model structure of theorem~\ref{theorem2} and on $\mathsf{Lp}$ the $Q$-model structure, the previous
adjunction is a Quillen equivalence, see~\cite{Hirschhorn}.
\end{proposition}

\begin{proof}
The functor $L$ clearly sends Morita dg functors to weak
equivalences. By lemma~\ref{weak fibration} the evaluation functor $Ev_1$
preserves trivial fibrations. This shows that $L$ is a left Quillen functor.
Let $\mathcal{A}$ be a cofibrant object in $\mathsf{dgcat}$ and
$(\mathcal{B}_{contr} \subset \mathcal{B})$ a $Q$-fibrant object in
$\mathsf{Lp}$. Let $\mathcal{A} \stackrel{F}{\rightarrow} \mathcal{B}$
be a dg functor in $\mathsf{dgcat}$. We need to show that $F$ is a
Morita dg functor if and only if the induced morphism of localization
pairs $(\emptyset \subset \mathcal{A} ) \rightarrow
(\mathcal{B}_{contr} \subset \mathcal{B})$ is a $Q$-weak
equivalence. But since the dg functor $\mathcal{B} \rightarrow
\mathcal{B}/\mathcal{B}_{contr}$ is a Morita dg functor this
automatically follows.
\end{proof}

\begin{proposition}\label{Quilleneq}
The total derived functors, $-\overset{\mathbb{L}}{\otimes}-$ and
$\mathcal{R}\mathsf{Hom}(-,-)$ in the category
$\mathsf{Ho}(\mathsf{Lp})$ correspond, under the equivalence:
$$
\xymatrix{
\mathsf{Ho}(\mathsf{Lp}) \ar@<1ex>[d]^{\mathcal{R}Ev_1} \\
\mathsf{Ho}(\mathsf{dgcat}) \ar@<1ex>[u]^L
}
$$
to the functors, $-\overset{\mathbb{L}}{\otimes}-$  and
$\mathsf{rep}_{dg}(-,-)$, see \cite{ICM}\cite{Toen}, in the
category $\mathsf{Ho}(\mathsf{dgcat})$.
\end{proposition}

\begin{proof}
Let $\mathcal{A}$ and $\mathcal{B}$ be dg categories. Then
$\mathcal{A} \overset{\mathbb{L}}{\otimes}\mathcal{B}$ identifies with
$\mathcal{A}_c \otimes \mathcal{B}$, where $\mathcal{A}_c$ is a
cofibrant resolution of $\mathcal{A}$ in $\mathsf{dgcat}$. Since
$L(\mathcal{A}_c)$ is cofibrant in $\mathsf{Lp}$ by lemma~\ref{tensorprod}, we
have the following zig-zag:
$$L(\mathcal{A}) \overset{\mathbb{L}}{\otimes} L(\mathcal{B})
\stackrel{\sim}{\leftarrow} L(\mathcal{A}_c) \otimes L(\mathcal{B})
\stackrel{\sim}{\rightarrow} L(\mathcal{A}_c \otimes \mathcal{B}) =
L(\mathcal{A} \overset{\mathbb{L}}{\otimes} \mathcal{B})\,,$$
of weak equivalences in $\mathsf{Lp}$. This proves that the total left derived tensor
products in $\mathsf{Ho}(\mathsf{Lp})$ and
$\mathsf{Ho}(\mathsf{dgcat})$ are identified.
Now, $\mathsf{rep}_{dg}(\mathcal{A},\mathcal{B})$ identifies with
$\mathsf{rep}_{dg}(\mathcal{A}_c,\mathcal{B}_f)$, where
$\mathcal{B}_f$ is a fibrant resolution of
$\mathcal{B}$ in $\mathsf{dgcat}$. By definition
$$ \mathcal{R}\mathsf{Hom}(L(\mathcal{A}), L(\mathcal{B})) =
\mathsf{Hom}((L(\mathcal{A})_c,(L(\mathcal{B})_f)_{\lambda})\,.$$
where $\lambda$ denotes an infinite cardinal whose size is greater or
equal to the cardinality of the set of isomorphism classes of objects
in the category $\mathsf{H}^0(\mathcal{A}_c)$.
We have the following $Q$-weak equivalent objects in $\mathsf{Lp}$:
\begin{eqnarray}
&&\mathcal{R}\mathsf{Hom}(L(\mathcal{A}), L(\mathcal{B})) \nonumber\\
&&\mathsf{Hom}((L(\mathcal{A})_c,(L(\mathcal{B})_f)_{\lambda})\nonumber\\
&&\mathsf{Hom}((\emptyset \subset \mathcal{A}_c), ((\mathcal{B}_f)_{\lambda})_{contr}
\subset (\mathcal{B}_f)_{\lambda}))\nonumber\\
&&(\mathsf{Fun}_{dg}(\mathcal{A}_c, ((\mathcal{B}_f)_{\lambda})_{contr}) \subset
\mathsf{Fun}_{dg}(\mathcal{A}_c, (\mathcal{B}_f)_{\lambda}))\nonumber\\
&&\overline{\mathsf{Fun}_{dg}(\mathcal{A}_c, ((\mathcal{B}_f)_{\lambda})_{contr})}
\subset \mathsf{Fun}_{dg}(\mathcal{A}_c,
(\mathcal{B}_f)_{\lambda}))/\mathsf{Fun}_{dg}(\mathcal{A}_c,
((\mathcal{B}_f)_{\lambda})_{contr})\nonumber\\
&&(\emptyset \subset \mathsf{rep}_{dg}(\mathcal{A}_c,\mathcal{B}_f))\nonumber\\
&&L(\mathsf{rep}_{dg}(\mathcal{A},\mathcal{B}))\nonumber\,.
\end{eqnarray}
This proves that the total right derived functor
$\mathcal{R}\mathsf{Hom}(-,-)$ in $\mathsf{Ho}(\mathsf{Lp})$
corresponds to the functor $\mathsf{rep}_{dg}(-,-)$, as in
\cite{ICM} \cite{Toen}.
\end{proof}

\chapter{Homotopy theory of well-generated algebraic triangulated categories}

\textit{\small{Ce chapitre correspond au article \cite{Triangcat}.}}

\section{Introduction}
Triangulated categories appear naturally in several branches
of mathematics such as algebraic geometry, algebraic analysis,
$K$-theory, representation theory and even in mathematical physics,
see for instance \cite{ENS} \cite{IHP}.

In his book~\cite{Neeman}, Neeman introduces the important class of {\it well-generated}
triangulated categories. Let us recall this concept. Let $\alpha$ be a regular cardinal,
see section $5.3$ of \cite{Set}. A triangulated category $\mathcal{T}$ is
$\alpha$-compactly generated \cite{Krause} \cite{Neeman} if it admits arbitrary coproducts and an {\it $\alpha$-good set of
  generators} $\mathcal{G}$, i.e. $\mathcal{G}$ is stable under shifts
in both directions and satisfies
\begin{itemize}
\item[1)] an object $X$ of $\mathcal{T}$ vanishes if and only if
  $\mathcal{T}(G,X)=0$ for each $G \in \mathcal{G}$,
\item[2)] each $G \in \mathcal{G}$ is $\alpha$-small, i.e. for each
  family of objects $X_i$, $i \in I$, of $\mathcal{T}$, each morphism
$$G \rightarrow \underset{i \in I}{\bigoplus}\, X_i$$
factors through a subsum $\underset{i \in J}{\bigoplus}X_i$ for some
subset $J$ of $I$ of cardinality strictly smaller than $\alpha$,
\item[3)] for each family of morphisms $f_i: X_i \rightarrow Y_i$, $i
  \in I$, of $\mathcal{T}$ which induces surjections 
$$ \mathcal{T}(G,X_i) \rightarrow \mathcal{T}(G,Y_i)$$
for all $G \in \mathcal{G}$ and all $i \in I$, the sum of the $f_i$
induces surjections
$$ \mathcal{T}(G, \bigoplus X_i) \rightarrow \mathcal{T}(G,\bigoplus Y_i)$$
for all $G \in \mathcal{G}$.
\end{itemize}
A triangulated category is well-generated if it is $\alpha$-compactly
generated for a regular cardinal $\alpha$, see \cite{Krause} \cite{Neeman}. Clearly each compactly
generated triangulated category is well-generated. Neeman proves in
\cite{Neeman} that the Brown representability theorem holds for
well-generated triangulated categories. This is one of the main
reasons for studying them. Another important result of \cite{Neeman}
is that if $\mathcal{T}$ is well generated and $\mathcal{S}
\rightarrow \mathcal{T}$ a localization (i.e. a fully faithful
triangle functor admitting a left adjoint whose kernel is generated by
a set of objects) then $\mathcal{S}$ is well-generated. Thus each
localization of a well-generated triangulated category is
well-generated.

\begin{example}
Let $\mathcal{B}$ be a Grothendieck abelian category, e.g. the
category of modules on a ringed space. Then by the Popescu-Gabriel
theorem \cite{Pop-Gab}, $\mathcal{B}$ is a localization of the
category of $\mathsf{Mod}\, A$ of $A$-modules over some ring $A$. One
can deduce from this \cite{DocNeeman} that the unbounded derived category of the
abelian category $\mathcal{B}$ is a localization of $\mathcal{D}(A)$
and thus is well-generated.
\end{example}

In his lecture at the international congress of mathematicians in Madrid
2006, Keller \cite{ICM} defines the notion of an algebraic triangulated
category in order to axiomatize the properties of all triangulated
categories appearing naturally in algebra and geometry. Recall
that a triangulated category $\mathcal{T}$ is {\it algebraic} if it is
triangle equivalent to the stable category $\underline{\mathcal{E}}$
for some Frobenius category $\mathcal{E}$. This holds if and only if
$\mathcal{T}$ is triangle equivalent to a full triangulated
subcategory of the category up to homotopy of complexes over an
additive category.

In this paper, we study the category of well-generated algebraic triangulated categories, see \cite{ICM}, using the tools of Quillen's homotopical
algebra, \cite{Quillen}, and the formalism of dg categories,
\cite{Drinfeld} \cite{ICM} \cite{ICM} \cite{dgquot} \cite{addendum}
\cite{IMRN} \cite{cras} \cite{Toen}.
More precisely, for a fixed regular cardinal $\alpha$ we construct a
category $\mathsf{dgcat}_{ex,\alpha}$, whose objects are essentially
the dg categories which are stable under suspensions, cosuspensions,
cones and $\alpha$-small sums.

The construction of $\mathsf{dgcat}_{ex,\alpha}$ is done in two steps:
first we construct a monad $\mathsf{T}_{\alpha}$ on the
  category of small differential graded categories $\mathsf{dgcat}$
  and we consider the associated category $\mathsf{T}_{\alpha}$-$\mathsf{alg}$ of
  $\mathsf{T}_{\alpha}$-algebras. Then the category $\mathsf{dgcat}_{ex,\alpha}$, is obtained by
  considering specific diagrams in $\mathsf{T}_{\alpha}$-$\mathsf{alg}$. 

In each of these two steps we have at our disposal an adjunction and by applying
an argument due to Quillen to our situation, we are able to lift the Quillen
model structure on $\mathsf{dgcat}$, see~\cite{cras}, along these
adjunctions to $\mathsf{dgcat}_{ex,\alpha}$.

Finally, we define a functor $\mathcal{D}_{\alpha}(-)$ from
$\mathsf{dgcat}_{ex, \alpha}$ to the category $\mathsf{Tri}_{\alpha}$
of $\alpha$-compactly generated algebraic triangulated categories,
see~\cite{Krause} \cite{Neeman}, which by \cite{Ober}~\cite{Porta}
is known to verify the following conditions:
\begin{enumerate}
\item[-] every category $\mathcal{T}$ in $\mathsf{Tri}_{\alpha}$ is equivalent to
  $\mathcal{D}_{\alpha}(\underline{A})$ for some $\underline{A}$ in $\mathsf{dgcat}_{ex,\alpha}$ and
\item[-] a morphism $F$ in $\mathsf{dgcat}_{ex, \alpha}$ is a
  weak equivalence if and only if $\mathcal{D}_{\alpha}(F)$ is an
  equivalence of triangulated categories.
\end{enumerate}

This shows that well-generated algebraic triangulated categories up to
equivalence admit a natural Quillen enhancement given by our model
category $\mathsf{dgcat}_{ex, \alpha}$.

\section{Preliminaries}
In what follows $\alpha$ will denote a regular cardinal and $k$ a
commutative ring with unit. The tensor product $\otimes$ will denote the tensor product over $k$. Let
$\mathsf{Ch}(k)$ denote the category of complexes over $k$. By a
{\it dg category}, we mean a differential graded $k$ category, see definition~\ref{dgcategorie}. For a dg category
$\mathcal{A}$, we denote by $\mathcal{C}_{dg}(\mathcal{A})$ the dg
category of right $\mathcal{A}$ dg modules and by $\,\widehat{ } : \mathcal{A} \rightarrow
\mathcal{C}_{dg}(\mathcal{A})$ the Yoneda dg functor.
We write $\mathsf{dgcat}$ for the category of small dg categories.
By theorem~\ref{mal} the category $\mathsf{dgcat}$ admits a structure of cofibrantly generated
model category whose weak equivalences are the quasi-equivalences.

\section{Monadic structure $\mathsf{T}$ on $\mathsf{dgcat}$}\label{monad}

In this section, we will construct a monad $\mathsf{T}_{\alpha}$ on
$\mathsf{dgcat}$. The $\mathsf{T}_{\alpha}$-algebras will be
essentially the dg categories which admit $\alpha$-small
sums, see proposition~\ref{sums}. This monad $\mathsf{T}_{\alpha}$ is
a variant of the well-known coproduct completion monad $\mathsf{Fam}$
of `families', \emph{cf.}~\cite{Fam}, see also~\cite{Kelly}.

For each ordinal $\beta$ strictly smaller than $\alpha$, we denote by
$I_{\beta}$ the underlying set of $\beta$, see
\cite{Set}~\cite{Hirschhorn}.
Let $\mathcal{A}$ be a small dg category. We denote by $F_{\beta}$ a
generic morphism from $I_{\beta}$ to $\mbox{obj}(\mathcal{A})$.

\begin{definition}\label{monadf}
Let $\mathsf{T}_{\alpha}(\mathcal{A})$ be the small dg category whose
objects are all the maps
$$ I_{\beta} \stackrel{F_{\beta}}{\longrightarrow}
\mbox{obj}\,(\mathcal{A})\,,$$
where $\beta < \alpha$, and whose complex of morphisms between $F_{\beta}$ and $F_{\beta '}$
is 
$$ \mathsf{Hom}_{\mathsf{T}_{\alpha}(\mathcal{A})}(F_{\beta}, F_{\beta
  '}):= \prod_{i \in I_{\beta}}  \bigoplus_{j \in I_{\beta'}} \mathsf{Hom}_{\mathcal{A}}(F_{\beta}(i), F_{\beta
'}(j))\,.$$
The composition and the identities are induced by those of $\mathcal{A}$.
\end{definition}

\begin{remark}
Clearly $\mathsf{T}_{\alpha}(\mathcal{A})$ is a small dg category and
the above construction is functorial in $\mathcal{A}$. We have a
functor
$$ \mathsf{T}_{\alpha}(-): \mathsf{dgcat} \longrightarrow \mathsf{dgcat}\,.$$
\end{remark}

\begin{definition}
Let $\eta$ be the natural transformation
$$ \eta: Id_{\mathsf{dgcat}} \longrightarrow
\mathsf{T}_{\alpha}(-)\,,$$
whose evaluation at $\mathcal{A}$ is the dg functor
$$ \eta_{\mathcal{A}}:\mathcal{A} \longrightarrow
\mathsf{T}_{\alpha}(\mathcal{A})$$
that sends an object $X$ of $\mathcal{A}$ to the map
$$ I_{\mathbf{1}} \stackrel{X}{\longrightarrow}{
  \mbox{obj}\,(\mathcal{A})}\,,$$
where $\mathbf{1}$ denotes the first successor ordinal, see \cite{Set}.
\end{definition}
Notice that $\eta_{\mathcal{A}}$ is a fully faithful dg functor.

\begin{definition}\label{assoc}
Let $\mu_{\mathcal{A}}$ be the dg functor
$$ \mu_{\alpha}:(\mathsf{T}_{\alpha}\circ
\mathsf{T}_{\alpha})(\mathcal{A}) \longrightarrow
\mathsf{T}_{\alpha}(\mathcal{A})\,,$$
which sends an object
$$
\begin{array}{ccc}
I_{\beta} & \stackrel{F_{\beta}}{\longrightarrow} &
\mbox{obj}\,(\mathsf{T}_{\alpha}(\mathcal{A})) \\
x & \longmapsto & (I_{\gamma_x}
\stackrel{F_{\gamma_x}}{\longrightarrow} \mbox{obj}\,(\mathcal{A}))\,,
\end{array}
$$
of $(\mathsf{T}_{\alpha} \circ \mathsf{T}_{\alpha})(\mathcal{A})$ to
the map
$$ \coprod_{x \in \beta} F_{\gamma_x} := I_{\underset{x \in \beta}{\sum}
  \gamma_x} \longrightarrow \mbox{obj} \,(\mathcal{A})\,,$$
where $\underset{x \in \beta}{\sum} \gamma_x$ denotes the increasing sum of the
ordinals $\gamma_x$, $x \in \beta$, along the ordinal $\beta$, see
section $4.2$ of \cite{Set}.
\end{definition}

\begin{remark}
Observe that the ordinal $\underset{x \in \beta}{\sum}\gamma_x$ is strictly
smaller than $\alpha$ since $\beta$ and each one of the $\gamma_x$'s
are strictly smaller than $\alpha$, which is by hypothesis a regular
cardinal. The above construction is functorial in $\mathcal{A}$ and so
we have a natural transformation
$$\mu:(\mathsf{T}_{\alpha} \circ \mathsf{T}_{\alpha})(-) \longrightarrow \mathsf{T}_{\alpha}(-)\,.$$
\end{remark}

\begin{proposition}\label{monade}
The above construction give us a monad, see~\cite{Macl},
$\mathsf{T}_{\alpha}=(\mathsf{T}_{\alpha}(-), \eta, \mu)$ on the
category $\mathsf{dgcat}$.
\end{proposition}

\begin{proof}
We need to prove that the diagrams of functors
$$
\xymatrix{
\mathsf{T}_{\alpha}\circ \mathsf{T}_{\alpha}\circ \mathsf{T}_{\alpha}
\ar[d]_{\mu \mathsf{T}_{\alpha}} \ar[r]^-{\mathsf{T}_{\alpha}\mu} &
\mathsf{T}_{\alpha}\circ \mathsf{T}_{\alpha} \ar[d]^{\mu} & &
\mathsf{T}_{\alpha} \ar[r]^-{\eta \mathsf{T}_{\alpha}} \ar@{=}[dr] & \mathsf{T}_{\alpha} \circ
\mathsf{T}_{\alpha} \ar[d]^{\mu} & \mathsf{T}_{\alpha} \ar[l]_-{\mathsf{T}_{\alpha} \eta}
\ar@{=}[dl] \\
\mathsf{T}_{\alpha}\circ \mathsf{T}_{\alpha} \ar[r]_{\mu} & \mathsf{T}_{\alpha} & & &
\mathsf{T}_{\alpha} & 
}
$$
are commutative. The left diagram is commutative because the ordinal
sum operation, see~\cite{Set}, on the set of ordinals strictly smaller
than $\alpha$, is associative. The right diagram is also commutative by definition of $\mu$.

This proves the proposition.
\end{proof}

Let $\mathcal{A}$ be a small dg category.

\begin{lemma}\label{sums1}
The dg category $\mathsf{T}_{\alpha}(\mathcal{A})$ admits
$\alpha$-small sums.
\end{lemma}

\begin{proof}
Let $\beta$ be a cardinal strictly smaller than $\alpha$. Let
$J$ be a set of cardinality $\beta$ and $G$ a morphism
$$ G: J \longrightarrow
\mbox{obj}\,(\mathsf{T}_{\alpha}(\mathcal{A}))\,.$$
Choose a bijection between $J$ and $I_{\beta}$ and consider the
object 
$$ F_{\beta}:I_{\beta} \longrightarrow
\mbox{obj}\,(\mathsf{T}_{\alpha}(\mathcal{A}))$$
of $(\mathsf{T}_{\alpha} \circ \mathsf{T}_{\alpha})$ associated to $G$. Now notice that
by definition of $\mathsf{T}_{\alpha}(\mathcal{A})$ and since 
$$I_{\underset{x \in \beta}{\sum}\gamma_x}= \underset{x \in
  I_{\beta}}{\cup} I_{\gamma_x} $$
the object $\mu_{\mathcal{A}}(F_{\beta})$ of
$\mathsf{T}_{\alpha}(\mathcal{A})$ is in fact the $\beta$-small sum of $G$.
\end{proof}

Recall from \cite{Macl} that by definition an algebra over this monad (=$\mathsf{T}_{\alpha}$-algebra) consists of a pair $A=(\mathcal{A},R)$, where
$\mathcal{A}$ is a small dg category and $R$ is a dg functor
$R:\mathsf{T}_{\alpha}(\mathcal{A}) \rightarrow \mathcal{A}$ which makes both diagrams
$$
\xymatrix{
(\mathsf{T}_{\alpha}\circ \mathsf{T}_{\alpha})(\mathcal{A})
\ar[d]_{\mu_{\mathcal{A}}} \ar[r]^-{\mathsf{T}_{\alpha}(R)} &
\mathsf{T}_{\alpha}(\mathcal{A}) \ar[d]^{R} & &
\mathcal{A} \ar@{=}[dr] \ar[r]^-{\eta_{\mathcal{A}}} & \mathsf{T}_{\alpha}(\mathcal{A}) \ar[d]^{R} \\
\mathsf{T}_{\alpha}(\mathcal{A}) \ar[r]_{R} & \mathcal{A} & & &
\mathcal{A} 
}
$$
commute. A morphism $F:(\mathcal{A},R) \rightarrow (\mathcal{B},G)$ of
$\mathsf{T}_{\alpha}$-algebras is a dg functor $F:\mathcal{A}
\rightarrow \mathcal{B}$ which renders commutative the diagram
$$
\xymatrix{
\mathsf{T}_{\alpha}(\mathcal{A}) \ar[d]_R
\ar[r]^{\mathsf{T}_{\alpha}(F)} & \mathsf{T}_{\alpha}(\mathcal{B})
\ar[d]^G \\
\mathcal{A} \ar[r]_F & \mathcal{B}\,.
}
$$
We denote by $\mathsf{T}_{\alpha}$-$\mathsf{alg}$ the category of
$\mathsf{T}_{\alpha}$-algebras. By theorem $1$ of chapter VI from \cite{Macl},
we have an adjunction
$$
\xymatrix{
\mathsf{T}_{\alpha}\mbox{-}\mathsf{alg} \ar@<1ex>[d]^U \\
\mathsf{dgcat} \ar@<1ex>[u]^F\,,
}
$$
where the functor $U$ associates to a $\mathsf{T}_{\alpha}$-algebra $A$ the dg category $\mathcal{A}$ and $F$ associates to a dg category $\mathcal{B}$ the
$\mathsf{T}_{\alpha}$-algebra $(\mathsf{T}_{\alpha}(\mathcal{B}),\mu_{\mathcal{B}})$. 

\begin{proposition}\label{sums}
Let $A=(\mathcal{A},R)$ be a $\mathsf{T}_{\alpha}$-algebra. Then the dg
category $\mathcal{A}$ admits $\alpha$-small sums.
\end{proposition}

\begin{proof}
Let $\beta$ be a cardinal strictly smaller than $\alpha$ and $J$
a set of cardinality $\beta$. Let $G$ be a morphism
$$ G:J \longrightarrow \mbox{obj}\,(\mathcal{A})\,.$$
Choose a bijection between $J$ and $I_{\beta}$ and consider the
object of $\mathsf{T}_{\alpha}(\mathcal{A})$
$$ F_{\beta}:I_{\beta} \longrightarrow \mbox{obj}\,(\mathcal{A})\,,$$ associated to $G$.
We will prove that the object $R(F_{\beta})$ of $\mathcal{A}$ is the
$\beta$-small sum of $G$. Consider the following object 
$$ \overline{F_{\beta}}:=
\mathsf{T}_{\alpha}(\eta_{\mathcal{A}})(F_{\beta})\,,$$
in $(\mathsf{T}_{\alpha}\circ \mathsf{T}_{\alpha})(\mathcal{A})$.
Notice that $F_{\beta} = \mu_{\mathcal{A}}(\overline{F_{\beta}})$ and so
  by lemma~\ref{sums1}, $F_{\beta}$ is the $\beta$-small sum of
  $\overline{F_{\beta}}$. In particular, we have for each $x \in
  I_{\beta}$ a closed morphism of degree zero 
$$ i_x : \eta_{\mathcal{A}}(F_{\beta}(x)) \rightarrow F_{\beta} $$
in $\mathsf{T}_{\alpha}(\mathcal{A})$.
Now, since $A$ is a $\mathsf{T}_{\alpha}$-algebra the
  equality $R \circ \eta_{\mathcal{A}} = Id$ implies that
  $R(F_{\beta})$ is a weak $\beta$-small sum of $G$ in
  $\mathcal{A}$. This means that for every object $Z$ in $\mathcal{A}$
  the morphism of complexes
$$ 
\mathsf{Hom}_{\mathcal{A}}(R(F_{\mathcal{B}}),Z) \longrightarrow 
\underset{x \in I_{\beta}}{\prod}
\mathsf{Hom}_{\mathcal{A}}(F_{\beta}(x),Z)$$
induced by the closed morphisms of degree zero $R(i_x)$, $x \in
I_{\beta}$, is surjective. We now show that it is also injective. Let $H_1$ and $H_2$ be two
elements of $\mathsf{Hom}_{\mathcal{A}}(R(F_{\beta}),Z)$ such that
$$H_1 \circ i_x = H_2 \circ i_x, \, \forall x \in I_{\beta}\,.$$
We will now prove that $H_1=H_2$.

Consider the following commutative diagram in $\mathcal{A}$
$$
\xymatrix{
F_{\beta}(x) \ar[r]^-{R(i_x)} \ar[dr]_{G_x} & R(F_{\beta}) \ar[d]_{H_1} \ar@<1ex>[d]^{H_2}\\
 & Z
}
$$
Apply the dg functor $\eta_{\mathcal{A}}$ to the previous diagram and
consider the following one in $\mathsf{T}_{\alpha}(\mathcal{A})$~:
$$
\xymatrix{
 &  & F_{\beta} \ar@{.>}[d]^-{\theta} \ar@/^5pc/@{.>}[dd]^-{\Phi}\\
\eta_{\mathcal{A}}(F_{\beta}(x)) \ar[urr]^-{i_x}
\ar[rr]^-{\eta_{\mathcal{A}}(R(i_x))}
\ar[drr]_-{\eta_{\mathcal{A}}(G_x)} & &
\eta_{\mathcal{A}}(R(F_{\beta})) \ar[d]_-{\eta_{\mathcal{A}}(H_1)}
\ar@<1ex>[d]^-{\eta_{\mathcal{A}}(H_2)} \\
 & & \eta_{\mathcal{A}}(Z) \,.
}
$$
Since $F_{\beta}$ is the $\beta$-small sum of $\overline{F_{\beta}}$
there is a unique morphism $\theta$ in
$\mathsf{T}_{\alpha}(\mathcal{A})$ such that
$$ \theta \circ i_x = \eta_{\mathcal{A}}(R(i_x)), \, \forall x \in
I_{\beta}$$
and a unique morphism $\Phi$ in
$\mathsf{T}_{\alpha}(\mathcal{A})$ such that 
$$ \Phi \circ i_x = \eta_{\mathcal{A}}(G_x), \, \forall x \in
I_{\beta}\,.$$
This implies that
$$ \eta_{\mathcal{A}}(H_1)\circ \theta = \eta_{\mathcal{A}}(H_2)\circ
\theta \,.$$
We will show that $R(\theta)=Id$, which immediately implies the
proposition. Recall that since $A$ is a $\mathsf{T}_{\alpha}$-algebra,
we have the following commutative diagram
$$
\xymatrix{
(\mathsf{T}_{\alpha}\circ \mathsf{T}_{\alpha})(\mathcal{A})
\ar[d]_{\mu_{\mathcal{A}}} \ar[r]^-{\mathsf{T}_{\alpha}(R)} &
\mathsf{T}_{\alpha}(\mathcal{A}) \ar[d]^{R} \\
\mathsf{T}_{\alpha}(\mathcal{A}) \ar[r]_{R} & \mathcal{A}\,.
}
$$
Notice that since  $\overline{F_{\beta}}$ is the $\beta$-small sum of the objects
$\eta_{\mathsf{T}_{\alpha}(\mathcal{A})}(\eta_{\mathcal{A}}(F_{\beta}(x)))$,
the morphisms $i_x$ induce a morphism
$$ \Psi: \overline{F_{\beta}} \longrightarrow
\eta_{\mathsf{T}_{\alpha}(\mathcal{A})}(F_{\beta})\,,$$
in $(\mathsf{T}_{\alpha} \circ \mathsf{T}_{\alpha})(\mathcal{A})$.
Finally remark that $\mathsf{T}_{\alpha}(R)(\Psi)=\theta$ and $\mu_{\mathcal{A}}(\Psi)=Id$. This implies that $R(\theta)=Id$ and so
the proposition is proven.
\end{proof}

\begin{remark}\label{sums2}
Let $F:(\mathcal{A},R) \rightarrow (\mathcal{B},G)$ be a morphism of $\mathsf{T}_{\alpha}$-algebras.
Observe that the proof of proposition~\ref{sums} and the
commutative square
$$
\xymatrix{
\mathsf{T}_{\alpha}(\mathcal{A}) \ar[d]_R
\ar[r]^{\mathsf{T}_{\alpha}(F)} & \mathsf{T}_{\alpha}(\mathcal{B})
\ar[d]^G \\
\mathcal{A} \ar[r]_F & \mathcal{B}
}
$$
imply that the dg functor $F:\mathcal{A} \rightarrow \mathcal{B}$ preserves
$\alpha$-small sums. This fact will be used in section~\ref{finale}, see definition~\ref{defporta}.
\end{remark}

\section{Quillen's lifting argument}

In this section, we consider an argument due to Quillen,
see~\cite{Quillen}, that allows us to lift a Quillen model structure
along an adjunction.

Let $\mathcal{N}$ be a complete and cocomplete category. Consider a functor
$U:\mathcal{N} \rightarrow \mathcal{M}$, with $\mathcal{M}$ a Quillen
model category. Assume that $U$ admits a left adjoint 
$$F:\mathcal{M}\rightarrow \mathcal{N}\,.$$

\begin{definition}\label{lift}
A morphism $f:A \longrightarrow B$ in $\mathcal{N}$ is:
\begin{enumerate}
\item[-] a weak equivalence if $U(f)$ is a quasi-equivalence in
  $\mathcal{M}$.
\item[-] a fibration if $U(f)$ is a fibration in $\mathcal{M}$.
\item[-] a cofibration if it has the left lifting property with respect
  to all trivial fibrations in $\mathcal{N}$.
\end{enumerate}
\end{definition}

\begin{theorem}\label{liftarg}
Suppose that $\mathcal{M}$ is a cofibrantly generated Quillen model
category and that $U$ commutes with $\alpha$-filtered colimits for a
regular cardinal $\alpha$, see~\cite{Hirschhorn}.

Then with the notions of weak equivalence, fibration and cofibration
defined above, $\mathcal{N}$ is a Quillen model category provided the
following assumption on cofibrations holds: every cofibration with the
left lifting property with respect to fibrations is a weak
equivalence. Moreover, the adjunction $(F,U)$ becomes a Quillen adjunction.
\end{theorem}

\begin{proof}
We denote by $I$ the set of generating cofibrations of $\mathcal{M}$
and by $J$ the set of generating trivial cofibrations of
$\mathcal{M}$. Notice that since $\mathcal{M}$ is cofibrantly
generated, the domains of the elements of the sets $I$ and $J$ are
$\beta$-small for a cardinal $\beta$, see definition $11.1.2$ of \cite{Hirschhorn}. 
Since $U$ commutes with $\alpha$-small filtered colimits, the images of
these domains under the functor $F$ will be $\gamma$-small, where
$\gamma$ is the maximum of $\alpha$ and $\beta$.
This shows that the sets $F(I)$ and $F(J)$ of $\mathcal{N}$ allow the
small object argument.

Now the proof follows the lines of the one of
theorem $4.1$ from \cite{Jardine}: We simply use the set $I$, respectively $J$, instead
of the generating cofibrations of simplicial sets, respectively
generating trivial cofibrations of simplicial sets and consider $\gamma$-transfinite
compositions, see~\cite{Hirschhorn}, for the construction of the
factorizations.

Observe also that the class of cofibrations that have the left lifting
property with respect to fibrations is stable under $\gamma$-transfinite
compositions. Clearly the adjunction $(F,U)$ is a Quillen
adjunction.

This proves the theorem.
\end{proof}

Now let $\mathcal{N}$ and $\mathcal{M}$ be as at the beginning of this
section and consider definition~\ref{lift}.
 
\begin{proposition}\label{Quillen}
Suppose that
\begin{itemize}
\item[-] for every object $A$ in $\mathcal{N}$, the unique morphism $A
  \rightarrow *$, where $*$ denotes the terminal object in
  $\mathcal{N}$, is a fibration.
\item[-] for every object $A$ in $\mathcal{N}$, we dispose of a
  factorization
$$
\xymatrix{
A \ar[rr]^{\Delta} \ar[dr]^{\sim}_{i_A} & & A \times A \\
 & P(A) \ar@{->>}[ur]_{q_A} & \,,
}
$$
where $i_A$ is a weak equivalence and $q_A$ is a fibration.
\end{itemize}
Then every morphism in $\mathcal{N}$ that has the left lifting property with
respect to all fibrations is a weak equivalence.
\end{proposition}

\begin{proof}
Let $i:A \rightarrow B$ be a morphism in $\mathcal{N}$ that
has the left lifting property with respect to all fibrations.
Consider the following diagram
$$
\xymatrix{
A \ar@{=}[r] \ar[d]_i &  A \ar@{->>}[d]\\
B \ar@{->>}[r] \ar@{.>}[ur]^u & \ast \,.
}
$$
By the hypotheses on $i$ we have a morphism $u$ such that $u\circ
i=\mbox{Id}$. We now show that the morphism $i \circ u$ is right
homotopic to the
identity of $B$.
By hypothesis, we have at our disposal a factorization
$$
\xymatrix{
B \ar[rr]^{\Delta} \ar[dr]^{\sim}_{i_B} & & B\times B \\
 & P(B) \ar@{->>}[ur]_{q_B} & 
}
$$
which allows us to construct the diagram
$$
\xymatrix{
A \ar[rr]^{i_B \circ i} \ar[d]_i && P(B) \ar@{->>}[d]^{q_B} \\
B \ar[rr]_{[Id\, , \, i\circ u]} \ar@{.>}[urr]^H & & B\times B\,.
}
$$
By the hypothesis on $i$, we have at our disposal a morphism $H$, which by
definition is a right homotopy between the identity of $B$ and
$i\circ u$. Since the functor $U: \mathcal{N} \rightarrow \mathcal{M}$ preserves products, fibrations and weak
equivalences the identity on  $U(B)$ and $U(i)\circ U(u)$
are right homotopic in $\mathcal{M}$ and so they become equal in the
homotopy category $\mathsf{Ho}(\mathcal{M})$. Since we already know that
$U(u)\circ U(i)$ is the identity on $U(B)$, we conclude that the morphism $U(i)$ is an
isomorphism in $\mathsf{Ho}(\mathcal{M})$. By proposition $1.14$ from
\cite{Jardine}, $U(i)$ is in fact a weak equivalence in
$\mathcal{M}$ which implies by definition that $i$ is a weak
equivalence in $\mathcal{N}$. This proves the lemma.
\end{proof}

\section{Homotopy theory of $\mathsf{T}$-algebras}\label{homotopy2}

Recall from section~\ref{monad} that we have an adjunction
$$
\xymatrix{
\mathsf{T}_{\alpha}\mbox{-}\mathsf{alg} \ar@<1ex>[d]^U \\
\mathsf{dgcat} \ar@<1ex>[u]^F\,.
}
$$
Since the category $\mathsf{dgcat}$ is complete, proposition $4.3.1$
from \cite{Bor} implies that $\mathsf{T}_{\alpha}$-$\mathsf{alg}$ is
also complete. Now notice that the functor
$$ \mathsf{T}_{\alpha}(-): \mathsf{dgcat} \longrightarrow
\mathsf{dgcat}\,,$$
see definition~\ref{monadf}, commutes with $\alpha$-filtered
colimits, see \cite{Bor}. This implies by proposition $4.3.2$ and
$4.3.6$ from \cite{Bor} that the category $\mathsf{T}_{\alpha}$-$\mathsf{alg}$ is
cocomplete and that the functor $U$ commutes with $\alpha$-filtered
colimits.

From now on and until the end of this section we consider the
definition~\ref{lift} applied to our particular adjunction $(F,U)$.

Let $\mathcal{B}$ be a small dg category.

\begin{definition}
Let $P(\mathcal{B})$ be the dg category, see~\cite{Drinfeld}, whose
objects are the closed morphisms of degree zero in $\mathcal{B}$
$$ X \stackrel{f}{\longrightarrow} Y\,,$$
that become invertible in $\mathsf{H}^0(\mathcal{B})$.
We define the complex of morphisms
$$ \mathsf{Hom}_{P(\mathcal{B})}(X\stackrel{f}{\rightarrow}Y,
X'\stackrel{f'}{\rightarrow}Y')$$
as the homotopy pull-back in $\mathsf{Ch}(k)$ of the diagram
$$
\xymatrix{
& \mathsf{Hom}_{\mathcal{B}}(Y,Y') \ar[d]^{f^*} \\
\mathsf{Hom}_{\mathcal{B}}(X,X') \ar[r]^{f'_*} & \mathsf{Hom}_{\mathcal{B}}(X,Y')\,.
}
$$
\end{definition}
It is proven in lemma~\ref{lempath} that $P(\mathcal{B})$ is a path object
for $\mathcal{B}$ in the Quillen model structure on $\mathsf{dgcat}$
of theorem~\ref{mal}. Notice that the above construction is functorial in
$\mathcal{B}$ and so we have at our disposal a functor
$$P(-):\mathsf{dgcat} \rightarrow \mathsf{dgcat}\,.$$

Let $B=(\mathcal{B},S)$ be a $\mathsf{T}_{\alpha}$-algebra.

\begin{proposition}\label{path}
The category $P(\mathcal{B})$ carries a natural
$\mathsf{T}_{\alpha}$-algebra structure and so $B$ admits a
path-object $P(B)$.
\end{proposition}

\begin{proof}
Since $B$ is a $\mathsf{T}_{\alpha}$-algebra, we have the dg functor `sum'
$$ S: \mathsf{T}_{\alpha}(\mathcal{B}) \longrightarrow \mathcal{B}$$
and we will construct a dg functor 
$$\overline{S}:\mathsf{T}_{\alpha}(P(\mathcal{B})) \longrightarrow
P(\mathcal{B})\,.$$
Observe that we have a faithful dg functor
$$ \mathsf{T}_{\alpha}(P(\mathcal{B})) \longrightarrow P(\mathsf{T}_{\alpha}(\mathcal{B}))\,.$$
Define $\overline{S}$ as the composition of this dg functor with 
$$ P(\mathsf{T}_{\alpha}(\mathcal{B})) \longrightarrow
P(\mathcal{B})\,.$$
Since $B$ is a $\mathsf{T}_{\alpha}$-algebra, this construction shows us that
$(P(\mathcal{B}), \overline{S})$ is also a $\mathsf{T}_{\alpha}$-algebra. It is
also clear by construction that the dg functors $\mathcal{B}
\stackrel{i_{\mathcal{B}}}{\longrightarrow} P(\mathcal{B})$ and 
$$
\xymatrix{
 P(\mathcal{B}) \ar@{->>}[r]^-{q_{\mathcal{B}}} &
 \mathcal{B}\times\mathcal{B}
}
$$ are in fact morphisms of
$\mathsf{T}_{\alpha}$-algebras, where $\mathcal{B}\times\mathcal{B}$
carries the diagonal $\mathsf{T}_{\alpha}$-action. This proves the proposition. 
\end{proof}

\begin{theorem}\label{main2}
The category $\mathsf{T}_{\alpha}$-$\mathsf{alg}$ when endowed with
the notions of weak equivalence, fibration and cofibration as in
definition~\ref{lift}, becomes a cofibrantly generated Quillen model
category and the adjunction $(F,U)$ becomes a Quillen adjunction.
\end{theorem}

\begin{proof}
Recall from theorem~\ref{mal} that we have an explicit set
$I=\{Q,S(n), n \in \mathbb{Z}\}$ of generating cofibrations and an
explicit set $J=\{ F, R(n), n \in \mathbb{Z}\}$ of generating trivial
cofibrations for $\mathsf{dgcat}$.

Now notice that all conditions of
theorem~\ref{liftarg} are satisfied. In particular
proposition~\ref{path} and the fact that every object in
$\mathsf{dgcat}$ is fibrant, see remark~\ref{toutfibrant}, imply
proposition~\ref{Quillen} which implies that the assumption on cofibrations of
theorem~\ref{liftarg} holds. 

Observe that $F(I)$ is a set of generating
cofibrations on $\mathsf{T}_{\alpha}$-$\mathsf{alg}$ and that $F(J)$ is a set of generating
trivial cofibrations on $\mathsf{T}_{\alpha}$-$\mathsf{alg}$. This implies that the Quillen
model structure on $\mathsf{T}_{\alpha}$-$\mathsf{alg}$ is cofibrantly
generated. Since the functor $U$ preserves by definition weak
equivalences and fibrations the adjunction $(F,U)$ is a Quillen
adjunction.
This proves the theorem.
\end{proof}

\section{Exact  $\alpha$-cocomplete dg categories}
In this section, we will construct a category
$\mathsf{dgcat}_{ex,\alpha}$ by considering specific diagrams in
$\mathsf{T}_{\alpha}$-$\mathsf{alg}$. The objects of
$\mathsf{dgcat}_{ex,\alpha}$ will be essentially the dg categories
which are stable under suspensions, cosuspensions, cones and
$\alpha$-small sums, see remark~\ref{choices1}.

\begin{definition}\label{exact}
Let $\mathcal{P}$ be the dg category with only one object $X$ and
whose dg algebra of endomorphisms is $k$ concentrated in degree
$0$. Let $\mathcal{S}$, respectively $\mathcal{S}^{-1}$, be the full
sub dg category of $\mathcal{C}_{dg}(\mathcal{P})$, whose objects are
$\widehat{X}$ and $\widehat{X}[1]$, respectively $\widehat{X}$ and
$\widehat{X}[-1]$. We have a fully faithful dg functor $\mathcal{P}
\stackrel{S}{\rightarrow} \mathcal{S}$, respectively $\mathcal{P}
\stackrel{S^{-1}}{\rightarrow} \mathcal{S}^{-1}$.
Let $\mathcal{M}$ be the dg category which has two objects $0$ and $1$
and is generated by a morphism $f$ from $0$ to $1$ that satisfies
$d(f)=0$. Let $\mathcal{C}$ be the full sub dg category of
$\mathcal{C}_{dg}(\mathcal{M})$, whose objects are $\widehat{0}$,
$\widehat{1}$ and $\mathsf{cone}(\widehat{f})$. We have a fully faithful
dg functor $\mathcal{M} \stackrel{C}{\rightarrow} \mathcal{C}$.
\end{definition}
Let $\mathcal{A}$ be a small dg category. 

\begin{remark}\label{choice}
Notice that giving a dg functor $H:\mathcal{S} \rightarrow
\mathcal{A}$, respectively $H':\mathcal{S}^{-1} \rightarrow
\mathcal{A}$, corresponds exactly to specifying two objects $X$ and
$Y$ in $\mathcal{A}$ and an isomorphism $\widehat{X}[1]
  \stackrel{\sim}{\longrightarrow} \widehat{Y}$, respectively $\widehat{X}[-1]
  \stackrel{\sim}{\longrightarrow} \widehat{Y}$, in
  $\mathcal{C}_{dg}(\mathcal{A})$. Notice also that giving a dg
  functor $R:\mathcal{C} \rightarrow \mathcal{A}$ corresponds exactly
  to specifying a morphism $f:X \rightarrow Y$ of degree zero in
  $\mathcal{A}$ such that $d(f)=0$, an object $Z$ in $\mathcal{A}$ and
  an isomorphism $Z \stackrel{\sim}{\longrightarrow} \mathsf{cone}(\widehat{f})$.
\end{remark}

Recall from section~\ref{monad} that we have at our disposal an adjunction
$$
\xymatrix{
\mathsf{T}_{\alpha}\mbox{-}\mathsf{alg} \ar@<1ex>[d]^U \\
\mathsf{dgcat} \ar@<1ex>[u]^F\,.
}
$$

\begin{definition}\label{exact1}
Let $\mathsf{dgcat}_{ex,\alpha}$ be the category whose objects are the
$4$-tuples $\underline{A}=(A,S_A,S^{-1}_A,C_A)$, where $A$ is a
$\mathsf{T}_{\alpha}$-algebra and $S_A$, $S_A^{-1}$ and $C_A$, the
structure morphisms of $A$, are $\mathsf{T}_{\alpha}$-algebra morphisms
which make all diagrams
$$
\xymatrix{
\underset{\mathcal{P} \rightarrow \mathcal{A}}{\coprod} F(\mathcal{P}) \ar[r]
\ar[d] & A & \underset{\mathcal{P} \rightarrow \mathcal{A}}{\coprod} F(\mathcal{P}) \ar[r]
\ar[d] & A & \underset{\mathcal{M} \rightarrow
  \mathcal{A}}{\coprod} F(\mathcal{M}) \ar[r] \ar[d] & A \\
\underset{\mathcal{P} \rightarrow \mathcal{A}}{\coprod} F(\mathcal{S}) 
\ar[ur]_-{S_A} & & \underset{\mathcal{P} \rightarrow
  \mathcal{A}}{\coprod} F(\mathcal{S}^{-1}) \ar[ur]_-{S^{-1}_A}
& &  \underset{\mathcal{M} \rightarrow \mathcal{A}}{\coprod}
F(\mathcal{C}) \ar[ur]_-{C_A} & 
}
$$
commutative.

A morphism $G: \underline{A} \rightarrow \underline{B}$ in
$\mathsf{dgcat}_{ex,\alpha}$ consists of a morphism of
$\mathsf{T}_{\alpha}$-algebras $G: A \rightarrow B$ that makes the
following diagrams
$$
\xymatrix{
\underset{\mathcal{P} \rightarrow \mathcal{A}}{\coprod} F(\mathcal{S})
\ar[d]_-{S_A} \ar[r] & \underset{\mathcal{P} \rightarrow \mathcal{B}}{\coprod} F(\mathcal{S})
\ar[d]^-{S_B} &  \underset{\mathcal{P} \rightarrow
  \mathcal{A}}{\coprod} F(\mathcal{S}^{-1}) \ar[d]_-{S^{-1}_A} \ar[r]
&   \underset{\mathcal{P} \rightarrow
  \mathcal{B}}{\coprod} F(\mathcal{S}^{-1}) \ar[d]^-{S^{-1}_B}
&  \underset{\mathcal{M} \rightarrow \mathcal{A}}{\coprod}
F(\mathcal{C}) \ar[d]_-{C_A} \ar[r] &  \underset{\mathcal{M} \rightarrow \mathcal{B}}{\coprod}
F(\mathcal{C}) \ar[d]^-{C_B}\\
A \ar[r]_-{G} & B & A \ar[r]_-{G} & B & A \ar[r]_-{G} & B 
}
$$
commutative.
\end{definition}

\begin{remark}\label{choices1}
Observe that an object $\underline{A}$ in $\mathsf{dgcat}_{ex,\alpha}$
consists of a $\mathsf{T}_{\alpha}$-algebra $A$ and of a choice, in 
the sense of remark~\ref{choice}, for the suspensions and cosuspensions
of every object of the dg category $\mathcal{A}=U(A)$ and also of a
choice for the cone of every cycle of degree zero of the dg category
$\mathcal{A}$. In particular $\mathcal{A}$ is stable under
suspensions, cosuspensions and cones. 

Observe also that a morphism $G$
in $\mathsf{dgcat}_{ex,\alpha}$ consists of a morphism of
$\mathsf{T}_{\alpha}$-algebras that commutes with all these choices. 
\end{remark}

We have a forgetful functor
$$ U_1: \mathsf{dgcat}_{ex,\alpha} \longrightarrow
\mathsf{T}_{\alpha}\mbox{-}\mathsf{alg}\,,$$
that associates to an object $\underline{A}$ of
$\mathsf{dgcat}_{ex,\alpha}$ the $\mathsf{T}_{\alpha}$-algebra $A$.

\begin{proposition}\label{gauche}
The functor $U_1$ admits a left adjoint functor $F_1$.
\end{proposition}

\begin{proof}
The proof will consist in verifying the conditions of theorem~$2$
from section V.$6$. of \cite{Macl}, i.e. that
$\mathsf{dgcat}_{ex,\alpha}$ has small limits and satisfies the
solution set condition. We will now prove that the category $\mathsf{dgcat}_{ex,\alpha}$
admits small limits by showing that we have products and
equalizers. Then we will prove that these are preserved by the functor
$U_1$. Let $\{\underline{A_i} \}_{i \in I}$ be a family of objects in
$\mathsf{dgcat}_{ex,\alpha}$. Endow the $\mathsf{T}_{\alpha}$-algebra
$\underset{i \in I}{\prod} A_i$ with the structure morphisms induced
by those of $\underline{A_i},\, i \in I$. In this way, the
$\mathsf{T}_{\alpha}$-algebra $\underset{i \in I}{\prod} A_i$ belongs
naturally to $\mathsf{dgcat}_{ex,\alpha}$ and we observe that it is the
product in $\mathsf{dgcat}_{ex,\alpha}$ of the familly
$\{\underline{A_i} \}_{i \in I}$.

Consider now morphisms $G_1,G_2:\underline{A} \rightarrow
\underline{B}$ in $\mathsf{dgcat}_{ex,\alpha}$. Let $K$ be the
equalizer in $\mathsf{T}_{\alpha}$-$\mathsf{alg}$ of the pair
$G_1,G_2:A \rightarrow B$. By remark~\ref{choices1}, we need to show
that the dg category $U(K)$ is endowed with a choice for the
suspension and cosuspension for each object and with a choice for the
cone of every cycle of degree zero. Since $U$ is a right adjoint
functor, $U(K)$ identifies with the equalizer of the pair 
$$U(G_1),U(G_2): \mathcal{A} \rightarrow \mathcal{B}$$
and since $G_1$ and $G_2$
are morphisms in $\mathsf{dgcat}_{ex,\alpha}$, the dg functors
$U(G_1)$ and $U(G_2)$ commute with all the choices. This implies that
the the non-full dg subcategory $U(K)$ of $\mathcal{A}$ is stable
under all the choices of suspension, cosuspension and cones in
$\mathcal{A}$. This shows us that $K$ belongs naturally to
$\mathsf{dgcat}_{ex,\alpha}$ and that it is the equalizer in
$\mathsf{dgcat}_{ex,\alpha}$ of the pair $G_1,G_2:\underline{A}
\rightarrow \underline{B}$.
This proves that the category $\mathsf{dgcat}_{ex,\alpha}$ admits
small limits and by construction they are preserved by the forgetful functor $U_1$.

We will now prove that the solution set condition is verified,
see~\cite{Macl}.

Let $A$ be a $\mathsf{T}_{\alpha}$-algebra.
Consider the following set of morphisms in
$\mathsf{T}_{\alpha}$-$\mathsf{alg}$
$$ \mbox{Ens}:=\left\{ F(\mathcal{P}) \stackrel{F(S)}{\longrightarrow}
F(\mathcal{S}),\, F(\mathcal{P}) \stackrel{F(S^{-1})}{\longrightarrow}
F(\mathcal{S}^{-1}),\,  F(\mathcal{M}) \stackrel{F(C)}{\rightarrow}
F(\mathcal{C})\right\}\,.$$
Since $\mathcal{P}$ and $\mathcal{M}$ are clearly small in
$\mathsf{dgcat}$ and the functor $U$
commutes with $\alpha$-filtered colimits, the objects $F(\mathcal{P})$ and
$F(\mathcal{M})$ are $\alpha$-small in
$\mathsf{T}_{\alpha}$-$\mathsf{alg}$, see \cite{Hirschhorn}.

Apply the small object argument, see \cite{Hirschhorn}, to the
morphism
$$ A \longrightarrow 0\,,$$
where $0$ denotes the terminal object in
$\mathsf{T}_{\alpha}$-$\mathsf{alg}$, using the set of morphisms
$\mbox{Ens}$. We obtain a factorization
$$ 
\xymatrix{
A \ar[rr] \ar[dr]_-i & & 0 \\
 & \mathsf{Ex}_{\alpha}(A) \ar[ur]_-q & \,, 
}
$$
where $q$ is a morphism of $\mathsf{T}_{\alpha}$-algebras that has the
right lifting property with respect to all elements of $\mbox{Ens}$.

Now, for each one of the following (solid) commutative squares
$$
\xymatrix{
F(\mathcal{P}) \ar[r] \ar[d]_{F(S)} & \mathsf{Ex}_{\alpha}(A) \ar[d] &
F(\mathcal{P}) \ar[d]_-{F(S^{-1})} \ar[r] & \mathsf{Ex}_{\alpha}(A)
\ar[d] & F(\mathcal{M}) \ar[d]_-{F(C)} \ar[r] & \mathsf{Ex}_{\alpha}(A)
\ar[d] \\
F(\mathcal{S}) \ar[r] \ar@{.>}[ur] & 0 & F(\mathcal{S}^{-1}) \ar[r]
\ar@{.>}[ur] & 0 & F(\mathcal{C}) \ar@{.>}[ur] \ar[r] & 0 \,,
}
$$
choose a morphism of $\mathsf{T}_{\alpha}$-algebras, (here denoted by a
dashed arrow), as in the proof of proposition $10.5.16$ from
\cite{Hirschhorn}, that makes both triangles commutative. Notice that
a set of morphisms as this one specifies structure
morphisms for the $\mathsf{T}_{\alpha}$-algebra
$\mathsf{Ex}_{\alpha}(A)$. This shows us that when endowed with these
choices $\mathsf{Ex}_{\alpha}(A)$ belongs to $\mathsf{dgcat}_{ex,\alpha}$.

Let now $\underline{B}$ be an object of $\mathsf{dgcat}_{ex,\alpha}$
and $Q:A \rightarrow B$ be a morphism of
$\mathsf{T}_{\alpha}$-algebras.
Observe that the structure morphisms of $B$ and the
construction of the $\mathsf{T}_{\alpha}$-algebra
$\mathsf{Ex}_{\alpha}(A)$ by the small object argument, see the proof of
proposition $10.5.16$ in \cite{Hirschhorn}, allow us to define, by
transfinite induction, a morphism $\overline{Q}$ of
$\mathsf{T}_{\alpha}$-algebras such that the diagram
$$
\xymatrix{
A \ar[r]^-{i} \ar[dr]_-{Q} &  \mathsf{Ex}_{\alpha}(A)
\ar[d]^-{\overline{Q}} \\
 & B
}
$$
commutes. Observe also that when the $\mathsf{T}_{\alpha}$-algebra
$\mathsf{Ex}_{\alpha}(A)$ is endowed with the above choices the
morphism $\overline{Q}$ becomes a morphism in
$\mathsf{dgcat}_{ex,\alpha}$. This proves the solution set condition.

The proposition is now proven.
\end{proof}

We have the following adjunctions
$$
\xymatrix{
\mathsf{dgcat}_{ex,\alpha} \ar@<1ex>[d]^{U_1}\\
\mathsf{T}_{\alpha}\mbox{-}\mathsf{alg} \ar@<1ex>[d]^U
\ar@<1ex>[u]^{F_1} \\
\mathsf{dgcat} \ar@<1ex>[u]^F\,.
}
$$

\begin{proposition}\label{monadic}
The functor $U_1$ is monadic, see \cite{Macl}.
\end{proposition}

\begin{proof}
To prove that the functor $U_1$ is monadic, see section $3$ of
\cite{Macl}, we will verify condition $(iii)$ of theorem $1$
from section VI.$7$. of \cite{Macl}.

Let $G_1,G_2:\underline{A} \rightarrow \underline{B}$ be a pair of
morphisms in $\mathsf{dgcat}_{ex,\alpha}$. Consider the following
split coequalizer, see \cite{Macl}, in
$\mathsf{T}_{\alpha}$-$\mathsf{alg}$
$$
\xymatrix{
A \ar@<0.5ex>[r]^{G_1} \ar@<-0.5ex>[r]_{G_2} & B \ar@/^1.5pc/[l]^-{R}  \ar[r]^L & D
\ar@/^1.5pc/[l]^-{Q} \,,
}
$$
where $L\circ G_1 = L \circ G_2$, $L\circ Q =Id$, $G_1 \circ R = Id$,
and $G_2\circ R=Q\circ L$.

We will now construct structure morphisms for $D$ such that $D$ will
become an object of $\mathsf{dgcat}_{ex,\alpha}$ and $L$ a morphism
in $\mathsf{dgcat}_{ex,\alpha}$. Apply the functor $U$ to the previous
split coequalizer in $\mathsf{T}_{\alpha}$-$\mathsf{alg}$ and obtain
$$
\xymatrix{
\mathcal{A} \ar@<0.5ex>[r]^{G_1} \ar@<-0.5ex>[r]_{G_2} & \mathcal{B} \ar@/^1.5pc/[l]^-{R}  \ar[r]^L & \mathcal{D}
\ar@/^1.5pc/[l]^-{Q} \,.
}
$$
Now, apply the functors
$$ \underset{\mathcal{M} \rightarrow ?}{\coprod} \mathcal{M}\,,
\underset{\mathcal{M} \rightarrow ?}{\coprod}\mathcal{C} :
\mathsf{dgcat} \rightarrow \mathsf{dgcat}$$
to the previous split coequalizer in $\mathsf{dgcat}$ and obtain the
following diagram
$$
\xymatrix{
 & \underset{\mathcal{M} \rightarrow \mathcal{A}}{\coprod} \mathcal{C}
 \ar@<0.5ex>[rr] \ar@<-0.5ex>[rr] \ar[ldd]^(.3){C_{\mathcal{A}}}|\hole & & \underset{\mathcal{M} \rightarrow \mathcal{B}}{\coprod} \mathcal{C}
 \ar[rr] \ar[ldd]^(.3){C_{\mathcal{B}}}|\hole & & \underset{\mathcal{M} \rightarrow
   \mathcal{C}}{\coprod} \mathcal{C} \ar@{.>}[ldd]^-{C_{\mathcal{D}}} \\
\underset{\mathcal{M} \rightarrow \mathcal{A}}{\coprod} \mathcal{M}
\ar@<0.5ex>[rr] \ar@<-0.5ex>[rr] \ar[d]  \ar[ur] & & \underset{\mathcal{M} \rightarrow
  \mathcal{B}}{\coprod} \mathcal{M} \ar[rr] \ar[d] \ar[ur] & &
\underset{\mathcal{M} \rightarrow \mathcal{C}}{\coprod} \mathcal{M}
\ar[d]  \ar[ur] & \\
\mathcal{A} \ar@<0.5ex>[rr]^-{G_1}  \ar@<-0.5ex>[rr]_-{G_2} & &
\mathcal{B} \ar[rr]^L  \ar@/^1.5pc/[ll]^-{R}& & \mathcal{D} \ar@/^1.5pc/[ll]^-{Q}& 
}
$$

Notice that since $L$ is a split coequalizer in $\mathsf{dgcat}$, the rows in the
diagram are coequalizers. This implies that the dg functors
$C_{\mathcal{A}}$ and $C_\mathcal{B}$, which correspond under the
adjunction $(F,U)$ to the structure morphisms $C_A$ and $C_B$, induce
a dg functor $C_{\mathcal{D}}$. Now, since $L$ admits a right inverse $Q$, a
diagram chasing argument shows us that the right triangle in the
diagram is commutative.
 
Observe that under the adjunction $(F,U)$ this commutative diagram
corresponds exactly to a structure morphism $C_D$ on $D$. Clearly by construction the morphism $U$
preserves this structure morphism. Now, consider an analogous
argument for the construction of structure morphisms $Q_D$ and
$Q^{-1}_D$.

We will now prove that $L$ is a coequalizer in the category
$\mathsf{dgcat}_{ex,\alpha}$. Let $\underline{E}$ be an object of
$\mathsf{dgcat}_{ex,\alpha}$ and $H:\underline{B} \rightarrow
\underline{E}$ a morphism such that $H \circ G_1 = H \circ G_2$. Since
the morphism $L$ is a coequalizer in
$\mathsf{T}_{\alpha}$-$\mathsf{alg}$, there exists an unique morphism
of $\mathsf{T}_{\alpha}$-algebras $R$ which makes the following
diagram
$$
\xymatrix{
A \ar@<0.5ex>[r]^{G_1} \ar@<-0.5ex>[r]_{G_2} & B \ar[r]^L \ar[dr]_H & D
\ar@{.>}[d]^R \\
 & & E
}
$$
commutative.
We now prove that $R$ is a morphism in
$\mathsf{dgcat}_{ex,\alpha}$. Since $L \circ Q = Id$ we have $R=H\circ
Q$. Now apply the functors 
$$\underset{\mathcal{P}
  \rightarrow?}{\coprod}\mathcal{S}\,,\underset{\mathcal{P}
  \rightarrow ?}{\coprod}\mathcal{S}^{-1}\,,\underset{\mathcal{M}
  \rightarrow ?}{\coprod}\mathcal{C} : \mathsf{dgcat} \longrightarrow \mathsf{dgcat}\,,$$
to the image of the diagram above under the functor $U$, and use a diagram
chasing argument to conclude that $L$ belongs in fact to $\mathsf{dgcat}_{ex,\alpha}$. This proves the
proposition.
\end{proof}
Notice that since $U_1$ is monadic and the category $\mathsf{T}_{\alpha}$-$\mathsf{alg}$ is
complete, proposition $4.3.1$ from \cite{Bor} implies that $\mathsf{dgcat}_{ex,\alpha}$ is also complete.

\begin{proposition}\label{filtered2}
The category $\mathsf{dgcat}_{ex,\alpha}$ admits $\alpha$-filtered
colimits and these are preserved by the functor $U_1$.
\end{proposition}

\begin{proof}
Let $\{\underline{A_i}\}_{i \in I}$ be an $\alpha$-filtered diagram in
$\mathsf{dgcat}_{ex,\alpha}$. Consider the colimit
$$ Y :=\underset{i \in I}{\mbox{colim}}\,A_i\,,$$
of the $\alpha$-filtered diagram $\{A_i\}_{i \in I}$ of
$\mathsf{T}_{\alpha}$-algebras. 

We will now construct structure
morphisms for $Y$ such that $Y$ becomes the colimit in
$\mathsf{dgcat}_{ex,\alpha}$ of the diagram $\{\underline{A_i}\}_{i
  \in I}$. Since the functor $U$ commutes with $\alpha$-filtered
colimits, see section~\ref{homotopy2}, we have
$$ U(Y) = \underset{i \in I}{\mbox{colim}}\, \mathcal{A}_i\,.$$
We now construct a structure morphism $C_Y$, see
definition~\ref{exact1}, for $Y$. We have the following
$\alpha$-filtered diagrams in $\mathsf{dgcat}$ and morphisms
between them
$$
\xymatrix{
\left\{ \underset{\mathcal{M} \rightarrow
  \mathcal{A}_i}{\coprod}\mathcal{M} \right\}_{i \in I} \ar[r] \ar[d] & \left\{
\mathcal{A}_i \right\}_{i \in I} \\
\left\{ \underset{\mathcal{M} \rightarrow
  \mathcal{A}_i}{\coprod}\mathcal{C} \right\}_{i \in I}
\ar[ur]_-{C_{\mathcal{A}_i}} & 
}
$$
Now notice that since we are considering $\alpha$-filtered colimits,
we have
$$ \underset{i \in I}{\mbox{colim}} \underset{\mathcal{M}\rightarrow
  \mathcal{A}_i}{\coprod}\mathcal{M} \stackrel{\sim}{\longrightarrow} \underset{\mathcal{M}
  \rightarrow \underset{i \in I}{\mbox{colim}}\,\mathcal{A}_i}{\coprod}
\mathcal{M}$$
and
$$ \underset{i \in I}{\mbox{colim}} \underset{\mathcal{M}\rightarrow
  \mathcal{A}_i}{\coprod}\mathcal{C} \stackrel{\sim}{\longrightarrow} \underset{\mathcal{M}
  \rightarrow \underset{i \in I}{\mbox{colim}}\,\mathcal{A}_i}{\coprod}
\mathcal{C}\,.$$
This implies that the morphism of $\mathsf{T}_{\alpha}$-algebras which
corresponds under the adjunction $(F,U)$ to the dg functor
$ \underset{i \in I}{\mbox{colim}} \,C_{\mathcal{A}_i} $
is a structure morphism $C_Y$ of $Y$. Consider now an analogous
argument for the construction of structure morphisms $S_Y$ and
$S^{-1}_Y$.

Finally since the functor $U$ commutes with
$\alpha$-filtered colimits, $Y$ is clearly the colimit in
$\mathsf{dgcat}_{ex,\alpha}$ of the diagram  $\{\underline{A_i}\}_{i
  \in I}$. This proves the proposition.
\end{proof}

\begin{proposition}
The category $\mathsf{dgcat}_{ex,\alpha}$ is cocomplete.
\end{proposition}

\begin{proof}
Recall that by proposition~\ref{monadic} the adjunction
$$
\xymatrix{
\mathsf{dgcat}_{ex,\alpha} \ar@<1ex>[d]^{U_1}\\
\mathsf{T}_{\alpha}\mbox{-}\mathsf{alg}
\ar@<1ex>[u]^{F_1} \,,
}
$$
is monadic. Now by propositions $4.3.2$ and $4.3.6$ of \cite{Bor} we
only need to show that the functor $U_1 \circ F_1$ commutes with
$\alpha$-filtered colimits. But this follows from the fact that $F_1$ is
a left adjoint and that by proposition~\ref{filtered2} $\alpha$-filtered
colimits exist in $\mathsf{dgcat}_{ex,\alpha}$ and are preserved by
$U_1$.
This proves the proposition.
\end{proof}

We will now construct path objects in
$\mathsf{dgcat}_{ex,\alpha}$. For this we consider definition~\ref{Quillen} applied to our particular adjunction
$(F_1,U_1)$, see theorem~\ref{main2}.

Let $\underline{A}$ be an object of $\mathsf{dgcat}_{ex,\alpha}$.
\begin{proposition}\label{2path1}
The $\mathsf{T}_{\alpha}$-algebra $P(A)$, see proposition~\ref{path},
is endowed with natural structure morphisms and so $\underline{A}$
admits a path object in $\mathsf{dgcat}_{ex,\alpha}$.
\end{proposition}
\begin{proof}
We will construct structure morphisms $S_{P(A)}$, $S_{P(A)}^{-1}$ and
$C_{P(A)}$ for $P(A)$, see definition~\ref{exact1}, in such a way that
$P(A)$ becomes a path object in $\mathsf{dgcat}_{ex,\alpha}$.
Construct the structure morphism $S_{P(A)}$, respectively
$S_{P(A)}^{-1}$, by applying $S_A$, respectively $S_A^{-1}$, componentwise.

By remark~\ref{choices1}, to construct a
structure morphism $C_{P(A)}$, we need to construct a cone in the
category $P(\mathcal{A})$ for every cycle of degree zero in $P(\mathcal{A})$. Let now $(m_X,m_Y,h)$ be a cycle of degree $0$ in $P(\mathcal{A})$
between the objects $X \stackrel{f}{\rightarrow} Y$ and $X'
\stackrel{f'}{\rightarrow} Y'$ of $P(\mathcal{A})$. In particular $h$
is a morphism in $\mathcal{A}$ of degree $-1$ such that $d(h)=m_Y\circ
f - f'\circ m_X$. Consider the following diagram
$$
\xymatrix{
X \ar[d]_{m_X} \ar[rr]^f  \ar[drr]|{h}& &  Y \ar[d]^{m_Y} \\
X' \ar[rr]^{f'} \ar[d]_{i_X}  \ar@{.>}[drr]|{h'} &  & Y' \ar[d]^{i_Y} \\
\mathsf{cone}(m_X) \ar@{.>}[rr]_-{\Phi} &  & \mathsf{cone}(m_Y)\,,
}
$$
where $h'=0$ and $\Phi$ is the morphism defined by the matrix
$$
\begin{bmatrix} 
f' & h \\ 0 & S(f)
\end{bmatrix}\,.
$$
Observe that the object in $P(\mathcal{A})$
$$ 
\xymatrix{
\mathsf{cone}(m_X) \ar@{.>}[rr]^-{\Phi} & & \mathsf{cone}(m_Y)\,,
}
$$
corepresents the cone of the morphism $(m_X,m_Y,h)$ in $P(\mathcal{A})$.
This shows us that $P(A)$ belongs to $\mathsf{dgcat}_{ex,\alpha}$. It
is also clear that the morphisms $A
\stackrel{i_A}{\longrightarrow} P(A)$ and 
$$
\xymatrix{
 P(A) \ar@{->>}[r]^-{q_A} &
 A\times A
}
$$ are in fact morphisms in
$\mathsf{dgcat}_{ex,\alpha}$ and that $i_A$ is a weak equivalence. This proves the proposition.
\end{proof}

We will now prove the main result.
\begin{theorem}\label{final}
The category $\mathsf{dgcat}_{ex,\alpha}$ when endowed with
the notions of weak equivalence, fibration and cofibration as in
definition~\ref{lift}, becomes a cofibrantly generated Quillen model
category and the adjunction $(F_1,U_1)$ becomes a Quillen adjunction.
\end{theorem}

\begin{proof}
Recall from theorem~\ref{main2} that we have an explicit set
$F(I)$ of generating cofibrations and an
explicit set $F(J)$ of generating trivial
cofibrations for $\mathsf{T}_{\alpha}$-$\mathsf{alg}$. 

Now notice that all conditions of
theorem~\ref{liftarg} are satisfied. In particular
proposition~\ref{2path1} and the fact that every object in
$\mathsf{T}_{\alpha}$-$\mathsf{alg}$ is fibrant imply
proposition~\ref{Quillen} which implies that the assumption on cofibrations of
theorem~\ref{liftarg} holds. 

Observe that $F_1(F(I))$ is a set of generating
cofibrations on $\mathsf{dgcat}_{ex,\alpha}$ and that $F_1(F(J))$ is a set of generating
acyclic cofibrations on $\mathsf{dgcat}_{ex,\alpha}$. This implies that the Quillen
model structure on $\mathsf{dgcat}_{ex,\alpha}$ is cofibrantly
generated. Since the functor $U_1$ preserves by definition weak
equivalences and fibrations the adjunction $(F_1,U_1)$ is a Quillen
adjunction.
This proves the theorem.
\end{proof}

\section{Enhancement of well-generated algebraic triangulated
  categories}\label{finale}

In this section, we show that the category of $\alpha$-compactly generated algebraic
triangulated categories up to equivalence admits a natural Quillen
enhancement given by our model category $\mathsf{dgcat}_{ex,\alpha}$.

Let $\mathsf{Tri}_{\alpha}$ denote the category of $\alpha$-compactly generated algebraic
triangulated categories in the sense of Neeman, see \cite{Krause} \cite{Neeman}.

Let $\underline{A}$ be an object of $\mathsf{dgcat}_{ex,\alpha}$, with
underlying dg category $\mathcal{A}$.
Recall from proposition~\ref{sums} that $\mathcal{A}$ admits
$\alpha$-small sums.

\begin{definition}[\cite{Porta}]\label{defporta}
The $\alpha$-continuous derived category $\mathcal{D}_{\alpha}(\underline{A})$ of
$\underline{A}$ is the classical triangle quotient, see \cite{Neeman} \cite{Verdier}, of the derived category
$\mathcal{D}(\mathcal{A})$ of $\mathcal{A}$ by the localizing
triangulated subcategory, i.e. stable under infinite sums, of $\mathcal{D}(\mathcal{A})$ generated by
the cones on the canonical morphisms
$$ \underset{i \in I}{\bigoplus} \widehat{X_i} \longrightarrow
\widehat{\underset{i \in I}{\bigoplus} X_i}\,,$$
where $(X_i)_{i \in I}$ is a family of objects of $\mathcal{A}$ and $I$ is a set of
cardinality strictly smaller than $\alpha$.
\end{definition}

\begin{remark}
By a theorem of \cite{Porta}, $\mathcal{D}_{\alpha}(\underline{A})$ belongs to
$\mathsf{Tri}_{\alpha}$. 

Observe that this construction is functorial in
$A$. In fact, let $F:\underline{A} \rightarrow \underline{B}$ be a
morphism in $\mathsf{dgcat}_{ex,\alpha}$. Since the dg functor
$F:\mathcal{A} \rightarrow \mathcal{B}$ commutes with $\alpha$-small
sums, see remark~\ref{sums2}, it induces a functor
$$ \mathcal{D}_{\alpha}(F): \mathcal{D}_{\alpha}(\underline{A})
\rightarrow D_{\alpha}(\underline{B})$$
between triangulated categories.
\end{remark}
Thus, we have defined a functor
$$ \mathcal{D}_{\alpha}(-): \mathsf{dgcat}_{ex,\alpha}
\longrightarrow \mathsf{Tri}_{\alpha}\,.$$
Observe that if $F:\underline{A} \rightarrow \underline{B}$ is a weak
equivalence in $\mathsf{dgcat}_{ex,\alpha}$, i.e. the dg functor
$F:\mathcal{A} \rightarrow \mathcal{B}$ induces an equivalence of
categories $\mathsf{H}^0(F):\mathsf{H}^0(\mathcal{A}) \rightarrow
\mathsf{H}^0(\mathcal{B})$, then the triangulated functor
$\mathcal{D}_{\alpha}(F)$ is an equivalence of triangulated categories.

The following theorem is proven in \cite{Porta}, \emph{cf.}~\cite{Ober}.

\begin{theorem}[\cite{Porta}]
The functor $\mathcal{D}_{\alpha}(-)$ satisfies the conditions:
\begin{itemize}
\item[-]  every category $\mathcal{T}$ in $\mathsf{Tri}_{\alpha}$ is equivalent to
  $\mathcal{D}_{\alpha}(\underline{A})$ for some object
  $\underline{A}$ in $\mathsf{dgcat}_{ex,\alpha}$ and
\item[-] a morphism $F$ in $\mathsf{dgcat}_{ex,\alpha}$ is a weak
  equivalence if and only if $\mathcal{D}_{\alpha}(F)$ is an
  equivalence of triangulated categories.
\end{itemize}
\end{theorem}

\part{Applications {\`a} la DG-(d{\'e})stabilisation}

\chapter{On the structure of Calabi-Yau categories with a cluster
  tilting subcategory}

\textit{\small{Ce chapitre correspond {\`a} l'article \cite{Documenta}.}}

\section{Introduction}
In this article, we propose a description of a class of
Calabi-Yau categories using the formalism of dg-categories and the
notion of `stabilization', as used for the description of triangulated
orbit categories in section $7$ of \cite{orbit}.
For $d \geq 2$, let $\mathcal{C}$ be an algebraic $d$-Calabi-Yau
triangulated category endowed with a $d$-cluster tilting subcategory
$\mathcal{T}$,
\emph{cf.}~\cite{cluster}~\cite{Iyama32}~\cite{Iyama33}, see also \cite{Raf}~\cite{Ginz}~\cite{Ginz1}. Such
categories occur for example,
\begin{itemize}
\item[-] in the representation-theoretic approach to
  Fomin-Zelevinsky's cluster algebras \cite{Cluster1},
  \emph{cf.}~\cite{BMRT} \cite{CK2} \cite{GLS} and the references
  given there,
\item[-] in the study of Cohen-Macaulay modules over certain isolated
  singularities, \emph{cf.}~\cite{IR}~\cite{cluster}~\cite{IY}, and the
  study of non commutative crepant resolutions~\cite{VdB}, \emph{cf.}~\cite{IR}.
\end{itemize}
From $\mathcal{C}$ and $\mathcal{T}$ we construct an exact dg category
$\mathcal{B}$, which is perfectly $(d+1)$-Calabi-Yau, and a
non-degenerate aisle $\mathcal{U}$,
\emph{cf.}~\cite{aisles}, in $\mathrm{H}^0(\mathcal{B})$ whose
heart has enough projectives.
We prove, in theorem~\ref{main3}, how to recover the category
$\mathcal{C}$ from $\mathcal{B}$ and $\mathcal{U}$ using a
general procedure of stabilization defined in
section~\ref{mainsec}. This extends previous results of
\cite{preprint} to a more general framework.

It follows from
\cite{Palu} that for $d=2$, up to derived equivalence, the category
$\mathcal{B}$ only depends on $\mathcal{C}$ (with its enhancement) and
not on the choice of $\mathcal{T}$.
In the appendix, we show how to naturally extend a $t$-structure, \emph{cf.}~\cite{Ast100}, on the
compact objects of a triangulated category to the whole category.
\\

\textbf{Example}
Let $k$ be a field, $A$ a finite-dimensional hereditary $k$-algebra
and $\mathcal{C}=\mathcal{C}_A$ the cluster category of $A$, see
\cite{BMRRT}~\cite{CCS}, i.e. the quotient of the bounded derived
category of finitely generated modules over $A$ by the functor
$F=\tau^{-1}[1]$, where $\tau$ denotes the $\mathsf{AR}$-translation
and $[1]$ denotes the shift functor.

Then $\mathcal{B}$ is given by the dg algebra, see section $7$ of
\cite{orbit},
$$ B= A\oplus (DA)[-3]$$
and theorem~\ref{main3} reduces to the equivalence
$$ \mathcal{D}^b(\mathcal{B})/\mathsf{per}(\mathcal{B})
\stackrel{\sim}{\longrightarrow} \mathcal{C}_A$$
of section $7.1$ of \cite{orbit}.

\section{Preliminaries}\label{preli}

Let $k$ be a field. Let $\mathcal{E}$ be a $k$-linear Frobenius
category with split idempotents. Suppose that its stable category
$\mathcal{C}=\underline{\mathcal{E}}$, with suspension functor $S$,
has finite-dimensional $\mathrm{Hom}$-spaces and admits a Serre
functor $\Sigma$, see \cite{Bon-Kap}. Let $d\geq 2$ be an
integer. We suppose that $\mathcal{C}$ is Calabi-Yau of CY-dimension
$d$, \emph{i.e.}~\cite{ENS} there is an isomorphism of triangle functors
$$S^d \stackrel{\sim}{\rightarrow} \Sigma\,.$$
We fix such an isomorphism once and for all. See section $4$ of~\cite{cluster} for
several examples of the above situation.

For $X,Y \in \mathcal{C}$
and $n\in \mathbb{Z}$, we put 
$$ \mathrm{Ext}^n(X,Y)=\mathrm{Hom}_{\mathcal{C}}(X,S^nY)\,.$$
We suppose that $\mathcal{C}$ is endowed with a $d$-cluster tilting subcategory $\mathcal{T} \subset \mathcal{C}$, \emph{i.e.}
\begin{itemize}
\item[a)] $\mathcal{T}$ is a $k$-linear subcategory,
\item[b)] $\mathcal{T}$ is functorially finite in $\mathcal{C}$,
  i.e. the functors $\mathrm{Hom}_{\mathcal{C}}(?,X)|\mathcal{T}$ and
  $\mathrm{Hom}_{\mathcal{C}}(X,?)|\mathcal{T}$ are finitely
  generated for all $X \in \mathcal{C}$,
\item[c)] we have $\mathrm{Ext}^i(T,T')=0$ for all $T,T' \in \mathcal{T}$
  and all $0<i<d$ and
\item[d)] if $X \in \mathcal{C}$ satisfies $\mathrm{Ext}^i(T,X)=0$ for all
  $0<i<d$ and all $T \in \mathcal{T}$, then $T$ belongs to $\mathcal{T}$.
 
\end{itemize}
 
Let $\mathcal{M}\subset \mathcal{E}$ be the preimage of $\mathcal{T}$
under the projection functor. In particular, $\mathcal{M}$ contains
the subcategory $\mathcal{P}$ of the projective-injective objects in
$\mathcal{M}$. Note that $\mathcal{T}$ equals the quotient
$\underline{\mathcal{M}}$ of $\mathcal{M}$ by the ideal of morphisms
factoring through a projective-injective.

We have the following commutative square:
{\small
$$
\xymatrix{
*+<1pc>{\mathcal{M}} \ar@{^{(}->}[r] \ar@{->>}[d] & \mathcal{E} \ar@{->>}[d] \\
*+<1pc>{\mathcal{T}}  \ar@{^{(}->}[r]  & \underline{\mathcal{E}} =\mathcal{C} \,.
}
$$}

We use the notations of \cite{ICM}. In particular, for an additive
category $\mathcal{A}$, we denote by $\mathcal{C}(\mathcal{A})$ (resp. $\mathcal{C}^-(\mathcal{A})$, $\mathcal{C}^b(\mathcal{A})$,
$\ldots$) the category of unbounded (resp. right bounded,
resp. bounded, $\ldots$) complexes over $\mathcal{A}$ and by
$\mathcal{H}(\mathcal{A})$ (resp. $\mathcal{H}^-(\mathcal{A})$,
$\mathcal{H}^b(\mathcal{A})$, $\ldots$) its quotient modulo
the ideal of nullhomotopic morphisms. By \cite{KellerVos}, \emph{cf.} also
\cite{Rickard1}, the projection functor $\mathcal{E} \rightarrow
\underline{\mathcal{E}}$ extends to a canonical triangle functor
$\mathcal{H}^b(\mathcal{E})/\mathcal{H}^b(\mathcal{P}) \rightarrow
\underline{\mathcal{E}}$. This induces a triangle functor $\mathcal{H}^b(\mathcal{M})/\mathcal{H}^b(\mathcal{P}) \rightarrow
\underline{\mathcal{E}}$. It is shown in \cite{Palu} that this
functor is a localization functor. Moreover, the projection functor
$\mathcal{H}^b(\mathcal{M}) \rightarrow
\mathcal{H}^b(\mathcal{M})/\mathcal{H}^b(\mathcal{P})$ induces an
equivalence from the subcategory
$\mathcal{H}^b_{\mathcal{E}\mbox{-}ac}(\mathcal{M})$ of bounded
$\mathcal{E}$-acyclic complexes with components in $\mathcal{M}$ onto
its kernel. Thus, we have a short exact sequence of triangulated categories
$$ 0 \longrightarrow \mathcal{H}^b_{\mathcal{E}\mbox{-}ac}(\mathcal{M}) \longrightarrow
\mathcal{H}^b(\mathcal{M})/ \mathcal{H}^b\mathcal(\mathcal{P})
\longrightarrow \mathcal{C} \longrightarrow 0 \,.$$

Let $\mathcal{B}$ be the dg (=differential graded) subcategory of the
category $\mathcal{C}^b(\mathcal{M})_{dg}$ of bounded complexes over
$\mathcal{M}$ whose objects are the $\mathcal{E}$-acyclic complexes. We
denote by $G : \mathcal{H}^-(\mathcal{M}) \rightarrow
\mathcal{D}(\mathcal{B}^{op})^{op}$ the functor which takes a right
bounded complex $X$ over $\mathcal{M}$ to the dg module
$$B \mapsto \mathrm{Hom}^\bullet_{\mathcal{M}}(X,B) \,,$$
where $B$ is in $\mathcal{B}$.
\begin{remark}
By construction, the functor $G$ restricted to
$\mathcal{H}^b_{\mathcal{E}\mbox{-}ac}(\mathcal{M})$ establishes an
equivalence
$$ G : \mathcal{H}^b_{\mathcal{E}\mbox{-}ac}(\mathcal{M})
\stackrel{\sim}{\longrightarrow} \mathrm{per}(\mathcal{B}^{op})^{op} \,.$$

\end{remark}
Recall that if $P$ is a right bounded complex of projectives and $A$
is an acyclic complex, then each morphism from $P$ to $A$ is
nullhomotopic. In particular, the complex
$\mathrm{Hom}^\bullet_{\mathcal{M}}(P,A)$ is nullhomotopic for each $P$
in $\mathcal{H}^- (\mathcal{P})$. Thus $G$ takes $\mathcal{H}^- (P)$ to
zero, and induces a well defined functor (still denoted by $G$)
$$G : \mathcal{H}^b(\mathcal{M})/ \mathcal{H}^b\mathcal(\mathcal{P})
\longrightarrow \mathcal{D}(\mathcal{B}^{op})^{op} \,.$$

\section{Embedding}

\begin{proposition}\label{pleinfidele}
The functor $G$ is fully faithful.
\end{proposition}
For the proof, we need a number of lemmas. \\
It is well-known that the category $\mathcal{H}^-(\mathcal{E})$ admits
a semiorthogonal decomposition, \emph{cf.}~\cite{Bon-Orl}, formed by  $\mathcal{H}^-(\mathcal{P})$
and its right orthogonal
$\mathcal{H}^-_{\mathcal{E}\mbox{-}ac}(\mathcal{E})$, the full
subcategory of the right bounded $\mathcal{E}$-acyclic complexes. For $X$ in $\mathcal{H}^-(\mathcal{E})$,
we write 
$$ {\bf{p}}X \rightarrow X \rightarrow {\bf{a}}_pX \rightarrow S{\bf{p}}X$$
for the corresponding triangle, where ${\bf{p}}X$ is in $\mathcal{H}^-(\mathcal{P})$ and
${\bf{a}}_p X$ is in $\mathcal{H}^{-}_{\mathcal{E}\mbox{-}ac}(\mathcal{E})$. 
If $X$ lies in $\mathcal{H}^-(\mathcal{M})$, then clearly
${\bf{a}}_pX$ lies in
$\mathcal{H}^-_{\mathcal{E}\mbox{-}ac}(\mathcal{M})$ so that we have
an induced semiorthogonal decomposition of $\mathcal{H}^-(\mathcal{M})$.

\begin{lemma}\label{fidele}
The functor $\Upsilon:\mathcal{H}^b(\mathcal{M})/ \mathcal{H}^b\mathcal(\mathcal{P})
\longrightarrow \mathcal{H}^-_{\mathcal{E}\mbox{-}ac}(\mathcal{M})$
which takes $X$ to ${\bf{a}}_p X$ is fully faithful.
\end{lemma}

\begin{proof}
By the semiorthogonal decomposition of $\mathcal{H}^-(\mathcal{M})$,
the functor $X \mapsto {\bf{a}}_pX$ induces a right adjoint of the
localization functor
$$\mathcal{H}^-(\mathcal{M}) \longrightarrow \mathcal{H}^-(\mathcal{M})/
\mathcal{H}^-(\mathcal{P})$$
and an equivalence of the quotient category with the right orthogonal $\mathcal{H}^-_{\mathcal{E}\mbox{-}ac}(\mathcal{M})$.
$$
\xymatrix{
 & *+<1pc>{\mathcal{H}^-(\mathcal{P})} \ar@{_{(}->}@<-2ex>[d] & \\
 & *+<1pc>{\mathcal{H}^-(\mathcal{M})} \ar@<-2ex>[d] \ar[u] &
*+<1pc>{\mathcal{H}^{-}_{\mathcal{E}\mbox{-}ac}(\mathcal{M})=\mathcal{H}(\mathcal{P})^{\bot}} \ar@{_{(}->}[l]
\\
*+<1pc>{\mathcal{H}^b(\mathcal{M})/\mathcal{H}^b(\mathcal{P})}  \ar@{^{(}->}[r]
&  *+<1pc>{\mathcal{H}^-(\mathcal{M})/\mathcal{H}^-(\mathcal{P})} \ar@{_{(}->}[u]  \ar@{-->}[ru]^{\sim} & \,.\\
}
$$
Moreover, it is easy
to see that the canonical functor
$$\mathcal{H}^b(\mathcal{M})/\mathcal{H}^b(\mathcal{P}) \longrightarrow
\mathcal{H}^-(\mathcal{M})/\mathcal{H}^-(\mathcal{P})$$
is fully faithful so that we obtain a fully faithful functor
$$ \mathcal{H}^b(\mathcal{M})/\mathcal{H}^b(\mathcal{P}) \longrightarrow
\mathcal{H}^-_{\mathcal{E}\mbox{-}ac}(\mathcal{M})$$
taking $X$ to ${\bf{a}}_p X$.
\end{proof}

\begin{remark}
Since the functor $G$ is triangulated and takes
$\mathcal{H}^-(\mathcal{P})$ to zero, for $X$ in
$\mathcal{H}^b(\mathcal{M})$, the adjunction morphism $X
\rightarrow {\bf{a}}_pX$ yields an isomorphism 
$$ G(X) \stackrel{\sim}{\longrightarrow} G({\bf{a}}_pX)=G(\Upsilon X)\,.$$
\end{remark}

Let $\mathcal{D}^-_{\underline{\mathcal{M}}}(\mathcal{M})$ be the full
subcategory of the derived category $\mathcal{D}(\mathcal{M})$ formed
by the right bounded complexes whose homology modules lie in the
subcategory $\mathrm{Mod}\,\underline{\mathcal{M}}$ of
$\mathrm{Mod}\,\mathcal{M}$. The Yoneda functor $\mathcal{M}
\rightarrow \mathrm{Mod}\,\mathcal{M}$, $M \mapsto M^{\wedge}$, induces a full
embedding
$$ \Psi :\mathcal{H}^-_{\mathcal{E}\mbox{-}ac}(\mathcal{M}) \hookrightarrow
\mathcal{D}^-_{\underline{\mathcal{M}}}(\mathcal{M}) \,.$$ 
We write $\mathcal{V}$ for its essential image.
Under $\Psi$, the category
$\mathcal{H}^b_{\mathcal{E}\mbox{-}ac}(\mathcal{M})$ is identified
with $\mathrm{per}_{\underline{\mathcal{M}}}(\mathcal{M})$.
Let $\Phi : \mathcal{D}^-_{\underline{\mathcal{M}}}(\mathcal{M}) \rightarrow
\mathcal{D}(\mathcal{B}^{op})^{op}$ be the functor which takes $X$ to the dg module
$$ B \mapsto \mathrm{Hom}^\bullet(X_c, \Psi(B))\,,$$
where $B$ is in $\mathcal{H}^b_{\mathcal{E}\mbox{-}ac}(\mathcal{M})$
and $X_c$ is a cofibrant replacement of $X$ for the projective model
structure on $\mathcal{C}(\mathcal{M})$.
Since for each right bounded complex $M$ with components in
$\mathcal{M}$, the complex $M^{\wedge}$ is cofibrant in
$\mathcal{C}(\mathcal{M})$, it is clear that the functor $G :
\mathcal{H}^b(\mathcal{M})/ \mathcal{H}^b(\mathcal{P}) \rightarrow
\mathcal{D}(\mathcal{B}^{op})^{op}$ is isomorphic to the composition $
\Phi \circ \Psi \circ \Upsilon$.
We have the following commutative diagram

{\small
$$
\xymatrix{
\mathcal{H}^b(\mathcal{M})/ \mathcal{H}^b(\mathcal{P})
\ar@{^{(}->}[r]^-{\Upsilon} & \mathcal{H}^-_{\mathcal{E}\mbox{-}ac}(\mathcal{M}) \ar[rr]^{\Psi}
\ar[dr]^{\sim} &   & \mathcal{D}^-_{\underline{\mathcal{M}}}(\mathcal{M})
\ar[r]^{\Phi} &   \mathcal{D}(\mathcal{B}^{op})^{op} \\
 &  & {\mathcal{V}}
  \ar@{^{(}->}[ur] \ar@{_{(}-->}[rru] & & \\
 \mathcal{H}^b_{\mathcal{E}\mbox{-}ac}(\mathcal{M}) \ar@{=}[r]
 \ar@{^{(}->}[uu] &  \mathcal{H}^b_{\mathcal{E}\mbox{-}ac}(\mathcal{M}) \ar@{^{(}->}[uu]
\ar[rr]^{\sim} & & *+<1pc>{\mathrm{per}_{\underline{\mathcal{M}}}(\mathcal{M})}
\ar@{^{(}->}[lu] \ar@{^{(}->}[uu] \ar[r]^{\sim} & 
*+<1pc>{\mathrm{per}(\mathcal{B}^{op})^{op}} \ar@{^{(}->}[uu] \,\,\,.
}
$$
}
\begin{lemma}\label{caracterisation}
Let $Y$ be an object of $\mathcal{D}^-_{\underline{\mathcal{M}}}(\mathcal{M})$.
\begin{itemize}
\item[a)] $Y$ lies in $\mathrm{per}_{\underline{\mathcal{M}}}(\mathcal{M})$ iff $\mathrm{H}^p(Y)$ is a finitely presented $\underline{\mathcal{M}}$-module
for all $p \in \mathbb{Z}$ and vanishes for all but finitely many $p$.
\item[b)] $Y$ lies in $\mathcal{V}$ iff $\mathrm{H}^p(Y)$ is a finitely presented $\underline{\mathcal{M}}$-module
for all $p \in \mathbb{Z}$ and vanishes for all $p \gg 0$.
\end{itemize}
\end{lemma}

\begin{proof}
a) Clearly the condition is necessary. For the converse, suppose first
that $Y$ is a finitely presented
$\underline{\mathcal{M}}$-module. Then, as an $\mathcal{M}$-module,
$Y$ admits a resolution of length $d+1$ by finitely generated
projective modules by theorem $5.4$ b) of \cite{cluster}. It follows
that $Y$ belongs to $\mathrm{per}_{\underline{\mathcal{M}}}(\mathcal{M})$. Since $\mathrm{per}_{\underline{\mathcal{M}}}(\mathcal{M})$ is
triangulated, it also contains all shifts of finitely presented
$\underline{\mathcal{M}}$-modules and all extensions of shifts. This
proves the converse.\\ 
b) Clearly the condition is necessary.
For the converse, we can suppose
without loss of generality that $Y^n=0$, for all $n \geq 1$ and that
$Y^n$ belongs to $\mathrm{proj}\,\mathcal{M}$, for $n\leq 0$. We now construct a sequence
$$ \cdots \rightarrow P_n \rightarrow \cdots \rightarrow P_1
\rightarrow P_0$$
of complexes of finitely generated projective $\mathcal{M}$-modules
such that $P_n$ is quasi-isomorphic to $\tau_{\geq -n}Y$ for each $n$
and that, for each $p \in \mathbb{Z}$, the sequence of
$\mathcal{M}$-modules $P_n^p$ becomes stationary. By our assumptions,
we have $\tau_{\geq 0}Y \stackrel{\sim}{\rightarrow}
\mathrm{H}^0(Y)$. Since $\mathrm{H}^0(Y)$ belongs to
$\mathrm{mod}\,\underline{\mathcal{M}}$, we know by theorem $5.4$ c) of
\cite{cluster} that it belongs to $\mathrm{per}(\mathcal{M})$ as an
$\mathcal{M}$-module. We define $P_0$ to be a finite resolution of
$\mathrm{H}^0(Y)$ by finitely generated $\mathcal{M}$-modules. 
For the induction step, consider the following truncation triangle associated
with $Y$ 
$$ S^{i+1}\mathrm{H}^{-i-1}(Y) \rightarrow \tau_{\geq -i-1}Y \rightarrow
\tau_{\geq -i}Y \rightarrow S^{i+2}\mathrm{H}^{-i-1}(Y)\,,$$
for $i \geq 0$. 
By the induction hypothesis, we have constructed $P_0,\ldots,P_i$ and we have a
quasi-isomorphism $P_i \stackrel{\sim}{\rightarrow} \tau_{\geq
  -i}Y$. Let $Q_{i+1}$ be a finite resolution of $S^{i+2}\mathrm{H}^{-i-1}(Y)$ by
finitely presented projective $\mathcal{M}$-modules. We have a
morphism $f_i:P_i \rightarrow Q_{i+1}$ and we define
$P_{i+1}$ as the cylinder of $f_i$. We define $P$ as the limit of the
$P_i$ in the category of complexes. We remark that $Y$
is quasi-isomorphic to $P$ and that $P$
belongs to $\mathcal{V}$. This proves the converse.
\end{proof}

Let $X$ be in $\mathcal{H}^-_{\mathcal{E}\mbox{-}ac}(\mathcal{M})$.

\begin{remark}\label{t-strures}
Lemma~\ref{caracterisation} shows that the natural $t$-structure of
$\mathcal{D}(\mathcal{M})$ restricts to a $t$-structure on
$\mathcal{V}$. 
This allows us to express $\Psi(X)$ as
$$ \Psi(X) \stackrel{\sim}{\longrightarrow}
\underset{i}{\mbox{holim}}\,\tau_{\geq-i}\Psi(X) \,,$$
where  $\tau_{\geq-i}\Psi(X)$ is in $\mathrm{per}_{\underline{\mathcal{M}}}(\mathcal{M})$.
\end{remark}

\begin{lemma}\label{holim}
We have the following isomorphism
$$
\Phi(\Psi(X))=\Phi(\underset{i}{\mathrm{holim}}\,\tau_{\geq-i}\Psi(X))
\stackrel{\sim}{\longrightarrow}
\underset{i}{\mathrm{holim}}\,\Phi(\tau_{\geq -i}\Psi(X))\,.
$$
\end{lemma}

\begin{proof}
It is enough to show that the canonical morphism induces a
quasi-isomorphism when evaluated at any object $B$ of
$\mathcal{B}$. We have 
$$\Phi(\underset{i}{\mbox{holim}}\,\tau_{\geq-i}\Psi(X))(B)
=
\mathrm{Hom}^\bullet(\underset{i}{\mbox{holim}}\,\tau_{\geq-i}\Psi(X),
B)\,,$$
but since $B$ is a bounded complex, for each $ n \in \mathbf{Z}$, the sequence
$$ i \mapsto \mathrm{Hom}^n(\tau_{\geq -i}\Psi(X),B) $$
stabilizes as $i$ goes to infinity. This implies that
$$
\mathrm{Hom}^\bullet(\underset{i}{\mbox{holim}}\,\tau_{\geq-i}\Psi(X),B)
\stackrel{\sim}{\longleftarrow} \underset{i}{\mbox{holim}}\,\Phi(\tau_{\geq-i}\Psi(X))(B)\,.
$$
\end{proof}

\begin{lemma}\label{commute}
The functor $\Phi$ restricted to the category $\mathcal{V}$ is fully faithful.
\end{lemma}

\begin{proof}
Let $X,Y$ be in $\mathcal{H}^-_{\mathcal{E}\mbox{-}ac}(\mathcal{M})$. 
The following are canonically isomorphic~:

\begin{eqnarray}
&& \mathrm{Hom}_{\mathcal{D}(\mathcal{B}^{op})^{op}}(\Phi\Psi X,
\Phi\Psi Y)  \nonumber \\
&& \mathrm{Hom}_{\mathcal{D}(\mathcal{B}^{op})}(\Phi\Psi Y,
\Phi\Psi X)  \nonumber \\
&& \mathrm{Hom}_{\mathcal{D}(\mathcal{B}^{op})}(\underset{i}{\mbox{hocolim}}\,\Phi\tau_{\geq-i}\Psi
Y,\underset{j}{\mbox{hocolim}}\,\Phi\tau_{\geq-j}\Psi X
)\\
&& \underset{i}{\mbox{holim}}\,
\mathrm{Hom}_{\mathcal{D}(\mathcal{B}^{op})}(\Phi\tau_{\geq-i}\Psi
Y,\underset{j}{\mbox{hocolim}}\,\Phi\tau_{\geq-j}\Psi X
)  \nonumber \\
&& \underset{i}{\mbox{holim}}\,
  \underset{j}{\mbox{hocolim}}\,
  \mathrm{Hom}_{\mathcal{D}(\mathcal{B}^{op})}(\Phi\tau_{\geq-i}\Psi
  Y,\Phi\tau_{\geq-j}\Psi X) \\
&& \underset{i}{\mbox{holim}}\,
  \underset{j}{\mbox{hocolim}}\,
  \mathrm{Hom}_{\mathrm{per}_{\underline{\mathcal{M}}}(\mathcal{M})}(\tau_{\geq-j}\Psi X,\tau_{\geq-i}\Psi Y) \nonumber \\
&& \underset{i}{\mbox{holim}}\,
  \mathrm{Hom}_{\mathcal{V}}(\underset{j}{\mbox{holim}}\,\tau_{\geq-j}\Psi X,\tau_{\geq-i}\Psi Y) \\
&&
\mathrm{Hom}_{\mathcal{V}}(\Psi(X),\Psi(Y))\,. \nonumber
\end{eqnarray}
Here $(4.1)$ is by the lemma~\ref{holim} seen in
$\mathcal{D}(\mathcal{B}^{op})$, $(4.2)$ is by the fact that
$\Phi\tau_{\geq-i}\Psi Y$ is compact and $(4.3)$ is by the fact that
$\tau_{\geq-i}\Psi Y$ is bounded.
\end{proof}

It is clear now that lemmas~\ref{fidele}, \ref{holim} and \ref{commute} imply the  proposition~\ref{pleinfidele}. 

\section{Determination of the image of $G$}
Let $L_{\rho}:\mathcal{D}^-(\underline{\mathcal{M}}) \rightarrow
\mathcal{D}^-_{\underline{\mathcal{M}}}(\mathcal{M})$ be the
  restriction functor induced by the projection functor $\mathcal{M}
  \rightarrow \underline{\mathcal{M}}$. $L_{\rho}$ admits a left
  adjoint $L: \mathcal{D}^-_{\underline{\mathcal{M}}}(\mathcal{M})
\rightarrow \mathcal{D}^-(\underline{\mathcal{M}})$ which takes $Y$ to
$Y\otimes^{\mathbb{L}}_{\mathcal{M}}\underline{\mathcal{M}}$.
Let $\mathcal{B}^-$ be the dg subcategory of
$\mathcal{C}^-(\mathrm{Mod}\,\mathcal{M})_{dg}$ formed by the objects
of $\mathcal{D}^-_{\underline{\mathcal{M}}}(\mathcal{M})$ that are in the
  essential image of the restriction of $\Psi$ to
$\mathcal{H}^b_{\mathcal{E}\mbox{-}ac}(\mathcal{M})$. Let
$\mathcal{B}'$ be the dg quotient, \emph{cf.}~\cite{Drinfeld}, of
$\mathcal{B}^-$ by its quasi-isomorphisms. It is clear that the dg
categories $\mathcal{B}'$ and $\mathcal{B}$ are quasi-equivalent,
\emph{cf.}~\cite{DerivingDG}, and that the natural dg functor
$\mathcal{M} \rightarrow
\mathcal{C}^-(\mathrm{Mod}\,\mathcal{M})_{dg}$ factors through
$\mathcal{B}^-$. Let $R':\mathcal{D}(\mathcal{B}^{op})^{op}
\rightarrow \mathcal{D}(\underline{\mathcal{M}}^{op})^{op}$ be the
restriction functor induced by the dg functor $\underline{\mathcal{M}}
\rightarrow \mathcal{B}'$.
Let $\Phi': \mathcal{D}^-_{\underline{\mathcal{M}}}(\mathcal{M})
\rightarrow \mathcal{D}(\mathcal{B}'^{op})^{op}$ be the functor which
takes $X$ to the dg module 
$$ B' \mapsto \mathrm{Hom}^{\bullet}(X_c,B')\,,$$
where $B'$ is in $\mathcal{B}'$ and $X_c$ is a cofibrant replacement
of $X$ for the projective model structure on
$\mathcal{C}(\mathrm{Mod}\,\mathcal{M})$.
Finally let $\Gamma :\mathcal{D}¯(\underline{\mathcal{M}}) \rightarrow
\mathcal{D}(\underline{\mathcal{M}}^{op})^{op}$ be the functor that
sends $Y$ to 
$$ M \mapsto \mathrm{Hom}^{\bullet}(Y_c,
\underline{\mathcal{M}}(?,M))\,,$$
where $Y_c$ is a cofibrant replacement of $Y$ for the projective model
structure on $\mathcal{C}(\mathrm{Mod}\,\underline{\mathcal{M}})$ and $M$
is in $\underline{\mathcal{M}}$.

We have the following diagram~:
$$
\xymatrix{
& & & \mathcal{D}(\mathcal{B}^{op})^{op} & \mathcal{B} \ar[d] \\
\mathcal{H}^b(\mathcal{M})/\mathcal{H}^b(\mathcal{P})
\ar@{^{(}->}[r]^-{\Upsilon} & \mathcal{H}^-_{\mathcal{E}\mbox{-}ac}(\mathcal{M})
\ar[r]^{\Psi} & \mathcal{D}^-_{\underline{\mathcal{M}}}(\mathcal{M})
\ar[d]_L \ar[r]^{{\Phi}'} \ar[ur]^{\Phi}  & \mathcal{D}(\mathcal{B}'^{op})^{op}
\ar[d]^{R'} \ar[u]_{\sim} & \mathcal{B}' \\
 &  & \mathcal{D}^-(\underline{\mathcal{M}}) \ar[r]_{\Gamma} &
 \mathcal{D}(\underline{\mathcal{M}}^{op})^{op} & \underline{\mathcal{M}} \ar[u] \\
}
$$

\begin{lemma}\label{commutative}
The following square
$$
\xymatrix{
 \mathcal{D}^-_{\underline{\mathcal{M}}}(\mathcal{M}) \ar[r]^{\Phi'}
 \ar[d]_{L} & \mathcal{D}(\mathcal{B}'^{op})^{op} \ar[d]^{R'} &
 \mathcal{B}' \\
\mathcal{D}^-(\underline{\mathcal{M}}) \ar[r]_{\Gamma} &
\mathcal{D}(\underline{\mathcal{M}}^{op})^{op} &
\underline{\mathcal{M}} \ar[u]
}
$$
is commutative.
\end{lemma}

\begin{proof}
By definition $(R' \circ \Phi')(X)(M)$ equals
$\mathrm{Hom}^{\bullet}(X_c,
\underline{\mathcal{M}}(?,\mathcal{M}))$. Since
$\underline{\mathcal{M}}(?,M)$ identifies with $L_{\rho}M^{\wedge}$
and by adjunction, we have 
$$
\mathrm{Hom}^{\bullet}(X_c, \underline{\mathcal{M}}(?,M))
\stackrel{\sim}{\longrightarrow} \mathrm{Hom}^{\bullet}(X_c, L_{\rho}
M^{\wedge}) \stackrel{\sim}{\longrightarrow}
\mathrm{Hom}^{\bullet}((LX)_c, \underline{\mathcal{M}}(?,M))\,,$$
where the last member equals $(\Gamma \circ L)(X)(M)$.
\end{proof}

\begin{lemma}\label{reflet}
The functor $L$ reflects isomorphisms\,.
\end{lemma}

\begin{proof}
Since $L$ is a triangulated functor, it is enough to show that if $L(Y)=0$, then
$Y=0$. Let $Y$ be in
$\mathcal{D}^-_{\underline{\mathcal{M}}}(\mathcal{M})$ such that
$L(Y)=0$. We can suppose, without loss of generality, that
$\mathrm{H}^p(Y)=0$ for all $p>0$. Let us show that
$\mathrm{H}^0(Y)=0$. Indeed, since $\mathrm{H}^0(Y)$ is an
$\underline{\mathcal{M}}$-module, we have $\mathrm{H}^0(Y) \cong
L^0\mathrm{H}^0(Y)$, where $L^0:\mathrm{Mod}\,\mathcal{M} \rightarrow
\mathrm{Mod}\, \underline{\mathcal{M}}$ is the left adjoint of the
inclusion $\mathrm{Mod}\,\underline{\mathcal{M}} \rightarrow
  \mathrm{Mod}\, \mathcal{M}$. Since $\mathrm{H}^p(Y)$ vanishes in
  degrees $p>0$, we have 
$$L^0\mathrm{H}^0(Y) = \mathrm{H}^0(LY)\,.$$
By induction, one concludes that $\mathrm{H}^p(Y)=0$ for all $p \leq 0$.
\end{proof}

\begin{proposition}\label{caracterization}
An object $Y$ of $\mathcal{D}^-_{\underline{\mathcal{M}}}(\mathcal{M})$
lies in the essential image of the functor $\Psi \circ \Upsilon
: \mathcal{H}^b(\mathcal{M})/ \mathcal{H}^b(\mathcal{P}) \rightarrow
\mathcal{D}^-_{\underline{\mathcal{M}}}(\mathcal{M})$ iff $\tau_{\geq -n}Y$ is in
$\mathrm{per}_{\underline{\mathcal{M}}}(\mathcal{M})$, for all $n \in
\mathbb{Z}$ and $L(Y)$ belongs to $\mathrm{per}(\underline{\mathcal{M}})$.
\end{proposition}

\begin{proof}
Let $X$ be in $\mathcal{H}^b(\mathcal{M})/
\mathcal{H}^b(\mathcal{P})$. By lemma~\ref{caracterisation} a), $\tau_{\geq -n} \Psi \Upsilon (X)$
is in $\mathrm{per}_{\underline{\mathcal{M}}}(\mathcal{M})$, for all
$n \in \mathbb{Z}$. Since $X$ is a bounded complex, there exists an $s
\ll 0$ such that for all $m <s$ the $m$-components of $\Upsilon (X)$ are
in $\mathcal{P}$, which implies that $L \Psi \Upsilon (X)$ belongs to
$\mathrm{per}(\underline{\mathcal{M}})$.\\
Conversely, suppose that $Y$ is an object of
$\mathcal{D}^-_{\underline{\mathcal{M}}}(\mathcal{M})$ which
satisfies the conditions. By lemma~\ref{caracterisation}, $Y$ belongs
to $\mathcal{V}$. Thus we have $Y=\Psi(Y')$ for some $Y'$ in
$\mathcal{H}^-_{\mathcal{E}\mbox{-}ac}(\mathcal{M})$. We now consider
$Y'$ as an object of $\mathcal{H}^-(\mathcal{M})$ and also write
$\Psi$ for the functor $\mathcal{H}^-(\mathcal{M}) \rightarrow
\mathcal{D}^-(\mathcal{M})$ induced by the Yoneda functor.
We can express $Y'$ as 
$$ Y' \stackrel{\sim}{\longleftarrow} \underset{i}{\mbox{hocolim}}\,
\sigma_{\geq -i}Y'\,,$$
where the $\sigma_{\geq -i}$ are the naive truncations. By our
assumptions on $Y'$, $\sigma_{\geq -i}Y'$ belongs to
$\mathcal{H}^b(\mathcal{M})/ \mathcal{H}^b(\mathcal{P})$, for all $i
\in \mathbb{Z}$. The functors $\Psi$ and $L$ clearly commute with the naive
truncations $\sigma_{\geq -i}$ and so we have
$$ L(Y)=L(\Psi Y')  \stackrel{\sim}{\longleftarrow}
\underset{i}{\mbox{hocolim}}\,L(\sigma_{\geq -i}\Psi Y')
\stackrel{\sim}{\longrightarrow}  
\underset{i}{\mbox{hocolim}}\,\sigma_{\geq -i}L(\Psi Y')\,.$$
By our hypotheses, $L(Y)$ belongs to
$\mathrm{per}(\underline{\mathcal{M}})$ and so there exists an $m \gg
0$ such that 
$$ L(Y)= L(\Psi Y') \stackrel{\sim}{\longleftarrow} \sigma_{\geq -m}L(\Psi
Y') = L(\sigma_{\geq -m}\Psi Y')\,.$$
By lemma~\ref{reflet}, the inclusion
$$\Psi(\sigma_{\geq -m}Y)'= \sigma_{\geq -m}\Psi Y' \longrightarrow \Psi(Y')=Y$$
is an isomorphism.
But since $\sigma_{\geq -m}Y'$ belongs to $ \mathcal{H}^b(\mathcal{M})/
\mathcal{H}^b(\mathcal{P})$, $Y$ identifies with $\Psi(\sigma_{\geq -m}Y')$.
\end{proof}

\begin{remark}\label{perfect}
It is clear that if $X$ belongs to
$\mathrm{per}(\underline{\mathcal{M}})$, then $\Gamma(X)$ belongs to
$\mathrm{per}(\underline{\mathcal{M}}^{op})^{op}$. We also have the
following partial converse.
\end{remark}

\begin{lemma}
Let $X$ be in
$\mathcal{D}^-_{\mathrm{mod}\,\underline{\mathcal{M}}}(\underline{\mathcal{M}})$
such that $\Gamma(X)$ belongs to
$\mathrm{per}(\underline{\mathcal{M}}^{op})^{op}$. Then X is in $\mathrm{per}(\underline{\mathcal{M}})$.
\end{lemma}

\begin{proof}
By lemma~\ref{caracterisation} b) we can suppose, without loss of
generality, that $X$ is a right bounded complex with finitely
generated projective components. Applying $\Gamma$, we get a perfect
complex $\Gamma(X)$. In particular $\Gamma(X)$ is homotopic to zero in
high degrees. But since $\Gamma$ is an equivalence 
$$ \mathrm{proj}\, \underline{\mathcal{M}} \stackrel{\sim}{\longrightarrow}
(\mathrm{proj}\,\underline{\mathcal{M}}^{op})^{op}\,,$$
it follows that $X$ is already homotopic to zero in high degrees. 
\end{proof}

\begin{remark}
The natural right aisle on $\mathcal{D}(\mathcal{M})$ is the full
subcategory of the objects $X$ such that $\mathsf{H}^n(X)=0$ for all
$n < 0$. The associated truncation functor $\tau_{\geq 0}$ takes
$\mathrm{per}_{\underline{\mathcal{M}}}(\mathcal{M})$ to
itself. Therefore, the natural right aisle on
$\mathcal{D}(\mathcal{M})$ restricts to a natural right aisle
$\mathcal{U}^{op}$ on $\mathrm{per}_{\underline{\mathcal{M}}}(\mathcal{M})$.
\end{remark}

\begin{definition}
Let $\mathcal{U}$ be the natural left aisle in
$\mathrm{per}_{\underline{\mathcal{M}}}(\mathcal{M})^{op}$ associated
with $\mathcal{U}^{op}$.
\end{definition}

\begin{lemma}
The natural left aisle $\mathcal{U}$ on $\mathrm{per}_{\underline{\mathcal{M}}}(\mathcal{M})^{op}
\stackrel{\sim}{\rightarrow} \mathrm{per}(\mathcal{B}^{op})$ satisfies
the conditions of proposition~\ref{extension2} b).
\end{lemma}

\begin{proof}
Clearly the natural left aisle $\mathcal{U}$ in
$\mathrm{per}_{\underline{\mathcal{M}}}(\mathcal{M})^{op}$ is
non-degenerate. We need to show that for each $C \in
\mathrm{per}_{\underline{\mathcal{M}}}(\mathcal{M})^{op}$, there is an
integer $N$ such that $\mathrm{Hom}(C,S^NU)=0$ for each $U \in
\mathcal{U}$. We have the following isomorphism
$$
\mathrm{Hom}_{\mathrm{per}_{\underline{\mathcal{M}}}(\mathcal{M})^{op}}(C,S^N\mathcal{U})
\stackrel{\sim}{\rightarrow}
\mathrm{Hom}_{\mathrm{per}_{\underline{\mathcal{M}}}(\mathcal{M})}(S^{-N}\mathcal{U}^{op},C)\,,
$$
where $\mathcal{U}^{op}$ denotes the natural right aisle on
$\mathrm{per}_{\underline{\mathcal{M}}}(\mathcal{M})$. Since by
theorem $5.4$ c) of \cite{cluster} an $\underline{\mathcal{M}}$-module
admits a projective resolution of length $d+1$ as an
$\mathcal{M}$-module and $C$ is a bounded complex, we conclude that
for $N \gg 0$
$$
\mathrm{Hom}_{\mathrm{per}_{\underline{\mathcal{M}}}(\mathcal{M})}(S^{-N}\mathcal{U}^{op},C)=0\,.$$
This proves the lemma.
\end{proof}
 
We denote by $\tau_{\leq n}$ and $\tau_{\geq n}$, $n \in \mathbb{Z}$, the
associated truncation functors on $\mathcal{D}(\mathcal{B}^{op})^{op}$.

\begin{lemma}
The functor $\Phi: \mathcal{D}^-_{\underline{\mathcal{M}}}(\mathcal{M}) \rightarrow
\mathcal{D}(\mathcal{B}^{op})^{op}$ restricted to the category
$\mathcal{V}$ is exact with respect to the given $t$-structures.
\end{lemma}

\begin{proof}
We first prove that $\Phi(\mathcal{V}_{\leq 0}) \subset
\mathcal{D}(\mathcal{B}^{op})^{op}_{\leq 0}$. Let $X$ be in
$\mathcal{V}_{\leq 0}$. We need to show that $\Phi(X)$ belongs to
$\mathcal{D}(\mathcal{B}^{op})^{op}_{\leq 0}$. The following have the
same classes of objects~:
\begin{eqnarray}
&& \mathcal{D}(\mathcal{B}^{op})^{op}_{\leq 0} \nonumber \\
&& \mathcal{D}(\mathcal{B}^{op})_{> 0} \nonumber \\
&& (\mathrm{per}(\mathcal{B}^{op})_{\leq 0})^{\bot} \\
&& \overset{\bot}{}(\mathrm{per}(\mathcal{B}^{op})^{op})_{> 0} \,,
\end{eqnarray}
where in $(5.1)$ we consider the right orthogonal in
$\mathcal{D}(\mathcal{B}^{op})$ and in $(5.2)$ we consider the left
orthogonal in $\mathcal{D}(\mathcal{B}^{op})^{op}$. These isomorphisms
show us that $\Phi(X)$ belongs to
$\mathcal{D}(\mathcal{B}^{op})^{op}_{\leq 0}$ iff 
$$ \mathrm{Hom}_{\mathcal{D}(\mathcal{B}^{op})^{op}}(\Phi(X),
\Phi(P))=0\,,$$
for all $P \in \mathrm{per}_{\underline{\mathcal{M}}}(\mathcal{M})_{>
  0}$. 
Now, by lemma~\ref{commute} the functor $\Phi$ is fully faithful and
so
$$ \mathrm{Hom}_{\mathcal{D}(\mathcal{B}^{op})^{op}}(\Phi(X), \Phi(P))
\stackrel{\sim}{\longrightarrow}
\mathrm{Hom}_{\mathrm{per}_{\underline{\mathcal{M}}}(\mathcal{M})}(X,P)\,.$$
Since $X$ belongs to $\mathcal{V}_{\leq 0} $ and $P$ belongs to
$\mathrm{per}_{\underline{\mathcal{M}}}(\mathcal{M})_{> 0}$, we
conclude that
$$\mathrm{Hom}_{\mathrm{per}_{\underline{\mathcal{M}}}(\mathcal{M})}(X,P)=0
\,,$$
which implies that $\Phi(X) \in
\mathcal{D}(\mathcal{B}^{op})^{op}_{\leq 0}$.
Let us now consider $X$ in $\mathcal{V}$. We have the truncation
triangle
$$ \tau_{\leq 0}X \rightarrow X \rightarrow \tau_{>0}X \rightarrow
S\tau_{\leq 0}X \,.$$
The functor $\Phi$ is triangulated and so we have the triangle
$$ \Phi\tau_{\leq 0}X \rightarrow X \rightarrow \Phi\tau_{>0}X \rightarrow
S \Phi \tau_{\leq 0}X \,,$$
where $\Phi\tau_{\leq 0}X$ belongs to
$\mathcal{D}(\mathcal{B}^{op})^{op}_{\leq 0}$.
Since $\Phi$ induces an equivalence between
$\mathrm{per}_{\underline{\mathcal{M}}}(\mathcal{M})$ and
$\mathrm{per}(\mathcal{B}^{op})^{op}$ and $\mathrm{Hom}(P,
\tau_{>0}X) =0$, for all $P$ in $\mathcal{V}_{\leq 0}$, we conclude
that  $\Phi\tau_{>0}X$ belongs to
$\mathcal{D}(\mathcal{B}^{op})^{op}_{> 0}$. This implies the lemma.
\end{proof}

\begin{definition}\label{stable}
Let $\mathcal{D}(\mathcal{B}^{op})^{op}_f$ denote the full triangulated subcategory of $\mathcal{D}(\mathcal{B}^{op})^{op}$
formed by the objects $Y$ such that $\tau_{\geq-n}Y$ is in
$\mathrm{per}(\mathcal{B}^{op})^{op}$, for all $n \in \mathbb{Z}$, and
$R(Y)$ belongs to $\mathrm{per}(\underline{\mathcal{M}}^{op})^{op}$.
\end{definition}

\begin{proposition}\label{caracterisation2}
An object $Y$ of $\mathcal{D}(\mathcal{B}^{op})^{op}$ lies in the
essential image of the functor $G: \mathcal{H}^b(\mathcal{M})/
\mathcal{H}^b(\mathcal{P}) \rightarrow
\mathcal{D}(\mathcal{B}^{op})^{op}$ iff it belongs to $\mathcal{D}(\mathcal{B}^{op})^{op}_f$.
\end{proposition}

\begin{proof}
Let $X$ be in  $\mathcal{H}^b(\mathcal{M})/
\mathcal{H}^b(\mathcal{P})$. It is clear that the $\tau_{\geq -n} G(X)$ are
in $\mathrm{per}(\mathcal{B}^{op})^{op}$ for all $n \in \mathbb{Z}$. By
proposition~\ref{caracterization} we know that $L \Psi \Upsilon(X)$
belongs to $\mathrm{per}(\underline{\mathcal{M}})$. By lemma~\ref{commutative} and
remark~\ref{perfect} we conclude that $RG(X)$ belongs to $\mathrm{per}(\underline{\mathcal{M}}^{op})^{op}$. 
Let now $Y$ be in $\mathcal{D}(\mathcal{B}^{op})^{op}_f$. We can express
it,  by the dual of lemma~\ref{filtration} as the homotopy limit of the following diagram
$$ \cdots \rightarrow \tau_{\geq-n-1}Y \rightarrow \tau_{\geq-n}Y
\rightarrow \tau_{\geq-n+1}Y \rightarrow \cdots\,,$$
where $\tau_{\geq-n}Y$ belongs to
$\mathrm{per}(\mathcal{B}^{op})^{op}$, for all $n \in \mathbb{Z}$.
But since $\Phi$ induces an equivalence between
$\mathrm{per}_{\underline{\mathcal{M}}}(\mathcal{M})$ and
$\mathrm{per}(\mathcal{B}^{op})^{op}$, this last diagram corresponds
to a diagram
$$ \cdots \rightarrow M_{-n-1} \rightarrow M_{-n} \rightarrow M_{-n+1}
\rightarrow \cdots$$
in $\mathrm{per}_{\underline{\mathcal{M}}}(\mathcal{M})$.
Let $p \in \mathbb{Z}$. The relations among the truncation functors
imply that the image of the above diagram under each homology functor
$\mathrm{H}^p$, $p \in \mathbb{Z}$, is stationary as $n$ goes to $+ \infty$. This implies that 
$$ \mathrm{H}^p\, \underset{n}{\mbox{holim}}\,M_{-n}
\stackrel{\sim}{\longrightarrow}
\underset{n}{\mbox{lim}}\, \mathrm{H}^p\,M_{-n} \cong \mathrm{H}^p\,
M_j\,,
$$
for all $j<p$. We have the following commutative diagram
$$
\xymatrix{
\underset{n}{\mbox{holim}}\,M_{-n} \ar[r] \ar[d] &
\underset{n}{\mbox{holim}}\, \tau_{\geq -i}\,M_{-n} \cong M_{-i}\\
\tau_{\geq -i}\, \underset{n}{\mbox{holim}}\, M_{-n} \ar[ru]^{\sim}
}
$$
which implies that $$ \tau_{\geq -i} \,\underset{n}{\mathrm{holim}} M_{-n}
\stackrel{\sim}{\longrightarrow} M_{-i}\,,$$
for all $i \in \mathbb{Z}$.
Since $\underset{n}{\mbox{holim}}\,M_{-n}$ belongs to $\mathcal{V}$,
lemma~\ref{holim} allows us to conclude that
$\Phi(\underset{n}{\mbox{holim}}\,M_{-n})\cong Y$. We now show that
$\underset{n}{\mbox{holim}}\,M_{-n}$ satisfies the conditions of
proposition~\ref{caracterization}. We know that
$\tau_{\geq -i}\,\underset{n}{\mbox{holim}}\,M_{-n}$ belongs to
$\mathrm{per}_{\underline{\mathcal{M}}}(\mathcal{M})$, for all $i \in \mathbb{Z}$. 
By lemma~\ref{commutative} $(\Gamma \circ
L)(\underset{n}{\mbox{holim}}\,M_{-n})$ identifies with $R(Y)$, which is
in $\mathrm{per}(\underline{\mathcal{M}}^{op})^{op}$.
Since $\underset{n}{\mbox{holim}}\,M_{-n}$ belongs to $\mathcal{V}$,
its homologies lie in $\mathrm{mod}\,\underline{\mathcal{M}}$ and so we
are in the conditions of lemma~\ref{caracterization}, which implies
that $L(\underset{n}{\mbox{holim}}\,M_{-n})$ belongs to
$\mathrm{per}_{\underline{\mathcal{M}}}(\mathcal{M})$. This finishes
the proof.
\end{proof}

\section{Alternative description}
In this section, we present another characterization of the image of
$G$, which was identified as $\mathcal{D}(\mathcal{B}^{op})^{op}_f$
in proposition~\ref{caracterisation2}. Let $M$ denote an object of
$\mathcal{M}$ and also the naturally associated complex in
$\mathcal{H}^b(\mathcal{M})$. Since the category
$\mathcal{H}^b(\mathcal{M})/\mathcal{H}^b(\mathcal{P})$ is generated
by the objects $M \in \mathcal{M}$ and the functor $G$ is fully
faithful, we remark that $\mathcal{D}(\mathcal{B}^{op})^{op}_f$ equals
the triangulated subcategory of $\mathcal{D}(\mathcal{B}^{op})^{op}$
generated by the objects $G(M)$, $M \in \mathcal{M}$. 
The rest of this section is concerned with the problem
of characterizing the objects $G(M)$, $M \in \mathcal{M}$. We denote by
$P_M$ the projective $\underline{\mathcal{M}}$-module
$\underline{\mathcal{M}}(?,M)$ associated with $M \in \mathcal{M}$ and
by $X_M$ the image of $M$ under $\Psi \circ \Upsilon$.

\begin{lemma}\label{repre2}
We have the following isomorphism
$$
\mathrm{Hom}_{\mathcal{D}_{\underline{\mathcal{M}}}^-(\mathcal{M})}(X_M,Y)
\stackrel{\sim}{\longleftarrow}
\mathrm{Hom}_{\mathrm{mod}\,\underline{\mathcal{M}}}(P_M,
\mathrm{H}^0(Y))\,,$$
for all $Y \in \mathcal{D}_{\underline{\mathcal{M}}}^-(\mathcal{M})$.
\end{lemma}

\begin{proof}
Clearly $X_M$ belongs to
$\mathcal{D}_{\underline{\mathcal{M}}}(\mathcal{M})_{\leq 0}$ and is
of the form
$$ \cdots \rightarrow P_n^{\wedge} \rightarrow \cdots \rightarrow P_1^{\wedge} \rightarrow P_0^{\wedge} \rightarrow M^{\wedge} \rightarrow 0\,,$$
where $P_n \in \mathcal{P}$, $n \geq 0$. Now Yoneda's lemma and the
fact that $\mathrm{H}^m(Y)(P_n)=0$, for all $m \in \mathbb{Z}$, $n \geq
0$, imply the lemma.
\end{proof}

\begin{remark}\label{restric}
Since the functor $\Phi$ restricted to $\mathcal{V}$ is fully faithful
and exact, we have
$$ \mathrm{Hom}_{\mathcal{D}(\mathcal{B}^{op})^{op}}(G(M), \Phi(Y))
\stackrel{\sim}{\longleftarrow}
\mathrm{Hom}_{\mathrm{per}(\mathcal{B}^{op})^{op}}(\Phi(P_M),
\mathrm{H}^0(\Phi(Y)))\,,$$
for all $Y \in \mathcal{V}$.
\end{remark}

We now characterize the objects $G(M)=\Phi(X_M)$, $M \in \mathcal{M}$,
in the triangulated category $\mathcal{D}(\mathcal{B}^{op})$. More precisely, we give a description of the functor 
$$ R_M:= \mathrm{Hom}_{\mathcal{D}(\mathcal{B}^{op})}(?,
\Phi(X_M)):\mathcal{D}(\mathcal{B}^{op})^{op} \rightarrow
\mathrm{Mod}\,k$$
using an idea of M.~Van den Bergh, \emph{cf.} lemma~$2.13$ of
\cite{representability}. Consider the following functor
$$ F_M:=
\mathrm{Hom}_{\mathrm{per}(\mathcal{B}^{op})}(\mathrm{H}^0(?),
\Phi(P_M)): \mathrm{per}(\mathcal{B}^{op})^{op} \rightarrow
\mathrm{mod}\,k\,.$$
 
\begin{remark}
Remark~\ref{restric} shows that the functor $R_M$ when restricted to
$\mathrm{per}(\mathcal{B}^{op})$ coincides with $F_M$.
\end{remark}

Let $DF_M$ be the composition of $F_M$ with the duality
functor $D=\mathrm{Hom}(?,k)$. 
Note that $DF_M$ is homological.

\begin{lemma}\label{DF}
We have the following isomorphism of functors on $\mathrm{per}(\mathcal{B}^{op})$
$$
DF_M \stackrel{\sim}{\longrightarrow}
\mathrm{Hom}_{\mathcal{D}(\mathcal{B}^{op})}(\Phi(X_M), ?[d+1])\,.$$
\end{lemma}

\begin{proof}
The following functors are canonically isomorphic to
$DF \Phi~:$
\begin{eqnarray}
&&
D\mathrm{Hom}_{\mathrm{per}(\mathcal{B}^{op})}(\mathrm{H}^0\Phi(?),
\Phi(P_M)) \nonumber\\
&&
D\mathrm{Hom}_{\mathrm{per}(\mathcal{B}^{op})}(\Phi\mathrm{H}^0(?),
\Phi(P_M)) \\
&&
D\mathrm{Hom}_{\mathrm{per}_{\underline{\mathcal{M}}}(\mathcal{M})}(P_M,
\mathrm{H}^0(?)) \\
&&
D\mathrm{Hom}_{\mathcal{D}_{\underline{\mathcal{M}}}^-(\mathcal{M})}(X_M,?) \\
&&
\mathrm{Hom}_{\mathcal{D}_{\underline{\mathcal{M}}}^-(\mathcal{M})}(?[-d-1],X_M) \\
&&
\mathrm{Hom}_{\mathcal{D}(\mathcal{B}^{op})^{op}}(\Phi(?)[-d-1],\Phi(X_M)) \\
&&
\mathrm{Hom}_{\mathcal{D}(\mathcal{B}^{op})^{op}}(\Phi(X_M),\Phi(?)[d+1]) 
\end{eqnarray} 
 Step $(6.1)$ follows from the fact that $\Phi$ is exact. Step $(6.2)$
 follows from the fact that $\Phi$ is fully faithful and we are
 considering the opposite category. Step $(6.3)$ is a consequence of
 lemma~\ref{repre2}. Step $(6.4)$ follows from the $(d+1)$-Calabi-Yau  property and remark~\ref{t-strures}. Step $(6.5)$ is a consequence of $\Phi$ being fully
 faithful and step $(6.6)$ is a consequence of working in the opposite
 category. Since the functor $\Phi^{op}$ establish an equivalence
 between $\mathrm{per}_{\underline{\mathcal{M}}}(\mathcal{M})^{op}$ and
 $\mathrm{per}(\mathcal{B}^{op})$ the lemma is proven.
\end{proof}

Now, since the category $\mathrm{Mod}\,k$ is cocomplete, we can
consider the left Kan extension, \emph{cf.}~\cite{Macl}, $E_M$ of $DF_M$
along the inclusion $\mathrm{per}(\mathcal{B}^{op}) \hookrightarrow
  \mathcal{D}(\mathcal{B}^{op})$. We have the following
  commutative square~:
$$
\xymatrix{
*+<1pc>{\mathrm{per}(\mathcal{B}^{op})} \ar[rr]^{DF_M} \ar@{^{(}->}[d]
& & *+<1pc>{\mathrm{mod}\,k }   \ar@{^{(}->}[d]     \\
\mathcal{D}(\mathcal{B}^{op}) \ar@{-->}[rr]_{E_M} & & \mathrm{Mod}\,k\,.
}
$$
For each $X$ of $\mathcal{D}(\mathcal{B}^{op})$, the comma-category of
morphisms $P \rightarrow X$ from a perfect object $P$ to $X$ is
filtered. Therefore, the functor $E_M$ is homological. Moreover, it
preserves coproducts and so $DE_M$ is cohomological and transforms coproducts into
products. Since $\mathcal{D}(\mathcal{B}^{op})$ is a compactly
generated triangulated category, the Brown representability theorem, \emph{cf.}~\cite{Neeman}, implies that there is a $Z_M \in
\mathcal{D}(\mathcal{B}^{op})$ such that 
$$ DE_M \stackrel{\sim}{\longrightarrow}
\mathrm{Hom}_{\mathcal{D}(\mathcal{B}^{op})}(?, Z_M)\,.$$

\begin{remark}
Since the duality functor $D$ establishes and anti-equivalence
in $\mathrm{mod}\,k$, the functor $DE_M$ restricted to
$\mathrm{per}(\mathcal{B}^{op})$ is isomorphic to $F_M$.
\end{remark} 

\begin{theorem}\label{G(M)}
We have an isomorphism
$$ G(M) \stackrel{\sim}{\longrightarrow} Z_M\,.$$
\end{theorem}

\begin{proof}
We now construct a morphism of functors from $R_M$ to $DE_M$. Since $R_M$ is representable, by Yoneda's lemma it is enough to
construct an element in $DE_M(\Phi(X_M))$. Let $\mathcal{C}$ be the
category $\mathrm{per}(\mathcal{B}^{op})\downarrow \Phi(X_M)$, whose objects
are the morphisms $Y' \rightarrow \Phi(X_M)$ and let $\mathcal{C}'$ be
the category $X_M \downarrow
\mathrm{per}_{\underline{\mathcal{M}}}(\mathcal{M})$, whose objects
are the morphisms $X_M \rightarrow X'$. The following are
canonically isomorphic~:
\begin{eqnarray}
&& DE_M (\Phi(X_M)) \nonumber \\
&& D \, \underset{\mathcal{C}}{\mbox{colim}}
\mathrm{Hom}_{\mathcal{D}(\mathcal{B}^{op})}(\Phi(X_M), Y'[d+1])\\
&& D \, \underset{\mathcal{C}'}{\mbox{colim}}
\mathrm{Hom}_{\mathcal{D}_{\underline{\mathcal{M}}}^-(\mathcal{M})}(X'[-d-1], X_M)\\
&& \mathrm{D} \,\underset{i}{\mbox{colim}}\, 
\mathrm{Hom}_{\mathcal{D}_{\underline{\mathcal{M}}}^-(\mathcal{M})}((\tau_{\geq
  -i}X_M)[-d-1], X_M) \\
&& \underset{i}{\mbox{lim}} \,\mathrm{D}\mathrm{Hom}_{\mathcal{D}_{\underline{\mathcal{M}}}^-(\mathcal{M})}((\tau_{\geq
  -i}X_M)[-d-1], X_M) \nonumber \\
&& \underset{i}{\mbox{lim}}\,
\mathrm{Hom}_{\mathcal{D}_{\underline{\mathcal{M}}}^-(\mathcal{M})}(X_M,
\tau_{\geq -i}X_M)
\end{eqnarray}

Step $(6.7)$ is a consequence of the definition of the left Kan
extension and lemma~\ref{DF}. Step $(6.8)$ is obtained by considering
the opposite category. Step $(6.9)$ follows from the fact that the
system $(\tau_{\geq -i}X_M)_{i \in \mathbb{Z}}$ forms a cofinal system
for the index system of the colimit. Step $(6.10)$ follows from the
$(d+1)$-Calabi-Yau property. Now, the image of the identity by the
canonical morphism
$$
\mathrm{Hom}_{\mathcal{D}_{\underline{\mathcal{M}}}^-(\mathcal{M})}(X_M,
X_M) \longrightarrow \underset{i}{\mbox{lim}}\,
\mathrm{Hom}_{\mathcal{D}_{\underline{\mathcal{M}}}^-(\mathcal{M})}(X_M,
\tau_{\geq -i} X_M)\,,$$
gives us an element of $(DE_M)(\Phi(X_M))$ and so a morphism
of functors from $R_M$ to $DE_M$. We remark that this
morphism is an isomorphism when evaluated at the objects of
$\mathrm{per}(\mathcal{B}^{op})$. Since both functors $R_M$ and
$DE_M$ are cohomological, transform coproducts into products
and $\mathcal{D}(\mathcal{B}^{op})$ is compactly generated, we
conclude that we have an isomorphism
$$ G(M) \stackrel{\sim}{\longrightarrow} Z_M\,.$$
\end{proof}

\section{The main theorem}\label{mainsec}
Consider the following commutative square as in section~\ref{preli}:
$$
\xymatrix{
*+<1pc>{\mathcal{M}} \ar@{^{(}->}[r] \ar@{->>}[d] & \mathcal{E} \ar@{->>}[d]\\
*+<1pc>{\mathcal{T}}  \ar@{^{(}->}[r]  & \underline{\mathcal{E}}=\mathcal{C}\,.
}
$$
In the previous sections we have constructed, from the above data, a
dg category $\mathcal{B}$ and a left aisle $\mathcal{U}\subset
\mathrm{H}^0(\mathcal{B})$, see~\cite{aisles}, satisfying the
following conditions~:
\begin{itemize}
\item[-] $\mathcal{B}$ is an exact dg category over $k$ such that
  $\mathrm{H}^0(\mathcal{B})$ has finite-dimensional $\mathrm{Hom}$-spaces and is
  Calabi-Yau of CY-dimension $d+1$,
\item[-] $\mathcal{U} \subset \mathrm{H}^0(\mathcal{B})$ is a
  non-degenerate left aisle such that~:
\begin{itemize}
\item[-] for all $B \in \mathcal{B}$, there is an integer $N$ such
  that $\mathrm{Hom}_{\mathrm{H}^0 (\mathcal{B})}(B,
  S^NU)=0$ for each $U \in \mathcal{U}$,
\item[-] the heart $\mathcal{H}$ of the $t$-structure on
  $\mathrm{H}^0(\mathcal{B})$ associated with
  $\mathcal{U}$ has enough projectives.
\end{itemize}
\end{itemize}

Let now $\mathcal{A}$ be a dg category and $\mathcal{W}\subset
\mathrm{H}^0(\mathcal{A})$ a left aisle satisfying the above
conditions. We can consider the following general construction~:
Let $\mathcal{Q}$ denote the category of projectives of
the heart $\mathcal{H}$ of the $t$-structure on
$\mathrm{H}^0(\mathcal{A})$ associated with $\mathcal{W}$. We claim that the following inclusion
$$ \mathcal{Q} \hookrightarrow \mathcal{H} \hookrightarrow
\mathrm{H}^0(\mathcal{A})\,,$$
lifts to a morphism $\mathcal{Q} \stackrel{j}{\rightarrow}
\mathcal{A}$ in the homotopy category of small dg categories
$\mathsf{Heq}$. Indeed, recall the following argument from section~$7$ of
\cite{DerivingDG}: Let $\tilde{\mathcal{Q}}$ be the full dg subcategory of $\mathcal{A}$
whose objects are the same as those of $\mathcal{Q}$. Let $\tau_{\leq
  0}\tilde{\mathcal{Q}}$ denote the dg category obtained from
$\tilde{\mathcal{Q}}$ by applying the truncation functor $\tau_{\leq
  0}$ of complexes to each $\mathrm{Hom}$-space. We have he
following diagram in the category of small dg categories
$$
\xymatrix{
 & *+<1pc>{\tilde{\mathcal{Q}}} \ar@{^{(}->}[r] & \mathcal{A} \\
 & \tau_{\leq 0}\tilde{\mathcal{Q}} \ar[u]  \ar[d] & \\
\mathcal{Q} \ar@{=}[r] & \mathrm{H}^0(\tilde{\mathcal{Q}}) & \,.  
}
$$
Let $X$, $Y$ be objects of $\mathcal{Q}$. Since $X$ and $Y$ belong to
the heart of a $t$-structure in $\mathrm{H}^0(\mathcal{A})$, we have
$$ \mathrm{Hom}_{\mathrm{H}^0(\mathcal{A})}(X,Y[-n])=0\,, $$
for $n \geq 1$. The dg category $\mathcal{A}$ is exact, which implies
that 
$$ \mathrm{H}^{-n}
\mathrm{Hom}_{\tilde{\mathcal{Q}}}^{\bullet}(X,Y)
\stackrel{\sim}{\longrightarrow} \mathrm{Hom}_{\mathrm{H}^0(\mathcal{A})}(X,Y[-n])
=0\,,$$
for $n \geq 1$. This shows that the dg functor $\tau_{\leq
  0}\tilde{\mathcal{Q}} \rightarrow \mathrm{H}^0(\tilde{\mathcal{Q}})$
is a quasi-equivalence and so we have a morphism $\mathcal{Q}
\stackrel{j}{\rightarrow} \mathcal{A}$ in the homotopy category of
small dg categories. We have a triangle functor $j^*:\mathcal{D}(\mathcal{A})
\rightarrow \mathcal{D}(\mathcal{Q})$ given by restriction. By
proposition~\ref{extension2}, the left aisle $\mathcal{W} \subset
\mathrm{H}^0(\mathcal{A})$ admits a smallest extension to a left aisle
$\mathcal{D}(\mathcal{A}^{op})^{op}_{\leq 0}$ on
$\mathcal{D}(\mathcal{A}^{op})^{op}$.
Let $\mathcal{D}(\mathcal{A}^{op})^{op}_f$ denote the full triangulated subcategory of $\mathcal{D}(\mathcal{A}^{op})^{op}$
formed by the objects $Y$ such that $\tau_{\geq-n}Y$ is in
$\mathrm{per}(\mathcal{A}^{op})^{op}$, for all $n \in \mathbb{Z}$, and
$j^*(Y)$ belongs to $\mathrm{per}(\mathcal{Q}^{op})^{op}$.
 
\begin{definition}
The stable category of $\mathcal{A}$ with respect to $\mathcal{W}$ is
the triangle quotient 
$$\mathrm{stab}(\mathcal{A},\mathcal{W}) = \mathcal{D}(\mathcal{A}^{op})^{op}_f/\mathrm{per}(\mathcal{A}^{op})^{op}\,.$$
\end{definition}

We are now able to formulate the main theorem.
Let $\mathcal{B}$ be the dg category and $\mathcal{U}
\subset\mathrm{H}^0(\mathcal{B})$ the left aisle constructed in
sections $1$ to $5$.
\begin{theorem}\label{main3}
The functor $G$ induces an equivalence of categories
$$ \tilde{G}: \mathcal{C} \stackrel{\sim}{\longrightarrow}
\mathrm{stab}(\mathcal{B}, \mathcal{U})\,.$$
\end{theorem}

\begin{proof}
We have the following commutative diagram~:
$$
\xymatrix{
\mathcal{C} \ar@{-->}[rr]^-{\tilde{G}}_-{\sim} & &
\mathrm{stab}(\mathcal{B}, \mathcal{U}) \\
\mathcal{H}^b(\mathcal{M})/\mathcal{H}^b(\mathcal{P}) \ar[u]
\ar[rr]^-G_-{\sim} && \mathcal{D}(\mathcal{B}^{op})^{op}_f \ar[u]\\
\mathcal{H}^b_{\mathcal{E}\mbox{-}ac}(\mathcal{M}) \ar[u]
\ar[rr]_-{\sim} && \mathrm{per}(\mathcal{B}^{op})^{op} \ar[u] \,.\\
}
$$
The functor $G$ is an equivalence since it is fully faithful by
proposition~\ref{pleinfidele} and essentially surjective by
proposition~\ref{caracterisation2}. Since we have an equivalence
$\mathcal{H}^b_{\mathcal{E}\mbox{-}ac}(\mathcal{M})
\stackrel{\sim}{\longrightarrow} \mathrm{per}(\mathcal{B}^{op})^{op}$
by construction of $\mathcal{B}$ and the columns of the above diagram
are short exact sequences of triangulated categories, the theorem is proved.
\end{proof}

\section{Appendix: extension of $t$-structures}
Let $\mathcal{T}$ be a compactly generated triangulated category with
suspension functor $S$. We denote by
$\mathcal{T}_c$ the full triangulated sub-category of $\mathcal{T}$ formed by the
compact objects, see \cite{Neeman}. We use the terminology of \cite{aisles}. Let
$\mathcal{U} \subseteq \mathcal{T}_c$ be a left aisle on
$\mathcal{T}_c$, i.e. a full additive subcategory $\mathcal{U}$ of
$\mathcal{T}_c$ which satisfies:
\begin{itemize}
\item[a)] $S\mathcal{U} \subset \mathcal{U}$,
\item[b)] $\mathcal{U}$ is stable under extensions, i.e. for each
    triangle 
$$X \rightarrow Y \rightarrow Z \rightarrow SX$$
of $\mathcal{T}_c$, we have $Y \in \mathcal{U}$ whenever $X, Z \in
\mathcal{U}$ and
\item[c)] the inclusion functor $\mathcal{U} \hookrightarrow
  \mathcal{T}_c$ admits a right adjoint.
\end{itemize}

As shown in \cite{aisles}, the concept of aisle is equivalent to
that of $t$-structure.

\begin{proposition}\label{extension2}
\begin{itemize}
\item[a)] The left aisle $\mathcal{U}$ admits a smallest extension to a left
aisle  $\mathcal{T}_{\leq 0}$ on $\mathcal{T}$. 
\item[b)] If $\mathcal{U} \subseteq \mathcal{T}_c$ is non-degenerate
  (\emph{i.e.}, $f:X \rightarrow Y$ is invertible iff
  $\mathrm{H}^p(f)$ is invertible for all $p \in \mathbb{Z}$) and for
  each $X \in \mathcal{T}_c$, there is an integer $N$ such that
  $\mathrm{Hom}(X,S^{N}U)=0$ for each $U \in \mathcal{U}$, then
  $\mathcal{T}_{\leq 0}$ is also non-degenerate.
\end{itemize} 
\end{proposition}

\begin{proof}
a) Let $\mathcal{T}_{\leq 0}$ be the smallest full subcategory of
$\mathcal{T}$ that contains $\mathcal{U}$ and is stable under infinite
sums and extensions. It is clear that $\mathcal{T}_{\leq 0}$ is stable
by $S$ since $\mathcal{U}$ is.
We need to show that the inclusion functor $\mathcal{T}_{\leq 0}
\hookrightarrow \mathcal{T}$ admits a right adjoint. For completeness,
we include the following proof, which is a variant of the `small
object argument', \emph{cf.} also~\cite{souto}. We have the following recursive
procedure. Let $X=X_0$ be an object in $\mathcal{T}$. For the initial
step consider all morphisms from any object $P$ in $\mathcal{U}$
to $X_0$. This forms a set $I_0$ since $\mathcal{T}$ is compactly generated
and so we have the following triangle
$$ 
\xymatrix{
\underset{f \in I_0}{\coprod} P \ar[r] &  X_0 \ar[r] &  X_1 \ar@{~>}[r] & \underset{f \in
  I_0}{\coprod} P \,.
}
$$

For the induction step consider the above construction with $X_n$, $n \geq 1$, in the
place of $X_{n-1}$ and $I_n$ in the place of $I_{n-1}$. We have the following diagram
$$
\xymatrix{
X=X_0 \ar[r] & X_1 \ar[r] \ar@{~>}[dl]  & X_2 \ar[r] \ar@{~>}[dl] &
X_3 \ar[r] \ar@{~>}[dl]
& \cdots
\ar[r] & X' \\
\underset{f \in I_0}{\coprod} P
\ar[u] &\underset{f \in I_1}{\coprod} P
\ar[u] & \underset{f \in I_2}{\coprod} P
\ar[u]  &  \underset{f \in I_3}{\coprod} P \ar[u] & & \,,
}
$$
where $X'$ denotes the homotopy colimit of the diagram $(X_i)_{i \in \mathbb{Z}}$. 
Consider now the following triangle
$$ S^{-1}X'  \rightarrow X'' \rightarrow X
\rightarrow X' \,,$$
where the morphism $X \rightarrow X'$ is the transfinite
composition in our diagram.
Let $P$ be in $\mathcal{U}$. We remark that since $P$ is compact,
$\mathrm{Hom}_{\mathcal{T}}(P,X')=0$. This also implies, by
construction of $\mathcal{T}_{\leq 0}$, that
  $\mathrm{Hom}_{\mathcal{T}}(R,X')=0$, for all $R$ in
    $\mathcal{T}_{\leq 0}$.
The long exact sequence obtained by applying the functor
$\mathrm{Hom}_{\mathcal{T}}(R,?)$ to the triangle above shows that
$$ \mathrm{Hom}(R,X'') \stackrel{\sim}{\longrightarrow}
\mathrm{Hom}(R,X)\,.$$
Let $X''_{n-1}$, $n\geq 1$, be an object as in the following triangle
$$ X=X_0 \rightarrow X_n \rightarrow X''_{n-1} \rightarrow S(X)\,.$$
A recursive application of the octahedron axiom implies that
$X''_{n-1}$ belongs to $S(\mathcal{T}_{\leq 0})$, for all $n \geq
1$. We have the isomorphism
$$ \underset{n}{\mbox{hocolim}}\,X''_{n-1}
\stackrel{\sim}{\longrightarrow} S(X'')\,.$$
Since $\underset{n}{\mbox{hocolim}}\,X''_{n-1}$ belongs to
$S(\mathcal{T}_{\leq 0})$, we conclude that $X''$ belongs to
$\mathcal{T}_{\leq 0}$. This shows that the functor that sends $X$ to
$X''$ is the right adjoint of the inclusion functor  $\mathcal{T}_{\leq 0}
\hookrightarrow \mathcal{T}$. 
This proves that $\mathcal{T}_{\leq 0}$ is a left aisle on $\mathcal{T}$.
We now show that the $t$-structure associated to $\mathcal{T}_{\leq
  0}$, \emph{cf.}~\cite{aisles},  extends, from $\mathcal{T}_c$
to $\mathcal{T}$, the one associated with $\mathcal{U}$. Let $X$ be in $\mathcal{T}_c$. We have the following truncation triangle associated with $\mathcal{U}$
$$ X_{\mathcal{U}} \rightarrow X \rightarrow X^{\mathcal{U}^{\bot}}
  \rightarrow SX_{\mathcal{U}}\,.$$
Clearly $X_{\mathcal{U}}$ belongs to $\mathcal{T}_{\leq 0}$. We remark
  that $\mathcal{U}^{\bot}=\mathcal{T}_{\leq 0}^{\bot}$, and so  $X^{\mathcal{U}^{\bot}}$ belongs to
  $\mathcal{T}_{>0}:=\mathcal{T}_{\leq 0}^{\bot}$. 

We now show that
  $\mathcal{T}_{\leq 0}$ is the smallest extension of the left aisle $\mathcal{U}$. Let $\mathcal{V}$ be an aisle containing
  $\mathcal{U}$. The inclusion functor $\mathcal{V} \hookrightarrow
  \mathcal{T}$ commutes with sums, because it admits a right
  adjoint. Since $\mathcal{V}$ is stable under extensions and
  suspensions, it contains $\mathcal{T}_{\leq 0}$.\\

b) Let $X$ be in $\mathcal{T}$. We need to show that $X=0$ iff
$\mathrm{H}^p(X)=0$ for all $p \in \mathbb{Z}$. Clearly the condition
is necessary. For the converse, suppose that $\mathrm{H}^p(X)=0$ for
all $p \in \mathbb{Z}$. Let $n$ be an integer. Consider the following
truncation triangle
$$ \mathrm{H}^{n+1}(X) \rightarrow \tau_{> n} X \rightarrow \tau_{>
  n+1}X \rightarrow S\mathrm{H}^{n+1}(X)\,.$$
Since $\mathrm{H}^{n+1}(X)=0$ we conclude that 
$$ \tau_{> n} X \in \underset{m \in \mathbb{Z}}{\bigcap}\mathcal{T}_{>m}\,,$$
for all $n \in \mathbb{Z}$. Now, let $C$ be a compact object of
$\mathcal{T}$. We know that there is a $k \in \mathbb{Z}$ such that $C
\in \mathcal{T}_{\leq k}$. This implies that 
$$ \mathrm{Hom}_{\mathcal{T}}(C, \tau_{> n}X)=0$$ 
for all $n \in \mathbb{Z}$, since $\tau_{> n}X$ belongs to $(\mathcal{T}_{\leq
  k})^{\bot}$. The category $\mathcal{T}$ is compactly generated and so we conclude
  that $\tau_{> n}X=0$, for all $n \in \mathbb{Z}$.
The following truncation triangle
$$ \tau_{\leq n}X \rightarrow X \rightarrow \tau_{> n}X \rightarrow
S\tau_{\leq n}X\,,$$
implies that $\tau_{\leq n} X$ is isomorphic to $X$ for all $n \in
\mathbb{Z}$. This can be rephrased as saying that
$$ X \in \underset{n \in \mathbb{N}}{\bigcap} \mathcal{T}_{\leq
  -n}\,.$$
Now by our hypothesis there is an integer $N$ such that 
$$ \mathrm{Hom}_{\mathcal{T}}(C, \mathcal{U}_{\leq - N})=0\,.$$
Since $C$ is compact and by construction of $\mathcal{T}_{\leq -N}$,
we have
$$ \mathrm{Hom}_{\mathcal{T}}(C,\mathcal{T}_{\leq -N})=0\,.$$
This implies that $\mathrm{Hom}_{\mathcal{T}}(C,X)=0$, for all compact
objects $C$ of $\mathcal{T}$. Since $\mathcal{T}$ is compactly
generated, we conclude that $X=0$. This proves the converse.
\end{proof}

\begin{lemma}\label{coproducts}
Let $(Y_p)_{p \in \mathbb{Z}}$ be in $\mathcal{T}$. We have the
following isomorphism
$$ \mathrm{H}^n\,(\underset{p}{\coprod}Y_p)
\stackrel{\sim}{\longleftarrow}\underset{p}{\coprod}\, \mathrm{H}^n(Y_p)\,,
$$
for all $n \in \mathbb{Z}$.
\end{lemma}

\begin{proof}
By definition $\mathrm{H}^n:= \tau_{\geq n}\,\tau_{\leq n}\,, n \in
\mathbb{Z}$. Since $\tau_{\geq n}$ admits a right adjoint, it is enough to show that $\tau_{\leq n}$
commute with infinite sums. 
We consider the following triangle
$$ \underset{p}{\coprod} \,\tau_{\leq n} Y_p \rightarrow
\underset{p}{\coprod}\,Y_p \rightarrow \underset{p}{\coprod}\, \tau_{>n} Y_p
\rightarrow S(\underset{p}{\coprod}\, \tau_{\leq n} Y_p)\,.$$
Here $\underset{p}{\coprod} \,\tau_{\leq n} Y_p$ belongs to $\mathcal{T}_{\leq
  n}$ since $\mathcal{T}_{\leq n}$ is stable under infinite sums. Let
$P$ be an object of $S^n\mathcal{U}$. Since $P$ is compact, we have
$$ \mathrm{Hom}_{\mathcal{T}}(P,\underset{p}{\coprod} \,\tau_{>n} Y_p)
\stackrel{\sim}{\longleftarrow} \underset{p}{\coprod}\,
\mathrm{Hom}_{\mathcal{T}}(P, \tau_{>n}Y_p)=0\,.$$
Since $\mathcal{T}_{\leq n}$ is generated by $S^n\mathcal{U}$,
$\underset{i}{\coprod} \,\tau_{>n} Y_p$ belongs to
$\mathcal{T}_{>n}$. Since the truncation triangle of
$\underset{p}{\coprod}\, Y_p$ is unique, this implies the following
isomorphism
$$ \underset{p}{\coprod} \,\tau_{\leq n}Y_p \stackrel{\sim}{\longrightarrow}
\tau_{\leq n}(\underset{p}{\coprod}\,Y_p)\,.$$
This proves the lemma.
\end{proof}

\begin{proposition}\label{filtration}
Let $X$ be an object of $\mathcal{T}$. Suppose that we are in the
conditions of proposition~\ref{extension2} b). We have the following isomorphism
$$ \underset{i}{\mathrm{hocolim}}\,\tau_{\leq i}X
\stackrel{\sim}{\longrightarrow} X\,.$$
\end{proposition}

\begin{proof}
We need only show that
$$\mathrm{H}^n(\underset{i}{\mbox{hocolim}}\, \tau_{\leq i}X)
\stackrel{\sim}{\longrightarrow} \mathrm{H}^n(X)\,,$$
for all $n \in \mathbb{Z}$. 
We have the following triangle, \emph{cf.}~\cite{Neeman},
$$ \underset{p}{\coprod}\,\tau_{\leq p}X \rightarrow \underset{q}{\coprod}\,
\tau_{\leq q}X \rightarrow \underset{i}{\mbox{hocolim}}\, \tau_{\leq
  i} X \rightarrow S(\underset{p}{\coprod}\,\tau_{\leq p}X)\,.$$
Since the functor $\mathrm{H}^n$ is homological, for all $n \in
\mathbb{Z}$ and it commutes with infinite sums by
lemma~\ref{coproducts}, we obtain a long exact sequence
\begin{eqnarray*}
 \cdots \rightarrow \underset{p}{\coprod}\, \mathrm{H}^n(\tau_{\leq
  p}X) \rightarrow \underset{q}{\coprod}\, \mathrm{H}^n(\tau_{\leq q}X)
  \rightarrow \mathrm{H}^n\, (\underset{i}{\mbox{hocolim}}\,
  \tau_{\leq i}X) \rightarrow \\
\rightarrow \underset{p}{\coprod}\, \mathrm{H}^n\,
  S(\tau_{\leq p}X) \rightarrow \underset{q}{\coprod}\, \mathrm{H}^n\,
  S(\tau_{\leq q}X) \rightarrow \cdots
\end{eqnarray*}

We remark that the morphism $\underset{p}{\coprod}\, \mathrm{H}^n\,
  S(\tau_{\leq p}X) \rightarrow \underset{q}{\coprod}\, \mathrm{H}^n\,
  S(\tau_{\leq q}X)$ is a split monomorphism and so we obtain $$ \mathrm{H}^n(X) = \underset{i}{\mbox{colim}}\,\mathrm{H}^n(\tau_{\leq i} X)
\stackrel{\sim}{\longrightarrow}
\mathrm{H}^n(\underset{i}{\mbox{hocolim}}\, \tau_{\leq i} X)\,.$$ 
\end{proof}


\appendix
\chapter{Drinfeld's DG quotient}

In this appendix, we give a simple and purely homotopic proof of the
main result proved by Drinfeld in \cite{Drinfeld}. Our proof is based
only on the Quillen model structure on $\dgcat$ of theorem~\ref{mal}.

Recall that the functor 
$$\mathsf{H}^0(-):\dgcat \rightarrow \mathsf{cat}\,,$$
where $\mathsf{cat}$ denotes
the category of small categories, descends to the localized categories
$$ \mathsf{H}^0(-): \mathsf{Heq} \rightarrow \mathsf{Ho}(\mathsf{cat})\,,$$
where $\mathsf{Heq}$ denotes the localization of $\mathsf{dgcat}$
by the quasi-equivalences and $\mathsf{Ho}(\mathsf{cat})$ the
localization of $\mathsf{cat}$ by the equivalences of categories.

Let $\mathcal{A}$ be a small dg category and
$\mathcal{N}$ a set of objects in $\mathcal{A}$. 

\begin{definition}
A morphism $Q : \mathcal{A} \rightarrow \mathcal{B}$
in $\mathsf{Heq}$ {\it annihilates} $\mathcal{N}$ if the induced
morphism in $\mathsf{Ho}(\mathsf{cat})$
$$ \mathsf{H}^0(Q): \mathsf{H}^0(\mathcal{A}) \rightarrow
\mathsf{H}^0(\mathcal{B})$$
takes all objects of $\mathcal{N}$ to zero objects (i.e. objects whose
identity vanishes in $\mathsf{H}^0(\mathcal{B})$). 
\end{definition}

\begin{remark}
\item[-] Notice that the morphism $\mathsf{H}^0(Q)$ in
  $\mathsf{Ho}(\mathsf{cat})$ corresponds to a functor in
  $\mathsf{cat}$ up to natural isomorphism.
\item[-] Observe also that if $\mathcal{N}$ equals the set of objects
  of $\mathcal{A}$ there exists at most one morphism $Q: \mathcal{A}
  \rightarrow \mathcal{B}$ in $\mathsf{Heq}$ which annihilates $\mathcal{A}$. We denote
  it by $0$. 
\end{remark}

In section $3$ of \cite{Drinfeld}, Drinfeld made the following
construction: let 
$$
\xymatrix{ \widetilde{\mathcal{A}} \ar[r]^{\pi}_{\sim}   &
  \mathcal{A} }
$$
be a $k$-homotopically flat resolution of $\mathcal{A}$ (for example, we could take
a cofibrant resolution of $\mathcal{A}$) and consider the dg category $\mathcal{A}/\mathcal{N}$ obtained from
$\widetilde{\mathcal{A}}$ by introducing a new morphism $h_X$ of
degree $-1$ for every object $X$ whose image under $\pi$ is
homotopically equivalent to an object of $\mathcal{N}$ and by imposing the relation $d(h_X)=\mathbf{1}_X$.

We have the following diagram in $\dgcat$
$$
\xymatrix{
\widetilde{\mathcal{A}} \ar[d]^{\pi}_{\sim} \ar[r]^-i &
\mathcal{A}/\mathcal{N} \\
\mathcal{A} & \,,
}
$$
which gives rise to a morphism $Q:\mathcal{A} \rightarrow
\mathcal{A}/\mathcal{N}$ in $\mathsf{Heq}$.

\begin{theorem}[\cite{Drinfeld}]\label{1}
The morphism $Q$ annihilates $\mathcal{N}$ and is
universal in $\mathsf{Heq}$ among the morphisms annihilating $\mathcal{N}$.
\end{theorem}
We now present a simple and homotopic proof of this theorem, where the homotopic notations
used are those of chapter $1$.
\begin{proof}
Let 
$$
\xymatrix{ \widetilde{\mathcal{A}} \ar@{->>}[r]^{\pi}_{\sim}   &
  \mathcal{A} }
$$
be a cofibrant resolution of $\mathcal{A}$ and let
$\widetilde{\mathcal{N}}$ be the full dg subcategory of
$\widetilde{\mathcal{A}}$ whose objects are those whose image under
$\pi$ is homotopically equivalent to an object of $\mathcal{N}$.

Recall from chapter $1$ that $\mathcal{C}(0)$ denotes the dg category
with two objects $8$ and $9$ such that \newline
$\mathsf{Hom}_{\mathcal{C}(0)}(8,8)=k$,
$\mathsf{Hom}_{\mathcal{C}(0)}(9,9)=k$,
$\mathsf{Hom}_{\mathcal{C}(0)}(9,8)=0$,
$\mathsf{Hom}_{\mathcal{C}(0)}(8,9)=S^{-1}$ and whose composition is
given by multiplication. Recall also that $\mathcal{P}(0)$ denotes the dg category
with two objects $6$ and $7$ such that
$\mathsf{Hom}_{\mathcal{C}(0)}(6,6)=k$,
$\mathsf{Hom}_{\mathcal{C}(0)}(7,7)=k$,
$\mathsf{Hom}_{\mathcal{C}(0)}(7,6)=0$,
$\mathsf{Hom}_{\mathcal{C}(0)}(6,7)=D^0$ and whose composition is
given by multiplication. We have a dg functor 
$$ S(0): \mathcal{C}(0) \rightarrow \mathcal{P}(0)\,,$$
that sends $8$ to $6$, $9$ to $7$ and $S^{-1}$ to $D^0$ by the
identity on $k$ in degree $-1$.

Now consider the following push-out
$$
\xymatrix{
*+<1pc>{\underset{X \in \widetilde{\mathcal{N}}}{\coprod} \mathcal{C}(0)}
\ar[r] \ar@{}[dr]|{\lrcorner} 
  \ar@{>->}[d]_{\underset{X \in \widetilde{\mathcal{N}}}{\coprod} S(0)} & *+<1pc>{\widetilde{\mathcal{A}}} \ar@{>->}[d]^i \\
 \underset{X \in \widetilde{\mathcal{N}}}{\coprod} \mathcal{P}(0)
   \ar[r] & \mathcal{A}/\mathcal{N} \,,
}
$$
where the uppper horizontal dg functor corresponds to
specifying all the identities of the objects in $\widetilde{\mathcal{N}}$.
This shows that $\mathcal{A}/\mathcal{N}$ is still a cofibrant dg
category. Notice that the dg functor $i$ induces a surjective map
$$ 
\xymatrix{ \mathsf{Hom}_{\dgcat}(\mathcal{A}/\mathcal{N}, \mathcal{B})
  \ar@{->>}[r]^-{i^{\ast}} & \{ F \in
  \mathsf{Hom}_{\dgcat}(\widetilde{\mathcal{A}},\mathcal{B})|\, F(X)\,
  \mbox{contractible}, \, \forall X \in \widetilde{\mathcal{N}} \}\,.
}
$$
Since every object in $\dgcat$ is fibrant, we can calculate the
morphisms in $\mathsf{Heq}$ to $\mathcal{B}$ using the good path object
$P(\mathcal{B})$ from definition~\ref{1path1}. 
Observe that by definition of the path object $P(\mathcal{B})$,
the set 
$$ \{ F \in
  \mathsf{Hom}_{\dgcat}(\widetilde{\mathcal{A}},\mathcal{B})| \,F(X)\,
  \mbox{contractible}, \, \forall X \in \widetilde{\mathcal{N}} \}$$
is stable under homotopies.

Clearly the map $i^{\ast}$ induces a surjective one
$$ 
\xymatrix{
\mathsf{Hom}_{\mathsf{Heq}}(\mathcal{A}/ \mathcal{N}, \mathcal{B})
\simeq \mathsf{Hom}_{\dgcat}(\mathcal{A}/ \mathcal{N}, \mathcal{B})/htp
\ar@{->>}[r]^-{i^{\ast}}  &  \{ F \in
  \mathsf{Hom}_{\dgcat}(\widetilde{\mathcal{A}},\mathcal{B})| \,F(X)\,
  \mbox{contractible}, \, \forall X \in \widetilde{\mathcal{N}}\} /htp \,.
} 
$$
We will now prove that the map $i^{\ast}$ is also injective.
Let $F$ and $G$ be dg functors from $\mathcal{A}/\mathcal{N}$ to
$\mathcal{B}$ such that $F \circ i$ and $G \circ i$ are homotopic. We
  have the following commutative diagram 
$$
\xymatrix{
 & \mathcal{B} \\
\widetilde{\mathcal{A}} \ar[r]^H \ar[ur]^{F\circ i} \ar[dr]_{G \circ
  i} & P(\mathcal{B}) \ar[u]_{p_0} \ar[d]^{p_1} \\
 & \mathcal{B} 
}
$$
in $\dgcat$. Observe that an extension of the homotopy $H$ from $\widetilde{\mathcal{A}}$
to $\mathcal{A}/\mathcal{N}$ corresponds exactly to specifying a contraction
for each object $H(X)= A \stackrel{f}{\rightarrow} B$ of
$P(\mathcal{B})$, where $X$ belongs to $\widetilde{\mathcal{N}}$.

Recall from lemma~\ref{lempath} that a contraction of $A
\stackrel{f}{\rightarrow} B$ in $P(\mathcal{B})$ corresponds
to morphisms $c_A \in \mathsf{Hom}_{\mathcal{B}}^{-1}(A,A)$, $c_B \in
\mathsf{Hom}_{\mathcal{B}}^{-1}(B,B)$ and $h \in
\mathsf{Hom}_{\mathcal{B}}^{-2}(A,B)$ which satisfy
$d(c_A)=\mathbf{1}_A$, $d(c_B)=\mathbf{1}_B$ and $d(h)=c_B\circ f + f
\circ c_A$.

By definition of $\mathcal{A}/\mathcal{N}$, the dg functors $F$ and
$G$ give us already the contractions $c_A$ and $c_B$. For $h$ it is
enough to take
$$ h= c_B \circ f \circ c_A\,.$$

This shows us that the dg functors $F$ and $G$ were already homotopic
and so the theorem is proven.
\end{proof}

In fact, in \cite{Drinfeld}, Drinfeld has proved a refined (=$2$-universal) property of his dg quotient construction.

\begin{theorem}[\cite{Drinfeld}]\label{2}
The morphism $Q:\mathcal{A} \rightarrow \mathcal{A}/\mathcal{N}$ in
$\mathsf{Heq}$ induces an equivalence of categories
$$ \mathsf{rep}(\mathcal{A}/\mathcal{N})
\stackrel{Q^{\ast}}{\rightarrow}
\mathsf{rep}_{\mathcal{N}}(\mathcal{A},\mathcal{B})\,,$$
where $\mathsf{rep}_{\mathcal{N}}(\mathcal{A},\mathcal{B})$ denotes
the full subcategory of quasi-functors whose associated functors
$\mathsf{H}^0(\mathcal{A}) \rightarrow \mathsf{H}^0(\mathcal{B})$
annihilate $\mathcal{N}$.
\end{theorem}

\begin{remark}
Notice that since we have a bijection, see
\cite{Toen}
$$ \mathsf{Hom}_{\mathsf{Heq}}(\mathcal{A},\mathcal{B})
\stackrel{\sim}{\leftarrow}
\mathsf{iso}(\mathsf{rep}(\mathcal{A},\mathcal{B}))\,,$$
where $\mathsf{iso}$ denotes the set of isomorphism classes, theorem~\ref{2} implies theorem~\ref{1}.
\end{remark}
We now prove that theorem~\ref{1} also implies theorem~\ref{2}.

\begin{theorem}\label{3}
The morphism $Q:\mathcal{A} \rightarrow \mathcal{A}/\mathcal{N}$ in
$\mathsf{Heq}$ induces an isomorphism
$$ \mathsf{rep}_{dg}(\mathcal{A}/\mathcal{N},\mathcal{B}) \rightarrow
\mathsf{rep}_{dg, \mathcal{N}}(\mathcal{A},\mathcal{B})$$
in $\mathsf{Heq}$.
\end{theorem}
The proof of the theorem is based on the following proposition.

Let $F:\mathcal{C} \rightarrow \mathcal{D}$ be a morphism in
$\mathsf{Heq}$. We denote by $K(F) \stackrel{i}{\hookrightarrow}
\mathcal{C}$ the full dg subcategory of $\mathcal{C}$ whose objects
are those which are sent to contractible objects by $F$. We denote by 
$$\mathsf{Ker}(F) : \mathsf{Heq} \rightarrow \mathsf{Set}$$
the functor which, to a dg category $\mathcal{E}$, associates the set 
$$\{ G \in \mathsf{Hom}_{\mathsf{Heq}}(\mathcal{E},\mathcal{A})
|\, F \circ G =0 \}\,.$$

\begin{proposition}\label{4}
The functor $\mathsf{Ker}(F)$ is corepresented by the dg category $K(F)$.
\end{proposition} 

\begin{proof}
Consider the following diagram in $\dgcat$
$$ K(\tilde{F}) \stackrel{\tilde{i}}{\rightarrow} \mathcal{A}_c
\stackrel{\tilde{F}}{\rightarrow} \mathcal{B}\,,$$
where $\mathcal{A}_c$ is a cofibrant resolution of $\mathcal{A}$ and
$\tilde{F}$ a representative of the morphism $F$. Let $\mathcal{E}$ be
a small dg category and $\mathcal{E}_c$ a cofibrant resolution.

Notice that we have a bijective map
$$ \mathsf{Hom}_{\dgcat}(\mathcal{E}_c, K(\tilde{F}))
\stackrel{\tilde{i}_{\ast}}{\longrightarrow} \{ P \in
\mathsf{Hom}_{\dgcat}(\mathcal{E}_c, \mathcal{A}_c)|\, \tilde{F}\circ
P =0 \}$$
which induces a surjective one
$$ \mathsf{Hom}_{\mathsf{Heq}}(\mathcal{E}_c, K(\tilde{F}))\simeq
\mathsf{Hom}_{\dgcat}(\mathcal{E}, K(\tilde{F}))/htp
\stackrel{\tilde{i}_{\ast}}{\rightarrow} \{ P \in
\mathsf{Hom}_{\dgcat}(\mathcal{E}_c, \mathcal{A}_c)|\, \tilde{F}\circ
P =0 \}/htp \,.$$
We now show that $\tilde{i}_{\ast}$ is also injective.
Let $S$ and $R$ be dg functors from $\mathcal{E}_c$ to $K(\tilde{F})$
such that $\tilde{i}\circ S$ and $\tilde{i} \circ R$ are homotopic. We
have the following commutative diagram
$$
\xymatrix{
 & \mathcal{A}_c \\
\mathcal{E}_c \ar[ur]^{\tilde{i}\circ S} \ar[r]^H
\ar[dr]_{\tilde{i}\circ R} & P(\mathcal{A}_c) \ar[u]_{p_0} \ar[d]^{p_1}\\
 & \mathcal{A}_c 
}
$$
in $\dgcat$.
Now, notice that the dg functor $H$ factors through
$$ P(K(\tilde{F})) \stackrel{P(\tilde{i})}{\hookrightarrow} P(\mathcal{A}_c)$$
and so gives us a homotopy between $S$ and $R$. This proves the proposition.
\end{proof}

Let us now prove theorem~\ref{3}.
\begin{proof}
Apply the functor
$$ \mathsf{rep}_{dg}(?,B) : \mathsf{Heq}^{op} \rightarrow
\mathsf{Heq}$$
to the short exact sequence
$$ \mathcal{N} \stackrel{i}{\rightarrow} \mathcal{A}
\stackrel{Q}{\rightarrow} \mathcal{A}/\mathcal{N}$$
and obtain
$$ \mathsf{rep}_{dg}(\mathcal{A}/\mathcal{N},\mathcal{B})
\stackrel{Q^{\ast}}{\rightarrow}
\mathsf{rep}_{dg}(\mathcal{A},\mathcal{B})
\stackrel{i^{\ast}}{\rightarrow}
\mathsf{rep}_{dg}(\mathcal{N},\mathcal{B})\,.$$
We now show that the dg category
$\mathsf{rep}_{dg}(\mathcal{A}/\mathcal{N},\mathcal{B})$ corepresents the functor $\mathsf{Ker}(i^{\ast})$ in $\mathsf{Heq}$, see proposition~\ref{4}.

Since $i^{\ast} \circ Q^{\ast}$ is zero it is enough to prove the
following: let $\mathcal{D}$ be a small dg category and consider the
diagram
$$
\xymatrix{
\mathsf{rep}_{dg}(\mathcal{A}/ \mathcal{N},\mathcal{B})
\ar[r]^{Q^{\ast}} & \mathsf{rep}_{dg}(\mathcal{A},\mathcal{B})
  \ar[r]^{i^{\ast}} & \mathsf{rep}_{dg}(\mathcal{N},\mathcal{B}) \\
 & \mathcal{D} \ar[u]^F \ar[ur]_0 & \,,
}
$$ 
where $F$ is a morphism in $\mathsf{Heq}$ such that $i^{\ast} \circ F
=0$.
Notice that we have at our disposal a short exact sequence
$$ \mathcal{D} \overset{\mathbb{L}}{\otimes} \mathcal{N} \stackrel{Id
  \overset{\mathbb{L}}{\otimes}i}{\longrightarrow}  \mathcal{D}
\overset{\mathbb{L}}{\otimes} \mathcal{A} \stackrel{Id
  \overset{\mathbb{L}}{\otimes}Q}{\longrightarrow} \mathcal{D}
\overset{\mathbb{L}}{\otimes} \mathcal{A}/\mathcal{N}$$
in $\mathsf{Heq}$. Since $\mathsf{rep}_{dg}(-,-)$ is the internal
Hom-functor in $\mathsf{Heq}$, see \cite{Toen}, we obtain by adjunction
the following diagram
$$
\xymatrix{
\mathcal{D} \overset{\mathbb{L}}{\otimes} \mathcal{N} \ar[dr]_0
\ar[r]^{Id\overset{\mathbb{L}}{\otimes}i} & \mathcal{D}
\overset{\mathbb{L}}{\otimes} \mathcal{A} \ar[d]^{F^{\natural}} \ar[r]^{Id
  \overset{\mathbb{L}}{\otimes}Q} & \mathcal{D}
\overset{\mathbb{L}}{\otimes} \mathcal{A}/\mathcal{N} \ar@{.>}[ld]^G
\\
 & \mathcal{B} & \,,
}
$$
where $F^{\natural}$ is the morphism associated to $F$ by adjunction
and $G$ is the unique morphism induced by $F^{\natural}$, see theorem~\ref{1}. By
adjunction this implies that there is a unique
$$ G_{\natural} \in \mathsf{Hom}_{\mathsf{Heq}}(\mathcal{D},
  \mathsf{rep}_{dg}(\mathcal{A}/ \mathcal{N},\mathcal{B}))\,,$$
such that $Q^{\ast}\circ G_{\natural} =F$. Now by proposition~\ref{4}, we have
an induced isomorphism
$$ \mathsf{rep}_{dg}(\mathcal{A}/\mathcal{N},\mathcal{B})
\stackrel{\sim}{\longrightarrow} \mathsf{rep}_{dg,
  \mathcal{N}}(\mathcal{A},\mathcal{B})$$
in $\mathsf{Heq}$. This proves the theorem.
\end{proof}

\begin{remark}
Since we have an equivalence of categories
$$ \mathsf{H}^0(\mathsf{rep}_{dg}(\mathcal{A}/\mathcal{N},
\mathcal{B})) \stackrel{\sim}{\leftarrow}
\mathsf{rep}(\mathcal{A}/\mathcal{N},\mathcal{B})\,,$$
theorem~\ref{1} implies theorem~\ref{2}.
\end{remark}

\bibliographystyle{alpha}
\bibliography{biblio.bib}

\newcommand{\etalchar}[1]{$^{#1}$}
\begin{thebibliography}{ATJLSS03}

\bibitem[ATJLSS03]{souto}
Leovigildo Alonso~Tarr{\'{\i}}o, Ana Jerem{\'{\i}}as~L{\'o}pez, and
  Mar{\'{\i}}a~Jos{\'e} Souto~Salorio.
\newblock Construction of {$t$}-structures and equivalences of derived
  categories.
\newblock {\em Trans. Amer. Math. Soc.}, 355(6):2523--2543 (electronic), 2003.

\bibitem[BBD82]{Ast100}
A.~A. Be{\u\i}linson, J.~Bernstein, and P.~Deligne.
\newblock Faisceaux pervers.
\newblock In {\em Analysis and topology on singular spaces, I (Luminy, 1981)},
  volume 100 of {\em Ast\'erisque}, pages 5--171. Soc. Math. France, Paris,
  1982.

\bibitem[Ber07]{Bergner}
Julia~E. Bergner.
\newblock A model category structure on the category of simplicial categories.
\newblock {\em Trans. Amer. Math. Soc.}, 359(5):2043--2058 (electronic), 2007.

\bibitem[BF78]{Bos-Fri}
A.~K. Bousfield and E.~M. Friedlander.
\newblock Homotopy theory of {$\Gamma $}-spaces, spectra, and bisimplicial
  sets.
\newblock In {\em Geometric applications of homotopy theory (Proc. Conf.,
  Evanston, Ill., 1977), II}, volume 658 of {\em Lecture Notes in Math.}, pages
  80--130. Springer, Berlin, 1978.

\bibitem[BK90]{Bon-Kap}
A.~I. Bondal and M.~M. Kapranov.
\newblock Framed triangulated categories.
\newblock {\em Mat. Sb.}, 181(5):669--683, 1990.

\bibitem[BLL04]{Bondal}
Alexey~I. Bondal, Michael Larsen, and Valery~A. Lunts.
\newblock Grothendieck ring of pretriangulated categories.
\newblock {\em Int. Math. Res. Not.}, (29):1461--1495, 2004.

\bibitem[BMR{\etalchar{+}}06]{BMRRT}
Aslak~Bakke Buan, Robert Marsh, Markus Reineke, Idun Reiten, and Gordana
  Todorov.
\newblock Tilting theory and cluster combinatorics.
\newblock {\em Adv. Math.}, 204(2):572--618, 2006.

\bibitem[BMRT]{BMRT}
Aslak~Bakke Buan, Robert Marsh, Idun Reiten, and Gordana Todorov.
\newblock Clusters and seeds in acyclic cluster algebras.
\newblock {\em Preprint math.RT/0510359}.

\bibitem[BO]{Bon-Orl}
A.~I. Bondal and D.~Orlov.
\newblock Semiorthogonal decomposition for algebraic varieties.
\newblock {\em Preprint MPIM 95/15, preprint math.AG/9506012}.

\bibitem[Boc]{Raf}
Raf Bocklandt.
\newblock Graded {C}alabi {Y}au algebras of dimension 3.
\newblock {\em Preprint math.RT/060358}.

\bibitem[Bor94]{Bor}
Francis Borceux.
\newblock {\em Handbook of categorical algebra. 2}, volume~51 of {\em
  Encyclopedia of Mathematics and its Applications}.
\newblock Cambridge University Press, Cambridge, 1994.

\bibitem[BS01]{Balmer}
Paul Balmer and Marco Schlichting.
\newblock Idempotent completion of triangulated categories.
\newblock {\em J. Algebra}, 236(2):819--834, 2001.

\bibitem[CCS06]{CCS}
P.~Caldero, F.~Chapoton, and R.~Schiffler.
\newblock Quivers with relations arising from clusters ({$A\sb n$} case).
\newblock {\em Trans. Amer. Math. Soc.}, 358(3):1347--1364 (electronic), 2006.

\bibitem[Cie97]{Set}
Krzysztof Ciesielski.
\newblock {\em Set theory for the working mathematician}, volume~39 of {\em
  London Mathematical Society Student Texts}.
\newblock Cambridge University Press, Cambridge, 1997.

\bibitem[Cisa]{catder}
Denis-Charles Cisinski.
\newblock Cat\'egories d\'erivables.
\newblock {\em Preprint available at www.math.univ-paris13.fr/$\sim$cisinski}.

\bibitem[Cisb]{Letter}
Denis-Charles Cisinski.
\newblock Localisation de {B}ousfield des d\'erivateurs.
\newblock {\em Letter to the author, Paris 3, March 2007}.

\bibitem[Cisc]{Cisinski}
Denis-Charles Cisinski.
\newblock Propri\'et\'es universelles et extensions de {K}an d\'eriv\'ees.
\newblock {\em Preprint available at www.math.univ-paris13.fr/$\sim$cisinski}.

\bibitem[Cis03]{Cisinski1}
Denis-Charles Cisinski.
\newblock Images directes cohomologiques dans les cat\'egories de mod\`eles.
\newblock {\em Ann. Math. Blaise Pascal}, 10(2):195--244, 2003.

\bibitem[CJ95]{Fam}
Aurelio Carboni and Peter Johnstone.
\newblock Connected limits, familial representability and {A}rtin glueing.
\newblock {\em Math. Structures Comput. Sci.}, 5(4):441--459, 1995.
\newblock Fifth Biennial Meeting on Category Theory and Computer Science
  (Amsterdam, 1993).

\bibitem[CK]{CK2}
Philippe Caldero and Bernhard Keller.
\newblock From triangulated categories to cluster algebras {I}{I}.
\newblock {\em Preprint math.RT/0510251, to appear in Annales de l'E.N.S.}

\bibitem[CKN01]{representability}
J.~Daniel Christensen, Bernhard Keller, and Amnon Neeman.
\newblock Failure of {B}rown representability in derived categories.
\newblock {\em Topology}, 40(6):1339--1361, 2001.

\bibitem[CN]{Cis-Nee}
D.-C. Cisinski and A.~Neeman.
\newblock Additivity for derivator ${K}$-theory.
\newblock {\em Preprint available at www.math.univ-paris13.fr/$\sim$cisinski}.

\bibitem[CnT]{Cortinas}
G.~Corti\~nas and A.~Thom.
\newblock Bivariant algebraic ${K}$-theory.
\newblock {\em Preprint math/0603531}.

\bibitem[DK80]{Dwyer}
W.~G. Dwyer and D.~M. Kan.
\newblock Simplicial localizations of categories.
\newblock {\em J. Pure Appl. Algebra}, 17(3):267--284, 1980.

\bibitem[Dri02]{Chicagotalk}
Vladimir Drinfeld.
\newblock D{G} categories.
\newblock {\em Series of talks at the Geometric Langlands Seminar, University
  of Chicago. Notes taken by D.~Ben-Zvi.}, Fall 2002.

\bibitem[Dri04]{Drinfeld}
Vladimir Drinfeld.
\newblock D{G} quotients of {DG} categories.
\newblock {\em J. Algebra}, 272(2):643--691, 2004.

\bibitem[DS04]{Dugger-S}
Daniel Dugger and Brooke Shipley.
\newblock {$K$}-theory and derived equivalences.
\newblock {\em Duke Math. J.}, 124(3):587--617, 2004.

\bibitem[Dug01]{Dugger}
Daniel Dugger.
\newblock Universal homotopy theories.
\newblock {\em Adv. Math.}, 164(1):144--176, 2001.

\bibitem[FZ02]{Cluster1}
Sergey Fomin and Andrei Zelevinsky.
\newblock Cluster algebras. {I}. {F}oundations.
\newblock {\em J. Amer. Math. Soc.}, 15(2):497--529 (electronic), 2002.

\bibitem[Gar]{Garkusha}
G.~Garkusha.
\newblock Homotopy theory of associative rings.
\newblock {\em Preprint math/0608482}.

\bibitem[Gina]{Ginz}
Victor Ginzburg.
\newblock Calabi-{Y}au algebras.
\newblock {\em Preprint math.AG/0612139}.

\bibitem[Ginb]{Ginz1}
Victor Ginzburg.
\newblock Lectures in {N}oncommutative geometry.
\newblock {\em Preprint math.AG/0506603}.

\bibitem[GJ99]{Jardine}
Paul~G. Goerss and John~F. Jardine.
\newblock {\em Simplicial homotopy theory}, volume 174 of {\em Progress in
  Mathematics}.
\newblock Birkh\"auser Verlag, Basel, 1999.

\bibitem[GLS06]{GLS}
Christof Gei{\ss}, Bernard Leclerc, and Jan Schr{\"o}er.
\newblock Rigid modules over preprojective algebras.
\newblock {\em Invent. Math.}, 165(3):589--632, 2006.

\bibitem[Gro90]{Grothendieck}
A.~Grothendieck.
\newblock D\'erivateurs.
\newblock {\em Manuscript. Available at www.math.jussieu.fr/$\sim$maltsin},
  1983--1990.

\bibitem[H\"07]{Tilting}
Happel D. Krause~H. H\"ugel, L.
\newblock {\em Handbook of {T}ilting theory}, volume 332 of {\em L{M}{S}
  {L}ecture {N}otes {S}eries}.
\newblock Cambridge {U}niversity {P}ress, 2007.

\bibitem[Hel97]{Heller}
Alex Heller.
\newblock Stable homotopy theories and stabilization.
\newblock {\em J. Pure Appl. Algebra}, 115(2):113--130, 1997.

\bibitem[Hir03]{Hirschhorn}
Philip~S. Hirschhorn.
\newblock {\em Model categories and their localizations}, volume~99 of {\em
  Mathematical Surveys and Monographs}.
\newblock American Mathematical Society, Providence, RI, 2003.

\bibitem[Hov99]{Hovey}
Mark Hovey.
\newblock {\em Model categories}, volume~63 of {\em Mathematical Surveys and
  Monographs}.
\newblock American Mathematical Society, Providence, RI, 1999.

\bibitem[Hov01]{Spectra}
Mark Hovey.
\newblock Spectra and symmetric spectra in general model categories.
\newblock {\em J. Pure Appl. Algebra}, 165(1):63--127, 2001.

\bibitem[HSS00]{Sspectra}
Mark Hovey, Brooke Shipley, and Jeff Smith.
\newblock Symmetric spectra.
\newblock {\em J. Amer. Math. Soc.}, 13(1):149--208, 2000.

\bibitem[IR]{IR}
Osamu Iyama and Idun Reiten.
\newblock Fomin-{Z}elevinsky mutation and tilting modules over {C}alabi-{Y}au
  algebras.
\newblock {\em Preprint math.RT/0605136}.

\bibitem[IY]{IY}
Osamu Iyama and Y.~Yoshino.
\newblock Mutations in triangulated categories and rigid {C}ohen-{M}acaulay
  modules.
\newblock {\em Preprint math/0607736}.

\bibitem[Iya05]{Iyama33}
Osamu Iyama.
\newblock Maximal orthogonal subcategories of triangulated categories
  satisfying {S}erre duality.
\newblock {\em Oberwolfach report}, 6, 2005.

\bibitem[Iya07]{Iyama32}
Osamu Iyama.
\newblock Higher-dimensional {A}uslander-{R}eiten theory on maximal orthogonal
  subcategories.
\newblock {\em Adv. Math.}, 210(1):22--50, 2007.

\bibitem[JT91]{Joyal}
Andr{\'e} Joyal and Myles Tierney.
\newblock Strong stacks and classifying spaces.
\newblock In {\em Category theory (Como, 1990)}, volume 1488 of {\em Lecture
  Notes in Math.}, pages 213--236. Springer, Berlin, 1991.

\bibitem[Kas87]{Kassel}
Christian Kassel.
\newblock Cyclic homology, comodules, and mixed complexes.
\newblock {\em J. Algebra}, 107(1):195--216, 1987.

\bibitem[Kel91]{KellerUniv}
Bernhard Keller.
\newblock Derived categories and universal problems.
\newblock {\em Comm. Algebra}, 19(3):699--747, 1991.

\bibitem[Kel94]{DerivingDG}
Bernhard Keller.
\newblock Deriving {DG} categories.
\newblock {\em Ann. Sci. \'Ecole Norm. Sup. (4)}, 27(1):63--102, 1994.

\bibitem[Kel98]{cyclicDG}
Bernhard Keller.
\newblock Invariance and localization for cyclic homology of {DG} algebras.
\newblock {\em J. Pure Appl. Algebra}, 123(1-3):223--273, 1998.

\bibitem[Kel99]{cyclichomology}
Bernhard Keller.
\newblock On the cyclic homology of exact categories.
\newblock {\em J. Pure Appl. Algebra}, 136(1):1--56, 1999.

\bibitem[Kel02]{Keller2002}
Bernhard Keller.
\newblock From {G}rothendieck groups to the universal derived invariant.
\newblock {\em Talks at the meeting ``20 years of tilting theory",
  Fraueninsel}, 2002.

\bibitem[Kel05]{orbit}
Bernhard Keller.
\newblock On triangulated orbit categories.
\newblock {\em Doc. Math.}, 10:551--581 (electronic), 2005.

\bibitem[Kel06a]{keller}
Bernhard Keller.
\newblock {$A$}-infinity algebras, modules and functor categories.
\newblock In {\em Trends in representation theory of algebras and related
  topics}, volume 406 of {\em Contemp. Math.}, pages 67--93. Amer. Math. Soc.,
  Providence, RI, 2006.

\bibitem[Kel06b]{ICM}
Bernhard Keller.
\newblock On differential graded categories.
\newblock In {\em International Congress of Mathematicians. Vol. II}, pages
  151--190. Eur. Math. Soc., Z\"urich, 2006.

\bibitem[KjwMP]{Ober}
Bernhard Keller (joint with~{M}arco {P}orta).
\newblock Well generated triangulated categories.
\newblock {\em Oberwolfach report n.8, 2006}.

\bibitem[KL97]{Kelly}
G.~M. Kelly and Stephen Lack.
\newblock On property-like structures.
\newblock {\em Theory Appl. Categ.}, 3:No. 9, 213--250 (electronic), 1997.

\bibitem[Kon]{Konnew}
Maxim Kontsevich.
\newblock Notes on motives in finite characteristic.
\newblock {\em Preprint, math/0702206}.

\bibitem[Kon98]{ENS}
Maxim Kontsevich.
\newblock Triangulated categories and geometry.
\newblock {\em Course at E.N.S. Paris. Notes taken by J.~Bella\"iche,
  J.-F.~Dat, I.~Marin, G.~Racinet and H.~Randriambololona}, 1998.

\bibitem[Kon04]{IHP}
Maxim Kontsevich.
\newblock Topological field theory for triangulated categories.
\newblock {\em Talk at the conference on K-theory and Noncommutative geometry,
  Institut Henri Poincar\'e, Paris}, June 2004.

\bibitem[KR06]{preprint}
Bernhard Keller and Idun Reiten.
\newblock Acyclic {C}alabi-{Y}au categories are cluster categories.
\newblock {\em Oberwolfach report}, (23):27--29, 2006.

\bibitem[KR07]{cluster}
Bernhard Keller and Idun Reiten.
\newblock Cluster-tilted algebras are {G}orenstein and stably {C}alabi-{Y}au.
\newblock {\em Adv. Math.}, 211(1):123--151, 2007.

\bibitem[Kra01]{Krause}
Henning Krause.
\newblock On {N}eeman's well generated triangulated categories.
\newblock {\em Doc. Math.}, 6:121--126 (electronic), 2001.

\bibitem[KS]{KonY}
Maxim Kontsevich and Y.~Soibelmann.
\newblock Notes on {A}-infinity algebras, {A}-infinity categories and
  non-commutative geometry {I}.
\newblock {\em Preprint math/0606241}.

\bibitem[KV87]{KellerVos}
Bernhard Keller and Dieter Vossieck.
\newblock Sous les cat\'egories d\'eriv\'ees.
\newblock {\em C. R. Acad. Sci. Paris S\'er. I Math.}, 305(6):225--228, 1987.

\bibitem[KV88]{aisles}
B.~Keller and D.~Vossieck.
\newblock Aisles in derived categories.
\newblock {\em Bull. Soc. Math. Belg. S\'er. A}, 40(2):239--253, 1988.

\bibitem[Low]{lowen}
Wendy Lowen.
\newblock Hochschild cohomology, the characteristic morphism and derived
  deformations.
\newblock {\em Preprint, math.KT/0707.2602}.

\bibitem[Lur]{Lurie}
Jacob Lurie.
\newblock Stable $\infty$-categories.
\newblock {\em Preprint, math/0608228}.

\bibitem[LVdB06]{lowen-Vdb}
Wendy Lowen and Michel Van~den Bergh.
\newblock Deformation theory of abelian categories.
\newblock {\em Trans. Amer. Math. Soc.}, 358(12):5441--5483 (electronic), 2006.

\bibitem[Mal01]{Malt}
Georges Maltsiniotis.
\newblock Introduction \`a la th\'eorie des d\'erivateurs (d'apr\`es
  {G}rothendieck).
\newblock {\em preprint, available at the author's homepage}, 2001.

\bibitem[McC94]{MacCarthy}
Randy McCarthy.
\newblock The cyclic homology of an exact category.
\newblock {\em J. Pure Appl. Algebra}, 93(3):251--296, 1994.

\bibitem[ML98]{Macl}
Saunders Mac~Lane.
\newblock {\em Categories for the working mathematician}, volume~5 of {\em
  Graduate Texts in Mathematics}.
\newblock Springer-Verlag, New York, second edition, 1998.

\bibitem[MN06]{Nest}
Ralf Meyer and Ryszard Nest.
\newblock The {B}aum-{C}onnes conjecture via localisation of categories.
\newblock {\em Topology}, 45(2):209--259, 2006.

\bibitem[Nee01a]{DocNeeman}
Amnon Neeman.
\newblock On the derived category of sheaves on a manifold.
\newblock {\em Doc. Math.}, 6:483--488 (electronic), 2001.

\bibitem[Nee01b]{Neeman}
Amnon Neeman.
\newblock {\em Triangulated categories}, volume 148 of {\em Annals of
  Mathematics Studies}.
\newblock Princeton University Press, Princeton, NJ, 2001.

\bibitem[Pal]{Palu}
Yann Palu.
\newblock Cat\'egories triangul\'ees et alg\`ebres amass\'ees.
\newblock {\em Ph.{D}. thesis in preparation}.

\bibitem[PG64]{Pop-Gab}
Nicolae Popesco and Pierre Gabriel.
\newblock Caract\'erisation des cat\'egories ab\'eliennes avec g\'en\'erateurs
  et limites inductives exactes.
\newblock {\em C. R. Acad. Sci. Paris}, 258:4188--4190, 1964.

\bibitem[Por]{Porta}
Marco Porta.
\newblock {\em Ph.{D}. thesis in preparation}.

\bibitem[Qui67]{Quillen}
Daniel~G. Quillen.
\newblock {\em Homotopical algebra}.
\newblock Lecture Notes in Mathematics, No. 43. Springer-Verlag, Berlin, 1967.

\bibitem[Rez]{Rezk}
Charles Rezk.
\newblock A note on a certain model category structure on the category of
  categories.
\newblock {\em available at www.math.uiuc.edu/$\sim$rezk}.

\bibitem[Ric89]{Rickard}
Jeremy Rickard.
\newblock Morita theory for derived categories.
\newblock {\em J. London Math. Soc. (2)}, 39(3):436--456, 1989.

\bibitem[Ric91]{Rickard1}
Jeremy Rickard.
\newblock Derived equivalences as derived functors.
\newblock {\em J. London Math. Soc. (2)}, 43(1):37--48, 1991.

\bibitem[RZ03]{Rou-Zim}
Rapha{\"e}l Rouquier and Alexander Zimmermann.
\newblock Picard groups for derived module categories.
\newblock {\em Proc. London Math. Soc. (3)}, 87(1):197--225, 2003.

\bibitem[Sch97]{Schwede}
Stefan Schwede.
\newblock Spectra in model categories and applications to the algebraic
  cotangent complex.
\newblock {\em J. Pure Appl. Algebra}, 120(1):77--104, 1997.

\bibitem[Sch06]{Marco}
Marco Schlichting.
\newblock Negative {$K$}-theory of derived categories.
\newblock {\em Math. Z.}, 253(1):97--134, 2006.

\bibitem[Taba]{univ}
Gon{\c{c}}alo Tabuada.
\newblock Higher ${K}$-theory via universal invariants.
\newblock {\em Preprint arXiv:0706.2420, submitted}.

\bibitem[Tabb]{Triangcat}
Gon{\c{c}}alo Tabuada.
\newblock Homotopy theory of well-generated algebraic triangulated categories.
\newblock {\em Preprint math.KT/0703172, to appear in Journal of $K$-theory}.

\bibitem[Tabc]{dgquot}
Gon{\c{c}}alo Tabuada.
\newblock The ${Q}$-model for the {M}orita homotopy theory of {DG} categories.
\newblock {\em Preprint math.KT/0701205, submitted}.

\bibitem[Tab05a]{IMRN}
Gon{\c{c}}alo Tabuada.
\newblock Invariants additifs de {DG}-cat\'egories.
\newblock {\em Int. Math. Res. Not.}, (53):3309--3339, 2005.

\bibitem[Tab05b]{cras}
Gon{\c{c}}alo Tabuada.
\newblock Une structure de cat\'egorie de mod\`eles de {Q}uillen sur la
  cat\'egorie des dg-cat\'egories.
\newblock {\em C. R. Math. Acad. Sci. Paris}, 340(1):15--19, 2005.

\bibitem[Tab06]{addendum}
Gon{\c{c}}alo Tabuada.
\newblock Addendum to: ``{A}dditive invariants of {DG}-categories'' ({F}rench)
  [{I}nt. {M}ath. {R}es. {N}ot. {\bf 2005}, no. 53, 3309--3339; 2196100].
\newblock {\em Int. Math. Res. Not.}, pages Art. ID 75853, 3, 2006.

\bibitem[Tab07]{Documenta}
Gon{\c{c}}alo Tabuada.
\newblock On the structure of {C}alabi-{Y}au categories with a cluster tilting
  subcategory.
\newblock {\em Doc. Math.}, 12:193--213 (electronic), 2007.

\bibitem[To{\"e}07]{Toen}
Bertrand To{\"e}n.
\newblock The homotopy theory of {$dg$}-categories and derived {M}orita theory.
\newblock {\em Invent. Math.}, 167(3):615--667, 2007.

\bibitem[TT90]{Thomason}
R.~W. Thomason and Thomas Trobaugh.
\newblock Higher algebraic {$K$}-theory of schemes and of derived categories.
\newblock In {\em The Grothendieck Festschrift, Vol.\ III}, volume~88 of {\em
  Progr. Math.}, pages 247--435. Birkh\"auser Boston, Boston, MA, 1990.

\bibitem[TV]{Toen-Vaq}
Bertrand To{\"e}n and Michel Vaqui\'e.
\newblock Moduli of objects in dg-categories.
\newblock {\em Preprint math/0503418, to appear in Annales de l'E.N.S.}

\bibitem[TV05]{HAG}
Bertrand To{\"e}n and Gabriele Vezzosi.
\newblock Homotopical algebraic geometry. {I}. {T}opos theory.
\newblock {\em Adv. Math.}, 193(2):257--372, 2005.

\bibitem[VdB02]{VdB}
M.~Van~den Bergh.
\newblock Non-commutative crepant resolutions.
\newblock {\em The {L}egacy of {N}iels {H}endrik {A}bel, Springer}, pages
  749--770, 2002.

\bibitem[Ver96]{Verdier}
Jean-Louis Verdier.
\newblock Des cat\'egories d\'eriv\'ees des cat\'egories ab\'eliennes.
\newblock {\em Ast\'erisque}, (239):xii+253 pp. (1997), 1996.
\newblock With a preface by Luc Illusie, Edited and with a note by Georges
  Maltsiniotis.

\bibitem[Wal85]{Waldhausen}
Friedhelm Waldhausen.
\newblock Algebraic {$K$}-theory of spaces.
\newblock In {\em Algebraic and geometric topology (New Brunswick, N.J.,
  1983)}, volume 1126 of {\em Lecture Notes in Math.}, pages 318--419.
  Springer, Berlin, 1985.

\end{thebibliography}

\end{document}